\documentclass[11pt,reqno]{amsart}

\usepackage[leqno]{amsmath}
\usepackage{amsthm}
\usepackage{amsfonts, tikz-cd}
\usepackage{amssymb}
\usepackage{eucal}
\usepackage[all]{xy}
\usepackage{mathtools}
\usepackage{stmaryrd}
\usepackage{enumitem}
\usepackage{upgreek}

\setenumerate{itemsep=2pt,parsep=2pt,before={\parskip=2pt}}

\usepackage[colorlinks=true,hyperindex, linkcolor=magenta, pagebackref=false, citecolor=cyan,pdfpagelabels]{hyperref}

\usepackage{xspace}
\usepackage{epsfig,epic,eepic,latexsym,color}
\usepackage[all]{xy}
\usepackage{comment} 
\usepackage[autostyle]{csquotes}
\numberwithin{equation}{section}

\usepackage[backend=biber, style=alphabetic, maxalphanames=5, maxnames=99]{biblatex}
\usepackage[english]{babel}
\usepackage{latexsym}

\newcommand{\nc}{\newcommand}
\nc{\rnc}{\renewcommand}
\rnc{\P}{\mathbf P}
\nc{\R}{\mathbf R}
\rnc{\rm}{\mathrm}
\rnc{\bf}{\mathbf}
\nc{\sm}{\text{sm}}

\nc{\C}{\mathbf C}
\nc{\Q}{\mathbf Q}
\nc{\Z}{\mathbf Z}
\nc{\N}{\mathbf N}
\nc{\A}{\mathbf A}

\nc{\an}{\operatorname{an}}
\nc{\perfd}{\operatorname{perfd}}
\nc{\perf}{\operatorname{new, perf}}
\nc{\diam}{\diamondsuit}

\nc{\cal}{\mathcal}

\nc{\htt}{\operatorname{ht}}
\nc{\Nm}{\operatorname{Nm}}
\nc{\Ker}{\operatorname{Ker}}
\nc{\mmod}{\operatorname{mod}}
\nc{\End}{\operatorname{End}}
\nc{\Tor}{\operatorname{Tor}}
\nc{\coker}{\operatorname{coker}}
\nc{\Tr}{\operatorname{Tr}}
\nc{\Cat}{\cal{C}\rm{at}}
\nc{\hocolim}{\operatorname{hocolim}}

\nc{\Aut}{\operatorname{Aut}}
\nc{\cont}{\text{cont}}
\nc{\sep}{\text{sep}}
\nc{\Hom}{\mathrm{Hom}}
\nc{\Gal}{\mathrm{Gal}}
\nc{\Spec}{\text{Spec}\,}
\nc{\RZ}{\operatorname{RZ}}
\rnc{\sp}{\operatorname{sp}}

\rnc{\t}{\tau}
\nc{\mm}{\pmb{\mu}}
\rnc{\a}{\alpha}
\nc{\n}{\mathfrak n}
\nc{\m}{\mathfrak m}
\nc{\mfs}{\mathfrak s}
\nc{\pp}{\mathfrak p}
\nc{\e}{\varepsilon}
\nc{\dd}{\delta}
\nc{\Imm}{\operatorname{Im}}

\nc{\p}{\mathfrak p}
\nc{\q}{\mathfrak q}

\nc{\Sym}{\operatorname{Sym}}
\nc{\codim}{\operatorname{codim}}
\nc{\rk}{\operatorname{rk}}
\nc{\GL}{\operatorname{GL}}
\nc{\SL}{\operatorname{SL}}
\nc{\Lie}{\operatorname{Lie}}
\nc{\Ind}{\operatorname{Ind}}
\nc{\Div}{\underline{Div}}
\nc{\Pic}{\mathbf{Pic}}
\nc{\uPic}{\underline{ \mathbf{Pic}}}
\nc{\rH}{\mathrm{H}}
\nc{\Spf}{\operatorname{Spf}}
\nc{\Frac}{\operatorname{Frac}}
\nc{\colim}{\operatorname{colim}}
\nc{\Spa}{\operatorname{Spa}}

\rnc{\an}{\operatorname{an}}
\nc{\et}{\text{\'et}}
\nc{\Ett}{\text{\'Et}}
\rnc{\Tr}{\operatorname{Tr}}

\rnc{\et}{\text{\'et}}
\nc{\proet}{\text{pro\'et}}

\nc{\xr}{\xrightarrow}

\nc{\eps}{\epsilon}
\nc{\ov}{\overline}
\nc{\ud}{\underline}
\nc{\wdh}{\widehat}

\nc{\F}{\mathcal F}
\nc{\G}{\mathcal G}
\nc{\E}{\mathcal E}
\nc{\K}{\mathcal K}
\nc{\I}{\mathcal I}
\nc{\sQ}{\mathcal Q}

\nc{\X}{\mathfrak X}
\nc{\Y}{\mathfrak Y}
\nc{\T}{\mathfrak T}
\nc{\mf}{\mathfrak f}
\nc{\mg}{\mathfrak g}
\nc{\M}{\mathfrak M}

\nc{\LL}{\mathcal{L}}
\rnc{\S}{\mathcal S}
\nc{\sU}{\mathfrak U}
\nc{\V}{\mathfrak V}
\nc{\ra}{\rangle}

\nc{\os}{\overset}

\rnc{\O}{\mathcal O}
\nc{\J}{\mathcal J}
\theoremstyle{definition}

\newtheorem{thm}{Theorem}[subsection]
\newtheorem{lemma}[thm]{Lemma}
\newtheorem{defn}[thm]{Definition}
\newtheorem{lemma-defn}[thm]{Lemma-Definition}
\newtheorem{warning}[thm]{Warning}

\newtheorem{prop}[thm]{Proposition}

\newtheorem{rmk}[thm]{Remark}

\newtheorem{examples}[thm]{Examples}

\newtheorem{notation}[thm]{Notation}

\newtheorem{thm-def}[thm]{Theorem/Definition}
\newtheorem{cor}[thm]{Corollary}

\newtheorem{construction}[thm]{Construction}

\newtheorem{thm1}{Theorem}[section]
\newtheorem{lemma1}[thm1]{Lemma}
\newtheorem{defn1}[thm1]{Definition}
\newtheorem{lemma-defn1}[thm1]{Lemma-Definition}

\newtheorem{rmk1}[thm1]{Remark}

\newtheorem{cor1}[thm1]{Corollary}

\begin{document}

\title{Mod-$p$ Poincar\'{e} Duality in $p$-adic Analytic Geometry}
\author{Bogdan Zavyalov}
\maketitle

\begin{abstract} We show Poincar\'e Duality for $\bf{F}_p$-\'etale cohomology of a smooth proper rigid-analytic space over a non-archimedean field $K$ of mixed characteristic $(0, p)$. It positively answers the question raised by P. Scholze in \cite{Sch1}. We prove duality via constructing Faltings' trace map relating Poincar\'e Duality on the generic fiber to (almost) Grothendieck Duality on the mod-$p$ fiber of a formal model. We also formally deduce Poincar\'e Duality for $\Z/p^n\Z$, $\Z_p$, and $\Q_p$-coefficients. 
\end{abstract}
\tableofcontents

\section{Introduction}
\subsection{Historical overview}

An important result in the classical theory of singular cohomology of complex manifolds is the Poincar\'e Duality Theorem that says that, for a compact complex manifold $X$ of pure dimension $d$, there is a trace map 
\[
t_X \colon \rm{H}^{2d}_{\rm{sing}}(X, K) \to K
\]
satisfying:

\begin{thm}\label{thm:intro-poincare-classical}(Classical Poincar\'e Duality) Let $X$ be a compact complex manifold $X$ of pure dimension $d$. Then the pairing 
\[
\rm{H}^{i}_{\rm{sing}}(X, K) \otimes_K \rm{H}^{2d-i}_{\rm{sing}}(X, K) \xr{-\cup -} \rm{H}^{2d}_{\rm{sing}}(X, K) \xr{t_X} K
\]
is perfect for any $i\geq 0$. 
\end{thm}

It is a meaningful question to ask if Poincar\'e Duality holds for other cohomology theories and other ``geometries''. In fact, Poincar\'e Duality is one of the axioms of a Weil cohomology theory (see \cite[\textsection 1.2]{Kleiman}). So it is expected that interesting cohomology theories should satisfy (a version) of Poincar\'e Duality. We now discuss examples of other cohomology theories satisfying Poincar\'e Duality and then formulate the main result of this paper. \smallskip

For a fixed field  $k$ and a prime number $\ell$, A.\,Grothendieck has defined (geometric) \'etale cohomology groups of any $k$-variety. More precisely, there is a functorial assignment 
\[
X \mapsto \rm{H}^i\left(X_{\ov{k}}, \bf{F}_\ell\right)
\]
that sends a $k$-variety $X$ to its $i$-th geometric \'etale cohomology group with coefficients in $\bf{F}_\ell$. It turns out that the theory of \'etale cohomology groups behaves especially well if $\ell$ is coprime to the characteristic of $k$. In this situation, one can show analogues of many results in the theory of singular cohomology of complex-analytic spaces. In particular, it turns out that the groups $\rm{H}^i(X_{\ov{k}}, \bf{F}_\ell)$ are finite for qcqs $X$, and the following version of Poincar\'e Duality holds:

\begin{thm}\label{thm:intro-poincare-varieties}(\cite[Exp.\,XVIII, Th\'eor\`eme 3.2.5]{SGA4_3}) Let $X$ be a smooth proper $k$-variety of pure dimension $d$, and $\ell$ a prime number coprime to the characteristic of $k$. Then there is a Galois-equivariant trace map\footnote{As always, $\bf{F}_\ell(d)$ is short for the \'etale sheaf $\mu_\ell^{\otimes d}$. In particular, the groups $\rm{H}^{i}_{\et}(X_{\ov{k}}, \bf{F}_\ell)$ and $\rm{H}^{i}_{\et}(X_{\ov{k}}, \bf{F}_\ell(d))$ are non-canonically isomorphic as abstract groups.} 
\[
t_X \colon \rm{H}^{2d}_{\et}\left(X_{\ov{k}}, \bf{F}_\ell(d)\right) \to \bf{F}_\ell
\]
such that the induced pairing 
\[
\rm{H}^{i}_{\et}\left(X_{\ov{k}}, \bf{F}_\ell\right) \otimes_{\bf{F}_\ell} \rm{H}^{2d-i}_{\et}\left(X_{\ov{k}}, \bf{F}_\ell(d)\right) \xr{-\cup -} \rm{H}^{2d}_{\et}\left(X_{\ov{k}}, \bf{F}_\ell(d)\right) \xr{t_X} \bf{F}_\ell
\]
is perfect for any  $i\geq 0$. 
\end{thm}

Now we go into the realm of rigid-analytic spaces. These spaces are non-archimedean analogues of complex-analytic varieties and, for the purpose of this introduction, it suffices to think about them as some kind of analytic spaces over a non-archimedean field $K$. The foundations of such spaces were extensively (separately) studied by R.\,Huber and V.\,Berkovich in \cite{H3} and \cite{Ber}. In particular, they were able to define the (geometric) \'etale cohomology groups
\[
\rm{H}^i_\et\left(X_{\wdh{\ov{K}}}, \bf{F}_\ell\right)
\]
for any rigid-analytic $K$-variety $X$ and a prime number $\ell$. One of the major results about this theory of \'etale cohomology groups is the following version of Poincar\'e Duality:

\begin{thm}\label{thm:intro-poincare-Huber-Berkovich}(\cite[Theorem 7.5.3]{H3}, \cite[Theorem 7.3.1]{Ber}) Let $X$ be a smooth proper rigid-analytic space of pure dimension $d$ over a non-archimedean field $K$, and $\ell$ a prime number invertible in the ring of integers $\O_K$. Then there is a Galois-equivariant trace map 
\[
t_X \colon \rm{H}^{2d}_{\et}\left(X_{\widehat{\ov{K}}}, \bf{F}_\ell(d)\right) \to \bf{F}_\ell
\]
such that the induced pairing 
\[
\rm{H}^{i}_{\et}\left(X_{\widehat{\ov{K}}}, \bf{F}_\ell\right) \otimes_{\bf{F}_\ell} \rm{H}^{2d-i}_{\et}\left(X_{\widehat{\ov{K}}}, \bf{F}_\ell(d)\right) \xr{-\cup -} \rm{H}^{2d}_{\et}\left(X_{\widehat{\ov{K}}}, \bf{F}_\ell(d)\right) \xr{t_X} \bf{F}_\ell
\]
is perfect for any $i\geq 0$. 
\end{thm}

Let us point out one little but very important difference between Theorem~\ref{thm:intro-poincare-varieties} and Theorem~\ref{thm:intro-poincare-Huber-Berkovich}. The former result allows the prime number $\ell$ to be any prime number different to the characteristic of $k$, while the latter result puts the stronger assumption on $\ell$ of being invertible in $\O_K$. For example, if $K$ is a non-archimedean field of mixed characteristic $(0, p)$, Theorem~\ref{thm:intro-poincare-varieties} holds for {\it all} prime numbers $\ell$, but Theorem~\ref{thm:intro-poincare-Huber-Berkovich} needs to assume that $\ell\neq p$. \smallskip

So Theorem~\ref{thm:intro-poincare-Huber-Berkovich} still leaves open the question whether Poincar\'e Duality holds for $\bf{F}_p$-\'etale cohomology of rigid-analytic spaces over non-archimedean fields of mixed characteristic $(0, p)$. The question of proving this version of Poincar\'e Duality was raised by P.\,Scholze in \cite[Introduction, page 2]{Sch1}. We give a positive answer to this question by proving the following theorem: 

\begin{thm}\label{thm:intro-poincare-p}(Theorem~\ref{thm:poincare-duality-mod-p}) Let $X$ be a smooth proper rigid-analytic space of pure dimension $d$ over a non-archimedean field $K$ of mixed characteristic $(0, p)$. Then there is a Galois-equivariant trace map 
\[
t_X \colon \rm{H}^{2d}_{\et}\left(X_{\widehat{\ov{K}}}, \bf{F}_p(d)\right) \to \bf{F}_p
\]
such that the induced pairing 
\[
\rm{H}^{i}_{\et}\left(X_{\widehat{\ov{K}}}, \bf{F}_p\right) \otimes_{\bf{F}_p} \rm{H}^{2d-i}_{\et}\left(X_{\widehat{\ov{K}}}, \bf{F}_p(d)\right) \xr{-\cup -} \rm{H}^{2d}_{\et}\left(X_{\widehat{\ov{K}}}, \bf{F}_p(d)\right) \xr{t_X} \bf{F}_p
\]
is perfect for any $i\geq 0$. 
\end{thm}

\begin{rmk} Theorem~\ref{thm:intro-poincare-p} settles the last remaining open case of Poincar\'e Duality for smooth proper rigid-analytic spaces. Namely, Theorem~\ref{thm:intro-poincare-Huber-Berkovich} and Theorem~\ref{thm:intro-poincare-p} cover all primes $\ell$ that are coprime to the characteristic of $K$. In contrast, for a proper rigid-analytic $K$-space $X$ of dimension $>0$ and a prime $\ell=\rm{char}\,K$, one can see $\rm{H}^{2d}_{\et}(X_{\wdh{\ov{K}}}, \bf{F}_\ell)=0$ showing that there is no chance to have duality in this case.
\end{rmk}

We also give a version of Theorem~\ref{thm:intro-poincare-p} for $\bf{F}_p$-local systems (see Theorem~\ref{thm:local-systems}), and for $\Z/p^n\Z$, $\Z_p$, and $\Q_p$-coefficients (see Theorems~\ref{thm:Poincare-duality-mod-pn}, \ref{thm:poincare-duality-integral}, and \ref{thm:poincare-duality-rational}). \medskip

A proof of Theorem~\ref{thm:intro-poincare-p} was also announced by O.\,Gabber in 2015. But a written account of his proof has never appeared since then. Our proof was motivated by Gabber's  \href{https://www.youtube.com/watch?v=uREGEdFJPsk}{Youtube lecture} on the subject. In particular, both proofs use the perspective of formal models to eventually reduce Poincar\'e Duality on the generic fiber to Grothendieck Duality on the special fiber. However, it seems that our proof is quite different in many aspects. We use a refined version of Temkin's local uniformization theorem from \cite{Z1} and Berkovich's construction of the trace map from \cite{Ber} to simplify the proof. We also use some techniques\footnote{These techniques are not explicitly used in this particular paper. However, they play an important role in our companion paper \cite{Z3}, especially in the proof of \cite[Theorem 6.13.5]{Z3} that plays a crucial role for our construction of Faltings' trace map.} (e.g. the theory of diamonds \cite{Sch2} and a generalization of the almost purity theorem from \cite{BS3}) that were not available at the time.

\subsection{Main steps of our proof} In this section, we explain the plan of the proof of Theorem~\ref{thm:intro-poincare-p}. Roughly, the main idea is to reduce Poincar\'e Duality for $\bf{F}_p$-coefficients to almost duality for $\O_X^+/p$-coefficients by using the Primitive Comparison Theorem \cite[Theorem 5.1]{Sch1} and then, after choosing a formal model $\X$, we use ``$p$-adic nearby cycles'' to reduce duality to (the almost version of) Grothendieck duality on the mod-$p$ fiber $\X_0$. \smallskip

More precisely, we choose a formal model $\X$ of $X$ with the mod-$p$ fiber $\X_0$. Then we use the natural morphism of ringed sites $\nu\colon (X_\proet, \O_X^+/p) \to (\X_0, \O_{\X_0})$ to rewrite the complex $\bf{R}\Gamma(X_\proet, \O_X^+/p)$ in the following form:
\[
\bf{R}\Gamma(X_\proet, \O_X^+/p) \simeq \bf{R}\Gamma\left(\X_0, \bf{R}\nu_*\left(\O_X^+/p\right)\right).
\]
The advantage of this formula is that it allows to decompose the question of proving almost duality for $\bf{R}\Gamma(X_\proet, \O_X^+/p)$ into proving almost duality for the ``nearby cycles'' $\bf{R}\nu_*\left(\O_X^+/p\right)$ and (the almost version of) Grothendieck duality on $\X_0$.\smallskip

Before we discuss the main steps of this argument in more detail, we wish to emphasize that this proof seriously depends on our other paper \cite{Z3} where we developed the theory of almost coherent sheaves on nice (formal) schemes. In particular, \cite[\textsection{5.5} and Theorem 6.13.6]{Z3} play the crucial role in our proof. \smallskip

Now we explain the main steps of our proof in more detail\footnote{We warn the reader that the step order below does not coincide with the order of exposition in this paper.}. In what follows, we always consider cohomology of the sheaf $\O_X^+/p$ with respect to the pro-\'etale topology on $X$.

\begin{enumerate}\itemsep0.5em

\item(Section~\ref{section:berkovich-trace}) We translate the construction of the trace map 
\[
t_X\colon \rm{H}^{2d}_{\et}\left(X_{\widehat{\ov{K}}}, \bf{F}_p(d)\right) \to \bf{F}_p
\]
from \cite[\textsection 7.2]{Ber} to the world of rigid-analytic spaces. This trace map is Galois-equivariant by design, so it allows us to replace $K$ with $C\coloneqq \wdh{\ov{K}}$ and $X$ with $X_C$ to assume that $K=C$ is algebraically closed. Then a small argument also allows us to assume that $X$ is connected. 

\item(Section~\ref{section:poincare-duality-mod-p}) Using the Primitive Comparison Theorem
\[
\rm{H}^{i}_{\et}\left(X, \bf{F}_p(d)\right)\otimes_{\bf{F}_p} \O_C/p \simeq^a \rm{H}^i\left(X, (\O_X^+/p)(d)\right),
\]
we reduce Poincar\'e Duality to showing that the natural pairing
\begin{equation}\label{diagram:intro-pairing}
\begin{tikzcd}
\rm{H}^{i}\left(X, \O_X^{+, a}/p\right) \otimes_{\O_C/p} \rm{H}^{2d-i}\left(X, (\O_X^{+, a}/p)(d)\right) \arrow{r}{-\cup -} & \rm{H}^{2d}\left(X, (\O_X^{+, a}/p)(d)\right)  \arrow{d}{\wr}\\
 \O_C^a/p & \arrow{l}{t_X\otimes \O_C/p} \rm{H}^{2d}_{\et}\left(X, \bf{F}_p(d)\right)\otimes_{\bf{F}_p} \O_C^a/p
\end{tikzcd}
\end{equation}
is almost perfect for all $i$. Using surjectivity of the trace map $t_X$ already established by Berkovich, we can reduce even further to showing that the pairing
\[
\rm{H}^{i}\left(X, \O_X^{+, a}/p\right) \otimes_{\O_C/p} \rm{H}^{2d-i}\left(X, (\O_X^{+, a}/p)(d)\right) \xr{-\cup -}  \rm{H}^{2d}\left(X, (\O_X^{+, a}/p)(d)\right) \xr{\rm{Tr}_X} \O_C^a/p
\]
is almost perfect for {\it any} choice of a trace morphism $\rm{Tr}_X \colon \rm{H}^{2d}\left(X, (\O_X^{+, a}/p)(d)\right) \to \O_C^a/p$. From now on, we never need to use $t_X$ anymore.

\item(Derived reformulation) So far, we have reduced the question to constructing a trace morphism
\begin{equation}\label{diagram:intro-trace}
\rm{Tr}_X\colon \rm{H}^{2d}\left(X, (\O_X^{+, a}/p)(d)\right) \to \O_C^a/p
\end{equation}
that induces an almost perfect pairing between $\rm{H}^{i}(X, \O_X^{+, a}/p)$ and $\rm{H}^{2d-i}(X, (\O_X^{+, a}/p)(d))$. \smallskip 

For technical reasons, it will be convenient to reformulate the question in the derived world. For this, we recall that \cite[Theorem 6.13.5]{Z3} guarantees that $\bf{R}\Gamma(X, \O_X^+/p)$ is almost concentrated in degrees $[0, 2d]$, so the map~(\ref{diagram:intro-trace}) is equivalent (up to a twist) to a map
\[
\rm{Tr}_X \colon \bf{R}\Gamma(X, \O_X^{+, a}/p) \to (\O_C^a/p)(-d)[-2d].
\]
Moreover, almost perfectness of the pairing induced by (\ref{diagram:intro-trace}) is equivalent to almost perfectness (in the derived sense) of the pairing
\[
\bf{R}\Gamma(X, \O_X^{+, a}/p) \otimes^L_{\O_C/p} \bf{R}\Gamma(X, \O_X^{+, a}/p) \xr{-\cup -} \bf{R}\Gamma(X, \O_X^{+, a}/p) \xr{\rm{Tr}_X} (\O_C^a/p)(-d)[-2d].
\]

Before we move further, we want to emphasize that almost duality for $\bf{R}\Gamma(X, \O_X^{+, a}/p)$ boils down to two essentially independent problems of constructing a trace morphism and showing that the induced pairing is almost perfect. 

\item(Section~\ref{section:poor-coherent-duality}, digression) As a preliminary work to construct $\rm{Tr}_X$, we have to develop part of the theory of dualizing sheaves and complexes on admissible formal $\O_C$-schemes. We do not develop a general theory in this paper, and only prove some results that we need. Here, we mention some results and constructions that will play an important role in the construction of $\rm{Tr}_X$. \smallskip 

For any separated admissible formal $\O_C$-scheme $\X$ with mod-$p$ fiber $\X_0$, we define its {\it dualizing complex} $\omega^\bullet_\X$ (see Definition~\ref{defn:dualizing-complex}). This complex lives in $\bf{D}_{coh}^{[-\dim \X_0, 0]}(\X)$ and its bottom cohomology sheaf $\omega_\X \coloneqq \cal{H}^{-\dim \X_0}\left(\omega^\bullet_\X\right)$ is called the {\it dualizing module}. \smallskip

The mod-$p$ fiber admits its own (relative) dualizing complex $\omega^\bullet_{\X_0}\coloneqq \mf_0^!\left(\O_C/p\right)$, where $\mf_0\colon \X_0 \to \Spec \O_C/p$ is the structure morphism and $\mf_0^!$ is the twisted inverse image from Grothendieck duality. It turns out that $\omega^\bullet_\X$ is compatible with $\omega^\bullet_{\X_0}$ on $\X_0$ via the formula
\[
\omega^\bullet_\X\otimes^L_{\O_{\X}} \O_{\X_0} \cong \omega^\bullet_{\X_0}.
\]

Another important result we show is a version of Hartogs' extension for $\omega_\X$. More precisely, for a separated admissible formal $\O_C$-scheme with smooth generic fiber $\X_C$ and reduced special fiber $\ov{\X}$, the natural morphism
\begin{equation}\label{equation:extend-from-codim-2}
\omega_\X \to j_{*}(\omega_\X|_{\X^\sm}) \cap ({\omega_\X})_C
\end{equation}
is an isomorphism, where $j\colon \X^\sm \to \X$ is the open immersion of the smooth locus of $\X$ into $\X$. %If $X$ is instead a sufficiently nice scheme over a discretely valued ring $\O_K$, the above property of the dualizing sheaf is well-known, and roughly follows from the fact that $\omega_X$ can be recovered from the codimension-$1$ points. 

\item(Section~\ref{section:construction-faltings-trace} and Section~\ref{section:global-almost-duality}) To construct a trace morphism $\rm{Tr}_X$, we choose an admissible formal $\O_C$-model $\X$ of $X$ with the structure morphism $\mathfrak{f} \colon \X \to \Spf \O_C$, and denote its mod-$p$ fiber by 
\[
\mathfrak{f}_0\colon \X_0 \to \Spec \O_C/p.
\]
The $\O_C/p$-scheme $\X_0$ is automatically proper by Lemma~\ref{lemma:proper-adic-formal}, and we can assume that the special fiber $\ov{\X}$ is reduced by the Reduced Fiber Theorem. The almost version of Grothendieck duality (see \cite[Theorem 5.5.5]{Z3}) ensures that the functor
\[
\bf{R}\mf_{0, *} \colon \bf{D}^+_{aqc}(\X_0)^a \to \bf{D}^+_{aqc}(\O_C/p)^a
\]
admits a right adjoint $\mf_0^! \colon  \bf{D}^+_{aqc}(\O_C/p)^a \to \bf{D}^+_{aqc}(\X_0)^a$ that is compatible with the usual quasi-coherent adjoint $\mf_0^!\colon \bf{D}^+_{qc}(\O_C/p) \to \bf{D}^+_{qc}(\X_0)$ under the almostification functors $\bf{D}^+_{qc} \to \bf{D}^+_{aqc}$. Using the isomorphism
\[
\bf{R}\Gamma\left(X, \O_X^+/p\right) \simeq \bf{R}\Gamma\left(\X_0, \bf{R}\nu_*\, \O_X^+/p\right),
\]
we conclude that constructing $\rm{Tr}_X$ is equivalent to constructing Faltings' trace
\[
\rm{Tr}_{F, \X} \colon \bf{R}\nu_*\left(\O_X^{+, a}/p\right) \to \mathfrak{f}_0^!\left(\O_C/p\right)(-d)[-2d],
\]
where $\mathfrak{f}_0^!\left(\O_C/p\right)$ is simply the almostification $\omega^{\bullet, a}_{\X_0}$ of the dualizing complex of $\X_0$. Using the compatibility between $\omega^\bullet_\X$ and $\omega^\bullet_{\X_0}$, and cohomological bounds coming from \cite[Theorem 6.13.5, Theorem 6.13.6]{Z3} and Theorem~\ref{thm:dualizing-complex}, we eventually reduce the question of constructing $\rm{Tr}_{F, \X}$ to the question of constructing the {\it integral Faltings' trace}
\[
\rm{Tr}_{F, \X}^{d, +} \colon \rm{R}^d\nu_*\,\wdh{\O}_X^{+, a} \to \omega^a_\X(-d). 
\]
The key step is to use (\ref{equation:extend-from-codim-2}) to ensure that it suffices to define $\rm{Tr}_{F, \X}^{d, +}$ on the generic fiber and on the smooth locus of $\X$ in a compatible manner. We essentially use the map from \cite[Proposition 3.23]{Sch1} on the generic fiber, and we adapt the map from \cite[\textsection 8.2]{BMS1} on the smooth locus, and check their compatibility. This finishes the construction of $\rm{Tr}_{F, \X}^{d, +}$ and, therefore, of $\rm{Tr}_X$. \smallskip

We want to emphasize that even though we are mostly interested in the ``mod-$p$ nearby cycles'' $\bf{R}\nu_* \left(\O_X^+/p\right)$, it is crucial to work with the ``integral'' complex $\rm{R}^d\nu_*\,\wdh{\O}_X^{+}$ to have access to the ``extend from codimension-$2$'' type argument used above. \smallskip

A priori, the construction of $\rm{Tr}_X$ depends on the choice of an admissible model $\X$ (with reduced special fiber), but we show that it is canonically independent of this choice in Lemma~\ref{lemma:global-trace-independent}.

\item(Section~\ref{section:local-duality} and Section~\ref{section:global-almost-duality}) The final step is to show that the trace map $\rm{Tr}_X$ constructed in (5) induces an almost perfect pairing on $\bf{R}\Gamma(X, \O_X^{+, a}/p)$. Again, using the almost version of Grothendieck Duality, we reduce the question to showing almost perfectness of the pairing
\begin{equation}\label{diagram:intro-duality-local}
\bf{R}\nu_*\left(\O_X^{+, a}/p\right) \otimes^L_{\O_{\X_0}} \bf{R}\nu_*\left(\O_X^{+, a}/p\right) \xr{-\cup -} \bf{R}\nu_*\left(\O_X^{+, a}/p\right) \xr{\rm{Tr}_{F, \X}} \left(\omega^{\bullet, a}_{\X_0}\right)(-d)[-2d].
\end{equation}
This is now a {\it local} question on $\X$. Thus we are in a good shape to use the local uniformization theorem \cite[Theorem 1.4]{Z1} that roughly says that any admissible formal $\O_C$-model of $X$ locally looks like a ``nice'' finite group quotient of a polystable formal $\O_C$-scheme up to some rig-isomorphisms. So it suffices to justify almost perfectness of (\ref{diagram:intro-duality-local}) for polystable models, and then show that almost perfectness of (\ref{diagram:intro-duality-local}) descends through rig-isomorphisms and ``nice'' quotients by finite groups. \smallskip

In the case of polystable models, we argue by explicit computations eventually reducing the claim to almost duality in continuous group cohomology of the profinite group $\bf{Z}_p(1)^d$. Then the almost version of Grothendieck Duality and duality between homotopy invariants and coinvariants for an action of a finite group $G$ ensure that almost perfectness of (\ref{diagram:intro-duality-local}) descends through rig-isomorphisms and nice finite group quotients.
\end{enumerate}

\subsection{Acknowledgements}
We are very grateful to B.\,Conrad for suggesting the problem and reading the first draft of this paper and making useful suggestions on how to improve the paper. The author benefited a lot from various fruitful discussions with B.\,Bhatt, B.\,Conrad, and S.\,Petrov. We express additional gradititude to B.\,Bhatt for inviting the author to the University of Michigan for one semester and numerous illuminating discussions there. We also thank the University of Michigan for their hospitality.  A big part of this work was carried out there. Finally, we thank the referees for their very careful reading. Their comments have greatly improved the paper.

\subsection{Conventions}\label{section:notation}

Let $K$ be a complete rank-$1$ valued field. We denote its ring of integers by $\O_K$, its maximal ideal by $\m_K$, and the residue field by $k\coloneqq \O_K/\m_K$. We denote by $C\coloneqq \wdh{\ov{K}}$ the completed algebraic closure of $K$. \smallskip

In this paper, we always write {\it qcqs} as a shortcut for quasi-compact quasi-separated. It applies to adic spaces, formal schemes, and schemes. \smallskip

An {\it admissible} formal $\O_K$-scheme is a flat topologically finitely presented\footnote{We recall that a topologically finitely presented formal $\O_K$-scheme is assumed to be qcqs.} formal $\O_K$-scheme $\X$. If $\varpi \in \O_K$ is a fixed pseudo-uniformizer, we define 
\[
    \X_i\coloneqq \X \times_{\Spf \O_K} \Spec \O_{K}/\varpi^{i+1}\O_K
\]
to be the mod-$\varpi^{i+1}$ fiber of $\X$. We denote the {\it special fiber} by 
\[
\ov{\X} \coloneqq \X \times_{\Spf \O_K} \Spec \O_K/\m_K.
\]

In this paper, a {\it rigid-analytic $K$-space} will always mean an adic space locally of finite type over $\Spa(K, \O_K)$. We recall that \cite[\textsection 4]{H1} constructs a fully faithful functor $r_K$ from the category of (classical) Tate rigid $K$-spaces to the category of rigid $K$-spaces in our definition. This functor induces an equivalence between quasi-separated Tate rigid $K$-spaces and quasi-separated rigid $K$-spaces by \cite[Proposition 4.5]{H1}. \smallskip

We recall the notion of the {\it cotangent complex} $L_f=L_{X/S}=L_{\O_X/\O_S}$  for a morphism of ringed sites $f\colon (X, \O_X) \to (S, \O_S)$. We refer to \cite[\href{https://stacks.math.columbia.edu/tag/08UT}{Tag 08UT}]{stacks-project} for the definition and a self-contained systematic development of this theory. We follow Stacks Project and use cohomological notations for the cotangent complex, in particular, $L_f\in \bf{D}^{\leq 0}(\O_X)$ for a morphism $f$. If $\mf\colon \X \to \Y$ is a morphism of topologically finitely presented formal $\O_C$-schemes, we define $\wdh{L}_{\X/\Y}$ as the derived $p$-adic completion of the complex $L_{\X/\Y}$. This definition coincides with the definition of the {\it analytic cotangent complex} $L_{\X/\Y}^{\text{an}}$ from \cite[Definition 7.2.3]{GR}. In particular, \cite[Proposition 7.2.10]{GR} implies that $\wdh{L}_{\X/\Y}\in \bf{D}^{\leq 0}_{coh}(\X)$ and the natural morphism $\wdh{\Omega}^1_{\X/\Y} \to \mathcal{H}^0(\wdh{L}_{\X/\Y})$ is an isomorphism for any $\mf$. \smallskip

In this paper, we always do almost mathematics with respect to the ideal $\m\subset \O_C$. It is straightforward to see that $\m$ is $\O_C$-flat and $\m^2=\m$, so this does define the ideal of almost mathematics. We also extensively use the formalism of {\it almost coherent} sheaves on (formal) $\O_C$-schemes as developed in \cite{Z3}. In particular, for a (topologically) finitely presented $\O_C$-algebra $R$, an $R$-module is {\it almost coherent} if, for every finitely generated ideal $\m_0\subset \m$, there is a finitely presented $R$-module $N_{\m_0}$ and an $R$-linear morphism
\[
f_{\m_0} \colon N_{\m_0} \to M
\]
such that $\ker(f_{\m_0})$ and $\coker(f_{\m_0})$ are annihilated by $\m_0$. It turns out that this notion appropriately globalizes to the notion of almost coherent sheaves on (topologically) finitely presented (formal) $\O_C$-schemes. We refer to \cite[\textsection 4]{Z3} for a thorough discussion of this notion. 
\smallskip

We define {\it the pro-\'etale site} of a rigid-space $X$ similarly to \cite{Sch1} and \cite{Sch-err}, but we fix a cardinal $\kappa$ and consider only $\kappa$-small pro-systems and coverings by $\kappa$-small sets of objects. This does not change the results of this paper. The pro-\'etale site comes with the {\it completed integral structure sheaf} $\wdh{\O}_X^+$ defined in \cite[Definition 4.1(ii)]{Sch1}. \smallskip

We also consider the morphisms of ringed topoi $\lambda\colon (X_{\proet}, \wdh{\O}^+_X) \to (X_{\et}, \O^+_X)$ and $t\colon  (X_{\et}, \O_X^+) \to (\X, \O_\X)$ for any rigid-analytic space $X$ with an admissible formal model $\X$. These morphisms fit into the commutative triangle  
\[
\begin{tikzcd}
\left(X_{\proet}, \wdh{\O}^+_X\right) \arrow{d}{\lambda} \arrow{dr}{\nu} &  \\
\left(X_{\et}, \O_X^+\right) \arrow{r}{t} & \left(\X, \O_\X\right) 
\end{tikzcd}
\]
that also induces the commutative diagram
\[
\begin{tikzcd}
\left(X_{\proet}, \O_X^+/p\right) \arrow{d}{\lambda} \arrow{dr}{\nu} & \\
\left(X_{\et}, \O_X^+/p\right) \arrow{r}{t} & \left(\X, \O_\X/p\right)=\left(\X_0, \O_{\X_0}\right).
\end{tikzcd}
\]
\medskip

We point out that the morphism denoted by $\nu$ above differs from the morphism denoted by $\nu$ in \cite{Sch1}. The morphism $\nu$ from \cite{Sch1} instead coincides with the morphism $\lambda$ above. \smallskip

Whenever we talk about cohomology of $\bf{Z}_p$ or $\bf{Q}_p$ local systems on a rigid space $X$, we always mean pro-\'etale cohomology groups. In particular, if $X$ is a (quasi-compact, quasi-separated) rigid space over a non-archimedean field $K$ with $\wdh{\ov{K}}=C$, 
\[
\bf{R}\Gamma(X_C, \bf{Z}_p)\coloneqq \bf{R}\Gamma_\proet(X_C, \wdh{\bf{Z}}_p) \text{ and }
\]
\[
\bf{R}\Gamma(X_C, \bf{Q}_p)\coloneqq \bf{R}\Gamma_\proet(X_C, \wdh{\bf{Q}}_p),
\]
where $\wdh{\bf{Z}}_p \coloneqq \lim_n \ud{\Z/p^n\Z}$ in $X_\proet$ and $\wdh{\bf{Q}}_p=\wdh{\bf{Z}}_p[1/p]$. \smallskip

If $(\cal{C}, \otimes, \bf{1})$ is a closed symmetric monoidal category with the inner Hom-functor $\ud{\Hom}$, we say that a pairing
\[
A\otimes B \to C
\]
is {\it perfect} if both duality morphisms
\[
B \to \ud{\Hom}_{\cal{C}}(A, C)
\]
\[
A \to \ud{\Hom}_{\cal{C}}(B, C)
\]
are isomorphisms. Throughout the paper, we are mostly interested in the cases $\cal{C}$ is equal to one of the following symmetric monoidal categories
\[
\left(\bf{Mod}_{R}, \otimes_{R}\right), \ \left(\bf{D}(R), \otimes^L_{R}\right), \ \left(\bf{Mod}^a_{R}, \otimes_{R}\right), \ \left(\bf{D}(R)^a, \otimes^L_{R}\right)
\]
for some $\O_C$-algebra $R$, or
\[
\left(\bf{Mod}_{X}, \otimes_{\O_X}\right), \ \left(\bf{D}(X), \otimes^L_{\O_X}\right), \ \left(\bf{Mod}^a_{X}, \otimes_{\O_X}\right), \ \left(\bf{D}(X)^a, \otimes^L_{\O_X}\right)
\]
for an $\O_C$-ringed space $(X, \O_X)$. If the symmetric monoidal category $(\cal{C}, \otimes)$ is one of $\left(\bf{Mod}^a_{R}, \otimes_{R}\right)$,  $\left(\bf{D}(R)^a, \otimes^L_{R}\right)$, $\left(\bf{Mod}^a_{X}, \otimes_{\O_X}\right)$, or $\left(\bf{D}(X)^a, \otimes^L_{\O_X}\right)$, we say that a pairing is {\it almost perfect} instead of just perfect. 
\medskip

We say that a diagram of the form
\[
\begin{tikzcd}
A \otimes A'\arrow{r}{\alpha} \arrow[d, shift left=2ex, "g"] & A'' \arrow{d}{h}  \\
B \otimes  B' \arrow[u, shift left=2ex, "f"] \arrow{r}{\beta} & B''
\end{tikzcd}
\]
{\it commutes} if the two natural morphisms 
\[
B\otimes A' \xr{f\otimes \rm{id}} A\otimes A' \xr{\alpha} A'' \xr{h} B'',
\]
\[
B\otimes A' \xr{\rm{id}\otimes g} B\otimes B' \xr{\beta} B''
\]
are equal. \smallskip

We also recall the definitions of the Tate and Breuil-Kisin twists. For this, we assume now that $K$ has mixed characteristic $(0, p)$ and consider the natural continuous action of the absolute Galois group $G_K\coloneqq \Gal(\ov{K}/K)$ on $C=\wdh{\ov{K}}$ and $\O_C$.\smallskip

The {\it Tate twist} $\bf{Z}_p(1)$ is defined as 
\[
\bf{Z}_p(1) \coloneqq \rm{T}_p(\bf{G}_{m, K})=\lim \mu_{p^n}(\ov{K}),
\]
where $\rm{T}_p(\bf{G}_{m, K})$ stands for the Tate module of $\bf{G}_{m, K}$. This is a rank-$1$ free $\bf{Z}_p$-module with a continuous action of $G_K$ that does not admit any canonical trivialization. We define $\Z_p(n)$, for $n\geq 0$, as the tensor product $\Z_p(n)\coloneqq \Z_p(1)^{\otimes n}$. We extend this definition for $n\leq -1$ by the formula $\Z_p(n) \coloneqq \Hom_{\Z_p}(\Z_p(-n), \Z_p)$. We also define the $G_K$-modules $\O_C(n)$ as $\O_C\otimes_{\Z_p} \Z_p(n)$ for any integer $n$. \smallskip

If $\mathcal C$ is a site, and $\F\in \mathbf{Shv}(\mathcal C)$ is a sheaf of $\O_C$-modules, we define its Tate twist $\F(n)$ as $\F\otimes_{\O_C}\O_C(n)$. It is straightforward to see that there is a canonical isomorphism of $G_K$-modules $\rm{H}^i(U, \F(n))\cong \rm{H}^i(U, \F)(n)$ for $U\in \mathcal C$ and $i\geq 0$. We will freely use this isomorphism in this paper. \smallskip

We also recall the definition of the {\it Breuil-Kisin twist} $\O_C\{1\}$ from \cite[Definition 8.2]{BMS1}. We set 
\[
\O_C\{1\}\coloneqq \rm{T}_p(\Omega^1_{\O_C/\bf{Z}_p})=\footnote{To get this equality, one notes that \cite[Theorem 6.5.8]{GR} implies that $L_{\O_C/\Z_p}\simeq \Omega^1_{\O_C/\Z_p}$ while \cite[Remark 3.1.12]{bhatt-arizona} implies that $\wdh{L}_{\O_C/\Z_p}$ is concentrated in degree $-1$. This formally implies that $\wdh{L}_{\O_C/\Z_p}\simeq \rm{T}_p(\Omega^1_{\O_C/\Z_p})[1]$.}\wdh{L}_{\O_C/\Z_p}[-1],
\]
where $\wdh{L}$ stands for the derived $p$-adically complete cotangent complex. We define the $\rm{dlog}_{\Z_p}$ map $\rm{dlog}\colon \bf{Z}_p(1) \to \O_C\{1\}$ as the morphism induced by the map 
\[
\mu_{p^\infty}(\ov{K})\to \Omega^1_{\O_C/\Z_p}
\]
that sends $f\in \mu_{p^\infty}(\ov{K})$ to $\frac{df}{f}$. The theorem of Fontaine \cite[Theorem 1]{Fon82}, \cite[\textsection 1.3]{Bei12}, \cite[Theorem 3.1]{Sza} says $\O_C\{1\}$ is a free rank-$1$ $\O_C$-module, and the $\rm{dlog}$ map, after an $\O_C$-linearization, 
\[
\rm{dlog} \colon \O_C(1) \to \O_C\{1\}
\]
is injective with image being equal to $(\zeta_p-1)\O_C\{1\}$ for any choice of a primitive $p$-th root of unity $\zeta_p$ (we review this result in Theorem~\ref{thm:Fontaine}). We define the Breuil-Kisin twists $\O_C\{n\}$ and $\F\{n\}$ similarly to the case of analogous Tate twists. 

\section{Poor man's version of duality on formal schemes}\label{section:poor-coherent-duality}

\subsection{Finitely presented compactifications}

In this section, we recall some basic facts about compactifications of (not necessarily noetherian) schemes. In particular, we show that, for a qcqs scheme $Y$ and a separated finitely presented morphism $f\colon X \to Y$, the category of finitely presented compactifications of $f$ is cofiltered. This will be an important technical role in our extension of Grothendieck duality to the universally coherent case. 

\begin{defn}\label{defn:fp-compactifications} For a morphism of schemes $f\colon X \to Y$, the {\it category of finitely presented compactifications} $\rm{Comp}^{\rm{fp}}(X/Y)$ (also denoted by $\rm{Comp}^{\rm{fp}}(f)$) is defined as follows:
\begin{enumerate}
    \item Objects are open immersions $j\colon X \hookrightarrow \ov{X}$ over $Y$ with $\ov{X} \to Y$ being proper and finitely presented;
    \item Morphisms $(j'\colon X \hookrightarrow \ov{X}') \to (j\colon X \hookrightarrow \ov{X})$ are $Y$-morphisms $g\colon \ov{X}' \to \ov{X}$ such that $g \circ j' = j$.
\end{enumerate}
    We say that a morphism $(j'\colon X \hookrightarrow \ov{X}') \xrightarrow{g} (j\colon X \hookrightarrow \ov{X})$ is {\it strict} if $g^{-1}(j(X))=j'(X)$.
\end{defn}

\begin{rmk}\label{rmk:noetherian-the-same} Definition~\ref{defn:fp-compactifications} coincides with the one given in \cite[\href{https://stacks.math.columbia.edu/tag/0ATT}{Tag 0ATT}]{stacks-project} if $Y$ is noetherian and $f$ is a separated finite type morphism. 
\end{rmk}

%\begin{variant}\label{variant:strict-fp-compactifications} For a morphism of schemes $f\colon X \to Y$, we define a subcategory $\rm{Comp}_{\rm{strict}}^{\rm{fp}}(X/Y)$ of $\rm{Comp}^{\rm{fp}}(X/Y)$ as follows:
%\begin{enumerate}
 %   \item Objects are the same as the objects of $\rm{Comp}^{\rm{fp}}(X/Y)$;
  %  \item Morphisms are strict morphisms in $\rm{Comp}^{\rm{fp}}(X/Y)$.
%\end{enumerate}
%\end{variant}

Now we note that, for every morphism of schemes $g\colon Y' \to Y$, the pullback functor induces a functor $g^*\colon \rm{Comp}^{\rm{fp}}(X/Y) \to \rm{Comp}^{\rm{fp}}(X_{Y'}/Y')$ via the formula
\[
\big(X \xhookrightarrow{j} \ov{X}\big) \mapsto \big(X_{Y'} \xhookrightarrow{j_{Y'}} \ov{X}_{Y'}\big). 
\]

\begin{rmk}\label{rmk:preserves-strict-morphisms} The pullback functor $g^*\colon \rm{Comp}^{\rm{fp}}(X/Y) \to \rm{Comp}^{\rm{fp}}(X_{Y'}/Y')$ sends strict morphisms to strict morphisms. 
\end{rmk}

\begin{lemma}\label{lemma:approximate-compactifications} Let $I$ be a directed set with the smallest element $0\in I$, let $(Y_i, u_{ij})_{i, j\in I}$ be an $I$-indexed inverse system of qcqs schemes with affine transition map, and let $f_0 \colon X_0 \to Y_0$ be a separated finitely presented $Y_0$-scheme. Let $Y=\lim_I Y_i$. Then the natural functor
\[
    \colim_I \rm{Comp}^{\rm{fp}}(X_0\times_{Y_0} Y_i/ Y_i) \to \rm{Comp}^{\rm{fp}}(X_0 \times_{Y_0} Y/ Y),
\]
%\[
 %   \colim_I \rm{Comp}_{\rm{strict}}^{\rm{fp}}(X_0\times_{Y_0} Y_i/ Y_i) \to \rm{Comp}_{\rm{strict}}^{\rm{fp}}(X_0 \times_{Y_0} Y/ Y),
%\]
is an equivalence of categories. 
\end{lemma}
\begin{proof}
    This follows directly from \cite[\href{https://stacks.math.columbia.edu/tag/01ZM}{Tag 01ZM}]{stacks-project}, \cite[\href{https://stacks.math.columbia.edu/tag/0EUU}{Tag 0EUU}]{stacks-project}, and \cite[\href{https://stacks.math.columbia.edu/tag/081F}{Tag 081F}]{stacks-project}.% and \cite[\href{https://stacks.math.columbia.edu/tag/081E}{Tag 081E}]{stacks-project}.
\end{proof}

\begin{cor}\label{cor:compactifications-super-cofiltered} Let $Y$ be a qcqs scheme, and let $f\colon X \to Y$ be a separated finitely presented morphism. Then
\begin{enumerate}
    \item\label{cor:compactifications-super-cofiltered-1} the category $\rm{Comp}^{\rm{fp}}(X/Y)$ is non-empty;
    \item\label{cor:compactifications-super-cofiltered-2} for every pair of objects $a,b \in \rm{Comp}^{\rm{fp}}(X/Y)$ there exist an object $c\in \rm{Comp}^{\rm{fp}}(X/Y)$ and {\it strict} morphisms $c \to a$ and $c\to b$;
    \item\label{cor:compactifications-super-cofiltered-3} for every pair of objects $a,b \in \rm{Comp}^{\rm{fp}}(X/Y)$ and every pair of morphisms $\alpha, \beta \colon a\to b$ there exists a {\it strict} morphism $\gamma\colon c\to a$ such that $\alpha \circ \gamma = \beta\circ \gamma$.
\end{enumerate}
    In particular, $\rm{Comp}^{\rm{fp}}(X/Y)$ is a cofiltered category. 
\end{cor}
\begin{proof}
    First, \cite[\href{https://stacks.math.columbia.edu/tag/01ZA}{Tag 01ZA}]{stacks-project} ensures that we can find a directed set $I$, and an $I$-indexed inverse system of {\it noetherian} schemes $(Y_i, u_{ij})_{i,j \in I}$ with affine transition maps $u_{i, j}$ such that $Y = \lim_I Y_i$. Furthermore, \cite[\href{https://stacks.math.columbia.edu/tag/01ZM}{Tag 01ZM}]{stacks-project} implies that there is an index $i_0 \in I$ and a separated finite type $Y_0$-scheme $f_{i_0} \colon X_{i_0} \to Y_{i_0}$ such that $X_{i_0} \times_{Y_{i_0}} Y \simeq X$ as $Y$-schemes. Therefore, we may replace $I$ with $I_{\geq i_0}$ and rename $i_0$ as $0$ to assume that we are in the situation of Lemma~\ref{lemma:approximate-compactifications}. Then Lemma~\ref{lemma:approximate-compactifications} and Remark~\ref{rmk:preserves-strict-morphisms} guarantee that it suffices to assume that $Y=Y_i$ is a noetherian scheme. In this case, the result follows directly from \cite[\href{https://stacks.math.columbia.edu/tag/0ATU}{Tag 0ATU}]{stacks-project} and Remark~\ref{rmk:noetherian-the-same}.  
\end{proof}

%The main goal of this section is to show that $\rm{Comp}(f)$ is cofiltered when $Y$ is qcqs and $f$ is a separated morphism of finite type. The main subtlety in working with Definition~\ref{defn:fp-compactifications} is that the scheme-theoretic image of a morphism between finitely presented $Y$-schemes may not be finitely presented. 

\subsection{Grothendieck duality for universally coherent schemes}

The formalism of quasi-coherent Grothendieck duality is usually developed only in the context of noetherian schemes. In this section, we sketch the extension of this formalism to the context of universally coherent schemes. It turns out that the approach taken in \cite[\href{https://stacks.math.columbia.edu/tag/0DWE}{Tag 0DWE}]{stacks-project} works almost verbatim in this more general setup. For this reason, we do not give full proofs and just explain the main changes compared to the proofs in \cite[\href{https://stacks.math.columbia.edu/tag/0DWE}{Tag 0DWE}]{stacks-project}. \smallskip

\begin{defn} A scheme $S$ is {\it universally coherent} if any finitely presented $S$-scheme $X$ is coherent.
\end{defn}

For the rest of the section, we fix a qcqs universally coherent scheme $S$. We denote by $\rm{FPS}_S$ the category of separated, finitely presented $S$-schemes. In what follows, we will freely use that any morphism $f\colon X \to Y$ in $\rm{FPS}_S$ is pseudo-coherent in the sense of \cite[\href{https://stacks.math.columbia.edu/tag/067Z}{Tag 067Z}]{stacks-project}, and $f$ is perfect in the sense of \cite[\href{https://stacks.math.columbia.edu/tag/0687}{Tag 0687}]{stacks-project} if and only if it is of finite Tor dimension. We will also freely use that, for a scheme $X\in \rm{FPS}_S$ and an object $\F\in \bf{D}_{qcoh}(X)$, $\F$ is pseudo-coherent in the sense of \cite[\href{https://stacks.math.columbia.edu/tag/08CB}{Tag 08CB}]{stacks-project} if and only if $\F\in \bf{D}^-_{coh}(X)$. \smallskip

Now we recall the main technical ingredient that will allow us to extend Grothendieck duality to the universally coherent case: 

\begin{thm}\label{thm:finitess-Kiehl} \cite[Corollary I.8.1.4]{FujKato}, \cite[Theorem 2.9]{Kiehl} Let $f\colon X \to Y$ be a proper morphism in $\rm{FPS}_S$. Then $\bf{R}f_*$ maps $\bf{D}^*_{coh}(X)$ to $\bf{D}^*_{coh}(Y)$ for $*=``\text{ ''}, +, -, b$.
\end{thm}

%Now we are almost ready to define the upper shriek functor. Before we do this, we explicitly spell out the notion of compactification that we are going to use in this paper.

\begin{thm}\label{thm:existence-!} There is a pseudo-functor $(-)^! \colon \rm{FPS}_R \to \Cat$ such that
\begin{enumerate}[label=\textbf{(\arabic*)}]
    \item $(X)^!=\bf{D}^+_{qc}(X)$,
    \item for an open immersion $f\colon X \to Y$, $f^! \simeq \bf{L}f^*(-)$,
    \item for a proper morphism $f\colon X \to Y$, $f^!$ is right adjoint to $\bf{R}f_*\colon \bf{D}^+_{qc}(X)\to \bf{D}^+_{qc}(Y)$.
\end{enumerate}
\end{thm}
\begin{proof}
    The proof of \cite[\href{https://stacks.math.columbia.edu/tag/0ATY}{Tag 0ATY}]{stacks-project} adapts to the universally coherent situation with minor changes. Namely, the main reason the StacksProject assumes noetheriannity of $S$ is to have access to the open base change result \cite[\href{https://stacks.math.columbia.edu/tag/0A9P}{Tag 0A9P}]{stacks-project}. The proof of \cite[\href{https://stacks.math.columbia.edu/tag/0A9P}{Tag 0A9P}]{stacks-project}, in turn, has noetherian hypothesis in order to ensure that derived pushforward along a proper finitely presented morphism preserves pseudo-coherent objects. In the universally coherent case, this is true due to Theorem~\ref{thm:finitess-Kiehl}. Besides this, one needs to use finitely presented compactifications from Definition~\ref{defn:fp-compactifications} as opposed to finite type compactifications used in the StacksProject. Corollary~\ref{cor:compactifications-super-cofiltered} implies that all geometric constructions with compactifications used in \cite[\href{https://stacks.math.columbia.edu/tag/0A9Y}{Tag 0A9Y}]{stacks-project} can be performed in the category of finitely presented compactifications.\footnote{We note that the proofs of \cite[\href{https://stacks.math.columbia.edu/tag/0AA0}{Tag 0AA0}]{stacks-project} and \cite[\href{https://stacks.math.columbia.edu/tag/0ATX}{Tag 0ATX}]{stacks-project} use that certain auxiliary morphisms between compactifications are strict. The same strictness claims can be achieved in the finitely presented setting due to Corollary~\ref{cor:compactifications-super-cofiltered}.}
\end{proof}

Now we study some of the properties of the upper shriek functor. 

\begin{lemma}\label{lemma:flat-base-change-!} Let 
\[
\begin{tikzcd}
X' \arrow{d}{f'} \arrow{r}{g'} & X \arrow{d}{f} \\
Y'\arrow{r}{g} & Y
\end{tikzcd}
\]
be a cartesian diagram in $\rm{FPS}_S$ with flat $g$. Then there is a canonical base change isomorphism
\[
\bf{L}g'^*\circ f^! \xr{\sim} f'^! \circ \bf{L}g^*.
\]
\end{lemma}
\begin{proof}
    The proof is essentially identical to that of \cite[\href{https://stacks.math.columbia.edu/tag/0E9U}{Tag 0E9U}]{stacks-project}. In {\it loc.\,cit.\,}, the noetherian hypothesis is only used to ensure that derived pushforward along a proper finitely presented morphism preserves pseudo-coherent complexes. In the universally coherent case, this is true due to Theorem~\ref{thm:finitess-Kiehl}.
\end{proof}

\begin{lemma}\label{lemma:coprojection-formula} Let $f\colon X \to Y$ be a perfect morphism in $\rm{FPS}_S$. Then there is a canonical isomorphism
\[
\mu_{K, f}\colon \bf{L}f^*K \otimes^L_{\O_X} f^!\O_Y \to f^!K
\]
for any $K\in \bf{D}^+_{qc}(Y)$.
\end{lemma}
\begin{proof}
    The construction of $\mu_{K, f}$ is similar to the construction of $\mu_{K, f}$ in \cite[\href{https://stacks.math.columbia.edu/tag/0B6T}{Tag 0B6T}]{stacks-project}. To prove that it is an isomorphism if $f$ is of finite Tor dimension, we can reduce to the case of a perfect finitely presented closed immersion $f$ and of the relative affine line $\bf{A}^1_Y \to Y$. The case of the relative affine line reduces to the case of the relative projective line $\bf{P}^1_Y \to Y$. Therefore, it suffices to prove the claim for a perfect proper $f$. In this case, the proof is essentially identical to that of \cite[\href{https://stacks.math.columbia.edu/tag/0A9U}{Tag 0A9U}]{stacks-project} with the changes similar to the ones already discussed in Theorem~\ref{thm:existence-!} and Lemma~\ref{lemma:flat-base-change-!}.
\end{proof}

\begin{lemma}\label{lemma:dualizing-complex-affine-space} Let $Y$ be an element in $\rm{FPS}_S$, and let $f\colon \bf{A}^d_Y \to Y$ be a relative affine space over $Y$. Then there is a (non-canonical) isomorphism $f^! \O_Y \cong \O_{\bf{A}^d_Y}[d]$.
\end{lemma}
\begin{proof}
    It suffices to show that, for the morphism $\ov{f}\colon \bf{P}^d_Y \to Y$, we have $\ov{f}^! \O_Y \simeq \O_{\bf{P}^d_Y}(-d-1)[d]$. For this, we note that the proof of \cite[\href{https://stacks.math.columbia.edu/tag/0A9W}{Tag 0A9W}]{stacks-project} applies verbatim in our situation.
\end{proof}

\begin{lemma}\label{lemma:finiteness-properties-!} Let $f\colon X \to Y$ be a morphism in $\rm{FPS}_S$. Then 
\begin{enumerate}
    \item $f^!$ sends $\bf{D}^+_{coh}(Y)$ to $\bf{D}^+_{coh}(X)$,
    \item if $f$ is of finite Tor dimension, $f^!$ sends $\bf{D}^b_{coh}(X)$ to $\bf{D}^b_{coh}(Y)$,
    \item if $f$ is flat, $f^!\O_Y$ is $Y$-perfect in the sense of \cite[\href{https://stacks.math.columbia.edu/tag/0DI0}{Tag 0DI0}]{stacks-project}.
\end{enumerate}
\end{lemma}
\begin{proof}
    For $(1)$ (resp. $(2)$), we can reduce to the case of a finitely presented closed immersion (resp. perfect closed immersion) and of the relative affine line. The former case can be seen using \cite[\href{https://stacks.math.columbia.edu/tag/0A9X}{Tag 0A9X}]{stacks-project} and standard properties of the $\bf{R}\ud{\cal{H}om}$-functor, the latter case follows from Lemma~\ref{lemma:coprojection-formula} and Lemma~\ref{lemma:dualizing-complex-affine-space}. \smallskip
    
    Now we show $(3)$. We already know that $f^!\O_Y\in \bf{D}^b_{coh}(X)$ and thus is pseudo-coherent. So we only need to show that it is of finite Tor dimension over $Y$. The claim is local, so we can assume that $X=\Spec B$ and $Y=\Spec A$. Then the natural morphism $A \to B$ can be factored as $A \to P=A[x_1, \dots, x_n] \to B$, where $A \to P$ is the natural inclusion and $P \to B$ is a surjective morphism. Using \cite[\href{https://stacks.math.columbia.edu/tag/0A9X}{Tag 0A9X}]{stacks-project}, \cite[\href{https://stacks.math.columbia.edu/tag/0A6H}{Tag 0A6H}]{stacks-project}, and Lemma~\ref{lemma:dualizing-complex-affine-space} it suffices to show that the object
    \[
    \rm{RHom}_P(B, P)[n] \in \bf{D}(P)
    \]
    has finite Tor dimension over $A$. Now \cite[\href{https://stacks.math.columbia.edu/tag/068Y}{Tag 068Y}]{stacks-project} implies $B$ considered as an element of $\bf{D}(P)$ is a perfect complex, and so $\rm{RHom}_P(B, P)[n]\in \bf{D}(P)$ is perfect. Thus, it is of finite Tor dimension over $A$. 
\end{proof}

\begin{construction} Let $f\colon X\to S$ be a flat, separated, finitely presented morphism with the diagonal morphism $\Delta\colon X \to X\times_S X$ and the projection morphisms $p_{i}\colon X \times_S X \to X$. Then there is the natural morphism 
\[
\xi \colon \Delta_* \O_X \to \bf{L}p_1^*f^!\O_S
\]
defined as the adjoint to the isomorphism
\[
\O_X \simeq \Delta^! p_2^! \bf{L}f^* \O_S \simeq \Delta^! \bf{L} p_1^* f^! \O_S,
\]
where the second isomorphism comes from Lemma~\ref{lemma:flat-base-change-!}. 
\end{construction}

\begin{lemma}\label{lemma:relative-dualizing} Let $f\colon X \to S$ be a flat, separated, finitely presented morphism. Then the pair $(f^!\O_S, \xi)$ is a relative dualizing complex in the sense of \cite[\href{https://stacks.math.columbia.edu/tag/0E2T}{Tag 0E2T}]{stacks-project}.
\end{lemma}
\begin{proof}
    The proof is analogous to that of \cite[\href{https://stacks.math.columbia.edu/tag/0E9W}{Tag 0E9W}]{stacks-project} using the facts already established in this section. 
\end{proof}

\begin{cor}\label{cor:base-change-dualizing-object} Let $g\colon S' \to S$ be a morphism of qcqs universally coherent schemes, and 
\[
\begin{tikzcd}
X' \arrow{d}{f'} \arrow{r}{g'} & X \arrow{d}{f} \\
S'\arrow{r}{g} & S
\end{tikzcd}
\]
a cartesian diagram with $f$ flat, separated, and finitely presented. Then there is a canonical isomorphism
\[
\beta_g\colon \bf{L}g'^*\,f^!\O_S \xr{\sim} f'^!\O_{S'}.
\]
\end{cor}
\begin{proof}
    This follows directly from Lemma~\ref{lemma:relative-dualizing}, \cite[\href{https://stacks.math.columbia.edu/tag/0E2Y}{Tag 0E2Y}]{stacks-project}, and \cite[\href{https://stacks.math.columbia.edu/tag/0E2W}{Tag 0E2W}]{stacks-project}. 
\end{proof}

\begin{lemma}\label{lemma:smooth-morphism-dualizing-complex} Let $f\colon X \to S$ be a smooth, separated, finitely presented morphism of pure relative dimension $d$. Then there is a canonical isomorphism
\[
\alpha_f\colon f^!\,\O_S \xr{\sim} \Omega^d_{X/S}[d]
\]
that commutes with an arbitrary base change $g\colon S' \to S$ of qcqs universally coherent schemes. 
\end{lemma}
\begin{proof}
    We first note that the diagonal morphism $\Delta\colon X \to X\times_S X$ is a regular closed immersion by \cite[\href{https://stacks.math.columbia.edu/tag/067U}{Tag 067U}]{stacks-project}. In particular, it is a perfect closed immersion by \cite[\href{https://stacks.math.columbia.edu/tag/068C}{Tag 068C}]{stacks-project}. \smallskip
    
    Now we construct a structure of a relative dualizing complex (in the sense of \cite[\href{https://stacks.math.columbia.edu/tag/0E2T}{Tag 0E2T}]{stacks-project}) on $\Omega^d_{X/S}[d]$. Clearly, $\Omega^d_{X/S}[d]$ is $S$-perfect since $\Omega^d_{X/S}$ is a line bundle. Thus, it suffices to construct a morphism $\eta_f\colon \Delta_*\O_X \to \bf{L}p_1^*\,\Omega^d_{X/S}[d]$ inducing an isomorphism (see \cite[\href{https://stacks.math.columbia.edu/tag/0A9X}{Tag 0A9X}]{stacks-project})
    \[
    \Delta_*\O_X \to \bf{R}\ud{\cal{H}om}_{\O_{X\times_S X}}(\Delta_*\O_X, \bf{L}p_1^*\Omega^d_{X/S}[d]) \simeq \Delta_*\,\Delta^!\,\bf{L}p_1^*\,\Omega^d_{X/S}[d].
    \]
    By adjunction, it suffices to construct an isomorphism
    \[
    \eta'_f\colon \O_{X} \xr{\sim} \Delta^! \,\bf{L}p_1^*\,\Omega^d_{X/S}[d].
    \]
    We define $\eta'_f$ as the following composition of isomorphisms:
    \begin{equation}\label{eqn:eta-prime}
        \Delta^! \,\bf{L}p_1^* \,\Omega^d_{X/S}[d] \simeq \Delta^!\, \O_{X\times_S X} \otimes^L_{\O_X} \Omega^d_{X/S}[d] \simeq \Omega^{d, \vee}_{X/S}[-d] \otimes^L_{\O_X} \Omega^d_{X/S} [d] \simeq \O_X,
    \end{equation}
    where the first isomorphism comes from Lemma~\ref{lemma:coprojection-formula}, and the second isomorphism comes from \cite[\href{https://stacks.math.columbia.edu/tag/0BQZ}{Tag 0BQZ}]{stacks-project} and \cite[\href{https://stacks.math.columbia.edu/tag/08S2}{Tag 08S2}]{stacks-project}. \smallskip
    
    Furthermore, we note that $\eta'_f$ (and thus $\eta_f$) commutes with arbitrary base change $g\colon S' \to S$ of qcqs universally coherent schemes. Indeed, the first isomorphism in Diagram~(\ref{eqn:eta-prime}) commutes with base change by \cite[\href{https://stacks.math.columbia.edu/tag/0B6P}{Tag 0B6P}]{stacks-project} and \cite[\href{https://stacks.math.columbia.edu/tag/0CTA}{Tag 0CTA}]{stacks-project}, while the second one commutes with base change by inspection (see the proof of \cite[\href{https://stacks.math.columbia.edu/tag/0BR0}{Tag 0BR0}]{stacks-project}). \smallskip

    Now we note that Lemma~\ref{lemma:relative-dualizing} and \cite[\href{https://stacks.math.columbia.edu/tag/0E2W}{Tag 0E2W}]{stacks-project} imply that there is a unique isomorphism $\alpha_f\colon (f^!\,\O_S, \xi_f) \simeq (\Omega^d_{X/S}[d], \eta_f)$. We wish to show that this isomorphism commutes with base change $g\colon S' \to S$. We denote by $f'\colon X' \to S'$, $g'\colon X' \to X$ the base change morphisms. \smallskip 
    
    Now we note that the uniqueness part of \cite[\href{https://stacks.math.columbia.edu/tag/0E2W}{Tag 0E2W}]{stacks-project} implies that the diagram
    \[
    \begin{tikzcd}
        \bf{L}g^*(f^!\O_S, \xi_f) \arrow{d}{\beta_g} \arrow{r}{\alpha_f} & \bf{L}g^*(\Omega^d_{X/S}[d], \eta_f) \arrow{d}{\wr} \\
        (f'^!\O_{S'}, \xi_{f'}) \arrow{r}{\alpha_{f'}} & (\Omega^d_{X'/S'}[d], \eta_{f'})
    \end{tikzcd}
    \]
    must commute, where $\beta_g$ comes from Corollary~\ref{cor:base-change-dualizing-object} and the right vertical arrow is induced by the canonical isomorphism $g^*\Omega^1_{X/S} \simeq \Omega^1_{X'/S'}$. This implies that $\alpha_f$ commutes with base change. 
\end{proof}

\begin{cor}\label{cor:etale-morphism-!} Let $f\colon X \to Y$ be an \'etale morphism in $\rm{FPS}_S$. Then there is a canonical isomorphism of functors
\[
f^!(-) \simeq \bf{L}f^*(-).
\]
\end{cor}
\begin{proof}
    Without loss of generality, we can assume that $Y=S$. Then the result a direct consequence of Lemma~\ref{lemma:coprojection-formula} and Lemma~\ref{lemma:smooth-morphism-dualizing-complex}.
\end{proof}

\begin{lemma}\label{lemma:degree-dualizing-complex} Let $f\colon X \to S$ be a separated, flat, finitely presented morphism with all fibers of dimension less or equal to $d$. Then $f^!\O_S \in \bf{D}^{[-d, 0]}_{coh}(X)$ and $f^!\O_S$ has Tor amplitude $[-d, 0]$ as an object of $\bf{D}(X, f^{-1}\O_S)$.
\end{lemma}
\begin{proof}
We already know that $f^!\O_S \in \bf{D}^b_{coh}(X)$ due to Lemma~\ref{lemma:finiteness-properties-!}. Now we note \cite[\href{https://stacks.math.columbia.edu/tag/08CI}{Tag 08CI}]{stacks-project} implies that it suffices to show that $f^!\O_S$ has Tor amplitude $[-d, 0]$ as an object of $\bf{D}(X, f^{-1}\O_S)$.  For this, we note that a standard approximation argument (see \cite[\href{https://stacks.math.columbia.edu/tag/01YT}{Tag 01YT}]{stacks-project}) and Corollary~\ref{cor:base-change-dualizing-object} allow us to reduce to the case when $S$ is noetherian. In this situation, the result follows directly from \cite[\href{https://stacks.math.columbia.edu/tag/0BV6}{Tag 0BV6}]{stacks-project}.
\end{proof}

For the next definition, we fix a flat, finitely presented, separated morphism $f\colon X \to S$.

\begin{defn}\label{defn:dualizing-complex-scheme} The {\it (relative) dualizing complex} of $X$ is the object $\omega^\bullet_X\coloneqq f^!\O_S \in \bf{D}^b_{coh}(X)$.

If $f$ is of pure relative dimension $d$, the {\it dualizing sheaf} of $X$ is the coherent sheaf 
\[
\omega_X \coloneqq \cal{H}^{-d}(\omega^\bullet_X)\in \bf{Coh}(X).
\]
\end{defn}

\begin{rmk} The relative dualizing complex is usually denoted by $\omega^\bullet_{X/S}$. However, for the purposes of this paper, it will be more convenient to denote it simply by $\omega^\bullet_X$ as this will never cause any confusion.
\end{rmk}

Now note that, for a finite morphism $f\colon X\to Y$ between quasi-compact smooth separated $S$-schemes of pure relative dimension $d$, there is a canonical trace morphism
\[
\rm{Tr}_f\colon f_*\Omega^d_{X/S} \to \Omega^d_{Y/S}
\]
coming from Lemma~\ref{lemma:smooth-morphism-dualizing-complex} and the $(\bf{R}f_*, f^!)$-adjunction. For future reference, we show that this morphism is quite explicit for a flat $f$. \smallskip

For any finite, locally free morphism $f\colon X \to Y$, we can also construct a trace morphism $\rm{Tr}_f\colon f_*\O_X \to \O_Y$. Locally on $Y$, $f$ looks like $f\colon \Spec B \to \Spec A$ where $B$ is an $A$-algebra that is finite projective as an $A$-module. Therefore we have the trace morphism $\rm{Tr}_{B/A} \colon B \to A$. Since $\rm{Tr}_{B/A}$ commutes with flat base change, this morphism glues (see \cite[\href{https://stacks.math.columbia.edu/tag/0BVH}{Tag 0BVH}]{stacks-project}) to a well-defined morphism 
\[
\rm{Tr}_f\colon f_*\O_X \to \O_Y.
\]

\begin{lemma}\label{lemma:finite-flat-trace-scheme} Let $f\colon X \to Y$ be a finite flat morphism between quasi-compact smooth separated $S$-schemes of pure relative dimension $d$. Then the diagram 
\[
\begin{tikzcd}
f_*f^*\Omega^d_{Y/S} \arrow{r} \arrow{d}{\wr} & f_*\Omega^d_{X/S} \arrow{d}{\rm{Tr}_f} \\
f_*\O_X\otimes_{\O_Y} \Omega^d_{Y/S} \arrow{r}{\rm{Tr}_{f}\otimes \rm{id}} & \Omega^d_{Y/S}
\end{tikzcd}
\]
commutes, where the top horizontal map is induced by the canonical morphism $f^*\Omega^d_{Y/S} \to \Omega^d_{X/S}$, the left vertical map comes from the projection formula, the right vertical map is the trace morphism from Grothendieck duality, and the bottom horizontal morphism is the trace of a finite locally free morphism.
\end{lemma}
\begin{proof}
    If $S$ is noetherian, this follows from \cite[Theorem 9.2.14(i)]{Nayak-Sastry} (or \cite[pp.\,97-99, (2.7.41)]{base-change}). In general, one uses a standard approximation argument (see \cite[\href{https://stacks.math.columbia.edu/tag/01YT}{Tag 01YT}]{stacks-project}) and \cite[\href{https://stacks.math.columbia.edu/tag/0B6J}{Tag 0B6J}]{stacks-project} to reduce to the noetherian case. 
\end{proof}

\subsection{Reflexivity of the dualizing sheaf}

The main goal of this section is to study the reflexivity property of the (relative) dualizing sheaf from Definition~\ref{defn:dualizing-complex-scheme}. In this section, we restrict our attention to flat, finitely presented schemes over a rank-$1$ valuation ring. \smallskip

Throughout this section, we fix a base scheme $S=\Spec \O_K$ for a rank-$1$ valuation ring $\O_K$ with a pseudo-uniformizer $\varpi$ and fraction field $K$. \smallskip

%Throughout this section, we will freely use different local cohomology functors; we refer to \cite[\href{https://stacks.math.columbia.edu/tag/0G6Y}{Tag 0G6Y}]{stacks-project}, \cite[\href{https://stacks.math.columbia.edu/tag/0952}{Tag 0952}]{stacks-project}, and \cite[\href{https://stacks.math.columbia.edu/tag/0DWQ}{Tag 0DWQ}]{stacks-project} for the definitions and relations between these functors. \smallskip

We first recall the notion of depth from \cite[\textsection 10.4]{GRfoundations} as this notion will play a crucial role in the rest of this section.  

\begin{defn}\label{defn:depth}\cite[(10.4.19)]{GRfoundations} Let $X$ be a scheme, $Z\subset X$ a closed subset, and $\F\in \bf{D}(X)$. The {\it depth of $\F$ along $Z$} is the number
\[
\rm{depth}_X(Z, \F)\coloneqq \sup \{ n\in \Z \ |\ \cal{H}^i_{Z}(\F)=0 \text{ for all } i< n\} \in \bf{Z} \cup \infty,
\]
where $\cal{H}^i_Z(\F)$ is the sheaf of compactly supported cohomology as defined in \cite[\href{https://stacks.math.columbia.edu/tag/0G6Y}{Tag 0G6Y}]{stacks-project}. 

If $x\in X$ is a (not necessarilly closed) point of $X$, the {\it depth of $\F$ at $x$} is the number
\[
\delta_X(x, \F) \coloneqq \rm{depth}_{\Spec \O_{X,x}}(\{x\}, \F|_{\Spec \O_{X,x}}) \in \bf{Z} \cup \infty.
\]
\end{defn}

\begin{rmk} We refer to \cite[Proposition 10.4.25]{GRfoundations} for the relation to the more usual definition of depth.
\end{rmk}

\begin{lemma}\label{lemma:S1-noetherian-top-space} Let $X$ be a scheme, $\F$ a quasi-coherent $\O_X$-module, and $x\in X$. Then $x \in \Spec \O_{X,x}$ is a weakly associated prime of $M_x\coloneqq \F_x$ if and only if $\delta_{X}(x, \F)=0$.
\end{lemma}
\begin{proof}
Without loss of generality, we can replace $X$ with $\Spec \O_{X,x}$ to assume that $X$ is a spectrum of a local ring, and $x$ is its (unique) closed point. \smallskip

Now, for each integer $i$, the sheaf $\cal{H}^i_x(X, \F)$ is supported on the closed point $\{x\}$. Therefore, $\cal{H}^i_x(X, \F) =0$ if and only if $\rm{H}^i_x(X, \F) \coloneqq \Gamma(\{x\}, \cal{H}^i_x(X, \F))=0$. So we conclude that 
\[
    \delta_X(x, \F) = \sup \{ n\in \Z \ |\ \rm{H}^i_{x}(X, \F)=0 \text{ for all } i< n\},
\]
By construction,
\[
    \rm{H}^0_{x}(X, \F) = \{ s\in \F_x \ | \ \rm{Supp}(s) \subset \{x\}\} = \{s\in M_x \ | \ \rm{rad}(\rm{Ann}_{\O_{X,x}}(s)) \supset \m_x\}. 
\]
Since $\rm{rad}(\rm{Ann}_{\O_{X,x}}(s)) \supsetneq \m_x$ only if $s=0$, \cite[\href{https://stacks.math.columbia.edu/tag/0566}{Tag 0566}]{stacks-project} implies that $\rm{H}^0_{x}(X, \F) \neq 0$ if and only if $x$ is a weakly associated prime of $M_x=\F_x$.
\end{proof}

\begin{lemma}\label{lemma:torsion-free-modules} Let $R$ be a reduced ring such that the topological space $|\Spec R|$ is noetherian, and $M$ a torsion-free $R$-module\footnote{Recall that this means that, for every regular element $f\in R$ and every non-zero element $m\in M$, $fm\neq 0$.}. Then the set of weakly associated ideals of $M$ is contained in the set of minimal ideals of $R$.
\end{lemma}
\begin{proof}
    First, we note that \cite[\href{https://stacks.math.columbia.edu/tag/00ES}{Tag 00ES}]{stacks-project} and \cite[\href{https://stacks.math.columbia.edu/tag/0052}{Tag 0052}]{stacks-project} ensure that $R$ has finitely many minimal primes. Let us denote them by $\{\p_1, \dots, \p_n\}$. Then \cite[\href{https://stacks.math.columbia.edu/tag/0EMA}{Tag 0EMA}]{stacks-project} and \cite[\href{https://stacks.math.columbia.edu/tag/05C3}{Tag 05C3}]{stacks-project} (applied to $M=R$) imply that
    \[
    \{\text{zerodivisors in }R\} = \bigcup_{i=1}^n \p_i. 
    \]
    Now we apply \cite[\href{https://stacks.math.columbia.edu/tag/05C3}{Tag 05C3}]{stacks-project} to the $R$-module $M$ to conclude that, for any weakly associated prime $\q$ of $M$ and an element $f\in R$, there is a non-zero element $m\in M$ such that $fm=0$. Since $M$ is torsion-free, we conclude that $f$ must be a zerodivisor in $R$, so
    \[
    \q \subset \bigcup_{i=1}^n \p_i.
    \]
    Now the Prime Avoidance Theorem (see \cite[\href{https://stacks.math.columbia.edu/tag/00DS}{Tag 00DS}]{stacks-project}) implies that $\q\subset \p_i$ for some $i$. Since $\p_i$ is a minimal prime ideal, we conclude that $\q=\p_i$.
\end{proof}

\begin{lemma}\label{lemma:dualizing-module-flat} Let $X$ be a separated, flat, finitely presented $\O_K$-scheme of pure relative dimension $d$. Then the dualizing module $\omega_X$ is flat over $\O_K$.
\end{lemma}
\begin{proof}
The question is local on $X$, so we can assume that $X=\Spec A$ is affine. Then Lemma~\ref{lemma:degree-dualizing-complex} implies that we have an isomorphism
\[
\omega^\bullet_{X_0} \simeq \omega^\bullet_{X}\otimes^L_{\O_K} \O_K/\varpi \O_K,
\]
where $X_0$ is the mod-$\varpi$ fiber of $X$. Lemma~\ref{lemma:degree-dualizing-complex} ensures that both $\omega^\bullet_X$ and $\omega^\bullet_{X_0}$ are concentrated in degrees $[-d, 0]$. This implies that multiplication by $\varpi$ is injective on $\mathcal{H}^{-d}(\omega^\bullet_\X)$. To finish the proof, we recall that $\O_K$-flat modules are exactly $\varpi$-torsionfree modules. 
\end{proof}

\begin{lemma}\label{lemma:dualizing-module-torsionfree} Let $k$ be a field, and let $X$ be a separated finite type $k$-scheme of pure dimension $d$. Then the dualizing sheaf $\omega_X$ is a torsion-free $\O_X$-module, i.e., for any open affine subscheme $U\subset X$, $\omega_X(U)$ is a torsion-free $\O_X(U)$-module. 
\end{lemma}
\begin{proof}
    The question is clearly local on $X$, so we can assume that $X=\Spec A$ is affine. In this situation, we pick a regular element $f\in A$ and wish to show that multiplication by $f$ induces an injective morphism $\omega_X \xrightarrow{f} \omega_X$. For brevity, we denote by $i\colon Y=\Spec A/(f) \to X=\Spec A$ the natural closed immersion. Then we note that \cite[\href{https://stacks.math.columbia.edu/tag/0A9X}{Tag 0A9X}]{stacks-project} and the short exact sequence
    \[
    0 \to R \xrightarrow{f} R \to R/f \to 0
    \]
    imply that we have the following distinguished triangle
    \begin{equation}\label{eqn:dualizing-complexes}
    i_*\omega^\bullet_{Y} \to \omega^\bullet_X \xrightarrow{f} \omega^\bullet_X.
    \end{equation}
    Since $f$ is a regular element, we conclude that $Y$ is a $k$-scheme of pure dimension $d-1$. Therefore, Lemma~\ref{lemma:degree-dualizing-complex} ensures that $i_*\omega^\bullet_{Y}$ is concentrated in degrees $[-d+1, 0]$. Thus, we apply $\cal{H}^{-d}$ to Equation~(\ref{eqn:dualizing-complexes}) to conclude that $\omega_X \xrightarrow{f} \omega_X$ is an injective morphism. 
\end{proof}

We recall that a coherent module $\F$ on a scheme $X$ is called {\it reflexive} if the natural morphism $\F \to \F^{\vee \vee}$ to its double-dual is an isomorphism. \smallskip

\begin{thm}\label{thm:dualizing-reflexive-schemes} Let $X$ be a separated flat finitely presented $\O_K$-scheme of pure relative dimension $d$. If the generic fiber $X_K$ is smooth and the special fiber $\ov{X}$ is geometrically reduced, then the relative dualizing sheaf $\omega_X$ is a reflexive coherent $\O_X$-module.
\end{thm}
To motivate this theorem, we note that if the ring $\O_K$ is discretely valued, then the assumptions on $X$ imply that it is normal. Then it is a classical result that a dualizing sheaf on a normal scheme is reflexive. 
\begin{proof}
    We start the proof by noting that Lemma~\ref{lemma:dualizing-module-flat} shows that $\omega_X$ is $\O_K$-flat.

    {\it Step $1$: $\omega_{X}/\m_K\omega_X \subset \omega_{\ov{X}}$}. The short exact sequence
    \[
        0 \to \m_K \to \O_K \to \O_K/\m_K \to 0
    \]
    and Corollary~\ref{cor:base-change-dualizing-object} imply that we have the following  distinguished triangle:
    \[
        \m_K \otimes^L_{\O_K} \omega^\bullet_X \to \omega^\bullet_X \to \omega^\bullet_{\ov{X}}.
    \]
    Using that $\m_K$ is $\O_K$-flat and that both $\omega^\bullet_X$ and $\omega^\bullet_{\ov{X}}$ are concentrated in degrees $[-d, 0]$ (see Lemma~\ref{lemma:degree-dualizing-complex}), we conclude that there is an exact sequence
    \[
        0 \to \m_K \otimes_{\O_K} \omega_X \to \omega_X \to \omega_{\ov{X}}.
    \]
    Now we use that $\omega_X$ is $\O_K$-flat to conclude that $\m_K \otimes_{\O_K}\omega_X \simeq \m_K\omega_X$ to see that $\omega_X/\m_K\omega_X \subset \omega_{\ov{X}}$. \smallskip

    {\it Step $2$: The $\O_{\ov{X}}$-module $\omega_X/\m_K\omega_X$ is torsion-free.} Since $\omega_X/\m_K\omega_X \subset \omega_{\ov{X}}$, it is sufficient to show $\omega_{\ov{X}}$ is torsion-free. This follows directly from Lemma~\ref{lemma:dualizing-module-torsionfree}. \smallskip

    {\it Step $3$: The $\O_{X_K}$-module $(\omega_X)_K$ is a line bundle.} This follows from the assumption that $X_K$ is $K$-smooth and the sequence of isomorphisms:
    \[
        (\omega_X)_K \simeq \omega_{X_K} \simeq \Omega^d_{X_K}.
    \]

    {\it Step $4$: Some cohomological considerations.} We choose  $x\in X$ a non-generic point in the special fiber. We roughly want to show that $\omega_X$ is at least $(S_2)$ at this point and $\omega_X^{\vee\vee}$ is at least $(S_1)$. However, it is a bit tricky to make this argument precise because the involved rings are not necessarily noetherian. For this reason, we will need to use some extra input from \cite{GRfoundations}. \smallskip

    First, we wish to show that $\delta_{X}(x, \omega^{\vee\vee}_X)\geq 1$. For this, we can assume that $X=\Spec R$ is affine. Then we note that $\omega^{\vee\vee}_X$ is a torsion-free $R$-module for a reduced ring $R$ with the noetherian underlying topological space $|X|$. Therefore, Lemma~\ref{lemma:S1-noetherian-top-space} and Lemma~\ref{lemma:torsion-free-modules} imply that $\delta_{X}(x, \omega^{\vee\vee}_X)\geq 1$. \smallskip
    
    Now we wish to show that $\delta_X(x, \omega_X)\geq 2$. Step~$2$ implies $\omega_X/\m_K\omega_X$ is torsion-free, and so Lemma~\ref{lemma:S1-noetherian-top-space} and Lemma~\ref{lemma:torsion-free-modules} guarantee that $\delta_{\ov{X}}(x, \omega_X/\m_K\omega_X) \geq 1$. Next, we apply \cite[Corollary 10.4.46]{GRfoundations} to the finitely presented morphism $X \to \Spec \O_K$, $\O_K$-flat coherent $\O_X$-module $\omega_X$, and the quasi-coherent $\O_{\Spec \O_K}$-module $\O_{\Spec \O_K}$ to conclude that 
    \begin{equation*}\label{eqn:1}
        \delta_X(x, \omega_X)=\delta_{\ov{X}}(x, \omega_X/\m_K\omega_X) + \delta_{\Spec \O_K}({\{\m_K\}, \O_{\Spec \O_K}}) \geq 2. 
    \end{equation*}

    {\it Step $5$: $\omega_X$ is reflexive.} We need to show that the natural morphism 
    \[
        \alpha\colon \omega_X \to \omega_X^{\vee\vee}
    \]
    is an isomorphism. It is clearly an isomorphism over the smooth locus $X^{\rm{sm}}\subset X$ because $(\omega_X)|_{X^{\rm{sm}}}$ is a line bundle (see Lemma~\ref{lemma:smooth-morphism-dualizing-complex}), and both $\omega_X$ and $\omega_X^{\vee\vee}$ commute with open immersions. In particular, $\alpha$ is an isomorphism over $X_K$ and over the generic points in the special fiber. \smallskip
    
    Now we note that $\alpha$ is injective because $\alpha_K$ is an isomorphism (due to the argument above) and $\omega_X$ is $\O_K$-flat. \smallskip
    
    Finally, we show that $\alpha_x$ is an isomorphism for any non-generic point $x$ in the special fiber $\ov{X}$. Since $\alpha$ is injective, we have a short exact sequence
    \[
        0 \to \omega_X \to \omega_X^{\vee\vee} \to \cal{Q} \to 0
    \]
    for some coherent $\O_X$-module $\cal{Q}$. Then we know that $\delta_{X}(x, \omega_X)\geq 2$ and $\delta_{X}(x, \omega_X^{\vee\vee})\geq 1$ by Step~$4$. So the long exact sequence for $\cal{H}_Z^i$ implies that $\delta_X(x, \cal{Q})\geq 1$ for any non-generic point $x$ in the special fiber. We also know that $\cal{Q}_y=0$ for any generic point $y$ in the special fiber and for any $y$ in the generic fiber. This implies that $\delta_X(y, \cal{Q}) = \infty$ at these points. Therefore, $\delta_X(x, \cal{Q})\geq 1$ for {\it any} $x\in X$. Therefore, $\cal{Q}$ is a coherent sheaf that has no weakly associated points. Thus, \cite[\href{https://stacks.math.columbia.edu/tag/05AP}{Tag 05AP}]{stacks-project} guarantees that $\cal{Q} \simeq 0$. 
\end{proof}

\subsection{Dualizing modules on separated admissible formal schemes}

In this section, we define a ``naive dualizing complex'' on separated admissible formal $\O_K$-schemes. This object will not be quite functorial because its construction will involve certain derived limits that are not functorial. In order to deal with this problem, we mainly restrict our attention to the study of a {\it dualizing module} that is the bottom cohomology sheaf of the naive dualizing complex. This turns out to be a more functorial object, and this functoriality is sufficient for all our purposes. \smallskip

Throughout this section, we fix a base formal scheme $\S=\Spf \O_K$ for a complete rank-$1$ valuation ring $\O_K$ with a pseudo-uniformizer $\varpi$ and fraction field $K$. \smallskip

For each integer $i\geq 0$, we denote by $S_i$ the scheme $\Spec \O_K/\varpi^{i+1}\O_K$. Likewise, for any  formal $\O_K$-scheme $\X$, we denote by $\X_i$ the fiber product 
\[
\X_i\coloneqq \X\times_{\S} S_i.
\]

Recall that, for a separated admissible formal $\O_K$-scheme $\X$, Definition~\ref{defn:dualizing-complex-scheme} produces a compatible sequence of dualizing complexes $\omega^\bullet_{\X_n} \in \bf{D}^b_{coh}(\X_n)$. 

\begin{defn}\label{defn:dualizing-complex} A {\it naive dualizing complex} $\omega^\bullet_{\X}$ of a separated admissible formal $\O_K$-scheme $\X$ is the derived limit 
\[
\omega^\bullet_{\X} \coloneqq \bf{R}\lim_n \omega^{\bullet}_{\X_n} \in \bf{D}(\X)
\]
with the transition map coming from Corollary~\ref{cor:base-change-dualizing-object}. 

If $\X$ is of pure dimension\footnote{This means that the underlying noetherian topological space $|\X|$ is of pure Krull dimension $d$.} $d$, the {\it dualizing sheaf} of $\X$ is the $\O_\X$-module 
\[
\omega_\X \coloneqq \lim_n \omega_{\X_n} \in \bf{Mod}_X.
\]
\end{defn}

\begin{warning} The naive dualizing complex of a separated admissible formal $\O_K$-scheme is well-defined up to an isomorphism (but not up to a unique isomorphism). 
\end{warning}

\begin{rmk} The naive dualizing complex will play the role of an intermediate device in our paper. So the lack of functoriality of $\omega^\bullet_\X$ will not be a major problem for our purposes. The actual thing that will be important for us is the dualizing sheaf $\omega_\X$ that is, indeed, functorial in $\X$.
\end{rmk}

\begin{thm}\label{thm:dualizing-complex} Let $\X$ be a separated admissible formal $\O_K$-scheme of pure dimension $d$. Then 
\begin{enumerate}
    \item the dualizing complex $\omega^\bullet_\X$ lies in $\bf{D}^{[-d, 0]}_{coh}(\X)$ and the canonical morphism $\omega^{\bullet}_\X \otimes^{L}_{\O_\X} \O_{\X_n} \to \omega^\bullet_{\X_n}$ is an isomorphism for any $n\geq 0$;
    \item there is an isomorphism $\mathcal H^{-d}(\omega^\bullet_\X)\cong \omega_\X$, so the dualizing module $\omega_\X$ is a coherent sheaf on $\X$;
    \item there is a canonical isomorphism $r_{\X^{\text{sm}}}\colon \wdh{\Omega}_{\X^\text{sm}}^d[d] \to \omega^\bullet_{\X^\text{sm}}$, where $\X^{\sm}$ is the smooth locus of $\X$. In particular, this induces a canonical isomorphism $r_{\X^{\text{sm}}}\colon \wdh{\Omega}_{\X^\text{sm}}^d \to \omega_\X|_{\X^\text{sm}}$.
\end{enumerate}
\end{thm}
\begin{proof}
The first two claims are local on $\X$, so we can assume that $\X=\Spf A$ for a flat, topologically finitely presented $\O_K$-algebra $A$. \smallskip

{\it Step 1. The dualizing complex $\omega^\bullet_\X$ is in $\bf{D}^{\leq 0}_{coh}(\X)$ and commutes with base change}: We first note that $\omega^\bullet_{\X_0}\in \bf{D}^{[-d, 0]}_{coh}(\X_0)$ due to Lemma~\ref{lemma:degree-dualizing-complex}. In particular, $\omega^\bullet_{\X_0}$ is pseudo-coherent as an object of $\bf{D}(\X_0)$ because $\X_0$ is a coherent scheme. Thus there is a resolution of $\omega^\bullet_{\X_0}$ by finite free $\O_{\X_0}$-modules $\mathcal E^\bullet_0 \to \omega^\bullet_{\X_0}$. 

Now we recall that Corollary~\ref{cor:base-change-dualizing-object} provides us with the canonical isomorphisms
\[
\omega^\bullet_{\X_n} \otimes^{L}_{\O_{\X_n}}\O_{\X_{n-1}} \to \omega^\bullet_{\X_{n-1}}.
\]
So \cite[\href{https://stacks.math.columbia.edu/tag/0BCB}{Tag 0BCB}]{stacks-project} ensures the existence of resolutions $\mathcal E^\bullet_n \to \omega^\bullet_{\X_n}$ by finite free modules such that $\mathcal E^\bullet_n \otimes_{\O_{\X_n}} \O_{\X_{n-1}} \simeq \mathcal E^{\bullet}_{n-1}$. Then we see that
\[
\omega^\bullet_{\X} \cong \bf{R}\lim_n \omega^\bullet_{\X_n} \cong \bf{R}\lim_n \mathcal E^\bullet_n \cong \lim_n \mathcal E^\bullet_n
\]
where the last equality comes from \cite[\href{https://stacks.math.columbia.edu/tag/0D60}{Tag 0D60}]{stacks-project} as cohomology of coherent sheaves vanishes on affine schemes. In order to conclude that $\omega^\bullet_{\X}\in \bf{D}^{\leq 0}_{coh}(\X)$, it suffices to notice that, for any $i\leq 0$, the limit $\lim_n \mathcal E_n^i$ is a finite free $\O_\X$-module of the same rank as $\E_0^i$. Furthermore, the base change morphism 
\[
\omega^{\bullet}_\X \otimes^{L}_{\O_\X} \O_{\X_n} \to \omega^\bullet_{\X_n}
\]
is an isomorphism as it has a representative $(\lim_n \mathcal E^\bullet_n) \otimes_{\O_\X} \O_{\X_n} \to \mathcal E^\bullet_n$ that is an isomorphism by the construction. \smallskip

{\it Step 2. The dualizing complex $\omega^\bullet_\X$ is in $\bf{D}^{[-d, 0]}_{coh}(\X)$ and there is an isomorphism $\mathcal H^{-d}(\omega^\bullet_\X)\cong \omega_\X$}: Throughout this step, we still assume that $\X=\Spf A$ is affine. We note that we already know that $\omega^\bullet_\X\in \bf{D}^{\leq 0}_{coh}(\X)$ by Step~$1$, so \cite[Proposition I.3.2.1]{FujKato} and \cite[Lemma 4.8.13]{Z3} ensure that it suffices to show that 
\[
\rm{H}^m(\X, \omega^\bullet_\X)\coloneqq \rm{H}^m\mathbf{R}\Gamma(\X, \omega^\bullet_\X) \simeq \Gamma(\X, \mathcal{H}^m(\omega^\bullet_\X)) =0
\]
for $m<-d$ and the natural morphism $\rm{H}^{-d}(\X, \omega^\bullet_\X) \to \lim_n \rm{H}^{-d}(\X_{n}, \omega^\bullet_{\X_n})$ is an isomorphism. Lemma~\ref{lemma:degree-dualizing-complex} implies that $\rm{H}^m(\X_{n}, \omega^\bullet_{\X_n}) = 0$ for $m<-d$. Therefore, the Milnor exact sequence (see \cite[\href{https://stacks.math.columbia.edu/tag/0D60}{Tag 0D60}]{stacks-project}) 
\[
0 \to \rm{R}^1\lim_n \rm{H}^{m-1}(\X_{n}, \omega^\bullet_{\X_n}) \to \rm{H}^m(\X, \omega^\bullet_\X) \to \lim_n \rm{H}^m(\X_{n}, \omega^\bullet_{\X_n}) \to 0
\]
implies that $\rm{H}^m(\X, \omega^\bullet_\X)\cong 0$ for $m<-d$ and the natural morphism $\rm{H}^{-d}(\X, \omega^\bullet_\X) \to \lim_n \rm{H}^{-d}(\X_{n}, \omega^\bullet_{\X_n})$ is an isomorphism. This finishes the proof. \smallskip 

{\it Step 3. Construction of the canonical isomorphism $r_{\X^{\text{sm}}}\colon \wdh{\Omega}_{\X^\text{sm}}^d[d] \to \omega_{\X^\text{sm}}$}: As $\X$ is flat over $\O_K$, we conclude that $(\X^{\text{sm}})_n = (\X_n)^{\text{sm}}$ for every integer $n\geq 0$. So Lemma~\ref{lemma:smooth-morphism-dualizing-complex} provides us with canonical isomorphisms $r_{\X_n^{\text{sm}}}\colon \Omega_{\X_n^\text{sm}}^d[d] \to \omega^\bullet_{\X_n^\text{sm}}$ that commute with base change. In particular, each dualizing complex $ \omega^\bullet_{\X^{\text{sm}}_n}$ is concentrated in degree $-d$ and locally free. The fact that $r_{\X_n^{\text{sm}}}$ commutes with base change means that the diagram 
\[
\begin{tikzcd}[column sep=17ex]
\Omega_{\X_n^\text{sm}}^d[d]  \otimes_{\O_{\X^{\text{sm}}_n}} \O_{\X^{\text{sm}}_{n-1}} \arrow{r}{r_{\X_n^{\text{sm}}} \otimes_{\O_{\X^{\text{sm}}_n}} \O_{\X^{\text{sm}}_{n-1}}} \arrow{d} & \omega^\bullet_{\X^{\text{sm}}_n} \otimes^\bf{L}_{\O_{\X^{\text{sm}}_n}} \O_{\X^{\text{sm}}_{n-1}} \arrow{d} \\
\Omega_{\X_{n-1}^\text{sm}}^d[d] \arrow{r}{r_{\X_{n-1}^{\text{sm}}}} & \omega^\bullet_{\X^{\text{sm}}_{n-1}} \ ,
\end{tikzcd}
\]
where the vertical maps are canonical base change isomorphisms, is commutative. In particular, we see that the canonical morphisms $\lim_n \Omega_{\X_n^\text{sm}}^d[d] \to \bf{R}\lim_n \Omega_{\X_n^\text{sm}}^d[d] \to \bf{R}\lim_n \omega^\bullet_{\X^{\text{sm}}_n}\cong \omega^\bullet_{\X^\text{sm}}$ are isomorphisms. Now we apply $\mathcal H^{-d}(-)$ to get a canonical isomorphism $r_{\X^{\text{sm}}}\colon \wdh{\Omega}_{\X^\text{sm}}^d \to \omega_\X|_{\X^\text{sm}}$
\end{proof}

\begin{cor}\label{cor:dualizing-module-flat} Let $\X$ be a separated admissible formal $\O_K$-scheme of pure dimension $d$. Then the dualizing module $\omega_\X$ is flat over $\O_K$.
\end{cor}
\begin{proof}
Theorem~\ref{thm:dualizing-complex} implies that we have a distinguished triangle
\[
\omega^\bullet_\X \xr{\varpi} \omega^\bullet_\X \to \omega^\bullet_{\X_0}
\]
in $\bf{D}^b_{coh}(\X)$. We note that Theorem~\ref{thm:dualizing-complex} and Lemma~\ref{lemma:degree-dualizing-complex} ensure that both complexes $\omega^\bullet_\X$ and $\omega^\bullet_{\X_0}$ are concentrated in degrees $[-d, 0]$. This implies that multiplication by $\varpi$ is injective on $\mathcal{H}^{-d}(\omega^\bullet_\X)$. To finish the proof, we recall that $\O_K$-flat modules are exactly $\varpi$-torsionfree modules. 
\end{proof}

Now we study some basic properties of dualizing complexes on admissible formal $\O_K$-schemes. We will need the following preliminary lemma: 

\begin{lemma}\label{lemma:dbcoh-complete} Let $\X$ be an admissible formal $\O_K$-scheme, and $\F\in \bf{D}^b_{coh}(\X)$. Then $\F$ is derived $\varpi$-adically complete (see \cite[\href{https://stacks.math.columbia.edu/tag/0999}{Tag 0999}]{stacks-project}).
\end{lemma}
\begin{proof}
    As derived complete $\O_\X$-modules are closed under taking cones, it suffices to show that a coherent sheaf $\F\in \bf{Coh}_\X$ is derived $\varpi$-adically complete. Let us denote by $\F_n\in \bf{Coh}_{\X_n}$ the tensor product $\F\otimes_{\O_\X} \O_{\X_n}$. Then \cite[Definition I.3.1.3, Theorem I.7.1.1]{FujKato} and \cite[\href{https://stacks.math.columbia.edu/tag/0BKS}{Tag 0BKS}]{stacks-project} guarantee that 
    \[
    \F \simeq \lim_n \F_n \simeq \bf{R}\lim_n \F_n.
    \]
    Each $\O_\X$-module $\F_n$ is derived $\varpi$-adically complete as it is $\varpi^{n+1}$-torsion, so the derived limit is also derived $\varpi$-adically complete.
\end{proof}

\begin{lemma}\label{lemma:dualizing-etale-base-change} Let $\mf\colon \X \to \Y$ be an \'etale morphism of separated admissible formal $\O_K$-schemes. Then there is an isomorphism $\bf{L}\mf^*\,\omega^\bullet_\Y \cong \omega^\bullet_\X$. If $\X$ and $\Y$ are of pure dimension $d$, then this isomorphism induces a canonical isomorphism $\mf^*\omega_\Y \xr{\sim} \omega_\X$.
\end{lemma}
\begin{proof}
We first note that $\bf{L}\mf^*\,\omega^\bullet_\Y \in \bf{D}^b_{coh}(\X)$ because $\omega^\bullet_\Y \in \bf{D}^b_{coh}(\Y)$ and $\mf$ is \'etale. So the complex $\bf{L}\mf^*(\omega^\bullet_\Y)$ is derived $\varpi$-adically complete by Lemma~\ref{lemma:dbcoh-complete}. Therefore, we use \cite[\href{https://stacks.math.columbia.edu/tag/0A0E}{Tag 0A0E}]{stacks-project} to see that 
\[
    \bf{L}\mf^*\,\omega^\bullet_\Y \cong 
    \bf{R}\lim_n \left(\bf{L}\mf^*\, \omega^\bullet_\Y \otimes^L_{\O_\X} \O_{\X_n}\right).
\]
Here we used that $\X$ is flat over $\O_K$ to conclude that $\varpi$ is a regular element of $\O_\X$, so $(\O_\X \xr{\varpi^{n+1}} \O_\X)\cong \O_{\X_n}[0]$. Now 
\[
    \bf{L}\mf^*\,\omega^\bullet_\Y \otimes^L_{\O_\X} \O_{\X_n} \simeq \bf{L}\mf^*_n\left(\omega^\bullet_\Y \otimes^L_{\O_\Y} \O_{\Y_n}\right) \simeq \bf{L}\mf^*_n\,\omega^\bullet_{\Y_n}
\]
as $\omega^\bullet_\Y \otimes^L_{\O_\Y} \O_{\Y_n} \simeq \omega^\bullet_{\Y_n}$ by Theorem~\ref{thm:dualizing-complex}. Finally, Corollary~\ref{cor:etale-morphism-!} guarantees \'etale base change preserves dualizing complexes, so we have a (non-canonical) isomorphism
\[
\bf{L}\mf^*\,\omega^\bullet_\Y \cong \bf{R}\lim_n \bf{L}\mf^*_n\,\omega^\bullet_{\Y_n} \cong \bf{R}\lim_n \omega^{\bullet}_{\X_n} \cong \omega^\bullet_{\X}.
\]

To see canonicity of the induced map on dualizing modules, we just note that it coincides with the morphism
\[
\mf^* \omega_{\Y} \simeq \mf^* \,\lim_n \omega_{\Y_n} \simeq \lim_n \mf_n^*\,\omega_{\Y_n} \xrightarrow{\sim} \lim_n \omega_{\X_n} \simeq \omega_\X 
\]
that is functorial due to functoriality of $\lim_n$ (as opposed to $\bf{R}\lim_n$). 
\end{proof}

Now let $X$ be a flat finitely presented $\O_K$-scheme. We denote by $\wdh{X}$ its $\varpi$-adic completion (considered as a formal $\O_K$-scheme). In what follows, we denote by $c\colon \wdh{X} \to X$ the completion morphism, and by $i'_n\colon X_n \to X$ the closed immersion $X\times_{\Spec \O_K} \Spec \O_K/\varpi^{n+1}\O_K \to X$. 

\begin{lemma}\label{lemma:complete-pullback} Let $X$ be a flat, finitely presented $\O_K$-scheme, $c \colon \wdh{X} \to X$ the $\varpi$-adic completion morphism, and $\F\in \bf{D}^b_{coh}(X)$. Then there is a functorial (in $\F$) isomorphism $\bf{L}c^* \F \xr{\sim} \bf{R}\lim_n \bf{L}i'^*_n \F$. 
\end{lemma}
\begin{proof}
First, we note that $\bf{L}c^*\F$ lies in $\bf{D}^b_{coh}(\wdh{X})$ by \cite[Remark I.9.1.1]{FujKato}, so it is derived complete due to Lemma~\ref{lemma:dbcoh-complete}. The commutative square
\[
\begin{tikzcd}
\wdh{X}_n \arrow{d}{i_n} \arrow{r}{\rm{id}} & X_n \arrow{d}{i'_n} \\
\wdh{X} \arrow{r}{c} & X \ ,
\end{tikzcd}
\]
implies that $\bf{L}i_n^*\bf{L}c^*\F \cong \bf{L}i'^*_n\F$. Using that $\wdh{X}$ is $\O_K$-flat and that $\bf{L}c^*\F$ is derived complete, we conclude that
\[
\bf{L}c^*\F \simeq \bf{R}\lim_n \left(\bf{L}c^*\F\otimes^L_{\O_{\wdh{X}}} \O_{\wdh{X}_n}\right) \simeq \bf{R}\lim_n \bf{L}i_n^*\bf{L}c^*\F \simeq \bf{R}\lim_n \bf{L}i'^*_n \F.\qedhere
\]
\end{proof}

\begin{cor}\label{cor:complete} Let $X$ be a flat, separated, finitely presented $\O_K$-scheme, and $c \colon \wdh{X} \to X$ the $\varpi$-adic completion morphism. Then, there is an isomorphism $\bf{L}c^* \omega^\bullet_X \xr{\sim} \omega^\bullet_{\wdh{X}}$. If $X$ is of relative pure dimension $d$, this isomorphism induces a canonical isomorphism $c^*(\omega_X) \xr{\sim} \omega_{\wdh{X}}$.
\end{cor}
\begin{proof}
    Lemma~\ref{lemma:complete-pullback}, Lemma~\ref{lemma:finiteness-properties-!}, and Corollary~\ref{cor:base-change-dualizing-object} imply that 
    \[
    \bf{L}c^*\omega^\bullet_X \cong \bf{R}\lim_n \omega^\bullet_{\wdh{X}_n}.
    \]
    This provides us with an isomorphism $\bf{L}c^* \omega^\bullet_X \to \omega^\bullet_{\wdh{X}}$ of complexes concentrated in degrees $[-d, 0]$ (see Lemma~\ref{lemma:degree-dualizing-complex} and Theorem~\ref{thm:dualizing-complex}). In case $X$ is of pure relative dimension $d$, we use that the map $c$ is flat to see that this morphism is canonical in the bottom degree $-d$ and induces an isomorphism $c^*(\omega_X) \to \omega_{\wdh{X}}$ that comes as the limit of the base change morphisms $c^*(\omega_X) \to \omega_{X_n} = \omega_{\wdh{X}_n}$.  
\end{proof}

\subsection{Extension properties of dualizing sheaf}

The main goal of this section is to show that, under some assumptions on an admissible formal $\O_K$-scheme $\X$, the dualizing sheaf $\omega_\X$ satisfies Hartogs' style extension principle in a precise sense. \smallskip

Throughout this section, we fix a base formal scheme $\S=\Spf \O_K$ for a complete rank-$1$ valuation ring $\O_K$ with a pseudo-uniformizer $\varpi$ and fraction field $K$. \smallskip

We start with the following definition that is supposed to formalize pushforward from $``\X_K\cup \sU \text{''}$ for an open $\sU\subset \X$ in an admissible formal $\O_K$-scheme $\X$:

\begin{defn}\label{defn:intersection} Let $\X$ be an admissible formal $\O_K$-scheme with an open formal subscheme $\mathfrak U \subset \X$, and $\F$ a coherent sheaf on $\X$. We define the $\O_\X$-module $j_{\mathfrak U, *}\F|_\mathfrak U \cap \F_K$ as the kernel of the map 
\[
j_{\X_K, *}\F_K \oplus j_{\mathfrak U, *}\F|_{\mathfrak U} \xr{f-g} j_{\mathfrak U_K, *}(\F|_{\mathfrak U})_K,
\]
where $j_\mathfrak U\colon \mathfrak U \to \X$ (resp. $j_{\X_K, *}\colon \X_K \to \X$, resp. $j_{\mathfrak U_K}\colon \mathfrak U_K \to \X$) is the natural open immersion (resp. generic fiber map) and $f, g$ are the natural morphisms that come from the adjunction. 
\end{defn}

\begin{warning} The notation $j_{\mathfrak U, *}\F|_\mathfrak U \cap \F_K$ could be a bit dangerous as it is not necessarily a subsheaf of $j_{\mathfrak U, *}\F|_\mathfrak U$. However, $j_{\mathfrak U, *}\F|_\mathfrak U \cap \F_K$ is indeed an intersection if both $f\colon j_{\X_K, *}\F_K \to j_{\mathfrak U_K, *}(\F|_{\mathfrak U})_K$ and $g\colon j_{\mathfrak U, *}\F|_{\mathfrak U} \to j_{\mathfrak U_K, *}(\F|_{\mathfrak U})_K$ are injective. Since we will use this definition only in this situation, this notation should not cause any confusion.
\end{warning}

\begin{lemma}\label{lemma:injectivity-vector-bundles} Let $\X$ be an admissible formal $\O_K$-scheme with reduced special fiber, and $\cal{E}$ a vector bundle on $\X$. Then, for any dense open $\sU \subset \X$, the natural morphism $\cal{E} \to j_{\sU, *}\left(\cal{E}|_{\sU}\right)$ is injective.
\end{lemma}
\begin{proof}
The claim is local on $\X$, so we can assume that $\X=\Spf A$ is a connected affine formal $\O_K$-scheme and $\cal{E}=\O_\X$. Since connected open affine subsets form a base of topology on $\X$, it suffices to show that the morphism $\O_\X \to j_{\sU, *} \O_{\sU}$ is injective on global sections. Therefore, it is enough to show that the restriction morphism
\[
r\colon A \to \O_\X(\sU)
\]
is injective. In this case, Lemma~\ref{lemma:connected} ensures that $\X_K$ is connected as well. Then \cite[Lemma 2.1.4]{C} implies that $r_K \colon A_K \to \O_{\X_K}(\sU_K)$ is injective. This implies injectivity of $r$ since $A$ is $\O_K$-flat. 
\end{proof}

\begin{lemma}\label{intersection-reducedfiber} Let $\X$ be an admissible formal $\O_K$-scheme,  $\F\in \bf{Coh}(\X)$ be a reflexive coherent module, and let $\sU\subset \X$ be a dense open subset of $\X$. Suppose that the special fiber $\ov{\X}$ is reduced. Then the natural morphism $\F \to j_{\mathfrak U, *}\F|_\mathfrak U \cap \F_K$ is an isomorphism.
\end{lemma}
\begin{proof}
The question is local, so we can assume that $\X=\Spf A$ is a connected affine formal $\O_K$-scheme and $\sU$ is a dense open formal subscheme of $\X$. In particular, Lemma~\ref{lemma:connected} implies that the generic fiber $\X_K$ is connected. \smallskip

The dual sheaf $\F^{\vee}\coloneqq \ud{\cal{H}om}_\X(\F, \O_\X)$ is coherent, so there is a presentation
\[
\O_\X^n \to \O_\X^m \to \F^{\vee} \to 0.
\]
Thus we can use reflexivity of $\F$ to get a ``co-presentation''
\[
0 \to \F \to \O_\X^m \to \O_\X^n.
\]
Using that the functor $\mathcal G \mapsto j_{\mathfrak U, *}\G|_\mathfrak U \cap \G_K$ is left exact, we get a commutative diagram
\[
\begin{tikzcd}
0 \arrow{r} \arrow{d}& \F \arrow{r} \arrow{d}& \O_\X^m \arrow{r} \arrow{d} & \O_\X^n \arrow{d} \\
0 \arrow{r} & j_{\mathfrak U, *}\F|_\mathfrak U \cap \F_K \arrow{r} & j_{\mathfrak U, *}(\O_\X^m)|_\mathfrak U \cap (\O_\X^m)_K \arrow{r}  &j_{\mathfrak U, *}(\O^n_\X)|_\mathfrak U \cap (\O_\X^n)_K
\end{tikzcd}
\]
that shows that it is sufficient to show the result for $\O_\X^m$ for any $m$. As the functor $\mathcal G \mapsto j_{\mathfrak U, *}\G|_\mathfrak U \cap \G_K$ clearly commutes with finite direct sums, it is enough to show the claim for $\O_\X$. \\

{\it Injectivity of the map $\O_\X \to j_{\mathfrak U, *}\O_\sU$}: This follows from Lemma~\ref{lemma:injectivity-vector-bundles} applied to $\cal{E}=\O_\X$. \\

{\it The map $\O_\X \to j_{\mathfrak U, *}\O_\mathfrak U \cap (\O_\X)_K$ is an isomorphism}: The previous step clearly implies that the map is injective, so we only need to check surjectivity. Since connected open affine subsets form a base of topology on $\X$, it suffices to show that the morphism 
\[
\O_\X(\X) \to \left(j_{\mathfrak U, *}\O_{\sU} \cap (\O_\X)_K\right)(\X)
\]
is surjective. Now note that the natural morphism $\O_{\X_K}(\X_K) \to \O_{\X_K}(\sU_K)$ is injective by connectedness of $\X_K$ and \cite[Lemma 2.1.4]{C}, and the morphism $\O_{\sU}(\sU) \to \O_{\sU_K}(\sU_K)$ is injective by $\O_K$-flatness of $\sU$. Therefore, $\left(j_{\mathfrak U, *}\O_{\sU} \cap (\O_\X)_K\right)(\X)$ is identified with the intersection $\O_{\X}(\sU)\cap A_K$, and we have to show that the natural morphism
\[
A \to \O_\X(\mathfrak U) \cap A_K
\]
is surjective. We pick any element $f\in \O_\X(\mathfrak U) \cap A_K\subset A_K$ and want to show that $f$ lives in $A$. Without loss of generality, we may assume that $f$ is non-zero. Then we note that Lemma~\ref{lemma:enough-constants} implies that there is an element $c\in K^\times$ such that $cf\in A \subset A_K$ and the residue class $\ov{cf}\in A/\m A$ is non-zero. Now we note that, for the purpose of proving $f\in A$, it suffices to show that $|c|\geq 1$. \smallskip

Suppose to the contrary that $|c|<1$, i.e., $c\in \m$. Now we consider the commutative diagram
\[
\begin{tikzcd}
A \arrow{d}{\pi_A} \arrow{r}{r} & \O_{\X}(\sU) \arrow{d}{\pi_{\sU}} \\
A/\m A=\O_{\ov{\X}}(\ov{\X}) \arrow{r}{\ov{r}} & \O_{\X}(\sU)/\m\O_{\X}(\sU)=\O_{\ov{\X}}(\ov{\sU}),
\end{tikzcd}
\]
where $r$ is the natural restriction map. Our assumptions that $\ov{\X}$ is reduced and $\ov{\sU} \subset \ov{\X}$ is dense imply that $\ov{r}$ is injective. Therefore, we conclude that
\[
\ov{r(cf)} = \ov{r}(\ov{cf}) \neq 0 \in \O_{\X}(\sU)/\m\O_{\X}(\sU).
\]
However, this clearly leads to the contradiction since $f\in \O_{\X}(\sU)$ and $c\in \m$. Thus, we conclude that $|c|\geq 1$ finishing the proof. 
\end{proof}

\begin{thm}\label{thm:dualizing-reflexive-formal-schemes} Let $\X$ be a separated admissible formal $\O_K$-scheme of pure dimension $d$. If the generic fiber $\X_K$ is smooth and the special fiber $\ov{\X}$ is geometrically reduced, then 
\begin{enumerate}
    \item the dualizing sheaf $\omega_\X$ is a reflexive coherent $\O_\X$-module, and
    \item the natural morphism $\omega_\X \to j_{\mathfrak U, *}(\omega_\X|_\mathfrak U) \cap (\omega_\X)_K$ is an isomorphism for any open dense formal subscheme $j_{\mathfrak U}\colon \mathfrak U\hookrightarrow \X$. 
\end{enumerate} 
\end{thm}
\begin{proof}
Lemma~\ref{intersection-reducedfiber} implies that it suffices to show that $\omega_\X$ is reflexive. This question is local on $\X$, so we can assume that $\X=\Spf B$ is affine. Then Lemma~\ref{lemma:pure-algebraization} implies that there is a flat, finitely presented $\O_K$-algebra $A$ such that $X=\Spec A$ is of pure relative dimension $d$, $A_K$ is $K$-smooth, and $\wdh{A}\simeq B$. Denote the completion morphism by $c\colon \X \to X$, then Corollary~\ref{cor:complete} implies that $\omega_\X=c^*(\omega_X)$. Since the special fiber of $X$ coincides with the special fiber of $\X$, we conclude that it is geometrically reduced as well. Therefore, Theorem~\ref{thm:dualizing-reflexive-schemes} implies that the $\O_X$-module $\omega_X$ is reflexive. The fact that its pullback $c^*(\omega_X)$ is reflexive boils down to the fact that $M\otimes_A \wdh{A}$ is reflexive over $\wdh{A}$, if $M$ is a coherent and reflexive $A$-module. This is easily seen to hold true, as flatness of the morphism $A\to \wdh{A}$ implies that the natural morphism
\[
\Hom_A (K, A)\otimes_A \wdh{A} \to  \Hom_{\wdh{A}}(K\otimes_A \wdh{A}, \wdh{A})
\]
is an isomorphism for any finitely presented $A$-module $K$. 
\end{proof}

\subsection{Trace map for proper morphisms of separated admissible formal schemes}\label{trace-formal}

In this section, we discuss the definition of the trace map $\rm{Tr}_{\mf}\colon \mf_*\omega_\X \to \omega_\Y$ for a proper morphism $\mf\colon \X \to \Y$ of separated, admissible formal $\O_K$-schemes of the same pure dimension $d$. Of course, it would be nice to develop a good theory of $\mf^!$ functors for admissible formal $\O_K$-schemes such that $\rm{Tr}_{\mf}$ comes from the $(\bf{R}\mf_*, \mf^!)$-adjunction. But to the best of our knowledge such theory is not present in the literature and we do not develop this formalism here. \smallskip 

Throughout this section, we fix a base formal scheme $\S=\Spf \O_K$ for a complete rank-$1$ valuation ring $\O_K$ with a pseudo-uniformizer $\varpi$ and fraction field $K$. \smallskip

Before starting the construction, we recall that $\omega_\Y\coloneqq \lim_n \omega_{\Y_n}$ comes with the canonical $\O_\X$-linear base-change morphisms $\omega_\Y \to \omega_{\Y_n}$ for each $n\geq 0$. We denote the induced $\O_{\Y_n}$-linear morphisms by $\rm{BC}^n_{\omega_\Y}\colon \omega_\Y\otimes_{\O_\Y} \O_{\Y_n} \to \omega_{\Y_n}$. Analogously, we have canonical morphisms $\rm{BC}^n_{\mf_*\omega_\X}\colon \mf_*\,\omega_\X \otimes_{\O_\Y} \O_{\Y_n} \to \mf_{n, *}\,\omega_{\X_n}$.

\begin{lemma}\label{lemma:trace-coherent} Let $\mf\colon \X \to \Y$ be a proper morphism of separated admissible formal $\O_K$-schemes of the same pure dimension $d$. Then there is a unique trace map $\rm{Tr}_{\mf}\colon \mf_*\,\omega_\X \to \omega_\Y$ such that, for any $n\geq 0$, the diagram 
\begin{equation}\label{diagram:1}
\begin{tikzcd}[column sep =8ex, row sep = 8 ex]
\mf_*\omega_\X \otimes_{\O_\Y} \O_{\Y_n} \arrow[d, swap, "\rm{Tr}_{\mf}\otimes_{\O_\Y} \O_{\Y_n}"] \arrow{r}{\rm{BC}^n_{\mf_*\omega_{\X}}} & \mf_{n,*}\omega_{\X_n} \arrow{d}{\mathcal {H}^{-d}(\rm{Tr}_{\mf_n})}\\
\omega_\Y  \otimes_{\O_\Y} \O_{\Y_n}  \arrow{r}{\rm{BC}^n_{\omega_\Y}} & \omega_{\Y_n}, \ 
\end{tikzcd}
\end{equation}
where $\rm{Tr}_{\mf_n}\colon \bf{R}\mf_{n, *}\,\omega^\bullet_{\X_n} \to \omega^\bullet_{\Y_n}$ is the trace map in coherent duality, is commutative. 
\end{lemma}
\begin{proof}
The uniqueness part is easy. (The proof of) Corollary~\ref{cor:dualizing-module-flat} implies that $\rm{BC}^n_{\omega_\Y}$ is injective. Then $\rm{Tr}_{\mf}\otimes_{\O_\Y} \O_{\Y_n}$ is uniquely defined from the diagram~(\ref{diagram:1}). Using that $\omega_{\Y}=\lim_n \omega_{\Y_n}$, we conclude that $\rm{Tr}_{\mf}=\lim_n \rm{Tr}_{\mf} \otimes_{\O_\Y} \O_{\Y_n}$ is also unique. \smallskip

Now we show existence. We recall that $\mf_*$ commutes with all limits, so the natural morphism $\mf_*\,\omega_\X\, \to \lim_n \mf_*\,\omega_{\X_n}$ is an isomorphism. Thus, we can define $\rm{Tr}_{\mf}\colon \mf_*\,\omega_\X \to \omega_\Y$ as $\lim_n \mathcal H^{-d}(\rm{Tr}_{\mf_n})$. In order for this formula to make sense, we need to show that $\rm{Tr}_{\mf_n}$ are compatible for different $n$. This follows from \cite[\href{https://stacks.math.columbia.edu/tag/0B6J}{Tag 0B6J}]{stacks-project} as $\Y_{n-1}$ and $\X_n$ are tor-independent over $\Y_n$ by $\O_K$-flatness of both $\X$ and $\Y$. This construction defines the desired map $\rm{Tr}_\mf\colon \mf_*\,\omega_\X \to \omega_\Y$.  
\end{proof}

\begin{lemma}\label{almost-commutative-coherent} Let $\mf\colon \X'' \to \X'$ and $\mg\colon \X' \to \mathfrak X$ be two proper morphisms of separated admissible formal $\O_K$-schemes of the same pure dimension $d$. Then the diagram
\[
\begin{tikzcd}[column sep = 10 ex]
\mg_*\,\mf_*\omega_{\X''} \arrow{r}{\mg_*(\rm{Tr}_{\mf})} \arrow{d}{\wr} & \mg_*\, \omega_{\X'} \arrow{r}{\rm{Tr}_{\mg}} & \omega_\X \\
(\mg\circ \mf)_*\,\omega_{\X''} \arrow[rru, swap, "\rm{Tr}_{\mg\circ \mf}"], & & 
\end{tikzcd}
\]
where the vertical arrow is the canonical identification of $\mg_* \circ \mf_*$ with $(\mg\circ \mf)_*$, is commutative. 
\end{lemma}
\begin{proof}
We firstly note the canonical identification $\bf{R}\mg_{n, *} \circ \bf{R}\mf_{n, *} \simeq \bf{R}(\mg \circ \mf)_{n, *}$ implies via the usual adjunction properties that 
\[
    \rm{Tr}_{\mg_n}\circ \bf{R}\mg_{n, *}\left(\rm{Tr}_{\mf_n}\right) \simeq \rm{Tr}_{(\mg \circ \mf)_n}.
\]

Now we use the formal properties of derived limits to write: \begin{align*}
\rm{Tr}_{\mg} \circ \mg_*\left(\rm{Tr}_{\mf}\right)& \simeq \lim_n \mathcal H^{-d}\left(\rm{Tr}_{\mg_n}\right) \circ  \mg_*\left(\lim_n \mathcal H^{-d}\left(\rm{Tr}_{\mf_n}\right) \right)  \\
& \simeq \lim_n \mathcal H^{-d}\left(\rm{Tr}_{\mg_n}\right)\circ \lim_n\left(\mg_{n, *}\left(\mathcal H^{-d}\left( \rm{Tr}_{\mf_n}\right) \right) \right)  \\
&\simeq \lim_n \left(\mathcal H^{-d}\left( \rm{Tr}_{\mg_n}\right) \circ \mathcal H^{-d} \left(\bf{R}\mg_{n, *}\left(\rm{Tr}_{\mf_n}\right) \right)  \right)\\
&\simeq\lim_n\left(\mathcal H^{-d} \left(\rm{Tr}_{\left(\mg\circ \mf\right)_n} \right) \right) \\
&\simeq \rm{Tr}_{\mg\circ \mf} \ . 
\end{align*}
We emphasize that in the above equalities we crucially used that all complexes $\omega^\bullet_{\X_n}, \omega^\bullet_{\X'_n}$ and $\omega^\bullet_{\X''_n}$ are concentrated in degrees $[-d, 0]$. 
\end{proof}

\subsection{Dualizing complexes on smooth, separated rigid-analytic spaces}

Throughout this section, we fix a base formal scheme $\S=\Spf \O_K$ for a complete rank-$1$ valuation ring $\O_K$ with a pseudo-uniformizer $\varpi$ and {\it algebraically closed} fraction field $K$. For an admissible formal $\O_K$-scheme $\X$, we denote its smooth locus by $\X^{\rm{sm}}$. \smallskip

The reason we need to assume that $K$ is algebraically closed is due to the following fact:

\begin{thm}\label{thm:cofinal-reduced-family} (Bosch--L\"utkebohmert--Raynaud) Let $\O_K$ be a complete rank-$1$ valuation ring with algebraically closed fraction field $K$. Then the generic fiber functor
\[
\X \mapsto \X_K
\]
gives rise to an equivalence between 
\begin{enumerate}
    \item the category of admissible formal $\O_K$-schemes with reduced special fiber, localized by rig-isomorphisms, and
    \item the category of qcqs reduced rigid-analytic $K$-spaces.
\end{enumerate}
\end{thm}
\begin{proof}
    We first note that \cite[Theorem 4.1]{BL1} implies that the functor for the generic fiber induces an equivalence between the category of admissible formal $\O_K$-schemes, localized by admissible blow-ups, and the category of qcqs rigid-analytic $K$-spaces. Then $\X \mapsto \X_K$ induces equivalence of the desired full subcategories due to Lemma~\ref{lemma:special-fiber-reduced} and Theorem~\ref{thm:reduced-fiber-theorem}.
\end{proof}

\begin{rmk} In what follows, we will freely (and usually without reference) use Lemma~\ref{lemma:proper-adic-formal} and Corollary~\ref{cor:pure-dimension-rigid-formal}. The former result says that any formal model of a proper morphism is also necessarily proper. The latter result says that an admissible formal $\O_K$-model with generic fiber of pure dimension $d$ is itself pure of dimension $d$. 
\end{rmk}

\begin{lemma}\label{lemma:dualizing-forms} Let $\X$ be a separated admissible formal $\O_K$-scheme with smooth generic fiber $\X_K$ of pure dimension $d$, and geometrically reduced special fiber $\ov{\X}$. Then there is a canonical isomorphism $s_{\X}\colon \Omega^d_{\X_K} \to (\omega_\X)_K$ that coincides with the generic fiber of $r_{\X^{\text{sm}}}\colon \wdh{\Omega}_{\X^\text{sm}}^d \to \omega_{\X^\text{sm}}$ on $(\X^{\text{sm}})_K$ (see Theorem~\ref{thm:dualizing-complex}). 
\end{lemma}
\begin{proof}
{\it Step 1: Uniqueness of $s_{\X}$}. First of all, we can assume that $\X_K$ is connected by Lemma~\ref{lemma:connected}. Furthermore, we note that $\X^{\rm{sm}}$ is non-empty due to the assumption that $\ov{X}$ is geometrically reduced. Now suppose we have two {\it isomorphisms} $f, g\colon \Omega^d_{\X_K} \xr{\sim} (\omega_\X)_K$ extending the isomorphism
\[
(r_{\X^{\text{sm}}})_K\colon (\X^{\text{sm}})_K\colon \Omega^d_{\X^{\rm{sm}}_K} \xr{\sim} (\omega_{\X^{\rm{sm}}})_K.
\]
This, in particular, implies that $(\omega_\X)_K$ is a line bundle. So $\varphi\coloneqq f-g$ is a morphism of line bundles, and thus its vanishing locus $\rm{V}(\varphi)$ is a Zariski-closed subset of $\X_K$. By construction,  $\rm{V}(\varphi)$ contains a non-empty, quasi-compact open $(\X^{\text{sm}})_K$, so \cite[Lemma 2.1.4]{C} implies that $\rm{V}(\varphi)=\X_K$.   \smallskip

{\it Step 2: Existence of $s_{\X}$}. By Step~$1$, we can construct $s_{\X}$ locally on $\X$, so we can assume that $\X=\Spf B$ is affine. Then Lemma~\ref{lemma:pure-algebraization} implies that there is a flat, finitely presented $\O_K$-algebra $A$ such that $X=\Spec A$ is of pure relative dimension $d$, $A_K$ is $K$-smooth, and $\wdh{A}\simeq B$. Let $\omega_{A}$ denote the coherent $A$-module corresponding to $\omega_{X}$, the same for $\omega_B$ and $\omega_\X$. Then the claim follows from the sequence of isomorphisms\footnote{In the formula below, we write $\Omega^d_{B_K/K}$ for the module of continuous differentials, see \cite[(1.6.2)]{H3}.}
\[
\omega_B \otimes_B B_K \simeq \omega_A \otimes_A B_K \simeq (\omega_A)_K\otimes_{A_K} B_K \simeq \Omega^d_{A_K/K} \otimes_{A_K} B_K \simeq \Omega^d_{B_K/K}
\]
satisfying the desired compatibility over $\X^{\rm{sm}}$. Indeed, the first isomorphism follows from Corollary~\ref{cor:complete}, the second isomorphism is formal, the third one follows from Lemma~\ref{lemma:smooth-morphism-dualizing-complex}, Corollary~\ref{cor:base-change-dualizing-object}, and $K$-smoothness of $A_K$, and the third isomorphism follows from the classical isomorphism $\Omega^d_{A_K/K} \otimes_{A_K} B_K \simeq \Omega^d_{B_K/K}$.
\end{proof}

\begin{rmk} The proof of Lemma~\ref{lemma:dualizing-forms} actually shows a stronger claim that $\Omega^d_X[d]$ is canonically isomorphic to $(\omega^\bullet_\X)_K$. 
\end{rmk}

For the rest of the section, we fix a qcqs separated smooth rigid-analytic $K$-space $X$ of pure dimension $d$. The next goal is to show that the isomorphisms $s_\X$ are ``compatible'' for different admissible $\O_K$-models $\X$ of $X$. \smallskip

Recall that, for any rig-isomorphism $\mf\colon \X' \to \X$ of admissible formal $\O_K$-schemes, there is a canonical isomorphism $\alpha_{\X', \X}\colon \left(\mf_* \omega_{\X'}\right)_K \xr{\sim} \left(\omega_{\X'}\right)_K$ due to Lemma~\ref{generic-fiber-commutes} and coherence of $\omega_\X$ (see Theorem~\ref{thm:dualizing-complex}). 

\begin{lemma}\label{trace-first} Let $X$ be a qcqs separated smooth rigid-analytic $K$-space $X$ of pure dimension $d$, and $\mf\colon \X' \to \X$ a morphism of two admissible formal $\O_K$-models of $X$ with geometrically reduced special fibers, i.e. the diagram
\[
\begin{tikzcd}
X \arrow{r} \arrow{rd}& \X'\arrow{d}{\mf} \\
& \X
\end{tikzcd}
\]
is commutative. Then the diagram
\[
\begin{tikzcd}
\Omega^d_X \arrow{r}{s_{\X'}} \arrow{rd}{s_{\X}} & \left(\omega_{\X'}\right)_K \arrow{d}{\left(\rm{Tr}_\mf\right)_K} \\
&\left(\omega_\X\right)_K 
\end{tikzcd}
\]
is commutative, where $(\rm{Tr}_\mf)_K$ is the composition 
\[
\left(\omega_{\X'}\right)_K \xr{\alpha_{\X', \X}^{-1}} \left(\mf_* \omega_{\X'}\right)_K \xr{\left(\rm{Tr}_{\mf}\right)_K} (\omega_\X)_K.
\] 
\end{lemma}
\begin{proof}
    By Lemma~\ref{lemma:connected}, we can assume that $X$ is connected. Then we consider the composition
    \[
    \begin{tikzcd}
    \Omega^d_X \arrow{r}{s_{\X'}} \arrow[rrr, bend left, "g"]& \left(\omega_{\X'}\right)_{K} \arrow{r}{\left(\rm{Tr}_{\mf}\right)_K} & \left(\omega_\X\right)_K \arrow{r}{s_\X^{-1}} & \Omega^d_X.
    \end{tikzcd}
    \]
    This is a morphism of line bundles on $X$, and we wish to show that this is the identity morphism. Now \cite[Lemma 2.1.4]{C} and connectedness of $X$ imply that this can be checked on any non-empty open $U\subset X$. Then Corollary~\ref{cor:iso-on-open} implies that we can take $U=\sU_K$ where $\sU\subset \X$ is a dense open such that $\sU\subset \X^{\rm{sm}}$ and $\mf|_{\mf^{-1}(\sU)}\colon \mf^{-1}(\sU) \to \sU$ is an isomorphism (that must be equal to $\rm{id}_\sU$ because it is $\rm{id}_U$ on the generic fiber). Then clearly $g|_U=\rm{id}$ finishing the proof. 
\end{proof}

\begin{thm-def}\label{thm-def:rigid-analytic-trace} Let $f \colon X \to Y$ be a proper morphism of qcqs separated smooth rigid-analytic $K$-spaces of pure dimension $d$. Then there is a unique {\it trace morphism} 
\[
    \rm{Tr}_{f}\colon f_*\left(\Omega^d_{X}\right) \to \Omega^d_Y
\]
such that, for each morphism $\mf\colon \X \to \Y$ of admissible formal $\O_K$-models with geometrically reduced special fibers satisfying $\mf_K=f$, the following diagram
\[
\begin{tikzcd}[column sep = 5em]
f_*\left(\Omega^d_{X}\right) \arrow{r}{f_*\left(s_{\X}\right)} \arrow{d}{\rm{Tr}_f} & f_* \left(\omega_{\X}\right)_K \arrow{d}{\left(\rm{Tr}_{\mf}\right)_K} \\
\Omega^d_{Y} \arrow{r}{s_{\Y}} &\left(\omega_{\Y}\right)_K
\end{tikzcd}
\]
commutes (see Lemma~\ref{lemma:trace-coherent} for the definition of $\rm{Tr}_{\mf}$).
\end{thm-def}
\begin{proof}
Uniqueness is clear because $s_\X$ and $s_\Y$ are isomorphisms. In order to prove existence, Theorem~\ref{thm:cofinal-reduced-family} implies that it suffices to show that, for a commutative diagram of admissible formal $\O_K$-schemes with reduced special fibers 
\[
\begin{tikzcd}
\X \arrow{d}{\mf} & \X' \arrow{l}{\pi'} \arrow{d}{\mf'} \\
\Y & \Y' \arrow{l}{\pi} 
\end{tikzcd}
\]  
satisfying the equalities $\mf_K=f=\mf'_K$, $\pi_K=\rm{id}_Y$, and $\pi'_K=\rm{id}_X$, the diagram
\[
\begin{tikzcd}[row sep = 7ex, column sep = 11ex]
f_*\left(\Omega^d_{X}\right) \arrow[rr, bend left, "f_{*}\left(s_{\X}\right)"] \arrow{r}{f_*\left(s_{\X'}\right)} & f_*\left(\omega_{\X'}\right)_K \arrow{d}{\left(\rm{Tr}_{\mf'}\right)_K} \arrow{r}{f_*\left(\left(\rm{Tr}_{\pi'}\right)_K\right)} & f_*\left(\omega_{\X}\right)_K \arrow{d}{\left(\rm{Tr}_{\mf}\right)_K}\\
\Omega^d_{Y} \arrow[rr, bend right, "s_{\Y}"] \arrow{r}{s_{\Y'}} & \left(\omega_{\Y'}\right)_K \arrow{r}{\left(\rm{Tr}_{\pi}\right)_K} & \left(\omega_{\Y}\right)_K
\end{tikzcd}
\]
is commutative. Indeed, the top and bottom ``triangles'' commute by Lemma~\ref{trace-first}, while the right square commutes by Lemma~\ref{almost-commutative-coherent}.
\end{proof}

The last thing we discuss is the explicit model for the trace map in the case of a finite \'etale morphism $f\colon X \to Y$ of smooth rigid-analytic $K$-spaces of the same pure dimension $d$. We use (the easy version of) the projection formula for finite morphisms and \'etaleness of $f$ to get a canonical isomorphism
\[
f_*\O_X \otimes_{\O_Y} \Omega^d_Y \xr{\sim} f_*\left(f^*\Omega^d_Y\right) \xr{\sim} f_* \Omega^d_X.
\]

In order to define the explicit trace map $f_*\O_X \otimes_{\O_Y} \Omega^d_Y \simeq f_*\Omega^d_X \to \Omega^d_Y$, we first define the trace morphism $\rm{Tr}_f\colon f_*\O_X \to \O_Y$. Locally on $Y$, $f$ looks like $f\colon \Spa (B, B^\circ) \to \Spa(A, A^\circ)$ where $B$ is a finite \'etale $A$-algebra. Thus, $B$ is a finite projective $A$-module, so we have the trace morphism $\rm{Tr}_{B/A} \colon B \to A$. Since $\rm{Tr}_{B/A}$ commutes with flat base change, this morphism glues to a well-defined morphism
\[
\rm{Tr}_f\colon f_*\O_X \to \O_Y.
\]
\begin{defn} For $f\colon X \to Y$ as above, the {\it explicit trace map} $\rm{Tr}^{\rm{expl}}_f\colon f_*\,\Omega^d_X \to \Omega^d_Y$ is the morphism
\[
\rm{Tr}_f \otimes_{\O_Y} \rm{id}_{\Omega^d_Y}\colon f_*\,\Omega^d_X \simeq f_*\O_X \otimes_{\O_Y} \Omega^d_Y \to \O_Y \otimes_{\O_Y} \Omega^d_Y \simeq \Omega^d_Y.
\]
\end{defn}

Locally on $Y$, in \'etale coordinates $U=\Spa(A, A^\circ) \xr{(z_1, \dots, z_d)} \bf{D}^d$, the maps looks like 
\[
    f \rm{d} z_1 \wedge \dots \wedge \rm{d} z_d \mapsto \rm{Tr}_{B/A}(f)\rm{d} z_1 \wedge \dots \wedge \rm{d} z_d,
\]
where $\rm{Tr}_{B/A}$ is the trace map for $A \to B$ that comes from $\Spa(B, B^\circ) =f^{-1}\left(U\right) \to \Spa(A, A^\circ)=U$.

\begin{rmk} We note that, if $\mf\colon \X \to \Y$ is a finite flat morphism of admissible formal $\O_K$-schemes, one can similarly define the (explicit) trace morphism
\[
\rm{Tr}_{\mf}\colon \mf_{*}\O_\X \to \O_\Y
\]
by gluing $\rm{Tr}_{B/A}$ over affines $\Spf A\subset \Y$. 
\end{rmk}

\begin{lemma}\label{lemma:trace-explicit-rigid} Let $f\colon X \to Y$ be a finite \'etale morphism of qcqs separated smooth rigid-analytic $K$-spaces of pure dimension $d$. Then the morphisms 
\[
\rm{Tr}_{f}, \rm{Tr}^{\rm{expl}}_f \colon f_*\,\Omega^d_X \to \Omega^d_Y
\]
coincide.
\end{lemma}
\begin{proof}
{\it Step 1: Reduce to the case when $f$ admits a finite, locally free $\O_K$-model $\mf\colon \X \to \Y$ with smooth $\X$ and $\Y$}. We first note $\rm{Tr}_{f}$ and $\rm{Tr}_{f}^{\rm{expl}}$ are morphisms between vector bundles on $Y$, so \cite[Lemma 2.1.4]{C} guarantees that it suffices to check that these two maps coincide over some open subset $V\subset Y$ that meets each connected (or, equivalently, irreducible) component of $Y$. Therefore, we may assume that $Y=\Spa(A, A^\circ)$ is affinoid, and thus $X=\Spa(B, B^\circ)$ is affinoid as well.\smallskip

In this situation, \cite[Corollary 6.4.1/5]{BGR} implies that both $A^\circ$ and $B^\circ$ are topologically finite type and the morphism $A^\circ \to B^\circ$ is finite. Furthermore, Corollary~\ref{cor:normalization} ensures that $\Spf A^\circ$ and $\Spf B^\circ$ have reduced special fibers. Therefore, Corollary~\ref{cor:super-good-open} (applied to $\mf\colon \Spf B^\circ \to \Spf A^\circ$) ensures that we can find an open affine $\sU \subset \Spf A^\circ$ such that $\mf|_{\mf^{-1}(\sU)} \colon \mf^{-1}(\sU) \to \sU$ is finite flat, and both $\sU$ and $\mf^{-1}(\sU)$ are $\O_K$-smooth. \smallskip

Furthermore, Corollary~\ref{cor:normalization} implies that $\sU=\Spf B'^{\circ}$ and $\mf^{-1}(\sU) = \Spf A'^{\circ}$ for some $K$-affinoid algebras $A'$ and $B'$. In other words, after replacing $X$ with $\sU_K$, we can further assume that $B^\circ$ is a finite, locally free $A^\circ$-algebra, and both $A^\circ$ and $B^\circ$ are (topologically) $\O_K$-smooth. \smallskip

{\it Step $2$. Computation of $\rm{Tr}_f$ when $f$ admits a finite, locally free model $\mf\colon \X \to \Y$ with smooth $\X$ and $\Y$}. By construction, it suffices to show that the diagram 
\[
\begin{tikzcd}
\mf_*\left(\mf^*\wdh{\Omega}^d_{\Y}\right) \arrow{r} \arrow{d}{\wr} & \mf_*\,\wdh{\Omega}^d_{\X}   \arrow{d}{\rm{Tr}_{\mf}} \\
\mf_*\O_{\X} \otimes_{\O_\Y} \wdh{\Omega}^d_{\Y} \arrow{r}{\rm{Tr}_{\mf}\otimes \rm{id}}  & \wdh{\Omega}^d_{\Y}
\end{tikzcd}
\]
commutes. This follows from Lemma~\ref{lemma:finite-flat-trace-scheme} (and Lemma~\ref{lemma:trace-coherent}). 
\end{proof}

\newpage

\section{Faltings' Trace Map}\label{section:construction-faltings-trace}
\subsection{Idea of the construction}
Throughout this section, we fix a complete rank-$1$ valuation ring $\O_C$ of mixed characteristic $(0, p)$ with the algebraically closed fraction field $C$.\smallskip

The main goal of this section is the construction of Faltings' trace map 
\[
    \Tr_{F, \X} \colon \bf{R}\nu_*\O_X^{+, a}/p \to \omega^{\bullet, a}_{\X_0}(-d)[-2d]
\]
for an admissible separated formal $\O_C$-scheme $\X$ with smooth generic fiber $X=\X_C$ of pure dimension $d$ and reduced special fiber.  \smallskip 

In case $\X$ is also proper, this construction will be combined with the almost version of the Grothendieck duality on $\X_0$ to get the global trace map
\[
\mathbf{R}\Gamma(X, \O_X^{+, a}/p) \to \left(\O_C/p(-d)[-2d]\right)^a.
\]
The construction of this trace map is one of the crucial steps in our proof of Poincar\'e Duality. \smallskip

Now we summarize the main ideas behind our construction of Faltings' trace map: 

\begin{enumerate}
    \item We note that \cite[Theorem 6.13.5]{Z3} and Theorem~\ref{thm:dualizing-complex} ensure that $\bf{R}\nu_* \O_X^{+, a}/p\in \bf{D}^{[0, d]}_{acoh}(\X_0)^a$ and $\omega^\bullet_{\X_0} \in \bf{D}^{[-d, 0]}_{coh}(\X_0)$. So, for the purpose of constructing Faltings' trace, it suffices to construct a map
    \[
    \rm{Tr}^d_{F, \X}\colon \rm{R}^d\nu_*\O_X^{+, a}/p \to \omega^a_{\X_0}(-d).
    \]
    Furthermore, we also have $\bf{R}\nu_*\wdh{\O}_X^{+, a}\in \bf{D}^{[0, d]}_{acoh}(\X)^a$ and $\omega^\bullet_{\X} \in \bf{D}^{[-d, 0]}_{coh}(\X)$, so it is actually sufficient to construct an ``integral'' morphism
    \[
    \rm{Tr}_{F, \X}^{d,+}\colon \rm{R}^d\nu_* \wdh{\O}_X^{+, a} \to \omega_{\X}(-d)^a
    \]
    as then $\rm{Tr}^d_{F, \X}$ can be constructed as the composition
    \[
    \rm{R}^d\nu_*\O_X^{+, a}/p \xr{\sim}  \rm{R}^d\nu_*\wdh{\O}_X^{+, a} \otimes_{\O_\X} \O_{\X_0} \xr{\rm{Tr}_{F, \X}^{d,+} \otimes_{\O_\X} \rm{Id}} \omega_\X(-d)^a \otimes_{\O_\X} \O_{\X_0} \xr{\rm{BC}_{\omega_\X(-d)}^a} \omega^a_{\X_0}(-d).
    \]
    
    \item Then we actually define (under the assumption that $\X$ has reduced special fiber) an {\it honest} morphism of $\O_\X$-modules
    \[
     \rm{Tr}_{F, \X}^{d,+}\colon \rm{R}^d\nu_*\,\wdh{\O}_X^+ \to \omega_{\X}(-d).
    \]
    The key idea is to use Theorem~\ref{thm:dualizing-reflexive-formal-schemes} to reduce the question of constructing $\rm{Tr}_{F, \X}^{d, +}$ to the question of constructing this morphism on the smooth locus and on the generic fiber in a compatible way. On the smooth locus, we use a variation of the map from \cite{BMS1} and, on the generic fiber, we use a variation of the analogous map from \cite{Sch1}.
\end{enumerate}

%The rest of the section is devoted to fulfilling the plan above and constructing the desired trace map in full generality. 

\subsection{The BMS map for general admissible formal schemes}

Throughout this section, we fix an admissible formal $\O_C$-scheme $\X$ with generic fiber $X=\X_C$. \smallskip

The main goal of this section is to recall the construction of the maps 
\[
\Phi^n_\X\colon \wdh{\Omega}^n_{\X}\{-n\} \to \frac{\rm{R}^n\nu_* \wdh{\O}_X^+}{\left(\rm{R}^n\nu_* \wdh{\O}_X^+\right)[(\zeta_p-1)^\infty]} = \frac{\rm{R}^n\nu_* \wdh{\O}_X^+}{\left(\rm{R}^n\nu_* \wdh{\O}_X^+\right)[p^\infty]} 
\]
essentially introduced in \cite{BMS1}. The construction in \cite{BMS1} is written under the assumption that the formal model $\X$ is smooth, however it will be convenient for us to use this construction for more general $\X$. For brevity, we introduce the following notation: %And it will be important for our purposes to define this map for non-smooth admissible formal $\O_C$-schemes. \smallskip 

\begin{defn} For $\X$ as above and an integer $n$, we define 
\[
\widetilde{\rm{R}^n\nu_* \wdh{\O}_X^+} \coloneqq \frac{\rm{R}^n\nu_* \wdh{\O}_X^+}{\left(\rm{R}^n\nu_* \wdh{\O}_X^+\right)[(\zeta_p-1)^\infty]} \in \bf{D}(\X).
\]
\end{defn}

In order to define $\Phi_\X^n$, we will need to use the cup-product, so we start the section by discussing some of its properties. Firstly, we recall that \cite[\href{https://stacks.math.columbia.edu/tag/0FKU}{Tag 0FKU}]{stacks-project} construct, for each integer $n\geq 0$, the cup-product map $\cup^n \colon \left(\rm{R}^1\nu_*\wdh{\O}^+_X\right)^{\otimes n} \xr{\cup^n} \rm{R}^n\nu_*\wdh{\O}^+_X$ that can be seen to (uniquely) descend to the morphism
\[
\cup^n \colon \left(\widetilde{\rm{R}^1\nu_*\wdh{\O}^+_X}\right)^{\otimes n} \xr{\cup^n} \widetilde{\rm{R}^n\nu_*\wdh{\O}^+_X}.
\]

\begin{lemma}\label{lemma:cup-product-well-defined} There is a unique $\O_\X$-linear morphism $\cup^n \colon \bigwedge^n\left( \widetilde{\rm{R}^1\nu_*\wdh{\O}^+_X}\right) \to \widetilde{\rm{R}^n\nu_*\wdh{\O}^+_X}$ such that the diagram
\[
\begin{tikzcd}
\left(\widetilde{\rm{R}^1\nu_*\wdh{\O}^+_X}\right)^{\otimes n} \arrow{d} \arrow{rd}{\cup^n} & \\
\bigwedge^n \left(\widetilde{\rm{R}^1\nu_*\wdh{\O}^+_X}\right) \arrow{r}{\cup^n} & \widetilde{\rm{R}^n\nu_*\wdh{\O}^+_X}
\end{tikzcd}
\]
commutes.
\end{lemma}
\begin{proof}
    By the universal property of the wedge powers, it suffices to show that $x\cup x=0$ for any local section $x\in \widetilde{\rm{R}^1\nu_*(\wdh{\O}^+_X)}$. The cup product is always anti-commutative\footnote{For this implication, use that the braiding morphism on $\wdh{\O}^+_X[0]$ is equal to $\rm{id}$, and the braiding morphism on $\rm{R}^1\nu_*\wdh{\O}_X^+[-1]$ is equal to $-\rm{id}$.} by \cite[\href{https://stacks.math.columbia.edu/tag/0FP5}{Tag 0FP5}]{stacks-project}, i.e. $x\cup x = -x\cup x$ for any $x\in \widetilde{\rm{R}^1\nu_*(\wdh{\O}^+_X)}$. Therefore, if $2$ is invertible in $\bf{Z}_p$, this implies that $x\cup x = 0$. Therefore, we only need to deal with $p=2$. In this case, $-2=\zeta_2-1$ so $\widetilde{\rm{R}^n\nu_*\wdh{\O}^+_X}$ are all $2$-torsion free. Thus, $x\cup x = -x \cup x$ implies that $x\cup x=0$. 
\end{proof}

Now we start the construction of $\Phi_\X^n$. Functoriality of the cotangent complex provides us with the map
\[
L_{\O_\X/\bf{Z}_p} \to \bf{R}\nu_* L_{\wdh{\O}^+_{X}/\bf{Z}_p}
\]
that passing to the derived $p$-adic completions gives the map
\[
\Phi'_\X\colon \wdh{L}_{\O_\X/\bf{Z}_p} \to \bf{R}\nu_*\wdh{L}_{\wdh{\O}^+_{X}/\bf{Z}_p}.
\]
The morphism $\Phi_\X^1$ essentially comes from $\Phi'_\X$ via the following two computations: 

\begin{lemma}\label{lemma:ident-1} There is a natural isomorphism $\wdh{\Omega}^1_\X \simeq \cal{H}^0(\wdh{L}_{\O_\X/\bf{Z}_p})$. 
\end{lemma}
\begin{proof}
We consider the morphisms $\X \to \Spf \O_C$ and $\Spf \O_C \to \Spf \Z_p$ that give rise to the associated distinguished triangle
\[
L_{\O_C/\Z_p} \otimes^{L}_{\O_C} \O_\X \to L_{\O_\X/\Z_p} \to L_{\O_\X/\O_C}
\]
that, after taking its derived $p$-adic completion, induces the distinguished triangle
\begin{equation}\label{formula:dist-triangle}
\wdh{L}_{\O_C/\Z_p} \wdh{\otimes}^{L}_{\O_C} \O_\X \to \wdh{L}_{\O_\X/\Z_p} \to \wdh{L}_{\O_\X/\O_C}.
\end{equation}
Now \cite[Remark 3.19]{bhatt-arizona} implies that $\wdh{L}_{\O_C/\Z_p}\simeq \O_C\{1\}[1]$ is a free $\O_C$-module concentrated in degree $-1$. Moreover, $\O_\X$ is clearly a $p$-adically derived complete module on $\X$, so $\wdh{L}_{\O_C/\Z_p} \wdh{\otimes}^L_{\O_C}\O_\X \simeq \O_\X\{1\}[1]$. Thus, the distinguished triangle~(\ref{formula:dist-triangle}) implies that  $\mathcal{H}^0(\wdh{L}_{\O_\X/\Z_p}) \xr{\sim} \mathcal{H}^0(\wdh{L}_{\O_\X/\O_C})$. Finally, \cite[Section 7.2.8]{GR} guarantees that the natural morphism $\wdh{\Omega}^1_{\X} \to \mathcal{H}^0(\wdh{L}_{\O_\X/\O_C})$ is an isomorphism. Combining these two facts, we get the natural isomorphism 
\[
   \wdh{\Omega}^1_\X \xr{\sim} \mathcal{H}^0(\wdh{L}_{\O_\X/\bf{Z}_p}).     \qedhere
\]
\end{proof}

\begin{lemma}\label{lemma:ident-2} There is a natural isomorphism 
\[
\mathcal H^0\left(\bf{R}\nu_*\wdh{L}_{\wdh{\O}^+_{X}/\bf{Z}_p}\right) \simeq \rm{R}^1\nu_*\wdh{\O}^+_X\{1\}.
\] 
\end{lemma}
\begin{proof}
 We do a similar trick here: we consider the following morphisms of ringed spaces:
\[
\left(X_{\proet}, \wdh{\O}^+_X\right) \to \left(\Spa(C, \O_C)_{\rm{an}}, \O_C\right) \to \left(\Spa(\bf{Q}_p, \Z_p)_{\rm{an}}, \Z_p\right).
\]
This induces the distinguished triangle 
\[
L_{\O_C/\Z_p} \otimes^{L}_{\O_C} \wdh{\O}^+_X \to L_{\wdh{\O}^+_X/\Z_p} \to L_{\wdh{\O}^+_X/\O_C}
\]
that, after taking derived $p$-adic completion, induces the distinguished triangle
\[
\wdh{L}_{\O_C/\Z_p} \wdh{\otimes}^{L}_{\O_C} \wdh{\O}^+_X \to \wdh{L}_{\wdh{\O}^+_X/\Z_p} \to \wdh{L}_{\wdh{\O}^+_X/\O_C}.
\]
Using derived completeness of $\wdh{\O}^+_X$ (see \cite[Remark 5.5]{BMS1}), the same argument as in Lemma~\ref{lemma:ident-1} implies that $\wdh{L}_{\O_C/\Z_p} \wdh{\otimes}^{L}_{\O_C} \wdh{\O}_X^+ \simeq \wdh{\O}^+_X\{1\}[1]$. Now \cite[Corollary 3.28]{bhatt-arizona} guarantees\footnote{\cite[Corollary 3.28]{bhatt-arizona} is written under the assumption that $X$ is smooth. However, the same proof works as long as one uses \cite[Proposition 4.8]{Sch1} in place of \cite[Corollary 4.7]{Sch1}.} that $\wdh{L}_{\wdh{\O}^+_X/\O_C} \simeq 0$, so we have a natural isomorphism $\wdh{\O}^+_X\{1\}[1] \xr{\sim} \wdh{L}_{\wdh{\O}^+_X/\Z_p}$. Thus, we get an isomorphism  
\[
\rm{R}^1\nu_*\wdh{\O}^+_X\{1\}  \simeq \mathcal H^0\left(\bf{R}\nu_*\wdh{L}_{\wdh{\O}^+_{X}/\bf{Z}_p}\right)=\rm{R}^0\nu_*\wdh{L}_{\wdh{\O}^+_{X}/\bf{Z}_p}.  \qedhere
\]
\end{proof}

Now we use Lemma~\ref{lemma:ident-1} and Lemma~\ref{lemma:ident-2} to define 
\[
\Phi'^1_\X\coloneqq \mathcal H^0(\Phi'_\X)\{-1\} \colon \wdh{\Omega}^1_\X\{-1\} \to \rm{R}^1\nu_*\wdh{\O}^+_X.
\]
And then we define $\Phi^1_\X \colon  \wdh{\Omega}^1_\X\{-1\} \to \widetilde{\rm{R}^1\nu_*\wdh{\O}^+_X}$ as the composition
\[
\wdh{\Omega}^1_\X\{-1\} \xr{\Phi'^1_\X} \rm{R}^1\nu_*\wdh{\O}^+_X \to \widetilde{\rm{R}^1\nu_*\wdh{\O}^+_X},
\]
where the second map is the natural quotient morphism. 

\begin{defn}\label{defn:phi-map} For each integer $n\geq 1$, we define the morphism $\Phi^n_\X \colon \wdh{\Omega}^n_{\X}\{-n\} \to \widetilde{\rm{R}^n\nu_*\wdh{\O}^+_X}$ as the following composition:
\[
    \wdh{\Omega}^n_\X\{-n\} \xr{\bigwedge^n \left(\Phi^1_\X\right)} \bigwedge^n \left( \widetilde{\rm{R}^1\nu_*\wdh{\O}^+_X}\right) \xr{\cup^n} \widetilde{\rm{R}^n\nu_*\wdh{\O}^+_X},  
\]
where $\cup^n$ is the cup-product map from Lemma~\ref{lemma:cup-product-well-defined}. 
\end{defn}

\begin{rmk}\label{rmk:functoriak-phi} By construction, the morphism $\Phi^n_\X$ is functorial with respect to flat morphisms. More precisely, if $\pi\colon \X' \to \X$ is flat morphism of admissible formal $\O_C$-schemes with generic fiber $\pi_C\colon X' \to X$, then the diagram
\begin{equation}\label{eqn:functorial-phi}
\begin{tikzcd}
    \pi^* \wdh{\Omega}^n_\X\{-n\} \arrow{d} \arrow{r}{\pi^*\left(\Phi_\X^n\right)} & \pi^*\widetilde{\rm{R}^n\nu_{\X, *}\wdh{\O}^+_X} \arrow{d} \\
    \wdh{\Omega}^n_{\X'}\{-n\}\arrow{r}{\Phi^n_{\X'}} & \widetilde{\rm{R}^n\nu_{\X', *}\wdh{\O}^+_{X'}}
\end{tikzcd}
\end{equation}
commutes for any $n\geq 1$. 
\end{rmk}

\begin{rmk}\label{rmk:etale-local-phi} If $\pi\colon \X' \to \X$ is \'etale, then the vertical arrows in Diagram~(\ref{eqn:functorial-phi}) are isomorphisms (the right vertical arrow is an isomorphism by \cite[Theorem 6.13.6(3)]{Z3}). Therefore, we conclude that $\Phi_{\X'}=\pi^*(\Phi_\X)$, i.e., $\Phi_\X$ is \'etale-local on $\X$.
\end{rmk}

\begin{rmk}\label{rmk:kunneth-phi} We apply Remark~\ref{rmk:functoriak-phi} to $\X'=\wdh{\bf{G}}^d_m$, the projection morphisms $\pi_i\colon \wdh{\bf{G}}^d_m \to \wdh{\bf{G}}_m$, and $n=1$ to conclude that the diagram
\[
\begin{tikzcd}[column sep = 6em]
    \bigoplus_{i=1}^d\pi^*_i\,\wdh{\Omega}^1_{\wdh{\bf{G}}_m}\{-1\} \arrow{d} \arrow{r}{\bigoplus_{i=1}^d \pi^*_i\left(\Phi_{\wdh{\bf{G}}_m}^1\right)} & \bigoplus_{i=1}^d \pi^*_i\,\widetilde{\rm{R}^1\nu_{\wdh{\bf{G}}_m, *}\wdh{\O}^+_{\bf{T}^1_C}} \arrow{d} \\
    \wdh{\Omega}^1_{ \wdh{\bf{G}}^d_m}\{-1\}\arrow{r}{\Phi^1_{\wdh{\bf{G}}^d_m}} & \widetilde{\rm{R}^1\nu_{\wdh{\bf{G}}^d_m, *}\wdh{\O}^+_{\bf{T}^d_C}}
\end{tikzcd}
\]
commutes for any $d\geq 1$.
\end{rmk}

\subsection{Faltings' trace for a smooth model}

The main goal of this section is to construct the integral Faltings' trace
\[
\Psi^d_\X \colon \rm{R}^d\nu_*\wdh{\O}_X^+ \to \omega_\X(-d)
\]
for a smooth admissible formal $\O_C$-scheme $\X$. The essential idea of the construction is to (slightly) modify the BMS map $\Phi^d_\X$ to simultaneously reverse its direction and change Breuil-Kisin twists to Tate twists. In order to do this, we need to recall the results from \cite{BMS1} regarding the map $\Phi^n_{\X}$ for a smooth admissible formal $\O_C$-scheme $\X$. \smallskip

\begin{lemma}\label{lemma:BMS-vector-bundle} Let $\X$ be a smooth admissible formal $\O_C$-scheme with generic fiber $X=\X_C$. Then the $\O_\X$-module $\frac{\rm{R}^n\nu_* \wdh{\O}_X^+}{\left(\rm{R}^n\nu_* \wdh{\O}_X^+\right)[\zeta_p-1]}$ is a finite, locally free $\O_\X$-module for any integer $n\geq 0$. In particular, the natural morphism
\[
\frac{\rm{R}^n\nu_* \wdh{\O}_X^+}{\left(\rm{R}^n\nu_* \wdh{\O}_X^+\right)[\zeta_p-1]} \to \frac{\rm{R}^n\nu_* \wdh{\O}_X^+}{\left(\rm{R}^n\nu_* \wdh{\O}_X^+\right)[(\zeta_p-1)^\infty]} = \widetilde{\rm{R}^n\nu_*\wdh{\O}_X^+}
\]
is an isomorphism. 
\end{lemma}
\begin{proof}
    We first note that \cite[Theorem 8.3]{BMS1} guarantees that the complex $L\eta_{\zeta_p-1}\left(\bf{R}\nu_*\wdh{\O}^+_{X}\right)$ has finite, locally free cohomology sheaves. On the other hand, \cite[Lemma 6.4]{BMS1} (applied to a K-injective resolution of $\bf{R}\nu_*\wdh{\O}^+_{X}$) implies that we have isomorphisms
    \[
    \mathcal H^n \left(L\eta_{\zeta_p-1}\left(\bf{R}\nu_*\wdh{\O}^+_{X}\right)\right) \xrightarrow{\sim}  \frac{\rm{R}^n\nu_* \wdh{\O}_X^+}{\left(\rm{R}^n\nu_* \wdh{\O}_X^+\right)[\zeta_p-1]}.
    \]
    Combining these two results, we conclude that the $\O_\X$-module $\frac{\rm{R}^n\nu_* \wdh{\O}_X^+}{\left(\rm{R}^n\nu_* \wdh{\O}_X^+\right)[\zeta_p-1]}$ is finite, locally free for any $n\geq 0$. The ``in particular'' part is a formal consequence of the fact that $\frac{\rm{R}^n\nu_* \wdh{\O}_X^+}{\left(\rm{R}^n\nu_* \wdh{\O}_X^+\right)[\zeta_p-1]}$ has already no $(\zeta_p-1)^\infty$-torsion.  
\end{proof}

\begin{lemma}\label{lemma:BMS} Let $\X$ be a smooth admissible formal $\O_C$-scheme with generic fiber $X=\X_C$. Then the morphism
\[
\Phi^n_{\X}\colon \wdh{\Omega}^n_{\X}\{-n\} \to \widetilde{\rm{R}^n\nu_*\wdh{\O}_X^+}
\]
is an isomorphism onto $(\zeta_p-1)^n \left(\widetilde{\rm{R}^n\nu_*\wdh{\O}_X^+}\right)$ for $n\geq 1$. 
\end{lemma}
\begin{proof}
In this proof we will freely use that $\frac{\rm{R}^n\nu_* \wdh{\O}_X^+}{\left(\rm{R}^n\nu_* \wdh{\O}_X^+\right)[\zeta_p-1]} = \widetilde{\rm{R}^n\nu_*\wdh{\O}_X^+}$ without explicitly referring to Lemma~\ref{lemma:BMS-vector-bundle}.

{\it Step 1. The map $\Phi^1_{\X}$ is an isomorphism onto $(\zeta_p-1)\left(\widetilde{\rm{R}^1\nu_*\wdh{\O}_X^+}\right)$}: We first note that \cite[Lemma 6.10]{BMS1} implies that there is a canonical morphism
\[
\alpha\colon L\eta_{\zeta_p-1}\left(\bf{R}\nu_*\wdh{\O}^+_{X}\right) \to \bf{R}\nu_*\wdh{\O}^+_{X}.
\]
Now the proofs of \cite[Theorem 8.3 and Theorem 8.7]{BMS1} show that $\Phi^1_{\X}\colon \wdh{\Omega}^1_\X\{-1\} \to \widetilde{\rm{R}^1\nu_*\wdh{\O}^+_X}$ factors as
\[
\begin{tikzcd}[column sep = 5 em, row sep = 4em]
    \wdh{\Omega}^1_\X\{-1\} \arrow{d}{\rm{BMS}^1_\X\{-1\}} \arrow{dr}{\Phi'^1_\X} \arrow{drr}{\Phi^1_\X}& &\\
    \cal{H}^1\left(L\eta_{\zeta_p-1}\left(\bf{R}\nu_*\wdh{\O}^+_{X}\right)\right) \arrow{r}{\cal{H}^1(\alpha)} & \rm{R}^1\nu_*\wdh{\O}_X^+ \arrow{r} & \frac{\rm{R}^n\nu_* \wdh{\O}_X^+}{\left(\rm{R}^n\nu_* \wdh{\O}_X^+\right)[\zeta_p-1]} =\widetilde{\rm{R}^1\nu_*\wdh{\O}_X^+}
\end{tikzcd}
\]
such that $\rm{BMS}^1_\X\{-1\}$ is an isomorphism. Then the result follows from the fact that the composition of the horizontal maps is an isomorphism onto $(\zeta_p-1)\left(\widetilde{\rm{R}^1\nu_*\wdh{\O}_X^+}\right)$. \smallskip

{\it Step 2. The map $\Phi^n_{\X}$ is an isomorphism onto $(\zeta_p-1)^n\left(\widetilde{\rm{R}^n\nu_*\wdh{\O}_X^+}\right)$ for any $n\geq 1$}: \cite[Proposition 6.7]{BMS1} guarantees that there is a natural cup-product structure on $L\eta_{\zeta_p-1}\left(\bf{R}\nu_* \wdh{\O}_X^+\right)$ (that, by construction, is compatible with the cup-product on $\bf{R}\nu_*\wdh{\O}_X^+$ via the map $\alpha$). Then \cite[Corollary 8.13(ii)]{BMS1} implies that the induced morphism\footnote{The cup-product descends to wedge-powers since the target $\cal{H}^n\left(L\eta_{\zeta_p-1}\left(\bf{R}\nu_*\wdh{\O}^+_{X}\right)\right)$ is torsion-free by Lemma~\ref{lemma:BMS-vector-bundle}.} 
\[
\cup^n\colon \bigwedge^n \cal{H}^1\left(L\eta_{\zeta_p-1}\left(\bf{R}\nu_*\wdh{\O}^+_{X}\right)\right) \to \cal{H}^n\left(L\eta_{\zeta_p-1}\left(\bf{R}\nu_*\wdh{\O}^+_{X}\right)\right)
\]
is an isomorphism for any $n\geq 1$. This allows us to define the {\it isomorphism} $\rm{BMS}_\X^n\{-n\}\colon \wdh{\Omega}^n_\X\{-n\} \to \cal{H}^n\left(L\eta_{\zeta_p-1}\left(\bf{R}\nu_*\wdh{\O}^+_{X}\right)\right)$ as the composition of $\bigwedge^n \rm{BMS}^1_\X\{-1\}$ with the above isomorphism $\cup^n$. By construction, we have the following commutative diagram:
\[
\begin{tikzcd}[column sep = 5em]
    \bigwedge^n \left( \wdh{\Omega}^1_\X\{-1\}\right) \arrow[bend left]{rr}{\bigwedge^n \Phi^1_\X} \arrow{d}{\wr}\arrow{r}{\bigwedge^n \rm{BMS}^1_\X\{-1\}} & \bigwedge^n \left(\cal{H}^1\left(L\eta_{\zeta_p-1}\left(\bf{R}\nu_*\wdh{\O}^+_{X}\right)\right)\right) \arrow{d}{\cup^n} \arrow{r}{\bigwedge^n \cal{H}^1(\alpha)} & \bigwedge^n \widetilde{\rm{R}^1\nu_*\wdh{\O}_X^+} \arrow{d}{\cup^n} \\
    \wdh{\Omega}^n_\X\{-n\} \arrow{r}{\rm{BMS}_\X^n\{-n\}} & \cal{H}^n\left(L\eta_{\zeta_p-1}\left(\bf{R}\nu_*\wdh{\O}^+_{X}\right)\right) \arrow{r}{\cal{H}^n(\alpha)} & \widetilde{\rm{R}^n\nu_*\wdh{\O}_X^+}.
\end{tikzcd}
\]
By going right and down in the diagram, we see that the composed morphism $\wdh{\Omega}^n_\X\{-n\} \to \widetilde{\rm{R}^n\nu_*\wdh{\O}_X^+}$ is equal to $\Phi^n_\X$. Now $\rm{BMS}_\X^n\{-n\}$ is an isomorphism, and $\cal{H}^n(\alpha)$ is an isomorphism onto $(\zeta_p-1)^n\left(\frac{\rm{R}^n\nu_* \wdh{\O}_X^+}{\left(\rm{R}^n\nu_* \wdh{\O}_X^+\right)[\zeta_p-1]}\right)=(\zeta_p-1)^n\left(\widetilde{\rm{R}^n\nu_*\wdh{\O}_X^+}\right)$ by its very construction (and Lemma~\ref{lemma:BMS-vector-bundle}). 
\end{proof}

We are almost ready to define the desired trace map $\Psi^d_\X \colon \rm{R}^d\nu_*\wdh{\O}^+_{\X_C} \to \omega_\X(-d)$ for a smooth admissible formal $\O_C$-scheme of pure dimension $d$. The last ingredient that we need to discuss is the relation between the Breuil-Kisin twist $\O_C\{1\}$ and the Tate twist $\O_C(1)$. \smallskip

We start by considering the $\rm{dlog}$-map $\rm{dlog}_{\Z_p}\colon \bf{Z}_p(1)=\rm{T}_p(\mu_{p^\infty}) \to \O_C\{1\}=\rm{T}_p(\Omega^1_{\O_C/\Z_p})$ as the morphism induced by the map 
\[
\mu_{p^\infty}(\ov{K}) = \mu_{p^\infty}(C) \xr{f\mapsto \frac{\rm{d}f}{f}} \Omega^1_{\O_C/\Z_p}.
\]
By abuse of notation, we denote the $\O_C$-linearization of $\rm{dlog}_{\Z_p}$ simply as 
\[
\rm{dlog}\colon \O_C(1) \to \O_C\{1\}.
\]

Now we recall the theorem of Fontaine that describes the image of $\rm{dlog}$. Unfortunately, we can not find this result explicitly stated in the form we need. For this reason, we include a proof of the following result:

\begin{thm}\label{thm:Fontaine}(Fontaine) The natural morphism $\rm{dlog}\colon \O_C(1) \to \O_C\{1\}$ is injective with the image equal to $(\zeta_p-1)\O_C\{1\}$.
\end{thm}
\begin{proof}
    We denote by $\bf{C}_p\coloneqq \wdh{\ov{\Q}}_p$ the completed algebraic closure of $\Q_p$, and by $\O_{\C_p}$ its ring of integers.  \smallskip
    
    {\it Step~$1$. The natural morphism $\O_{\C_p}\{1\} \otimes_{\O_{\C_p}} \O_C \to \O_C\{1\}$ is an isomorphism.} By definition, $\O_C\{1\}\simeq \wdh{L}_{\O_C/\Z_p}[-1]$. Therefore, (using that $\O_{\C_p}\{1\}$ is a free $\O_{\C_p}$-module of rank one), we see that it suffices to show that the natural morphism
    \[
    \wdh{L}_{\O_{\C_p}/\Z_p} \wdh{\otimes}_{\O_{\C_p}} \O_C \to \wdh{L}_{\O_C/\Z_p}
    \]
    is an isomorphism. This morphism fits into an exact triangle
    \[
    \wdh{L}_{\O_{\C_p}/\Z_p} \wdh{\otimes}_{\O_{\C_p}} \O_C \to \wdh{L}_{\O_C/\Z_p} \to \wdh{L}_{\O_{C}/\O_{\C_p}}.
    \]
    Therefore, the result follows from the fact that $\wdh{L}_{\O_C/\O_{\C_p}}=0$ (see the proof of \cite[Corollary 3.2.5]{bhatt-arizona}).\smallskip
    
    {\it Step~$2$. There is a natural isomorphism $\O_{\C_p}\{1\}\simeq \rm{T}_p(\Omega^1_{\ov{\Z}_p/\Z_p})$.} A proof similar to the one used in previous step shows that the natural morphism $\wdh{L}_{\ov{\Z}_p/\Z_p} \to \wdh{L}_{\O_C/\Z_p}$ is an isomorphism. Now \cite[Theorem \textsection{1.3}]{Bei12} and \cite[Corollary 3.3]{Sza} imply that $L_{\ov{\Z}_p/\Z_p} \simeq \Omega^1_{\ov{\Z}_p/\Z_p}[0]$ and $\Omega^1_{\ov{\Z}_p/\Z_p}$ is $p$-divisible. Therefore, we get 
    \[
    \O_{\C_p}\{1\}\coloneqq \wdh{L}_{\ov{\Z}_p/\Z_p}[-1] \simeq \rm{T}_p\left(\Omega^1_{\ov{\Z}_p/\Z_p}\right).
    \]
    
    {\it Step~$3$. Finish the argument.} Step~$1$ implies that it suffices to show that the claim for $C=\C_p$. Now \cite[Theorem \textsection{1.3}]{Bei12} implies that the sequence
    \[
    0 \to \frac{(1-\zeta_p)^{-1}\ov{\Z}_p}{\ov{\Z}_p}(1) \to \ov{\Z}_p\otimes_{\Z_p} \mu_{p^\infty} \xr{\ov{\Z}_p\otimes \rm{dlog}_{\Z_p}} \Omega^1_{\ov{\Z}_p/\Z_p} \to 0
    \]
    is exact\footnote{To get the left arrow, we note that $\ov{\Q}_p/\ov{\Z}_p \otimes_{\Z_p} \Z_p(1) \simeq \colim_m \left(\frac{1}{p^n}\ov{\Z}_p/\ov{\Z}_p \otimes_{\Z_p}\Z_p(1) \right) \simeq \colim_m \left(\ov{\Z}_p\otimes_{\Z_p} \Z_p(1)/p^n\Z_p(1)\right) \simeq \colim_m \left(\ov{\Z}_p\otimes_{\Z_p} \mu_{p^n}\right)\simeq \ov{\Z}_p \otimes_{\Z_p} \mu_{p^\infty}$.}. Now we apply the $\rm{RHom}_{\Z_p}(\Q_p/\Z_p, -)$-functor and use Step~$2$ to get an exact sequence
    \[
    0\to \O_{\C_p}(1) \xr{\rm{dlog}} \O_{\C_p}\{1\} \to Q\to 0,
    \]
    where $Q$ is annihilated by exactly $(1-\zeta_p)$. This finishes the proof. 
\end{proof}

Now let $\X$ be a smooth admissible formal $\O_C$-scheme and $n\geq 1$ an integer. Then  Theorem~\ref{thm:Fontaine} (after passing to duals and tensor powers) and local freeness of $\wdh{\Omega}^n_\X$ imply that we have a short exact sequence
\[
0 \to \wdh{\Omega}^n_\X\{-n\} \xr{\rm{id} \otimes (\rm{dlog}^{\vee})^{\otimes n}} \wdh{\Omega}^n_\X(-n) \to \cal{Q} \to 0
\]
such that the image of $\wdh{\Omega}^n_\X\{-n\}$ in $\wdh{\Omega}^n_\X(-n)$ is equal to $(\zeta_p-1)^n\left(\wdh{\Omega}^n_{\X}(-n)\right)$.

\begin{thm-def} Let $\X$ be a smooth admissible formal $\O_C$-scheme with generic fiber $X=\X_C$, and let $n\geq 1$ be an integer. Then there is a unique $\O_\X$-linear {\it isomorphism} $\Psi'^n_\X\colon \widetilde{\rm{R}^n\nu_*\wdh{\O}^+_X} \to \wdh{\Omega}^n_\X(-n)$ such that the diagram
\[
\begin{tikzcd}[row sep = 3em, column sep = 4em]
\widetilde{\rm{R}^n\nu_*\wdh{\O}^+_X}  \arrow[r, "\Psi'^{n}_{\X}"] & \wdh{\Omega}^n_{\X}(-n) \\
(\zeta_p-1)^n\left(\widetilde{\rm{R}^n\nu_*\wdh{\O}^+_X}\right) \arrow{u} \arrow{r}{(\Phi^n_{\X})^{-1}}  &\wdh{\Omega}^n_{\X} \{-n\} \arrow{u}{\rm{id} \otimes (\rm{dlog}^{\vee})^{\otimes n}}
\end{tikzcd}
\]
commutes\footnote{Existence of the inverse map $(\Phi^n_{\X})^{-1}$ follows from Lemma~\ref{lemma:BMS}.}. 
\end{thm-def}
\begin{proof}
    For brevity, we denote $\rm{id} \otimes (\rm{dlog}^{\vee})^{\otimes n}$ by $\iota^n$. Since all sheaves of interest are $(\zeta_p-1)$-torsionfree, it is straightforward to see that if $\Psi'^n_\X$ exists, it must be given by the formula 
    \[
    \Psi'^{n}_{\X}(x)= \frac{\iota^n\left((\Phi^n_{\X})^{-1}((\zeta_p-1)^nx)\right)}{(\zeta_p-1)^n}.
    \]
    One also easily checks that this formula is a well-defined $\O_\X$-linear {\it isomorphism} since $(\Phi^n_{\X})^{-1}$ is an isomorphism and both vertical arrows identify the source with $(\zeta_p-1)^n\times$ the target.
\end{proof}

\begin{defn}\label{defn:trace-on-smooth-locus} For an admissible separated smooth formal $\O_C$-scheme $\X$ with generic fiber $X$ of pure dimension $d$, we define the {\it trace morphism}
\[
    \Psi^d_{\X}\colon \rm{R}^d\nu_*\wdh{\O}^+_{X} \to \omega_{\X}(-d)
\]
as the composition
\[
\rm{R}^d\nu_*\wdh{\O}^+_{X} \to \widetilde{\rm{R}^d\nu_*\wdh{\O}^+_X} \xr{\Psi'^{d}_{\X}} \wdh{\Omega}^d_{\X}(-d) \xr{r_{\X}(-d)} \omega_{\X}(-d),
\]
where $r_{\X}$ is the map from Theorem~\ref{thm:dualizing-complex}.
\end{defn}

\subsection{Faltings' trace on the generic fiber}

The main goal of this section is to construct the trace morphism
\[
\left(\Psi^d_\X\right)_C \colon \left(\rm{R}^d\nu_*\wdh{\O}_X^+\right)_C \to \left(\omega_\X\right)_C(-d)
\]
for a separated admissible formal $\O_C$-scheme $\X$ with smooth generic fiber $X=\X_C$ and reduced special fiber. The key step is to relate the generic fiber of the morphism $ \Phi^n_\X \colon \wdh{\Omega}^n_{\X}\{-n\} \to \widetilde{\rm{R}^n\nu_*\wdh{\O}^+_X}$ from Definition~\ref{defn:phi-map} to another map (defined purely in terms of the generic fiber)
\[
\rm{Sch}^n_X\colon \Omega^n_{X}(-n) \to \rm{R}^n\mu_*\wdh{\O}_X,
\]
whose construction we now briefly recall. \smallskip

\subsubsection{Scholze's map}

First we recall that, for any rigid space $X$, we have morphisms of ringed sites $\lambda \colon (X_{\proet}, \wdh{\O}^+_X) \to (X_{\et}, \O^+_X)$ and $\pi \colon (X_{\et}, \O^+_X) \to (X_{\rm{an}}, \O^+_X)$. We denote the composition by $\mu \colon (X_{\proet}, \wdh{\O}^+_X) \to (X_{\rm{an}}, \O^+_X)$. \smallskip

If $X$ is smooth over $C$, P.\,Scholze constructs a canonical $\O_{X_\et}$-linear isomorphism
\[
\rm{Sch}'^1_X\colon \Omega^1_{X_{\et}} \xr{\sim} \rm{R}^1\lambda_*\wdh{\O}_X(1)
\] 
in \cite[Corollary 6.19, Remark 6.20]{Sch1} and \cite[Lemma 3.24]{Schsurvey}. In particular, we conclude that $\rm{R}^1\lambda_*\wdh{\O}_X$ is a coherent $\O_{X_{\et}}$-module. However, for our purposes, it will be convenient to work with the isomorphism
\[
\rm{Sch}^1_X \colon \Omega^1_X(-1) \xr{\sim} \rm{R}^1\mu_*\wdh{\O}_X
\]
defined as\footnote{We note that Scholze's isomorphisms $\O_{X_\et}(1)\simeq \lambda_*\wdh{\O}_X(1)$, $\Omega^1_{X_{\et}} \simeq \rm{R}^1\lambda_*\wdh{\O}_X(1)$, and the fact that coherent sheaves are $\pi_*$-acyclic imply that $\pi_*\left(\rm{R}^1\lambda_*\wdh{\O}_X(1)\right) \simeq \rm{R}^1\mu_*\wdh{\O}_X(1)$.} $\pi_*(\rm{Sch}'^1_X)(-1)$. 

\begin{rmk} \'Etale descent for coherent sheaves implies that $\pi^*(\rm{Sch}^1_X)(1)=\rm{Sch}'^1_X$.
\end{rmk}

\begin{defn} For a smooth rigid-analytic $C$-space $X$ and an integer $n\geq 1$, we define {\it Scholze's morphism}
\[
\rm{Sch}^n_X \colon \Omega^n_X(-n) \xr{\bigwedge^n(\rm{Sch}^1_X)} \bigwedge^n \left(\rm{R}^1\mu_*\wdh{\O}_X\right) \xr{\cup^n} \rm{R}^n\mu_*\wdh{\O}_X.
\]
\end{defn}

\begin{rmk} We note that the cup product factors through the wedge product since $2$ is invertible in $C$ (and so the cup-product is alternating as opposed to being just skew-symmetric).
\end{rmk}

\begin{rmk}\label{rmk:coherent-proetale-etale} \cite[Proposition 3.23]{Schsurvey} shows that the cup-product map 
\[
    \bigwedge^n \left(\rm{R}^1\mu_*\wdh{\O}_X\right) \xr{\cup^n} \rm{R}^n\mu_*\wdh{\O}_X
\]
is an isomorphism for any $n\geq 1$. Therefore, the map $\rm{Sch}^n_X$ is an isomorphism for any $n\geq 1$. In particular, the $\O_X$-module $\rm{R}^n\mu_*\wdh{\O}_X$ is coherent for any $n\geq 1$.  
\end{rmk} 

Now we wish to compute $\rm{Sch}'^1_X$ explicitly for an affinoid space $X$ admitting a finite \'etale morphism $X \to \bf{T}^d$ to the $d$-dimensional torus. We start with the following version of the K\"unneth formula for $\rm{Sch}'^1_X$:

\begin{lemma}\label{lemma:kunneth-scholze} Let $\pi_i\colon \bf{T}^d \to \bf{T}^1$ be the $i$-the projection map for $i=1, \dots, d$. Then the diagram
\begin{equation}\label{eqn:kunneth-scholze}
\begin{tikzcd}[column sep = 6em]
    \bigoplus_{i=1}^d\pi^*_i\,\Omega^1_{\bf{T}^1_{\et}} \arrow{d} \arrow{r}{\bigoplus_{i=1}^d \pi^*_i\left(\rm{Sch}'^1_{\bf{T}^1}\right)} & \bigoplus_{i=1}^d \pi^*_i\, \rm{R}^1\lambda_{\bf{T}^1, *}\wdh{\O}_{\bf{T}^1}(1) \arrow{d} \\
    \Omega^1_{\bf{T}^d_{\et}}\arrow{r}{\rm{Sch}'^1_{\bf{T}^d}} & \rm{R}^1\lambda_{\bf{T}^d, *}\wdh{\O}_{\bf{T}^d}(1)
\end{tikzcd}
\end{equation}
commutes for any $d\geq 1$. Furthermore, all arrows in this diagram are isomorphisms.
\end{lemma}
\begin{proof}
    Note that the torus $\bf{T}^d$ is defined over $\Q_p$, so we can use the construction of $\rm{Sch}'^1$ from \cite[Corollary 6.14 and 6.19]{Sch1}. Then the desired commutativity follows from functoriality of the $\bf{B}_{\rm{dR}}^+$-de Rham complex (see \cite[Definition 6.8 and the discussion after it]{Sch1}). Now we note that the left vertical arrow of Diagram~(\ref{eqn:kunneth-scholze}) is an isomorphism by construction, while horizontal arrows are isomorphisms because $\rm{Sch}'^1_X$ is an isomorphism for any smooth $X$. This formally implies that the right vertical arrow is also an isomorphism. 
\end{proof}

Later in the paper, we will need to sometime perform explicit computations in the ``framed'' case. All these computations follow a somewhat general pattern, so we fix the relevant notation here in order to not repeat it every time we need it. 

\begin{notation}\label{notation:perfectoid-covering} Let $R^+=\O_C\langle T_1^{\pm 1}, \dots, T_d^{\pm 1}\rangle$, $R=R^+[1/p]$, $\bf{T}^d = \Spa(R, R^+)$ a $d$-dimension torus, and $X=\Spa(S, S^+)\to \bf{T}^d$ an \'etale morphism that is a composition of rational embeddings and finite \'etale morphisms.

We consider the $R^+$-algebras $R_m^+=\O_C\langle T_1^{\pm 1/p^m}, \dots, T_d^{\pm 1/p^m}\rangle$, $R^+_\infty=\O_C\langle T_1^{\pm 1/p^\infty}, \dots, T_d^{\pm 1/p^\infty}\rangle$, $R_m=R_m^+[1/p]$, and $R_\infty=R_\infty^+[1/p]$. We also consider the corresponding pro-\'etale covering 
\[
\bf{T}^d_\infty \coloneqq \lim_m \bf{T}^d_m = \Spa(R_m, R_m^+) \to \bf{T}^d
\]
that is $\Gamma=\Z_p(1)^d$-torsor with the total space $\bf{T}^d_\infty$ being an affinoid perfectoid space with the associated adic space $\Spa(R_\infty, R_\infty^+)$. We denote the induced torsor over $X$ by
\[
X_\infty \coloneqq \bf{T}^d_\infty \times_{\bf{T}^d} X \to X
\]
that is also a $\Gamma$-torsor with $X_\infty$ being an affinoid perfectoid space (see \cite[Lemma 4.5]{Sch1}). Furthermore, the associated adic space $\wdh{X}_\infty$ is isomorphic to $\Spa(S_\infty, S_\infty^+)$ with $S_\infty = R_\infty \wdh{\otimes}_{R} S$ and $S_\infty^+$ equal to the integral closure of $\rm{Im}(R_\infty^+\wdh{\otimes}_{R^+} S^+ \to S_\infty)$ inside $S_\infty$. \smallskip

The Almost Purity Theorem implies that the natural morphism
\[ 
\rm{H}^i_{\cont}(\Gamma, S_{\infty}) \to \Gamma\left(X, \rm{R}^i\lambda_*\wdh{\O}_X\right) 
\]
is an isomorphism for any $i\geq 0$. Furthermore, the explicit computation with Koszul complexes (see \cite[Lemma 5.5]{Sch1}) ensures that the natural morphism 
\[
    \rm{H}^i_\cont(\Gamma, \Z_p)\otimes_{\Z_p} S \to \rm{H}^i_{\cont}(\Gamma, S_{\infty})
\]
is an isomorphism for $i\geq 0$. In particular, we get a canonical identification 
\[
\rm{Hom}_{\rm{cont}}\left(\Gamma, \Z_p(1)\right)\otimes_{\Z_p} S \simeq \rm{H}^1_{\rm{cont}}\left(\Gamma, S_\infty(1)\right) \simeq \Gamma\left(X, \rm{R}^i\lambda_*\wdh{\O}_X(1)\right).
\]
Since $\Omega^1_{X_{\et}}$ and $\rm{R}^1\lambda_*\wdh{\O}_X(1)$ are coherent, the map $\rm{Sch}'^1_X$ can be identified with an $S$-linear map
\[
\rm{Sch}'^1_X\colon \bigoplus_{i=1}^d S\cdot \frac{\rm{d}T_i}{T_i}=\Omega^1_{X_{\et}}(X) \to \rm{Hom}_{\rm{cont}}\left(\Z_p(1)^d, \Z_p(1)\right)\otimes_{\Z_p} S = \rm{H}^1_{\cont}(\Gamma, S_{\infty}(1)).
\]
\end{notation}

\begin{lemma}\label{lemma:nice-formula-generic-fiber} Let $X \to \bf{T}^d$ be an \'etale morphism that is a composition of rational embeddings and finite \'etale morphisms. Then, following Notation~\ref{notation:perfectoid-covering}, the morphism 
\[
\rm{Sch}'^1_X\colon \bigoplus_{i=1}^d S\cdot \frac{\rm{d}T_i}{T_i} \to \rm{Hom}_{\rm{cont}}\left(\Z_p(1)^d, \Z_p(1)\right)\otimes_{\Z_p} S
\]
is equal to the unique $S$-linear map that sends $\frac{\rm{d}T_i}{T_i}$ to the projection morphism
\[
\rm{pr}_i \in \rm{Hom}_{\rm{cont}}\left(\Z_p(1)^d, \Z_p(1)\right) \subset \rm{Hom}_{\rm{cont}}\left(\Z_p(1)^d, \Z_p(1)\right)\otimes_{\Z_p} S.
\]
\end{lemma}
\begin{proof}
    Since $\rm{Sch}'^1$ is \'etale local, it suffices to prove the claim for $X=\bf{T}^d$. Using Lemma~\ref{lemma:kunneth-scholze}, one reduces to the case $d=1$. This case we prove in Lemma~\ref{lemma:explicit-scholze} below.
\end{proof}

\begin{lemma}\label{lemma:explicit-scholze} In the notation as above, we have 
\[
\rm{Sch}'^1_{\bf{T}^1} \left(\frac{\rm{d}T}{T}\right) = \rm{id}\in \rm{Hom}_{\rm{cont}}\left(\bf{Z}_p(1), \bf{Z}_p(1)\right) \subset \rm{H}^1_{\rm{cont}}\left(\Gamma, \Z_p(1)\right)\otimes_{\Z_p} R.
\]
\end{lemma}
\begin{proof}
    We recall that $R=C\langle T^{\pm 1}\rangle$, $R_\infty=C\langle T^{\pm 1/p^\infty}\rangle$, $X=\bf{T}^1$, and $\Gamma=\Z_p(1)$. \smallskip
    
    {\it Step~$1$.} We now note that \cite[Lemma 3.24]{Schsurvey} guarantees that there is a commutative diagram
    \[
    \begin{tikzcd}\label{diagram:faltings-extension}
        R^\times \arrow{d}{\rm{dlog}}\arrow{r}{\delta} & \rm{H}^1_{\proet}\left(\Spa(R, R^+), \wdh{\bf{Z}}_p(1)\right) \arrow{d}{\alpha} \\
        \Omega^1_{\bf{T}^1}(\bf{T}^1) \arrow{r}{\rm{Sch}'^1_X} & \rm{H}^1_{\rm{cont}}\left(\Gamma, R_{\infty}(1)\right)=\rm{H}^1_{\proet}\left(X, \wdh{\O}_X(1)\right),
    \end{tikzcd}
    \]
    where the top horizontal map is the connecting map corresponding to the extension\footnote{\cite[Lemma 3.24]{Schsurvey} uses the ``uncompleted'' short exact sequence $0\to \wdh{\bf{Z}}_p(1) \to \lim_{\times p} \O_X^\times \to \O_X^\times \to 0$, but it is straightforward to see that it gives the same boundary map.} of sheaves 
    \begin{equation}\label{eqn:nedo-faltings}
        0 \to \wdh{\bf{Z}}_p(1) \to \lim_{\times p} \wdh{\O}_X^\times \to \wdh{\O}_X^\times \to 0
    \end{equation}
    on $X_\proet$. 
    Thus the question boils down to computing the element 
    \[
    \rm{Sch}'^1_X(\rm{dlog}\,T)=\alpha(\delta(T)).
    \]

    {\it Digression.} In order to compute $\alpha(\delta(X))$, we will need to consider the pro-etale $\Gamma$-covering
    \[
    X_\infty = \Spa (R_\infty, R^+_\infty) \to X=\Spa(R, R^+).
    \]
    Then the Cartan-Leray spectral sequence defines a canonical morphism
    \[
    \iota\colon \rm{H}^1_{\rm{cont}}(\Gamma, \Z_p(1))=\Check{\rm{H}}^1(X_\infty/X, \wdh{\Z}_p(1)) \to \rm{H}^1_{\proet}\left(\Spa(R, R^+), \wdh{\Z}_p(1)\right)
    \]
    such that $\alpha \circ \iota$ can be identified with the canonical morphism
    \[
    \rm{H}^1_{\rm{cont}}\left(\Gamma, \Z_p(1)\right) \to \rm{H}^1_{\rm{cont}}\left(\Gamma, R_\infty(1)\right)
    \]
    induced by $\Z_p(1) \to R_\infty(1)$. \smallskip
    
    {\it Step~$2$.} Now we pass from the short exact sequence (\ref{eqn:nedo-faltings}) to the associated left exact sequence of \v{C}ech complexes (associated to the pro-\'etale covering $X_\infty \to X$) to get the following left exact sequence of complexes
    \[
    0\to \Check{\rm{C}}(X_\infty/X, \wdh{\Z}_p(1)) \to \Check{\rm{C}}(X_\infty/X, \lim_{\times p} \wdh{\O}_X^\times) \to \Check{\rm{C}}(X_\infty/X, \wdh{\O}_X^\times).
    \]
    We note that the element $T\in R_\infty^\times$ lies in the image of 
    \[
    \check{\rm{C}}^0(X_\infty/X, \lim_{\times p}\wdh{\O}_X^\times)=\lim_{\times p}R_\infty^\times \to \Check{\rm{C}}^0(X_\infty/X, \wdh{\O}_X^\times)=R_\infty^\times
    \]
    with the explicit pre-image given by $\{T^{1/p^n}\}_n\in \lim_{\times p}R_\infty^\times$. \smallskip
    
    Therefore, \cite[Proposition F.2.1]{Zev} implies that we can compute $\delta(T)$ as follows:
    \[
    \delta(T) = \iota\left(\ov{\rm{d}\big(\{T^{1/p^n}\}_n\big)}\right) \in \rm{H}^1_{\proet}\left(\Spa(R, R^+), \wdh{\Z}_p(1)\right),
    \]
    where $\rm{d}\colon \Check{\rm{C}}^0(X_\infty/X, \lim_{\times p}\wdh{\O}_X^\times) \to \Check{\rm{C}}^1(X_\infty/X, \lim_{\times p}\wdh{\O}_X^\times)$ is the differential, and $\ov{\rm{d}\big(\{T^{1/p^n}\}_n\big)}$ is the corresponding class\footnote{Note that this class automatically lies in $\Check{\rm{C}}^1(X_\infty/X, \wdh{\Z}_p(1))\subset \Check{\rm{C}}^1(X_\infty/X, \lim_{\times p}\wdh{\O}_X^\times)$.} in $\Check{\rm{H}}^1(X_\infty/X, \wdh{\Z}_p(1))=\rm{H}^1_{\rm{cont}}(\Gamma, \Z_p(1))$.\smallskip
    
    {\it Step~$3$.} Now we choose a compatible sequence of primitive $p$-power roots of unity $\zeta_{p^n}$. This defines a topological generator $\gamma=(1, \zeta_p, \zeta_{p^2}, \dots)\in \Gamma$, and so defines an isomorphism $\Z_p(1)\cong \Z_p$. The action of $\gamma$ on $T^{1/p^m}$ is given by $\zeta_{p^m}$ for any $m$. Then one sees that 
    \[
    \rm{d}\big(\{T^{1/p^n}\}_n\big) \in \rm{Map}_{\rm{cont}}\left(\Gamma, \Z_p(1)\right) = \Check{\rm{C}}^1(X_\infty/X, \wdh{\Z}_p(1)) 
    \]
    is given by the unique continuous morphism such that
    \[
    \rm{d}\big(\{T^{1/p^n}\}_n\big) (\gamma^n) = \{1, \zeta_p^n, \zeta_{p^2}^n, \dots\}=\gamma^n
    \]
    for any $n\geq 1$. In other words, $\rm{d}\big(\{T^{1/p^n}\}_n\big)$ is equal to $\rm{id}\in \rm{Map}_{\rm{cont}}(\Z_p(1), \Z_p(1))$. In particular, its class in $\rm{H}^1_{\rm{cont}}\left(\Z_p(1),\Z_p(1)\right)=\rm{Hom}_{\rm{cont}}(\Z_p(1), \Z_p(1))$ is equal to $\rm{id}\in \rm{Hom}_{\rm{cont}}(\Z_p(1), \Z_p(1))$. Therefore, we finally conclude that
    \[
     \rm{Sch}'^1_X\left(\rm{dlog}\,T\right)=\alpha\left(\delta(X)\right) = \alpha\circ \iota(\rm{id})\in \rm{H}^1_{\rm{cont}}(\Gamma, R_{\infty})=\rm{H}^1_{\rm{cont}}(\Gamma, \Z_p(1))\otimes_{\Z_p} R.
    \]
    Using that the composition $\alpha\circ \iota$ is naturally identified with the canonical morphism $\rm{H}^1_{\rm{cont}}\left(\Gamma, \Z_p(1)\right) \to \rm{H}^1_{\rm{cont}}\left(\Gamma, R_\infty(1)\right)$ (see {\it Digression}), we conclude that
    \[
    \rm{Sch}'^1_X\left(\frac{\rm{d}T}{T}\right)= \rm{id}\in \rm{Hom}_{\rm{cont}}\left(\bf{Z}_p(1), \bf{Z}_p(1)\right) \subset \rm{H}^1_{\rm{cont}}\left(\Gamma, \Z_p(1)\right)\otimes_{\Z_p} R.\qedhere
    \]
\end{proof}

\begin{cor}\label{cor:compatible-traces-generic-fiber} Let $f\colon X' \to X$ be a finite \'etale morphism of qcqs smooth separated rigid-analytic $C$-varieties of pure dimension $d$. Then the following diagram
\begin{equation}\label{eqn:scholze-commutative}
\begin{tikzcd}[column sep = 4em]
f_* \Omega^d_{X'}(-d) \arrow{r}{\rm{Tr}_f(-d)} \arrow{d}{f_*\left(\rm{Sch}^d_{X'} \right)}& \Omega^d_X(-d) \arrow{d}{\rm{Sch}^d_X} \\
f_* \rm{R}^d\mu_{X', *}\wdh{\O}_{X'} \arrow{r}{\cal{H}^d(\rm{Tr}_{\rm{an}, f})} & \rm{R}^d\mu_{X, *} \wdh{\O}_X \\
\end{tikzcd}
\end{equation}
commutes, where $\rm{Tr}_f$ is the coherent trace morphism from Theorem/Definition~\ref{thm-def:rigid-analytic-trace} and $\rm{Tr}_{\rm{an}, f}$ is the pro-\'etale trace morphism from Definition~\ref{defn:analytic-trace}.
\end{cor}
\begin{proof}
    This question can be checked locally on $X$, so \cite[Lemma 5.2]{Sch1} (or \cite[Corollary D.5]{Z3} and the standard embedding $\bf{D}^n \subset \bf{T}^n$ as a rational subset $|X_i-1|\leq |p|$) implies that we can assume that $X=\Spa(A, A^+)$ is an affinoid and admits an \'etale morphism $g\colon X \to \bf{T}^d$ that is a composition of finite \'etale maps and rational embeddings. In this case, $X'=\Spa(B, B^+)$ is automatically an affinoid by \cite[p. 3.6.20]{H2}. \smallskip
    
    We note that $f_* \rm{R}^d\mu_{X', *}\wdh{\O}_{X'}$ and $\rm{R}^d\mu_{X, *} \wdh{\O}_X$ are coherent by Remark~\ref{rmk:coherent-proetale-etale}. So it suffices to check that Diagram~(\ref{eqn:scholze-commutative}) commutes after applying global sections. Then we use the map
    \[
    g\colon X \to \bf{T}^d=\Spa(C\langle T_{1}^{\pm 1}, \dots, T_d^{\pm 1}\rangle, \O_C\langle T_{1}^{\pm 1}, \dots, T_d^{\pm 1}\rangle)
    \]
    to define ``coordinates'' $z_i\coloneqq g^\sharp(T_i)$ on $X$, and 
    \[
        \rm{H}^0(X, \Omega^d_X)=A\left(\frac{\rm{d}z_1}{z_1}\wedge\frac{\rm{d}z_2}{z_2}\wedge \dots \wedge \frac{\rm{d}z_d}{z_d}\right)
        \]
    and, similarly,
    \[
        \rm{H}^0(X', \Omega^d_{X'})=B\left(\frac{\rm{d}z_1}{z_1}\wedge\frac{\rm{d}z_2}{z_2}\wedge \dots \wedge \frac{\rm{d}z_d}{z_d}\right).
    \]
    Now we choose a compatible system of $p$-power roots of unity $\{1, \zeta_p, \zeta_{p^2}, \dots\}$ to trivialize $\Z_p(1)$, so Lemma~\ref{lemma:trace-explicit-rigid}, Lemma~\ref{lemma:nice-formula-generic-fiber}, and Lemma~\ref{lemma:explicit-trace-etale} reduce the question to showing that the diagram
    \[
    \begin{tikzcd}[column sep = 10em, row sep = 4em]
        B\left(\frac{\rm{d}z_1}{z_1}\wedge\frac{\rm{d}z_2}{z_2}\wedge \dots \wedge \frac{\rm{d}z_d}{z_d}\right) \arrow{r}{\rm{Tr}_{B/A}\cdot \left(\frac{\rm{d}z_1}{z_1}\wedge\frac{\rm{d}z_2}{z_2}\wedge \dots \wedge \frac{\rm{d}z_d}{z_d}\right)} \arrow{d}{b\cdot \frac{\rm{d}z_1}{z_1}\wedge \dots \wedge \frac{\rm{d}z_d}{z_d} \mapsto b} & A\left(\frac{\rm{d}z_1}{z_1}\wedge\frac{\rm{d}z_2}{z_2}\wedge \dots \wedge \frac{\rm{d}z_d}{z_d}\right) \arrow{d}{a\cdot \frac{\rm{d}z_1}{z_1}\wedge \dots \wedge \frac{\rm{d}z_d}{z_d} \mapsto a} \\
        B \arrow{r}{\rm{Tr}_{B/A}} & A
    \end{tikzcd}
    \]
    is commutative. But this is clear by inspection.
\end{proof}

\subsubsection{Relationship between $\rm{Sch}_X$ and $\left(\Phi_\X\right)_C$}

\begin{rmk} Lemma~\ref{lemma:integral-rational-pullback} and Lemma~\ref{lemma:kill-torsion} imply that there is a canonical isomorphism 
\[
\left(\widetilde{\rm{R}^n\nu_*\wdh{\O}^+_X}\right)_C \simeq \rm{R}^n\mu_*\wdh{\O}_X
\]
for any $n\geq 0$. Therefore, we will everywhere implicitly identify the target of $\left(\Phi^n_\X\right)_C$ with $\rm{R}^n\mu_*\wdh{\O}_X$, and consider $\left(\Phi^n_\X\right)_C$ as a morphism
\[
\left(\Phi^n_\X\right)_C \colon \Omega^n_X\{-n\} \to \rm{R}^n\mu_*\wdh{\O}_X. 
\]
\end{rmk}

\begin{rmk} For any admissible formal $\O_C$-scheme $\X$, the map $\rm{dlog}\colon \O_C(1) \to \O_C\{1\}$ from Theorem~\ref{thm:Fontaine} induces a morphism
\[
\wdh{\Omega}^n_\X\{-n\} \xr{\rm{id} \otimes (\rm{dlog}^{\vee})^{\otimes n}} \wdh{\Omega}^n_{\X}(-n)
\]
such that its kernel and cokernel are annihilated by $(\zeta_p-1)^n$. This induces an isomorphism
\[
 \Omega^n_{X}\{-n\} \xr{\rm{id} \otimes (\rm{dlog}^{\vee})^{\otimes n}} \Omega^n_{X}(-n)
\]
on the generic fiber $X=\X_C$. 
\end{rmk}

\begin{thm}\label{thm:scholze-bms-coincide-on-generic-fibers} Let $\X$ be an admissible formal $\O_C$-scheme with reduced special fiber and smooth generic fiber $X=\X_C$. Then the diagram
\begin{equation}\label{282}
\begin{tikzcd}
\Omega^n_X(-n) \arrow{r}{\rm{Sch}_X^n} & \rm{R}^n\mu_*\,\wdh{\O}_X \\
\Omega^n_X\{-n\} \arrow{u}{j=\rm{id}\otimes \left(\rm{dlog}^{\vee}\right)^{\otimes n}} \arrow[ru, swap, "\left(\Phi^n_\X\right)_C"] &
\end{tikzcd}
\end{equation}
is commutative for any $n\geq 1$. In particular, the map $(\Phi^n_\X)_C$ is an isomorphism for $n\geq 1$.
\end{thm}

\begin{proof} 

{\it Step 1. Reduction to the case $n=1$}: We recall that both morphisms $\Phi^n_\X$ and $\rm{Sch}^n$ are defined as the composition of the wedge power of $\Phi^1_\X$ and $\rm{Sch}^1$ with the cup product map. Since cup product commutes with base change, it is sufficient to show the claim for $n=1$. \\

{\it Step 2. Reduce to the case of a smooth model $\X$ with ``good coordinates''}: Firstly, we can assume that $X$ is connected by replacing it with its connected component and corresponding connected component of $\X$ (that is well-defined by Lemma~\ref{lemma:connected}). \smallskip

Now both $\Omega^1_X\{-1\}$ and $\rm{R}^1\mu_*\wdh{\O}_X$ are vector bundles, so the locus where the maps $\rm{Sch}^n_X \circ j$ and $(\Phi^n_\X)_C$ are equal on fibers is a Zariski-closed subset. Since $X$ is reduced, it suffices to show that this Zariski-closed subset is equal to $X$. \smallskip

We use \cite[Lemma 2.1.4]{C} to see that it is sufficient to check the equality on some non-empty open subset $U\subset X$. In particular, it is sufficient to check equality on the open $\X^\sm_C$ that is non-empty as $\X$ has reduced special fiber. Therefore, we can replace $\X$ with $\X^\sm$ to assume that $\X$ is smooth. Moreover, we can localize $\X$ even further to assume that it is affine and has a finite \'etale morphism to $\wdh{\bf{G}}^d_{m, \O_C}$. The latter is achieved by \cite[Lemma 4.9]{Bhatt-spec}. \smallskip

{\it Step 3. Reduce to $\X=\wdh{\bf{G}}_m$}: We assume that $\X$ admits a finite \'etale map $\wdh{\bf{G}}^d_m$. We first note that $\Phi^1_\X$ is \'etale local on $\X$ by Remark~\ref{rmk:etale-local-phi} (alternatively, see \cite[Corollary 8.13]{BMS1} and the discussion after the proof of \cite[Lemma 8.16]{BMS1}). The morphism $\rm{Sch}^1$ is \'etale local on $X$ by its construction as it comes from a morphism of coherent sheaves in the \'etale topology. Thus, it suffices to prove the claim for $\X=\wdh{\bf{G}}_m^d$. Now one uses (the twisted version of) Lemma~\ref{lemma:kunneth-scholze} and Remark~\ref{rmk:kunneth-phi} to reduce to $\X=\bf{G}_m$. \smallskip

{\it Step~$4$. Finish the proof in case $\X=\wdh{\bf{G}}_m$.} In this case, we note that all sheaves in Diagram~(\ref{282}) are coherent sheaves on an affinoid $\bf{T}^1$. Therefore, we can check equality of maps on global sections. In what follows, we follow Notation~\ref{notation:perfectoid-covering} (in particular, $R=C\langle T^{\pm 1}\rangle$ and $R^+=\O_C\langle T^{\pm 1}\rangle$). With that notation in mid, we conclude that Diagram~(\ref{282}) commutes if and only if the diagram
\[
\begin{tikzcd}
\Omega^1_{\bf{T}^1_C/C}(-1) \arrow{r}{\rm{Sch}_{\bf{T}^1}^1} &  \rm{H}^1_{\cont}\left(\Gamma, R_{\infty}\right) \\
\Omega^1_{\bf{T}^1_C/C}\{-1\} \arrow{u}{\rm{id}\otimes \rm{dlog}^{\vee}} \arrow[ru, swap, "\left(\Phi^1_{\wdh{\bf{G}}_m}\right)_C"] &
\end{tikzcd}
\]
commutes. After untwisting (and using the computation of $\rm{H}^1_{\rm{cont}}(\Gamma, R_\infty)$), we see that the above diagram commutes if and only if the diagram
\[
\begin{tikzcd}[column sep=13ex, row sep = 4em]
R\cdot \frac{\rm{d}T}{T}=\wdh{\Omega}^1_{R^+/\O_C}[1/p] \arrow{r}{} \arrow[rd, swap, "\rm{Sch}'^1_{\bf{T}^1}"] \arrow{r}{\left(\rm{BMS}^1_{\wdh{\bf{G}}_m}\right)_C}& \rm{H}^1_{\rm{cont}}\left(\Gamma, \O_C\{1\}\right) \otimes_{\O_C} R \\
& \arrow[u, swap, "\rm{H}^1_{\rm{cont}}\left(\rm{dlog}\right) \otimes \rm{id}"]\rm{H}^1_{\rm{cont}}\left(\Gamma,\Z_p(1)\right)\otimes_{\Z_p} R
\end{tikzcd}
\]
commutes (see the proof of Lemma~\ref{lemma:BMS} for the recollection on the $\rm{BMS}$-map). Now \cite[Proposition 8.17]{BMS1} shows that 
\[
    \rm{BMS}^1_{\wdh{\bf{G}}_m}\left(\frac{\rm{d}T}{T}\right)=\rm{dlog}\otimes 1 \in \rm{Hom}_{\rm{cont}}\left(\bf{Z}_p(1), \O_C\{1\}\right)\otimes_{\O_C} R
\]
under the identification $\rm{H}^1_{\rm{cont}}(\Gamma, \O_C\{1\})\simeq \rm{Hom}_{\rm{cont}}(\bf{Z}_p(1), \O_C\{1\})$. Therefore, it suffices to show that 
\[
\rm{Sch}'^1_{\bf{T}^1} \left(\frac{\rm{d}T}{T}\right) = \rm{id}\in \rm{Hom}_{\rm{cont}}(\bf{Z}_p(1), \bf{Z}_p(1)) \subset \rm{H}^1_{\rm{cont}}(\Gamma, \O_C(1))\otimes_{\O_C} R.
\]
This was already verified in Lemma~\ref{lemma:explicit-scholze}.
\end{proof}

\begin{thm-def} Let $\X$ be an admissible formal $\O_C$-scheme with smooth generic fiber $X=\X_C$ and reduced special fiber, and let $n\geq 1$ be an integer. Then there is a unique $\O_X$-linear {\it isomorphism} $(\Psi'^n_\X)_C\colon (\widetilde{\rm{R}^n\nu_*\wdh{\O}^+_X})_C \to \Omega^n_X(-n)$ such that the diagram
\begin{equation}\label{eqn:psi-generic-fiber}
\begin{tikzcd}[row sep = 3em, column sep = 4em]
\left(\widetilde{\rm{R}^n\nu_*\wdh{\O}^+_X}\right)_C  \arrow[r, "\left(\Psi'^n_\X\right)_C"] & \Omega^n_{X}(-n) \\
\left(\widetilde{\rm{R}^n\nu_*\wdh{\O}^+_X}\right)_C \arrow{u}{\times (1-\zeta_p)^n} \arrow{r}{\left(\Phi^n_{\X}\right)_C^{-1}}  &\Omega^n_{X} \{-n\} \arrow{u}{\rm{id} \otimes (\rm{dlog}^{\vee})^{\otimes n}},
\end{tikzcd}
\end{equation}
commutes\footnote{Existence of the inverse map $\left(\Phi^n_{\X}\right)^{-1}$ follows from Theorem~\ref{thm:scholze-bms-coincide-on-generic-fibers}.}. 
\end{thm-def}
\begin{proof}
    The statement clearly follows from the fact that all vertical arrows and $(\Phi_\X^n)_C$ are isomorphisms. 
\end{proof}

\begin{defn}\label{defn:phi-generic-fiber} For an admissible separated formal $\O_C$-scheme $\X$ with smooth generic fiber $X=\X_C$ of pure dimension $d$ and reduced special fiber, we define the {\it trace morphism} 
\[
\left(\Psi^d_{\X}\right)_C\colon \left(\rm{R}^d\nu_*\wdh{\O}^+_{X}\right)_C \to (\omega_{\X})_C(-d)
\]
as the composition
\[
\left(\rm{R}^d\nu_*\wdh{\O}^+_{X}\right)_C \xr{\sim} \left(\widetilde{\rm{R}^d\nu_*\wdh{\O}^+_X}\right)_C \xr{\left(\Psi'^{d}_{\X}\right)_C} \Omega^d_{X}(-d) \xr{s_{\X}(-d)} (\omega_{\X})_C(-d) \ ,
\]
where $s_{\X}$ is the isomorphism from Lemma~\ref{lemma:dualizing-forms}.
\end{defn}

For the next remark, we denote by $\delta_\X \colon \left(\rm{R}^n\nu_*\wdh{\O}^+_X\right)_C \xr{\sim} \rm{R}^n\mu_*\wdh{\O}_X$ the composition of the isomorphisms from Lemma~\ref{lemma:integral-rational-pullback}. 

\begin{rmk}\label{rmk:faltings-scholze-generic-fiber} Theorem~\ref{thm:scholze-bms-coincide-on-generic-fibers} implies that, under the canonical identifications $\delta_\X \colon \left(\rm{R}^n\nu_*\wdh{\O}^+_X\right)_C \xr{\sim} \rm{R}^n\mu_*\wdh{\O}_X$ and $s^{-1}_{\X} \colon (\omega_\X)_C \xr{\sim} \Omega^d_X$, the map $\left(\Psi^d_{\X}\right)_C$ becomes equal to $\left(\rm{Sch}^d_X\right)^{-1}$. In particular, it is independent of the choice of $\X$ (up to these identifications). 
\end{rmk}

\subsection{Construction of Faltings' trace}

In this section, we define Faltings' trace map on any admissible separated formal $\O_C$-scheme with smooth generic fiber of pure dimension $d$ and reduced special fiber. The essential idea of the construction is to glue Definition~\ref{defn:trace-on-smooth-locus} and Definition~\ref{defn:phi-generic-fiber} together. Even though, we are ultimately interested in the $\O^+/p$-version of the trace map, it is crucial to first define the $\wdh{\O}^+$-version of the trace map.

\begin{defn}\label{defn:nice-formal-scheme} A {\it nice admissible formal $\O_C$-scheme of dimension $d$} is a separated admissible formal $\O_C$-scheme $\X$ with smooth generic fiber of pure dimension $d$ and reduced special fiber. 
\end{defn}

\begin{thm}\label{thm:faltings-1} Let $\X$ be nice admissible formal $\O_C$-scheme of dimension $d$, and let $X=\X_C$ and $\ov{\X}$ be its generic and special fiber respectively. Then there is a unique morphism 
\[
\rm{Tr}^{d, +}_{F, \X}\colon \rm{R}^d\nu_*\wdh{\O}^+_X \to \omega_\X(-d)
\]
such that $\rm{Tr}^{d, +}_{F, \X}|_{\X^\sm}$ coincides with $\Psi^d_{\X^\sm}$ and $(\rm{Tr}^{d, +}_F)_C$ coincides with $\Psi^d_{X}$.
\end{thm}
\begin{proof}
We note that Theorem~\ref{thm:dualizing-reflexive-formal-schemes} guarantees that the natural map $\omega_\X \to j_{\X^\sm, *}(\omega_{\X^\sm}) \cap {\omega_\X}_C$ is an isomorphism. Therefore, it shows that $\rm{Tr}^{d, +}_{F, \X}$ is uniquely defined by its restriction onto $\X^\sm$ and its pullback on $X$. Furthermore, it implies that it suffices to separately define 
\[
\left(\rm{R}^d\nu_*\wdh{\O}^+_X\right)_C \to \left(\omega_\X\right)_C(-d)
\]
and
\[
\left(\rm{R}^d\nu_*\wdh{\O}^+_X\right)|_{\X^\sm}=\rm{R}^d\nu_*\wdh{\O}^+_{\X^\sm_C} \to  \omega_{\X^\sm}(-d)
\]
such that they coincide on $(\X^\sm)_C$. In other words, we need to check that $\Psi^d_X$ and $(\Psi^d_{\X^\sm})_C$ are the same. This follows from their constructions and Theorem~\ref{thm:scholze-bms-coincide-on-generic-fibers} (and Lemma~\ref{lemma:dualizing-forms} to guarantee that the identifications of differential forms and dualizing modules on $\X^\sm$ and $X$ agree on $(\X^\sm)_C$).
\end{proof}

\begin{defn}\label{defn:falting-trace-full} Let $\X$ be a nice admissible formal $\O_C$-scheme of dimension $d$, and let $X=\X_C$ be its generic fiber. We define {\it integral Faltings' trace} 
\[
\rm{Tr}^+_{F, \X}\colon \bf{R}\nu_*\wdh{\O}^{+, a}_X \to \omega^a_\X(-d)[-d]
\] 
as the composition 
\[
\bf{R}\nu_*\wdh{\O}^{+, a}_X \to \rm{R}^d\nu_*\wdh{\O}^{+, a}_X[-d] \xr{\rm{Tr}_{F, \X}^{+, d}[-d]} \omega_\X^a(-d)[-d],
\]
where the first map is the projection of a complex on its top cohomology sheaf (we recall that $\bf{R}\nu_*\wdh{\O}^{+, a}_X \in \bf{D}^{[0, d]}_{acoh}(\X)$ by \cite[Theorem 6.13.6]{Z3}).
\end{defn}

\begin{lemma}\label{lemma:gen-fiber-Scholze} Let $\X$ be a nice admissible formal $\O_C$-scheme of dimension $d$, and let $X=\X_C$ be its generic fiber. Then the following diagram 
\[
\begin{tikzcd}
\left(\bf{R}\nu_* \wdh{\O}_X^{+, a}\right)_C \arrow{d}{\left(\rm{Tr}^+_{F, \X}\right)_C} \arrow{r}{\delta_\X} & \bf{R}\mu_*\wdh{\O}_X \arrow{r}& \rm{R}^d\mu_*\wdh{\O}_X[-d] \arrow{d}{\left(\rm{Sch}^d_X\right)^{-1}[-d]} \\
\left(\omega_{\X}^a\right)_C(-d)[-d] \arrow{rr}{s_{\X}(-d)[-d]} & & \Omega^d_X(-d)[-d]
\end{tikzcd}
\]
is commutative, where $\delta_\X\colon \left(\bf{R}\nu_*\wdh{\O}^+_X\right)_C \to \bf{R}\mu_*\wdh{\O}_X$ is the composition of the isomorphisms from Lemma~\ref{lemma:integral-rational-pullback}.
\end{lemma}
\begin{proof} 
    This follows directly from the very definition of $\rm{Tr}^+_{F, \X}$ and Remark~\ref{rmk:faltings-scholze-generic-fiber}.
\end{proof}

\begin{defn}\label{defn:F-trace-final} Let $\X$ be a nice admissible formal $\O_C$-scheme of dimension $d$, and let $X=\X_C$ be its generic fiber. We define {\it Faltings' trace} 
\[
\rm{Tr}_{F, \X}\colon \bf{R}\nu_*\left(\O^{+,a}_X/p\right) \to \omega^{\bullet, a}_{\X_0}(-d)[-2d]
\] 
as the composition 
\[
\bf{R}\nu_*\left(\O^{+, a}_X/p\right) \xr{\rm{pr}^{-1}} \bf{R}\nu_*\,\wdh{\O}^{+, a}_X \otimes^{L}_{\O_{\X}}\O_{\X_0} \xr{\rm{Tr}_{F, \X}\otimes^{L}_{\O_\X} \O_{\X_0}}  \omega^a_\X(-d)[-d] \otimes_{\O_{\X}} \O_{\X_0} \xr{\rm{BC}^a_{\omega_\X}(-d)[-2d]}  \omega^{\bullet, a}_{\X_0}(-d)[-2d]\,
\]
where $\rm{pr}^{-1}$ is the inverse of the projection formula isomorphism $\rm{pr}\colon \bf{R}\nu_*\,\wdh{\O}^{+, a}_X \otimes^{L}_{\O_\X}\O_{\X_0} \xr{\sim} \bf{R}\nu_*\left(\O^{+, a}_X/p\right)$, and $\rm{BC}_{\omega_\X}$ is\footnote{This slightly abuses the notations, but we hope that it does not cause any confusion in what follows.} the composition of the base change map defined before Lemma~\ref{lemma:trace-coherent}, and the natural ``inclusion'' $\omega_{\X_0}(-d)[-d] \to \omega^\bullet_{\X_0}(-d)[-2d]$.
\end{defn}

\begin{rmk} The first version of this paper contained a version of Faltings' trace map for any admissible separated rig-smooth formal $\O_C$-scheme $\X$. Based on the suggestion of the referee, we decided to restrict the discussion only to nice admissible formal $\O_C$-schemes because it significantly simplifies the exposition and it is sufficient for the main purposes of this paper. 
\end{rmk}

\begin{rmk} Note that we could have also defined the integral Faltings' trace as a map 
\[
    \widetilde{\rm{Tr}}^+_{F, \X}\colon \bf{R}\nu_*\,\wdh{\O}^{+,a}_X \to \omega^{\bullet, a}_\X(-d)[-2d]
\]
by composing it with the natural morphism $\omega_\X(-d)[-d] \to \omega^\bullet_\X(-d)[-2d]$. However, we prefer not to do this because the dualizing complex $\omega^{\bullet, a}_\X(-d)[-2d]$ does not admit a good theory of Grothendieck duality.
\end{rmk}

\subsection{Functoriality of Faltings' trace} The main goal of this section is to relate Faltings' trace morphisms on $\X'$ and $\X$ for a ``nice'' rig-finite, rig-\'etale morphism $\mf\colon \X' \to \X$.

\begin{defn} A morphism $\mf\colon \X' \to \X$ of admissible formal $\O_C$-models with generic fibers of pure dimension $d$ is {\it pseudo-finite} if it is rig-finite and $\bf{R}\mf_{*}\bf{R}\nu_{\X', *}\,\wdh{\O}_{X'}^{+,a}$ lies in $\bf{D}^{[0, d]}(\X)^a$.
\end{defn}

\begin{examples} Let $\X'$ and $\X$ be admissible formal $\O_C$-schemes with generic fibers $X'=\X'_C$ and $X=\X_C$ of pure dimension $d$. Then 
\begin{enumerate}
    \item any finite morphism $\mf\colon \X' \to \X$ is pseudo-finite. Indeed, \cite[Theorem 6.13.6]{Z3} ensures that $\bf{R}\nu_{\X', *} \wdh{\O}_{X'}^{+, a}\in \bf{D}^{[0, d]}_{acoh}(\X')^a$, and so its cohomology sheaves are $\mf_*$-acyclic. Therefore, 
    \[
    \bf{R}\mf_{*}\bf{R}\nu_{\X', *}\,\wdh{\O}_{X'}^{+,a} \in \bf{D}^{[0, d]}(\X)^a.
    \]
    \item any rig-isomorphism $\mf\colon \X' \to \X$ is pseudo-finite. This follows from \cite[Theorem 6.13.6]{Z3} and the isomorphism
    \[
    \rm{R}\mf_{*}\rm{R}\nu_{\X', *}\,\wdh{\O}_{X'}^{+, a} \simeq \rm{R}\nu_{\X, *}\, \wdh{\O}_X^{+, a} \in \bf{D}^{[0, d]}_{acoh}(\X)^a.
    \]
\end{enumerate}
\end{examples}

\begin{construction}\label{construction:factoring} Let $\X'$ and $\X$ be nice admissible formal $\O_C$-schemes of dimension $d$, let $X'=\X'_C$ and $X=\X_C$ be the generic fibers, and let $\mf\colon \X'\to \X$ be a pseudo-finite morphism. Then, by dimension reasons, the morphism
\[
\bf{R}\mf_* \rm{R}\nu_{\X', *}\,\wdh{\O}_{X'}^+ \xr{\bf{R}\mf_*\left(\rm{Tr}^+_{F, \X'}\right)} \rm{R}\mf_*\, \omega_{\X'}(-d)[-d]
\]
uniquely factors as the composition
\[
\bf{R}\mf_* \rm{R}\nu_{\X', *}\,\wdh{\O}_{X'}^+ \xr{\alpha} \mf_*\, \omega_{\X'}(-d)[-d] \to \bf{R}\mf_*\, \omega_{\X'}(-d)[-d].
\]
We define the morphism $\mf_*\left(\rm{Tr}^+_{F, \X}\right) \colon \bf{R}\mf_* \rm{R}\nu_{\X', *}\,\wdh{\O}_{X'}^+ \to \mf_*\, \omega_{\X'}(-d)[-d]$ as the morphism $\alpha$ from above.
\end{construction}

\begin{lemma}\label{lemma:traces-compatible-finite-etale} Let $\X'$ and $\X$ be nice admissible formal $\O_C$-schemes of dimension $d$, let $X'=\X'_C$ and $X=\X_C$ be their generic fibers, and let $\mf\colon \X'\to \X$ be a rig-\'etale, pseudo-finite morphism. Then the following diagram
\[
\begin{tikzcd}[column sep=10ex]
\bf{R}\mf_*\bf{R}\nu_{\X', *} \wdh{\O}^{+, a}_{X'}  \arrow{d}{\rm{Tr}^+_{\rm{Zar}, \mathfrak f}} \arrow{r}{\mf_*(\rm{Tr}^{+, a}_{F, \X'})} & \mf_*\omega^a_{\X'}(-d)[-d] \arrow{d}{\rm{Tr}_{\mf}(-d)[-d]} \\
\bf{R}\nu_{\X, *} \wdh{\O}^{+, a}_{X}\arrow{r}{\rm{Tr}^+_{F, \X}} & \omega^a_{\X}(-d)[-d]
\end{tikzcd}
\]
commutes, where $\rm{Tr}^+_{\rm{Zar}, \mathfrak f}$ is from Definition~\ref{defn:integral-proetale-trace}. 
\end{lemma}
\begin{proof}
By assumption, $\bf{R}\mf_*\bf{R}\nu_{\X', *} \wdh{\O}^{+, a}_{X'}\in \bf{D}^{[0, d]}_{acoh}(\X)^a$. Therefore, it suffices to show that both compositions are equal after applying $\cal{H}^d$. Lemma~\ref{lemma:dualizing-module-flat} ensures that the module $\omega_\X(-d)$ is $\O_C$-flat, so it suffices to check that the diagram commutes after taking the generic fiber. Then Lemma~\ref{zar-an-trace} and Lemma~\ref{lemma:gen-fiber-Scholze} guarantee that it is enough to check that the diagram
\begin{equation}\label{tr-tr-tr-2}
\begin{tikzcd}[column sep = 6em]
f_{\rm{an}, *}\circ \rm{R}^d\mu_{X',*} \wdh{\O}_{X'} \arrow{r}{f_{\rm{an}, *}\left(\rm{Sch}^d_{X'}\right)^{-1}} \arrow{d}{\cal{H}^d\left(\rm{Tr}_{\rm{an}, f}\right)} & f_{\rm{an}, *}\left(\Omega^d_{X'}\right)(-d) \arrow{d}{\rm{Tr}_{f}(-d)} \\
\rm{R}^d\mu_{X, *}\wdh{\O}_X \arrow{r}{(\rm{Sch}^{d}_{X})^{-1}} & \Omega^d_{X}(-d)
\end{tikzcd}
\end{equation}
commutes. This follows from Corollary~\ref{cor:compatible-traces-generic-fiber}.
\end{proof}

\begin{cor}\label{cor:descend-rig-finite-etale} Let $\X'$, $\X$, and $\mf$ be as in Lemma~\ref{lemma:traces-compatible-finite-etale}. Then the diagram 
\begin{equation}\label{eqn:new-eqn}
\begin{tikzcd}[column sep = 3em]
\bf{R}\mf_{0, *} \left(\bf{R}\nu_{\X', *} \O^{+, a}_{X'}/p \right)\otimes^{L} \bf{R}\mf_{0, *} \left(\bf{R}\nu_{\X', *} \O^{+, a}_{X'}/p\right) \arrow{r}{\cup} \arrow[d, shift left=7ex, "\rm{Tr}_{\rm{Zar}, \mf}"] & \bf{R}\mf_{0, *} \left(\bf{R}\nu_{\X', *} \O_{X'}^{+, a}/p\right) \arrow{r}{\bf{R}\mf_{0, *}\rm{Tr}_{F, \X'}} \arrow{d}{\rm{Tr}_{\rm{Zar}, \mf}}& \bf{R}\mf_{0, *}\left(\omega^{\bullet, a}_{\X'_0}(-d)[-2d]\right)\arrow{d}{\rm{Tr}_{\mf_{0}}(-d)[-2d]}  \\
\bf{R}\nu_{\X, *} \left(\O^{+, a}_{X}/p\right) \otimes^{L}  \bf{R}\nu_{\X, *}\left(\O^{+, a}_{X}/p\right) \arrow[u, shift left=7ex, "\rm{Res}_{\mf}"] \arrow{r}{\cup} & \bf{R}\nu_{\X, *} \left(\O_{X}^{+, a}/p\right) \arrow{r}{\rm{Tr}_{F, \X}} & \omega^{\bullet, a}_{\X_0}(-d)[-2d] \\
\end{tikzcd}
\end{equation}
is commutative in $\bf{D}(\X)^a$ (see Section~\ref{section:notation} for the precise meaning of this commutativity). 
\end{cor}
\begin{proof}
We first discuss commutativity of the right square. Our assumption on $\mf$ formally implies that $\bf{R}\mf_{0, *} \left(\bf{R}\nu_{\X', *} \O_{X'}^{+, a}/p\right)\in \bf{D}^{[0, d]}(\X_0)^a$, while Lemma~\ref{lemma:dualizing-module-flat} ensures that $\bf{R}\mf_{0, *}\left(\omega^\bullet_{\X'_0}(-d)[-2d]\right) \in \bf{D}^{\geq d}(\X_0)$. So, similarly to Construction~\ref{construction:factoring}, the right square in Diagram~(\ref{eqn:new-eqn}) can be decomposed as follows:
\[
\begin{tikzcd}
\bf{R}\mf_{0, *} \left(\bf{R}\nu_{\X', *} \O_{X'}^{+, a}/p\right) \arrow{d}{\rm{Tr}_{\rm{Zar}, \mf}}\arrow{r} & \mf_{0, *}\left( \omega^{a}_{\X'_0}(-d)[-d]\right)\arrow{d} \arrow{r} & \bf{R}\mf_{0, *}\left(\omega^{\bullet, a}_{\X'_0}(-d)[-2d]\right) \arrow{d}{\rm{Tr}_{\mf_{0}}(-d)[-2d]} \\
\bf{R}\nu_{\X, *} \left(\O_{X}^{+, a}/p\right) \arrow{r}{\rm{Tr}_{F, \X}} \arrow{r} & \omega^a_{\X_0}(-d)[-d] \arrow{r} & \omega^{\bullet, a}_{\X_0}(-d)[-2d].
\end{tikzcd}
\]
Commutativity of the left square is a consequence of Lemma~\ref{lemma:traces-compatible-finite-etale} (since $\omega_{\X}/p\omega_\X$ is a subsheaf of $\omega_{\X_0}$), and commutativity of the right square is formal. \smallskip

Now we show that the left square of Diagram~(\ref{eqn:new-eqn}) is commutative. Using the definition of the pro-\'etale trace map (see Definition~\ref{defn:integral-proetale-trace}), this diagram can be idenfied with 
\[
\begin{tikzcd}[column sep = 5em]
\bf{R}\nu_{\X, *} f_{\proet, *} \left(\O^{+, a}_{X'}/p\right) \otimes^{L} \bf{R}\nu_{\X, *} f_{\proet, *}\left(\O^{+, a}_{X'}/p \right) \arrow{r}{\bf{R}\nu_{\X, *}(-\cup-)} \arrow[d, shift left=7ex, "\bf{R}\nu_{\X, *}(\rm{Tr}_{\proet, f})"] & \bf{R}\nu_{\X, *} f_{\proet, *}\left(\O^{+, a}_{X'}/p\right) \arrow{d}{\bf{R}\nu_{\X, *}(\rm{Tr}_{\proet, f})}  \\
\bf{R}\nu_{\X, *} \left(\O^{+, a}_{X}/p\right) \otimes^{L}  \bf{R}\nu_{\X, *} \left(\O^{+, a}_{X}/p\right) \arrow[u, shift left=7ex, "\bf{R}\nu_{\X, *}(\rm{Res}_{f})"] \arrow{r}{\bf{R}\nu_{\X, *}(-\cup-)} & \bf{R}\nu_{\X, *} \left(\O_{X}^{+, a}/p\right).
\end{tikzcd}
\]
Thus, it is sufficient to show that the diagram
\[
\begin{tikzcd}
 f_{\proet, *} \left(\O^+_{X'}/p\right)\otimes^{L} f_{\proet, *}\left(\O^+_{X'}/p\right)  \arrow{r}{-\cup-} \arrow[d, shift left=6ex, "\rm{Tr}_{\proet, f}"] &  f_{\proet, *}\left(\O^+_{X'}/p\right) \arrow{d}{\rm{Tr}_{\proet, f}}  \\
 \O^+_{X}/p \otimes^{L} \O^+_{X}/p \arrow[u, shift left=6ex, "\rm{Res}_{f}"] \arrow{r}{-\cup-} & \O_{X}^+/p
\end{tikzcd}
\]
commutes. This equality can now be checked pro-\'etale locally, so one can assume that $f$ is a split \'etale cover. In this case, the claim is obvious. 
\end{proof}

We also check that Faltings' trace map is \'etale local on $\X$. More precisely, we have the following result:

\begin{lemma}\label{lemma:faltings-trace-commutes-with-base-change} Let $\X$ and $\Y$ be nice admissible formal $\O_C$-schemes of dimension $d$, $\mf\colon \X \to \Y$ an \'etale morphism with generic fiber $f\colon X \to X$. Then the diagram
\begin{equation}\label{diagram:faltings-trace-etale-local}
\begin{tikzcd}[column sep = 4em]
\bf{L}\mf^*\bf{R}\nu_{Y, *} \wdh{\O}^{+, a}_Y \arrow{d} \arrow{r}{\bf{L}\mf^*\rm{Tr}_{F, \Y}} & \bf{L}\mf^*\omega_{\Y}^a(-d)[-d] \arrow{d}\\
\bf{R}\nu_{X, *} \wdh{\O}^{+, a}_X \arrow{r}{\rm{Tr}_{F, \X}} & \omega_{\X}^a(-d)[-d] \ ,
\end{tikzcd}
\end{equation}
where the left vertical arrow is the base change map and the right vertical arrow is the map from Lemma~\ref{lemma:dualizing-etale-base-change}, is commutative in $\bf{D}(\X)^a$ with vertical arrows being (almost) isomorphisms. 
\end{lemma}
\begin{proof}
We note that $\bf{L}\mf^*\omega_{\Y}(-d)[-d] \simeq \mf^* \omega_{\Y}(-d)[-d]$ as $\mf^*$ is flat. So Lemma~\ref{lemma:dualizing-etale-base-change} guarantees that the right vertical map is an isomorphism. Moreover, \cite[Theorem 6.13.5]{Z3} guarantees that the left vertical map is an isomorphism. So we only need to show that the diagram commutes.\smallskip

Now we use that $\bf{L}\mf^*\bf{R}\nu_{Y, *} \wdh{\O}^{+, a}_Y$ is concentrated in degrees $[0, d]$ by \cite[Theorem 6.13.6]{Z3}. So it suffices to show that both morphisms 
\[
\mf^*(\rm{R}^d\nu_{Y, *} \wdh{\O}^{+}_Y)^a \to \omega_{\X}^a(-d)
\]
coincide. \smallskip

Now we recall that $\omega_{\X}^a(-d)$ is $\O_C$-flat by Lemma~\ref{lemma:dualizing-module-flat}. Therefore, equality of these maps can be checked on the generic fiber. However, on the generic fiber, Lemma~\ref{zar-an-trace} and Lemma~\ref{lemma:gen-fiber-Scholze} imply that Diagram~(\ref{diagram:faltings-trace-etale-local}) can be identified with 
\[
\begin{tikzcd}[column sep = 4em]
f^*(\rm{R}^d\mu_{Y, *}\wdh{\O}_Y) \arrow{r}{f^*(\rm{Sch}^d_Y)^{-1}} \arrow{d} & f^* \Omega^d_Y(-d) \arrow{d} \\
\rm{R}^d\mu_{X, *}\wdh{\O}_X \arrow{r}{(\rm{Sch}^d_X)^{-1}} &  \Omega^d_X(-d).
\end{tikzcd}
\]
This diagram commutes because $\rm{Sch}^d_X$ comes from a morphism of sheaves on the \'etale site.  
\end{proof}

\section{Local Duality}\label{section:local-duality}
\subsection{Overview}
Throughout this section, we fix an algebraically closed complete rank-$1$ valued field $C$ with ring of integers $\O_C$, maximal ideal $\m\subset \O_C$, and residue field $k$. We also assume that $\O_C$ is of mixed characteristic $(0, p)$. \smallskip

Throughout this section, we will freely use the notion of (derived) almost Hom sheaves $\bf{R}\ud{al\cal{H}om}$. We refer to \cite[Definition 3.5.6]{Z3} for the precise definition of this object. \smallskip

The main goal of Section~\ref{section:local-duality} is to show that, for any nice admissible formal $\O_C$-scheme $\X$ (see Definition~\ref{defn:nice-formal-scheme}) with generic fiber $X=\X_C$, Faltings' trace $\rm{Tr}_{F, \X}\colon \bf{R}\nu_*\left(\O_X^{+, a}/p\right) \to \omega^{\bullet, a}_{\X_0}(-d)[-2d]$ induces an {\it almost perfect pairing} (see Section~\ref{section:notation})
\begin{equation}\label{eqn:almost-perfect-pairing}
\bf{R}\nu_*\left(\O_X^{+, a}/p\right) \otimes^L_{\O_{\X_0}} \bf{R}\nu_* \left(\O_X^{+, a}/p\right) \xr{\cup} \bf{R}\nu_*\left(\O_X^{+, a}/p\right) \xr{\rm{Tr}_{F, \X}} \omega^{\bullet, a}_{\X_0}(-d)[-2d].
\end{equation}

\begin{defn}\label{defn:almost-perfect-pairing} Let $\X$ be a nice admissible formal $\O_C$-scheme with generic fiber $X=\X_C$. Then the paring~(\ref{eqn:almost-perfect-pairing}) is called {\it Faltings' pairing} with the corresponding duality morphism 
\[
D_\X\colon \bf{R}\nu_*\left(\O_X^{+, a}/p\right) \to \bf{R}\ud{al\cal{H}om}_{\X_0}\left(\bf{R}\nu_* \left(\O_X^{+, a}/p\right), \omega^{\bullet,a}_{\X_0}(-d)[-2d]\right).
\]
\end{defn}

\begin{thm}\label{thm:main-thm-?} (Theorem~\ref{thm:local-duality}) Let $\X$ be an admissible nice formal $\O_C$-scheme. Then Faltings' pairing is almost perfect, i.e. the duality morphism $D_\X$ is an isomorphism in $\bf{D}(\X_0)^a$. 
\end{thm}

We now explain the main steps involved in the proof in more detail:
\begin{enumerate}
    \item We show the claim for polystable formal $\O_C$-schemes (see Definition~\ref{defn:polystable}) with smooth generic fiber. The proof in this case is strongly motivated by the calculations in \cite[\textsection 3]{Ces}. Namely, we choose an explicit affinoid perfectoid cover of the model polystable formal scheme. This covering turns out to be a $\Gamma\simeq \bf{Z}_p(1)^d$-torsor that allows us to reduce the claim to the almost duality in the continuous cohomology groups of $\Gamma$. This duality claim is relatively easy in the case of a smooth model, but becomes quite subtle for a general polystable model.  \smallskip 
    
    \item Then we use \cite[Theorem 1.4]{Z1} to relate any admissible formal $\O_C$-scheme with smooth generic fiber of dimension $d$ to polystable formal $\O_C$-schemes. Roughly, \cite[Theorem 1.4]{Z1} says that, locally on $\X$, any rig-smooth formal model is isomorphic to a polystable formal scheme up to rig-isomorphisms and quotients by a finite group with free action on the generic fiber. So, we only need to show that the property that $D_\X$ is an isomorphism descends through rig-isomorphisms and ``good quotients'' by finite groups. 
    
    \item We show that almost perfectness of Faltings' pairing descends through rig-isomorphisms. The key input here is the almost version of Grothendieck Duality and Corollary~\ref{cor:descend-rig-finite-etale}. 
    
    \item We show that almost perfectness of Faltings' pairing descends through ``good'' quotients by an action of a finite group $G$. The key input here is the almost version of Grothendieck Duality, duality between homotopy invariants and coinvariants for an action of a finite group $G$, and Corollary~\ref{cor:descend-rig-finite-etale}.
\end{enumerate}

\subsection{Preliminaries on polystable formal schemes}\label{section:preliminary-polystable}

The main goal of this section is to perform certain computations on (model) polystable formal $\O_C$-schemes that will be later used to establish almost perfectness of Faltings' trace map. In particular, we discuss standard perfectoid coverings of (model) polystable formal $\O_C$-schemes, and we construct an isomorphism (see Theorem~\ref{thm:structure-nearby-framed})
\[
\frac{\rm{R}^d\nu_* \wdh{\O}^+_X}{\left(\rm{R}^d\nu_* \wdh{\O}^+_X\right)[1-\zeta_p]} \simeq \omega_\X(-d)
\]
generalizing (via similar methods) the isomorphism from \cite[Theorem 8.3]{BMS1} to the case of polystable formal schemes.\smallskip

There are slightly different non-equivalent definition of polystable formal schemes in the literature. For this reason, we explicitly recall the definition that we are going to use in this paper:

\begin{defn} A {\it model polystable formal $\O_C$-scheme} $\X$ is an affine formal $\O_C$-scheme of the form 
\[
    \Spf \frac{\O_C\langle t_{1,0},\dots, t_{1, n_1}, \dots, t_{l,0}, \dots, t_{l, n_l}\rangle}{(t_{1,0}\cdots t_{1, n_1}-\varpi_1, \dots, t_{l,0}\cdots t_{l, n_l}-\varpi_l)}
\]
for some $\varpi_i\in \O_C \setminus \{0\}$. 
\end{defn}

\begin{defn}\label{defn:polystable} An admissible formal $\O_C$-scheme $\X$ is called {\it polystable} if \'etale locally it admits an \'etale morphism to a model polystable formal $\O_C$-scheme. 
\end{defn}

\begin{rmk} Our definition of a polystable formal $\O_C$-scheme might differ from some sources. For example, our definition implies that a polystable formal $\O_C$-scheme $\X$ is rig-smooth, while some other definitions do not impose this assumption (i.e. allow $\varpi_i=0$).  
\end{rmk}

\begin{lemma}\label{lemma:integrally-closed} Let $\Spf A$ be an affine polystable formal $\O_C$-scheme. Then $A$ is equal to $A_K^{\circ}$. In particular, $A$ is integrally closed in $A_K$.
\end{lemma}
\begin{proof}
    We note that $A=A_K^{\circ}$ due to Lemma~\ref{lemma:special-fiber-reduced} and the fact that polystable formal schemes have reduced special fiber. In particular, $A$ is integrally closed in $A_K$ since it is always true for the ring of power-bounded elements.
\end{proof}

\subsubsection{Dualizing complex on polystable formal schemes}

This section is devoted to the study of dualizing complexes on polystable formal $\O_C$-schemes. 

\begin{lemma}\label{lemma:alg-polyst-dual} Let $\X$ be a separated polystable formal $\O_C$-scheme with special fiber $\ov{\X}$. Then $\omega_{\X_n}^{\bullet}$ is isomorphic to a rank-$1$ locally free $\O_{\X_n}$-module in degree $-\dim \X_n=-\dim \ov{\X}$. Furthermore, if $\X$ is a model polystable formal $\O_C$-scheme, we have a canonical\footnote{This isomorphism is only canonical after fixing the coordinates $t_{i, j_i}$. But these coordinates are indeed fixed in our definition of a model polystable formal $\O_C$-scheme.} isomorphism $\omega_{\X_n}^\bullet \simeq \O_{\X_n}[\dim \ov{\X}]$.
\end{lemma}
\begin{proof}
We prove the claim for $n=0$, the proof is absolutely the same for $n\geq 1$. Recall that the relative dualizing complex commutes with \'etale base change by Corollary~\ref{cor:etale-morphism-!}. So it suffices to treat the model example 
\[
\X_0=\Spf \frac{\O_C/p[t_{1,0}, \dots, t_{1, n_1}, \dots, t_{l, 0}, \dots t_{l, n_l}]}{\left(t_{1,0} \dots t_{1, n_1} - \varpi_1, \dots, t_{l,0} \dots t_{l, n_l} - \varpi_l \right)} \ .
\]
This can be realized as a closed subscheme 
\[
i\colon \X_0 \to Y_0\coloneqq \bf{A}_{ \O_C/p}^{\dim \X_0+l}
\]
defined by the ideal $I\coloneqq (t_{1,0} \dots t_{1, n_1} - \varpi_1, \dots, t_{l,0} \dots t_{l, n_l} - \varpi_l)$. The sequence $\{t_{1,0} \dots t_{1, n_1} - \varpi_1, \dots, t_{l,0} \dots t_{l, n_l} - \varpi_l\}$ is regular, thus the ideal $I$ is Koszul-regular by \cite[\href{https://stacks.math.columbia.edu/tag/063E}{Tag 063E}]{stacks-project}. \smallskip

Now we note that $\omega_{Y_0}^\bullet \simeq \O_{Y_0}[\dim Y_0]$ since $Y_0$ is an affine space (this isomorphism is canonical after fixing coordinates on $\bf{A}^{\dim \X_0+l}$). Now we recall that $\omega^\bullet_{\X_0}\cong i^!\,\omega_{Y_0}^\bullet$, so \cite[\href{https://stacks.math.columbia.edu/tag/0BR0}{Tag 0BR0}]{stacks-project} provides an identification
\[
\omega_{\X_0}^\bullet \cong \left(\bigwedge^{l} \left(\cal{I}/\cal{I}^2\right)^{\vee}\right)[\dim \X_0],
\]
where $\cal{I}=\widetilde{I}$. The last observation is that the sheaf $\cal{I}/\cal{I}^2$ is free of rank $l$ by the proof of \cite[\href{https://stacks.math.columbia.edu/tag/063H}{Tag 063H}]{stacks-project}. 
\end{proof}

\begin{rmk}\label{rmk:dualizing-complex-on-semistable-models} The same proof shows that the (relative) dualizing complex $\omega^\bullet_X \cong \O_X[d]$ for $X=\Spec k[x_0, \dots, x_n]/(x_0\dots x_n)$ for any field $k$.
\end{rmk}

\begin{lemma}\label{lemma:line-bundle-dualizing} Let $\X$ be a separated polystable formal $\O_C$-scheme of pure dimension $d$. Then $\omega_{\X}^{\bullet}$ is isomorphic to a rank-$1$ locally free $\O_{\X}$-module in degree $-d$. Furthermore, if $\X$ is a model polystable formal $\O_C$-scheme, then $\omega_{\X}^\bullet \simeq \O_\X[d]$. 
\end{lemma}
\begin{proof}
Since dualizing complexes (and modules) commute with \'etale base change (see Lemma~\ref{lemma:dualizing-etale-base-change}), it suffices to prove the claim for a model polystable formal $\O_C$-scheme. In this case, the result follows trivially from Lemma~\ref{lemma:alg-polyst-dual} since 
\[
\omega^\bullet_\X \simeq \bf{R}\lim_n \omega^\bullet_{\X_n}[d] \simeq \bf{R}\lim_n \O_{\X_n}[d] \simeq \O_\X[d]. \qedhere
\]
\end{proof}

\subsubsection{Pro-\'etale cover of a model polystable formal scheme}\label{section:poly-perfectoid-covering}

In this subsection we construct an explicit ``cover'' of the model polystable formal scheme such that on the generic fiber it becomes a pro-finite \'etale cover by an affinoid perfectoid space. The case $\X=\bf{G}_m^d$ was done in \cite{Sch1} and the case of a semi-stable model\footnote{With some extra mild assumptions.} was done in \cite[\textsection 3]{Ces}. The main difference between our approach and the approach in \cite[\textsection 3]{Ces} is that we write some formulas in a more canonical way. It will both later simply the proofs and the notation.  \smallskip

For the rest of the subsection, we fix a ring 
\begin{equation}\label{eqn:model-polystable}
R^+\coloneqq \frac{\O_C\left\langle t_{1,0}, \dots, t_{1, n_1}, \dots, t_{l, 0}, \dots, t_{l, n_l}\right\rangle}{\left(t_{1,0} \dots t_{1, n_1} - \varpi_1, \dots, t_{l,0} \dots t_{l, n_l} - \varpi_l \right)}
\end{equation}
and define rings
\begin{equation}\label{eqn:model-polystable-etale}
R_m^+\coloneqq \frac{\O_C\left\langle  t_{1,0}^{1/p^m}, \dots, t_{1, n_1}^{1/p^m}, \dots, t_{l, 0}^{1/p^m}, \dots, t_{l, n_l}^{1/p^m}\right\rangle}{\left(t_{1,0}^{1/p^m} \dots t_{1, n_1}^{1/p^m} - \varpi_1^{1/p^m}, \dots, t_{l,0}^{1/p^m} \dots t_{l, n_l}^{1/p^m} - \varpi_l^{1/p^m} \right)},
\end{equation}
where we implicitly choose a compatible sequence of roots $\varpi_i^{1/p^m}\in \O_C$ for all $i=1, \dots, l$. We define the rational version of those rings as $R_m\coloneqq R_m^+[1/p]$. \smallskip

Throughout this section, we put $\X=\Spf R^+$, $X=\X_C$, and $d=\dim X = n_1+n_2+\dots +n_l$. \smallskip

\begin{lemma}\label{lemma:finite-cover-integrally-closed} 
\begin{enumerate}
    \item We have an equality $R_m^+=R_m^{\circ}$, where $(-)^{\circ}$ stands for the set of power-bounded elements. In particular, $R_m^+$ is integrally closed in $R_m$;
    \item\label{lemma:finite-cover-integrally-closed-2} $R_m^+$ is finitely presented as an $R^+$-module.
%    \item\label{lemma:finite-cover-integrally-closed-3} there is an isomorphism $\bf{R}\Hom_{\ov{R}^+}(\ov{R}_m^+, \ov{R}^+) \simeq \ov{R}^+/p$, where $\ov{R}^+=R^+/\m_CR^+$ and $\ov{R}^+_m=R^+_m/\m_CR^+_m$. %$\bf{R}\Hom_{R^+}(R_m^+, R^+)$ is concentrated in degree $0$. 
\end{enumerate}  
\end{lemma}
\begin{proof}
(1): This follows formally from Lemma~\ref{lemma:integrally-closed} due to the observation that $\Spf R^+_m$ is a model polystable formal $\O_C$-scheme. \smallskip

(2): Clearly, $R^+_m$ is a finite $R^+$-module that is also $\O_C$-flat. Therefore, it is a finitely presented $R^+$-module due to \cite[Theorem 7.3/4]{B}. \end{proof}

We clearly have maps $R^+ \to R_m^+$ that just send $t_{i,j} \in R^+$ to $t_{i,j}$ considered as elements of $R_m^+$. It is easy to see that the associated map $\Spf R_m^+ \to \Spf R^+$ is finite and rig-\'etale, i.e. its generic fiber is \'etale as a map of rigid spaces. Furthermore, the group $\Gamma_m:=\prod_{i=1}^l \mu_{p^{m}}^{n_i}(C)$ admits a continuous $R^+$-linear action on $R^+_m$. Namely, we can realize $\Gamma_m$ as a subgroup
\begin{equation}\label{eqn:delta-m}
\Gamma_m=\left\{(\e_{1, i_1})_{i_1=0}^{n_1}\times \dots \times (\e_{l, i_l})_{i_l=0}^{n_l}  \in \prod_{i=1}^l \mu_{p^{m}}^{n_i+1}(C) \ | \  \prod_{j=0}^{n_i} \e_{i, j} =1\ \forall
 i=1, \dots l\right\} %\subset \prod_{i=1}^l \mu_{p^{m}}^{n_i+1}(C)
\end{equation}

In this presentation an element $\e\in \Gamma_m$  acts on $R^+_m$ as the multiplication of each coordinate $t_{i,j}^{1/p^m}$ by the corresponding root of unity $\e_{i,j}$. This action is clearly seen to be continuous and $R^+$-linear. Moreover, one can see that the generic fiber 
\[
\Spa(R_m, R_m^+) \to \Spa(R, R^+)
\]
is a finite \'etale $\Gamma_m$-torsor. Note that $\Gamma_m$ is isomorphic $\mu_{p^m}^{d}(C)$ for any $m\geq 1$. \smallskip

We define the ring $R^+_{\infty}\coloneqq (\rm{colim}_m R_{m}^+)^{\wedge}$, where $(-)^{\wedge}$ stands for the $p$-adic completion, and the rational version $R_{\infty}\coloneqq R^+_{\infty}[1/p]$. Our first goal is to see that $(R_{\infty}, R^+_{\infty})$ is a perfectoid pair. \smallskip

For this, we will need the following decomposition of $R^+_\infty$ as an $\O_C$-module:
\begin{equation}\label{model-decomp-OC}
R^+_{\infty} \cong \wdh{\bigoplus}_J \O_C t_{1,0}^{d_{1,0}} \dots  t_{l,n_l}^{d_{l, n_l}},
\end{equation}
\[
\text{where } J=\left\{\left(d_{i,j}\right)_{i=1, j=0}^{l, n_i} \in \Z\left[\frac{1}{p}\right]_{\geq 0}\, | \, \forall i=1, \dots, l \, \exists
 j\in [0, n_i] \text{ such that } d_{i,j}=0\right\}
\]
We give a more conceptual description of this decomposition later in this section.

\begin{defn}\label{defn:integral-perfectoid} We say that a $p$-torsionfree $\O_C$-algebra $A$ is an {\it integral perfectoid ring} if the Frobenius homomorphism $A/p^{1/p} A \xr{x\mapsto x^p} A/p A$ is an isomorphism. 
\end{defn}

\begin{rmk}\label{rmk:infty-integral-perfectoid} This definition coincides with \cite[Definition 3.5]{BMS1} for $p$-torsionfree $\O_C$-algebras by \cite[Lemma 3.10]{BMS1}.
\end{rmk}

\begin{lemma}\label{lemma:pro-etale-torsor} The ring $R^+_{\infty}$ is a $p$-torsionfree integral perfectoid ring, and the pair $(R_{\infty}, R^+_{\infty})$ is a perfectoid pair. Furthermore, in the pro-\'etale site $\Spa(R, R^+)_{\proet}$, the map $\lim_m \Spa(R_m, R_m^+) \to \Spa(R, R^+)$ is a pro-\'etale $\Gamma$-torsor for the pro-finite group $\Gamma\coloneqq \lim_m \Gamma_m$.
\end{lemma}
\begin{proof}
We note that \cite[Lemma 3.20]{BMS1} ensures that it suffices to show that $R^+_\infty$ is a $p$-torsionfree integral perfectoid and that $(R_\infty, R_\infty^+)$ is a Tate-Huber pair. The latter boils down to showing that $R_\infty^+$ is integrally closed in $R_\infty$. \smallskip

In other words, we have to show that $R^+_{\infty}$ is $p$-torsionfree, integrally closed in $R_{\infty}$, and the map $\varphi\colon R^+_{\infty}/p^{1/p} \to R^+_{\infty}/p$ induced by $x\mapsto x^p$ is an isomorphism. \smallskip

The fact $R^+_{\infty}$ is $p$-torsionfree can be either seen directly from decomposition~(\ref{model-decomp-OC}). Lemma~\ref{lemma:finite-cover-integrally-closed} implies that $\colim R_{m}^+$ is integrally closed in $\colim R_m$. Therefore, \cite[Lemma 5.1.2]{Bhatt-notes} guarantees that $R^+_{\infty}$ is integrally closed in $R_{\infty}$. Finally, we can see that the map $\varphi$ is an isomorphism directly from Decomposition~(\ref{model-decomp-OC}). \smallskip

Since each $\Spa(R_m, R_m^+) \to \Spa(R, R^+)$ is a $\Gamma_m$-torsor, we conclude that $\lim_n \Spa(R_m, R_m^+) \to \Spa(R, R^+)$ is a $\Gamma$-torsor.
\end{proof}

\begin{rmk} We pass to the limit in (\ref{eqn:delta-m}) to get a natural inclusion $\Gamma \subset \Z_p(1)^{d+l}$. Similarly, we conclude that $\Gamma\simeq \Z_p(1)^{d}$. Moreover, after choosing a compatible sequence of $p$-power roots of unity $(\zeta_p, \zeta_{p^2}, \dots)$ it becomes isomorphic to $\Z_p^{d}$. 
\end{rmk}

We consider the $\Gamma$-torsor 
\[
f\colon \lim_m \Spa(R_m, R_m^+) \to \Spa(R, R^+).
\]
Using almost vanishing of higher cohomology of affinoid perfectoid objects (see \cite[Lemma 4.12]{Sch1}), Lemma~\ref{lemma:pro-etale-torsor}, and the Cartan-Leray spectral sequence associated to the pro-\'etale $\Gamma$-torsor $f$, we conclude that the natural morphism
\[
\bf{R}\Gamma_{\rm{cont}}(\Gamma, R_\infty^+) \to \bf{R}\Gamma(X, \wdh{\O}^+_X) \simeq \bf{R}\Gamma(\X, \bf{R}\nu_*\, \wdh{\O}^+_X)
\]
is an almost isomorphism (see the proof \cite[Lemma 5.6]{Sch1} for more details). \smallskip

For the next lemma, we recall that there is a functor $(-)^{L\Updelta} \colon \bf{D}_{comp}(S^+, pS^+) \to \bf{D}(\Spf S^+)$ that sends a derived $p$-adically complete $S^+$-module $M$ to its associated ``quasi-coherent'' (in the derived sense) complex of sheaves $M^{L\Updelta}$. We refer to \cite[\textsection 4.8 and Definition 4.8.7]{Z3} for more details.

\begin{lemma}\label{lemma:almost-compute-nu} The map $\bf{R}\Gamma_{\cont}(\Gamma, R^+_\infty)^{L\Updelta} \to \bf{R}\nu_*\,\wdh{\O}^+_X$ is an almost isomorphism. Similarly, the natural map $\widetilde{\bf{R}\Gamma_{\cont}(\Gamma, R^+_\infty/p)} \to \bf{R}\nu_*(\O^+_X/p)$ is an almost isomorphism.
\end{lemma}
\begin{proof}
The discussion before the lemma implies that there is a canonical almost isomorphism
\[
\bf{R}\Gamma_{\rm{cont}}(\Gamma, R_\infty^+) \simeq \bf{R}\Gamma(\X, \bf{R}\nu_*\, \wdh{\O}^+_X).
\]
Since $\bf{R}\nu_*\wdh{\O}^+_X$ is almost (quasi-)coherent by \cite[Theorem 6.13.6]{Z3}, \cite[Corollary 4.8.14]{Z3} implies that the natural morphism
\[
\bf{R}\Gamma_{\cont}(\Gamma, R^+_\infty)^{L\Updelta} \simeq \bf{R}\Gamma(\X, \bf{R}\nu_*\, \wdh{\O}^+_X)^{L\Updelta} \to \bf{R}\nu_*\, \wdh{\O}^+_X
\]
is an almost isomorphism. The proof for $\O_X^+/p$ is the same. 
\end{proof}

The next goal is to get a more conceptual description of Decomposition~(\ref{model-decomp-OC}). We start by considering the split $\O_C$-torus $T$ defined by its functor of points
\[
T(S) = \left\{ \left(x_{1,0}, \dots, x_{1, n_1}, \dots, x_{l,0}, \dots, x_{l, n_l}\right) \in \bf{G}_m^{d+l}(S) \ | \ \prod_{j=0}^{n_i} x_{i, j} =1\ \forall
 i=1, \dots l\right\} \subset \bf{G}_m^{d+l}
\] 
One sees that $T$ is abstractly isomorphic to $\bf{G}_m^{d}$ (but there is no preferred isomorphism). We denote by $\bf{X}(T)\coloneqq \rm{Hom}_{\O_C\rm{-gp}}(T, \bf{G}_m)$ the character group of $T$.

\begin{rmk}\label{rmk:conceptual-delta} For the later use, it will be convenient to use a slightly different description of $\Gamma$. We define 
\[
\Gamma'\coloneqq \rm{T}_p(T) \simeq \rm{Hom}\left(\bf{X}\left(T\right)\left[1/p\right]/\bf{X}\left(T\right), \mu_{p^{\infty}}\left(C\right)\right).
\]
Using the standard inclusion $T \subset \bf{G}_m^{d+l}$, we get an inclusion
\[
\Gamma' \subset \rm{T}_p(\bf{G}_m^{d +l}) \simeq \bf{Z}_p(1)^{d+l}
\]
that coincides with $\Gamma \subset \Z_p(1)^{d+l}$ defined as a limit. 
\end{rmk}

Now we note $T$ admits an $\O_C$-action on $\Spf R^+$ defined on regular functions as 
\[
(x).t_{i,j}=x_{i,j}t_{i,j} \text{ for } x\in T.
\] 
Then we have a canonical decomposition of $R^+$ by weights: 
\begin{equation}\label{decom-char}
R^+ \simeq \wdh{\bigoplus}_{\chi \in \bf{X}(T)} V_{\chi}  
\end{equation}
with $V_{\chi}$ being a one dimensional free $\O_C$-module that corresponds to the character $\chi$. And the multiplication in $R^+$ is induced by canonical isomorphisms\footnote{We use additive notation for elements of $\bf{X}(T)$ in the formula below.}
\[
V_{\chi} \otimes V_{\chi'} \xr{\sim} V_{\chi + \chi'} \ . 
\]
More explicitly, we recall that the inclusion $T \subset \bf{G}_m^{d+l}$ induces the surjection 
\[
    \Z^{d+l}\simeq \bf{X}(\bf{G}_m^{d+l}) \to \bf{X}(T).
\]
So a direct summand $V_{\chi}$ explicitly corresponds to a rank-$1$ free $\O_C$-module
\[
\O_C \prod_{i=1}^n t_{i,0}^{a_{i,0}-a_{i,j_i}}\dots t_{i,n_i}^{a_{i,n_i}-a_{i,j_i}} \subset R^+,
\]
where $(a_{i,j}) \in \Z^{d+l}$ is any lift of $\chi$ and, for each $i$, $j_i$ is the integer such that $a_{i, j_i}$ is the smallest element among all $\{a_{i, j}\}_{j=0}^{n_i}$. Now we extend, the definition of $V_{\chi}$ for all $\chi \in \bf{X}(T)[1/p]$. Namely, for $\chi\in \bf{X}(T)[1/p]$, we choose an integer $m$ such that $p^m\chi\in \bf{X}(T)$ and so
\[
V_{p^m\chi}=\O_C t_{1,0}^{a_{1,0}}\dots t_{l,n_l}^{a_{l,n_l}} \text{ with the condition that at least one } a_{i,j}=0 \ \forall i
\]
for some non-negative integers $a_{i, j}$. Then we define
\[
V_{\chi}\coloneqq \O_C t_{1,0}^{a_{1,0}/p^m}\dots t_{l,n_l}^{a_{l,n_l}/p^m} \subset R_m^+.
\]
Moreover, one easily gets the following lemma:

\begin{lemma}\label{lemma:decomposition-R} In the notation as above, there is a canonical decomposition
\begin{equation}\label{eqn:decomposition:finite-level-OC}
R^+_m \simeq \wdh{\bigoplus}_{\chi \in \frac{1}{p^m}\bf{X}(T)} V_{\chi}
\end{equation}
for each $m\geq 0$. Furthermore, $V_{\chi}$ is a free $\O_C$-module of rank one for each $\chi\in \frac{1}{p^m}\bf{X}(T)$, and the multiplication map on $R^+_m$ induces isomorphisms 
\[
V_{\chi} \otimes V_{\chi'} \xr{\sim} V_{\chi + \chi'}
\]
for each $\chi, \chi' \in \frac{1}{p^m}\bf{X}(T)$. 
\end{lemma}

\begin{rmk} By passing to the limit, Lemma~\ref{lemma:decomposition-R} gives a decomposition
\[
R^+_\infty \simeq \wdh{\bigoplus}_{\chi \in \bf{X}(T)[1/p]} V_{\chi}
\]
that recovers Decomposition~(\ref{model-decomp-OC}).
\end{rmk}

One disadvantage of Lemma~\ref{lemma:decomposition-R} is that Decomposition~(\ref{eqn:decomposition:finite-level-OC}) is only $\O_C$-linear and not $R^+$-linear. We fix it by introducing the following definition: 

\begin{defn}\label{defn:+-chi} For each integer $m$ and an element $\chi\in \frac{1}{p^m}\bf{X}(T) \subset \bf{X}(T)[1/p]$, we define an $R^+$-module $R^+_{\ov{\chi}} \subset R^+_m\subset R_\infty^+$ as follows
\[
R^+_{\ov{\chi}} = \wdh{\bigoplus}_{\chi'\in \chi+\bf{X}(T)} V_{\chi'}.
\]
\end{defn}

\begin{rmk} Definition~\ref{defn:+-chi} is easily seen to depend only on the class $\ov{\chi} \in \frac{\frac{1}{p^m}\bf{X}(T)}{\bf{X}(T)} \subset \frac{\bf{X}(T)[1/p]}{\bf{X}(T)}$. In particular, this definition does not depend on the choice of $m$ such that $p^m\chi \in \bf{X}(T)$ or $\chi\in \ov{\chi}$. 
\end{rmk}

\begin{lemma}\label{lemma:decomposition-R-infty-2} In the notation as above, there is a canonical decomposition
\begin{equation}\label{eqn:O-module-decomposition}
R^+_m \simeq \bigoplus_{\ov{\chi} \in \frac{\frac{1}{p^m}\bf{X}(T)}{\bf{X}(T)}} R^+_{\ov{\chi}}.
\end{equation} 
\end{lemma}
\begin{proof}
    This follows directly from Lemma~\ref{lemma:decomposition-R}. We only note that the direct sum in (\ref{eqn:O-module-decomposition}) is already complete since $\frac{\frac{1}{p^m}\bf{X}(T)}{\bf{X}(T)}$ is a finite group.
\end{proof}

\begin{rmk}\label{rmk:action-explicit} By passing to the limit, Lemma~\ref{lemma:decomposition-R-infty-2} gives a decomposition
\begin{equation}\label{eqn:decomposition-R-infty}
    R^+_\infty \simeq \wdh{\bigoplus}_{\ov{\chi} \in \frac{\bf{X}(T)[1/p]}{\bf{X}(T)}} R^+_{\ov{\chi}}.
\end{equation}
Furthermore, using the presentation of $\Gamma=\rm{Hom}\left(\bf{X}\left(T\right)\left[1/p\right]/\bf{X}\left(T\right), \mu_{p^{\infty}}\left(C\right)\right)$ from Remark~\ref{rmk:conceptual-delta}, one sees that the action of $\Gamma$ on $R^+_{\ov{\chi}}$ is given by the rule
\[
\gamma.x= \gamma(\ov{\chi})x \text{ for any } \gamma \in \Gamma = \rm{Hom}\left(\bf{X}\left(T\right)\left[1/p\right]/\bf{X}\left(T\right), \mu_{p^{\infty}}\left(C\right)\right), \ x\in R_{\ov{\chi}}^+.
\]
\end{rmk}

\begin{lemma}\label{lemma:smooth-case} In the notation as above, suppose $l=1$ and $\varpi_1$ lies in $\O_C^\times$ in (\ref{eqn:model-polystable}) (i.e., $R^+=\frac{\O_C\left\langle t_{0}, \dots, t_{n}\right\rangle}{\left(t_{0} \dots t_{n} - \varpi\right)}$ for $\varpi \in \O_C^\times$). Then the natural morphism 
\[
V_{\chi} \otimes_{\O_C} R^+ \to R^+_{\ov{\chi}}
\]
is an isomorphism for any $\chi\in \bf{X}(T)[1/p]$. In particular, each $R^+_{\ov{\chi}}$ is a free $R^+$-module of rank one. 
\end{lemma}
\begin{proof}
    This can be seen by an explicit computation. We leave details to the interested reader.
\end{proof}

\begin{warning} In general, the map $V_{\chi} \otimes_{\O_C} R^+ \to R^+_{\ov{\chi}}$ is not an isomorphism. Furthermore, the $R^+$-modules $R_{\ov{\chi}}^+$ are not free $R^+$-modules. They are not even $R^+$-flat. 
\end{warning}

\subsubsection{Group cohomology}

Throughout this section, we keep the notation introduced in Section~\ref{section:poly-perfectoid-covering}. \smallskip

The main goal of this section is to get a good understanding of the cohomology groups $\bf{R}\Gamma_{\cont}(\Gamma, R_\infty^+)$. For this, it will be convenient to choose a trivialization $T\cong \bf{G}_m^{d}$ that provides us with a basis $e_1, \dots, e_d$ in $\bf{X}(T)$. In particular, any rational character $\chi\in \bf{X}(T)[1/p]$ admits a unique decomposition

\[
\chi = \frac{a_1}{p^{m_1}}e_1 + \dots + \frac{a_d}{p^{m_d}}e_d
\]
such that $a_i\in \Z$, $m_i\in \Z_{\geq 0}$, and $a_i$ is coprime with $p$ if $m_i>0$. Furthermore, we choose a compatible system of primitive $p$-power roots of $(1, \zeta_p, \zeta_{p^2}, \dots )$ which defines a trivialization $\Gamma \cong \bf{Z}_p^{d}$. Now we use the presentation $\Gamma=\rm{Hom}\big(\bf{X}\left(T\right)\left[1/p\right]/\bf{X}\left(T\right), \mu_{p^{\infty}}\left(C\right)\big)$ from Remark~\ref{rmk:conceptual-delta} to note that the above trivializations provide $\Gamma$ with a basis $\gamma_1, \dots, \gamma_d$ such that, for each $\frac{1}{p^{m}}e_j$, we have
\[
\gamma_i\left(\ov{\frac{1}{p^{m}}e_j}\right) =  \begin{cases} \zeta_{p^{m}} & i=j \\ 1 & i\neq j \end{cases},
\]
where $\ov{\frac{1}{p^{m}}e_j}$ is the class of $\frac{1}{p^{m}}e_j$ in $\bf{X}(T)[1/p]/\bf{X}(T)$.

\begin{lemma}\label{lemma:computation-chi} Let $\chi\in \bf{X}(T)[1/p]$ be a rational character with a decomposition $\chi= \frac{a_1}{p^{m_1}}e_1 + \dots + \frac{a_d}{p^{m_d}}e_d$ as above, and let $m=m_k$ be the maximum of $m_1,\dots, m_d$. Then 
\begin{enumerate}
    \item\label{lemma:computation-chi-0} $\rm{H}^i_{\cont}(\Gamma, R^+_{\ov{\chi}})=0$ for $i>d$. 
    \item\label{lemma:computation-chi-1} $\rm{H}^i_{\cont}(\Gamma, R^+_{\ov{\chi}})$ is a finitely presented $R^+$-module for any $i\geq 0$;
    \item\label{lemma:computation-chi-2} $\rm{H}^i_{\cont}(\Gamma, R^+_{\ov{\chi}})$ is annihilated by $\zeta_{p^{m_k}}-1$ for any $i\geq 0$;
    \item\label{lemma:computation-chi-3} there is a finitely presented $\O_C$-module $M_i$ such that $\rm{H}^i_{\cont}(\Gamma, R^+_{\ov{\chi}})\simeq M_i\otimes_{\O_C} R^+_{\ov{\chi}}$. Furthermore, $\rm{H}^i_{\cont}(\Gamma, R^+_{\ov{\chi}})$ has no non-zero $\m$-torsion elements for any $i\geq 0$.
\end{enumerate}
\end{lemma}
\begin{proof}
    First, we note that \cite[Lemma 7.3]{BMS1} ensures that there is a canonical isomorphism
    \[
    K\left(R^+_{\ov{\chi}}; \gamma_1-1,\dots, \gamma_d-1\right) \simeq \bf{R}\Gamma_{\cont}\left(\Gamma, R^+_{\ov{\chi}}\right),
    \]
    where $K\left(R^+_{\ov{\chi}}; \gamma_1-1,\dots, \gamma_d-1\right)$ is the Koszul complex from \cite[Definition 7.1]{BMS1}. Remark~\ref{rmk:action-explicit} and the discussion before the lemma imply that $\gamma_i$ acts on $R^+_{\ov{\chi}}$ via the multiplication by the element $\gamma_j(\ov{\chi})=\zeta_{p^{m_j}}^{a_j}$. Therefore, we have an isomorphism of complexes
    \[
    K\left(R^+_{\ov{\chi}}; \gamma_1-1,\dots, \gamma_d-1\right) \simeq K\left(\O_C; \zeta_{p^{m_1}}^{a_1}-1, \dots, \zeta_{p^{m_d}}^{a_d}-1\right) \otimes_{\O_C} R^+_{\ov{\chi}}.
    \]
    Since $R^+_{\ov{\chi}}$ is an $\O_C$-flat module, we conclude that 
    \begin{equation}\label{eqn:reduce-to-oc}
        \rm{H}^i_{\cont}(\Gamma, R^+_{\ov{\chi}}) \simeq \rm{H}^i\left( K\left(\O_C; \zeta_{p^{m_1}}^{a_1}-1, \dots, \zeta_{p^{m_d}}^{a_d}-1\right) \right) \otimes_{\O_C} R^+_{\ov{\chi}}.
    \end{equation}
    For brevity, we denote $K\left(\O_C; \zeta_{p^{m_1}}^{a_1}-1, \dots, \zeta_{p^{m_d}}^{a_d}-1\right)$ simply by $K^\bullet$.
    
    (\ref{lemma:computation-chi-0}): This is clear since the Koszul complex $K^\bullet$ is concentrated in degrees $[0, d]$. \smallskip
    
    (\ref{lemma:computation-chi-1}): Note that Lemma~\ref{lemma:finite-cover-integrally-closed}(\ref{lemma:finite-cover-integrally-closed-2}) implies $R^+_{\ov{\chi}}$ is finitely presented over $R^+$ since it is a direct summand of $R^+_{m_k}$ (by our choice of $m_k$). Therefore, Equation~(\ref{eqn:reduce-to-oc}) implies that it suffices to show that $\rm{H}^i\left(K^\bullet\right)$ is a finitely presented $\O_C$-module for any $i\geq 0$. This formally follows from the fact that the terms $K^\bullet$ are free $\O_C$-modules of finite rank, its differentials are $\O_C$-linear, and $\O_C$ is a coherent ring.\smallskip 
    
    (\ref{lemma:computation-chi-2}): Similarly, it suffices to show that $\rm{H}^i\left(K^\bullet\right)$ is annihilated by $\zeta_{p^{m_k}}-1$. First, we note that $\zeta_{p^{m_k}}-1=0$ if $m_k=0$, so we can assume that $m_k>0$. Now \cite[\href{https://stacks.math.columbia.edu/tag/0663}{Tag 0663}]{stacks-project} implies that these groups are annihilated by $\zeta^{a_k}_{p^{m_k}}-1$. Since $a_k$ is coprime with $p$, we conclude that $\zeta_{p^{m_k}}^{a_k}-1$ and $\zeta_{p^{m_k}}-1$ have the same $p$-adic valuation, so $\zeta_{p^{m_k}}-1$ must also annihilate these groups as well. \smallskip
    
    (\ref{lemma:computation-chi-3}): The first part of is clear from the discussion before. Since $\rm{H}^i\left(K^\bullet\right)$ is a finitely presented $\O_C$-module, \cite[Proposition 2.8(i) and 2.10(i)]{Sch1} ensures that it is isomorphic to a finite direct sum $\O_C^{r} \oplus \bigoplus_{j} \O_C/b_j \O_C$ for some $b_j\in \m$. Therefore, it suffices to show that $R^+_{\ov{\chi}}/b$ has no $\m$-torsion for any $b\in \m$. This can be easily seen from the fact that $R^+_{\ov{\chi}}$ is a topologically free $\O_C$-module (this follows directly from Definition~\ref{defn:+-chi}). 
\end{proof}

\begin{rmk}\label{rmk:zetap-1-torsion} We note that Lemma~\ref{lemma:computation-chi}(\ref{lemma:computation-chi-2}) implies that $\rm{H}^i_{\cont}(\Gamma, R^+_{\ov{\chi}})$ is annihilated by $\zeta_p-1$ for any non-trivial $\ov{\chi} \in \bf{X}(T)[1/p]/\bf{X}(T)$.
\end{rmk}

\begin{cor}\label{cor:computation-chi-mod-p} We have
\begin{enumerate}
    \item\label{cor:computation-chi-mod-p-0} $\rm{H}^i_{\cont}(\Gamma, R^+_{\ov{\chi}}/p)=0$ for $i>d$ and $\ov{\chi} \in \bf{X}(T)[1/p]/\bf{X}(T)$;
    \item\label{cor:computation-chi-mod-p-1} $\rm{H}^i_{\cont}(\Gamma, R^+_{\ov{\chi}}/p)$ is a finitely presented $R^+/p$-module for $i\geq 0$ and $\ov{\chi} \in \bf{X}(T)[1/p]/\bf{X}(T)$;
    \item\label{cor:computation-chi-mod-p-2} for each integers $i\geq 0$ and $n>0$, the group $\rm{H}^i_{\cont}(\Gamma, R^+_{\ov{\chi}}/p)$ is annihilated by $p^{1/n}$ for all but finitely many elements $\ov{\chi} \in \bf{X}(T)[1/p]/\bf{X}(T)$.
\end{enumerate}
\end{cor}
\begin{proof}
    First, we note that \cite[Lemma 7.3]{BMS1} implies that 
    \begin{equation}\label{eqn:cohomology-base-change}
        \bf{R}\Gamma_{\cont}(\Gamma, R^+_{\ov{\chi}}) \otimes_{\Z_p}^L \Z/p\Z \simeq \bf{R}\Gamma_{\cont}(\Gamma, R^+_{\ov{\chi}}/p).
    \end{equation}
    Therefore, (\ref{cor:computation-chi-mod-p-0}) and (\ref{cor:computation-chi-mod-p-1}) follow directly from Lemma~\ref{lemma:computation-chi}(\ref{lemma:computation-chi-0}, \ref{lemma:computation-chi-1}). Now (\ref{cor:computation-chi-mod-p-2}) follows directly from Equation~(\ref{eqn:cohomology-base-change}), Lemma~\ref{lemma:computation-chi}(\ref{lemma:computation-chi-2}), and the observation that $v_p(\zeta_{p^m}-1)=\frac{1}{p^m-p^{m-1}}$ approaches $0$ as $m$ grows.
\end{proof}

\begin{cor}\label{cor:group-coh-compute-model-case} Let $i$ be a non-negative integer. Then
\begin{enumerate}
    \item\label{cor:group-coh-compute-model-case-0} $\rm{H}^i_{\cont}(\Gamma, R^+_\infty)=0$ for $i>d$. 
    \item\label{cor:group-coh-compute-model-case-1} there is a canonical isomorphism $\rm{H}^i_{\cont}(\Gamma, R^+) \simeq \bigwedge^i_{R^+} (R^+)^{d}$;
    \item\label{cor:group-coh-compute-model-case-2} there is a decomposition $\rm{H}^i_{\cont}(\Gamma, R^+_\infty)=\rm{H}^i_{\cont}(\Gamma, R^+) \oplus N_i$ such that $N_i$ is an $R^+$-module annihilated by $1-\zeta_p$;
    \item\label{cor:group-coh-compute-model-case-2.5} $\rm{H}^i_{\cont}(\Gamma, R^+_\infty)$ has no $\m$-torsion elements;
    \item\label{cor:group-coh-compute-model-case-3} $\rm{H}^i_{\cont}(\Gamma, R^+_\infty)$ is almost finitely presented over $R^+$.
\end{enumerate}
\end{cor}
\begin{proof}
    (\ref{cor:group-coh-compute-model-case-0}): This can be seen similarly to Lemma~\ref{lemma:computation-chi}(\ref{lemma:computation-chi-0}) using the Koszul complex. \smallskip 

    (\ref{cor:group-coh-compute-model-case-1}): The proof of Lemma~\ref{lemma:computation-chi} shows that $\bf{R}\Gamma_{\cont}(\Gamma, R^+)$ is canonically isomorphic to the Koszul complex $K(R^+;0, 0, \dots, 0)$ with trivial differentials. Therefore, we get a canonical identification $\rm{H}^i_{\cont}(\Gamma, R^+) \simeq \bigwedge^i_{R^+} (R^+)^{d}$. \smallskip
    
    (\ref{cor:group-coh-compute-model-case-2}): First, we note that Equation~(\ref{eqn:decomposition-R-infty}) implies that 
    \[
    \rm{H}^i_{\cont}(\Gamma, R^+_\infty) = \rm{H}^i_{\cont}(\Gamma, R^+) \bigoplus \rm{H}^i_{\cont}\Big(\Gamma, \wdh{\bigoplus\limits_{0\neq \ov{\chi} \in \bf{X}(T)[1/p]/\bf{X}(T)}} R^+_{\ov{\chi}}\Big). 
    \]
    Therefore, it suffices to show that $N_i\coloneqq \rm{H}^i_{\cont}\Big(\Gamma, \wdh{\bigoplus\limits_{0\neq \ov{\chi} \in \bf{X}(T)[1/p]/\bf{X}(T)}} R^+_{\ov{\chi}}\Big)$ is annihilated by $\zeta_p-1$. For this, we note that Remark~\ref{rmk:zetap-1-torsion} ensures that $\rm{H}^i_{\cont}(\Gamma, R^+_{\ov{\chi}})$ is annihilated by $\zeta_p-1$ for each $0\neq \ov{\chi} \in \bf{X}(T)[1/p]/\bf{X}(T)$. Therefore, \cite[Lemma 3.6]{Ces} concludes that 
    \[
    \rm{H}^i_{\cont}\Big(\Gamma, \wdh{\bigoplus\limits_{0\neq \ov{\chi} \in \bf{X}(T)[1/p]/\bf{X}(T)}} R^+_{\ov{\chi}}\Big) \subset  \prod_{0\neq \ov{\chi} \in \bf{X}(T)[1/p]/\bf{X}(T)} \rm{H}^i_{\cont}(\Gamma, R^+_{\ov{\chi}})
    \]
    is annihilated by $\zeta_p-1$ as well. \smallskip
    
    (\ref{cor:group-coh-compute-model-case-2.5}): The proof of (\ref{cor:group-coh-compute-model-case-2}) implies that
    \[
    \rm{H}^i_{\cont}(\Gamma, R^+_\infty) \subset \prod_{\ov{\chi} \in \bf{X}(T)[1/p]/\bf{X}(T)} \rm{H}^i_{\cont}(\Gamma, R^+_{\ov{\chi}}).
    \]
    Therefore, it suffices to show that $\rm{H}^i_{\cont}(\Gamma, R^+_{\ov{\chi}})$ does not have $\m$-torsion elements for each $\chi\in \bf{X}(T)[1/p]/\bf{X}(T)$. This follows from Lemma~\ref{lemma:computation-chi}(\ref{lemma:computation-chi-3}). \smallskip 

    (\ref{cor:group-coh-compute-model-case-3}): We recall that \cite[Lemma 7.3]{BMS1} implies that $\bf{R}\Gamma_{\cont}(\Gamma, R^+_\infty)$ is $p$-adically derived complete and  $\bf{R}\Gamma_{\cont}(\Gamma, R^+_\infty) \otimes_{\Z_p}^L \Z/p\Z \simeq \bf{R}\Gamma_{\cont}(\Gamma, R^+_\infty/p)$. Therefore, \cite[Theorem 2.13.2]{Z3} ensures that it suffices to show that 
    \[
    \bf{R}\Gamma_{\cont}(\Gamma, R^+_\infty/p) \simeq \bigoplus_{\ov{\chi} \in \bf{X}(T)[1/p]/\bf{X}(T)} \bf{R}\Gamma_{\cont}(\Gamma, R^+_{\ov{\chi}}/p)
    \]
    lies in $\bf{D}_{acoh}^b(R^+/p)$. This follows directly from Corollary~\ref{cor:computation-chi-mod-p}. 
\end{proof}

\begin{lemma}\label{lemma:framed-case-decalage} Let $\X=\Spf R^+$ be a model polystable formal $\O_C$-scheme, and let $X=\X_C$ be its generic fiber. Then the natural map
\[
L\eta_{1-\zeta_p}\left(\bf{R}\Gamma_{\cont}\left(\Gamma, R^+_{\infty}\right)\right) \to L\eta_{1-\zeta_p}\left(\bf{R}\Gamma(X, \wdh{\O}^+_X)\right)
\]
is an (honest) isomorphism (see \cite[Corollary 6.5]{BMS1} for the definition of $L\eta_{1-\zeta_p}$). 
\end{lemma}
\begin{proof}
We know that the map $\bf{R}\Gamma_{\cont}\left(\Gamma, R^+_{\infty}\right) \to \bf{R}\Gamma(X, \wdh{\O}^+_X)$ is an almost isomorphism by Lemma~\ref{lemma:almost-compute-nu} (or really the discussion before it). Now \cite[Lemma 8.11(ii)]{BMS1} says that, in order to show that the map \[
L\eta_{1-\zeta_p}\left(\bf{R}\Gamma_{\cont}\left(\Gamma, R^+_{\infty}\right)\right) \to L\eta_{1-\zeta_p}\left(\bf{R}\Gamma(X, \wdh{\O}^+_X)\right)
\]
is an isomorphism, it suffices to show $\rm{H}^i_{\cont}(\Gamma, R^+_{\infty})$ and $\frac{\rm{H}^i_{\cont}(\Gamma, R^+_{\infty})}{\left(1-\zeta_p\right)\rm{H}^i_{\cont}(\Gamma, R^+_{\infty})}$ have no non-zero $\m$-torsion elements for all $i\geq 0$. This follows from Corollary~\ref{cor:group-coh-compute-model-case}(\ref{cor:group-coh-compute-model-case-1}, \ref{cor:group-coh-compute-model-case-2}, \ref{cor:group-coh-compute-model-case-2.5}). 
\end{proof}

\begin{cor}\label{cor:framed-case-decalage-sheaf} Let $\X$ and $X$ be as in Lemma~\ref{lemma:framed-case-decalage} . Then the natural map
\[
\left(L\eta_{1-\zeta_p}\,\bf{R}\Gamma_{\cont}\left(\Gamma, R^+_{\infty}\right)\right)^{L\Updelta} \to L\eta_{1-\zeta_p}\left(\bf{R}\nu_*\wdh{\O}^+_X\right)
\]
is an (honest) isomorphism. 
\end{cor}
\begin{proof}
Using that $\bf{R}\nu_*\wdh{\O}^+_X$ is quasi-coherent and almost coherent (see \cite[Theorem 6.13.6]{Z3}), we conclude that \cite[Lemma 4.8.13,Theorem 4.8.15]{Z3} and \cite[Lemma 6.4]{BMS1} imply that
\[
L\eta_{1-\zeta_p}\left(\bf{R}\nu_*\wdh{\O}^+_X\right)\simeq \left(L\eta_{1-\zeta_p}\bf{R}\Gamma(X, \wdh{\O}^+_X)\right)^{L\Updelta}.
\]
Therefore, the result follows directly from Lemma~\ref{lemma:framed-case-decalage}. 
\end{proof}

\begin{cor}\label{cor:line-bundle-push} Let $\X$ be a polystable formal $\O_C$-scheme with $X=\X_C$ of pure dimension $d$. Then $\frac{\rm{R}^i\nu_*\wdh{\O}^+_X}{\rm{R}^i\nu_*\wdh{\O}^+_X[1-\zeta_p]}$ is a locally free $\O_\X$-module of rank $\binom{d}{i}$ for any $i\geq 0$. In particular, the natural morphism 
\[
\frac{\rm{R}^i\nu_*\wdh{\O}^+_X}{\rm{R}^i\nu_*\wdh{\O}^+_X[1-\zeta_p]} \to \widetilde{\rm{R}^i\nu_*\wdh{\O}^+_X}
\]
is an isomorphism. 
\end{cor}
\begin{proof}
    Since the formation of $\rm{R}^i\nu_*\wdh{\O}^+_X$ commutes with \'etale base change (see \cite[Theorem 6.13.6]{Z3}), it suffices to treat the case of a model polystable formal $\O_C$-scheme $\X=\Spf R^+$. Now \cite[Lemma 6.4]{BMS1} implies that it suffices to show that $\cal{H}^i\left(L\eta_{1-\zeta_p}\left(\bf{R}\nu_*\wdh{\O}^+_X\right)\right)$ is a free $\O_\X$-module of rank $\binom{d}{i}$. This follows from Corollary~\ref{cor:framed-case-decalage-sheaf} and Corollary~\ref{cor:group-coh-compute-model-case}(\ref{cor:group-coh-compute-model-case-1}, \ref{cor:group-coh-compute-model-case-2}) (and \cite[Lemma 4.8.13]{Z3}). 
\end{proof}

\begin{thm}\label{thm:structure-nearby-framed} Let $\X$ be a separated polystable formal $\O_C$-scheme with $X=\X_C$ of pure dimension $d$. Then Faltings' map
\[
\rm{Tr}^{+, d}_{F, \X}\colon \rm{R}^d\nu_*\wdh{\O}^+_X \to \omega_{\X}(-d)
\]
induces the (honest) isomorphism
\[
t^+_\X\colon \frac{\rm{R}^d\nu_*\wdh{\O}^+_X}{(\rm{R}^d\nu_*\wdh{\O}^+_X)[1-\zeta_p]} \xr{\sim} \omega_{\X}(-d) \simeq \omega^\bullet_{\X}(-d)[-d].
\]
\end{thm} 
\begin{proof}
First of all, Lemma~\ref{lemma:line-bundle-dualizing} guarantees that $\omega_{\X}(-d) \simeq \omega^\bullet_{\X}(-d)[-d]$ is a locally free $\O_\X$-module of rank one, while Corollary~\ref{cor:line-bundle-push} ensures that $\frac{\rm{R}^d\nu_*\wdh{\O}^+_X}{(\rm{R}^d\nu_*\wdh{\O}^+_X)[1-\zeta_p]}$ is also a locally free $\O_\X$-module of rank one. In particular, both sheaves are reflexive. Now note that $\X$ has (geometrically) reduced special fiber and smooth generic fiber. Thus, Lemma~\ref{intersection-reducedfiber} says that it is sufficient to check that the map $t^+_\X$ is an isomorphism on the smooth locus $\X^\sm$ and the generic fiber $\X_C$. \smallskip

After unravelling the definition (see Theorem~\ref{thm:faltings-1}), we see that the former case follows from \cite[Theorem 8.3]{BMS1} and the latter case follows from \cite[Proposition 3.23]{Schsurvey}. 
\end{proof}

\subsection{Local duality on formal polystable schemes}

The main goal of this section is to show that Faltings' trace map induces an almost perfect pairing on any separated polystable formal $\O_C$-scheme $\X$. We start the section by establishing duality in continous group cohomology of the profinite group $\Gamma=\Z_p(1)^d$, and then use the standard perfectoid covering from Section~\ref{section:poly-perfectoid-covering} to relate it to perfectness of Faltings' trace on {\it model} polystable formal $\O_C$-schemes. The general case then follow by a version of \'etale descent.

\subsubsection{Digression: duality on semistable varieties}\label{section:digression}

In this section, we do some preliminary computations on semi-stable varieties over a finite field $\bf{F}_p$ (all computations stay the same over an arbitrary field). These computations will be used in the next subsection to show that the complex $\bf{R}\Gamma_\cont(\Gamma, R^+_\infty)$ is almost self-dual. \smallskip

Throughout this section, let us fix a ring $A=\bf{F}_p[t_0, \dots, t_n]/(t_0 \cdots  t_n)$, and $X=\Spec A$. We also put $A_m=\bf{F}_p[t_0^{1/p^m}, \dots, t_n^{1/p^m}]/(t_0^{1/p^m} \cdots t_n^{1/p^m})$ and $X_m=\Spec A_m$ for every integer $m\geq 0$. \smallskip

\begin{defn}\label{defn:trace-am} The {\it trace morphism} $\rm{tr}_m\colon A_m \to A$ is the unique $\bf{F}_p$-linear morphism such that, for every $a_0, \dots, a_n \in \frac{1}{p^m}\Z$, 
\[
\rm{tr}_m\left(t_0^{a_0} \dots  t_n^{a_n}\right) = \begin{cases}
  t_0^{a_0} \dots  t_n^{a_n},   & a_i\in \Z \text{ for all } i \\
  0, & \text{ otherwise.}
\end{cases} 
\]
The trace morphism is easily seen to be $A$-linear. 
\end{defn}

\begin{thm}\label{thm:algebraic-duality-semistable} The trace morphism induces a (derived) perfect pairing
\[
A_m\otimes^L_A A_m \xr{\cdot} A_m \xr{\rm{tr}_m} A
\]
for every integer $m\geq 0$. In other words, the induced morphism
\[
\alpha_m \colon A_m \to \bf{R}\Hom_A(A_m, A)
\]
is an isomorphism for all $m\geq 0$.
\end{thm}
\begin{proof}
    {\it Step~$1$. $\bf{R}\Hom_A(A_m, A)$ is concentrated in degree $0$.} A standard argument (for example, see Remark~\ref{rmk:dualizing-complex-on-semistable-models}) shows that the dualizing complexes (relative to $\bf{F}_p$) on $X=\Spec A$ and $X_m=\Spec A_m$ are isomorphic to $\O_X[n]$ and $\O_{X_m}[n]$ respectively. Since $X_m \to X$ is a finite morphism of noetherian schemes, \cite[\href{https://stacks.math.columbia.edu/tag/0AX2}{Tag 0AX2}]{stacks-project} implies that $A_m[n] \simeq \bf{R}\Hom_A(A_m, A[n])$. In particular, $\bf{R}\Hom_A(A_m, A)$ is concentrated in degree $0$. At this moment, we already know that $A_m$ is abstractly isomorphic to $\bf{R}\Hom_A(A_m, A)$, but we do not know that this isomorphism is defined via the trace morphism. 
    
    {\it Step~$2$. We show that $\rm{tr}_m$ induces an isomorphism $\alpha_m \colon A_m \to \Hom_A(A_m, A)$.} One easily checks that $\alpha_m$ is injective. So it suffices to show that it is surjective. Pick any $\varphi \in \Hom_A(A_m, A)$. Then, for any
    \[
    I\subset [n], \underline{a}=(a_i)_{i\in I} \in \left(\frac{1}{p^m}\Z \cap (0, 1)\right)^{\#I},
    \]
    we have
    \[
    \left(\prod_{j\in [n]\setminus I} t_j\right)\cdot\varphi\left(\prod_{i\in I} t_i^{a_i}\right) = \varphi\left(\prod_{i\in I} t_i^{a_i} \cdot \prod_{j\in [n]\setminus I} t_j\right)=0.
    \]
    Therefore, for any such $I$ and $\underline{a}$, there is an element $b_{I, \underline{a}}\in A$ such that 
    \[
    \varphi\left(\prod_{i\in I} t_i^{a_i}\right) = b_{I, \underline{a}} \prod_{i\in I} t_i. 
    \]
    Then one sees that 
    \[
    \alpha_m(x_\varphi)=\varphi,
    \]
    where 
    \[
    x_{\varphi} = \sum_{I\subset [n], \underline{a}=(a_i)_{i\in I} \in \left(\frac{1}{p^m}\Z \cap (0, 1)\right)^{\#I}} b_{I, \underline{a}} \prod_{i\in I} t_i^{1-a_i}. \qedhere
    \]
\end{proof}

\subsubsection{Duality on polystable formal schemes}

Throughout this section, we keep the notation of Section~\ref{section:poly-perfectoid-covering}. In particular, $R^+$ is a model polystable algebra as in (\ref{eqn:model-polystable}) and $R^+_m$ is its rig-\'etale cover as in (\ref{eqn:model-polystable-etale}). \smallskip

The main goal of this section is to show an analogue of Theorem~\ref{thm:algebraic-duality-semistable} for $R_m^+$ and $R^+$. \smallskip

\begin{defn} The {\it trace morphism} $\rm{tr}^+_m \colon R^+_m \to R^+$ is defined to be the projection 
\[
R_m^+ = \bigoplus_{\ov{\chi} \in \frac{\frac{1}{p^m}\bf{X}(T)}{\bf{X}(T)}} R^+_{\ov{\chi}} \to R^+
\]
onto the factor corresponding to the trivial character (see Lemma~\ref{lemma:decomposition-R-infty-2}). 
\end{defn}

\begin{lemma}\label{lemma:dual-complete} Let $A$ be a topologically finitely presented $\O_C$-algebra, and $M \in \bf{D}^-_{acoh}(A)$. Then $\bf{R}\Hom_A(M, A)\in \bf{D}^+_{acoh}(A)$, and it is derived $p$-adically complete. 
\end{lemma}
\begin{proof}
    The first claim follows from \cite[Proposition 2.6.19 and 2.4.8(i)]{Z3}. The second claim follows from \cite[\href{https://stacks.math.columbia.edu/tag/0A6E}{Tag 0A6E}]{stacks-project} and the observation that $A$ is $p$-adically complete (thus, it is derived $p$-adically complete by \cite[\href{https://stacks.math.columbia.edu/tag/091T}{Tag 091T}]{stacks-project}).
\end{proof}

\begin{lemma}\label{lemma:kunneth-polystable-model} Let $A$ and $B$ be flat, topologically finite type $\O_C$-algebras, let $M$ be a finitely presented $A$-module, and $N$ a finitely presented $B$-module. Suppose that both $M$ and $N$ are $\O_C$-flat, then the natural morphism 
    \[
    \bf{R}\Hom_{A}(M, A) \wdh{\otimes}^L_{\O_C}\bf{R}\Hom_{B}(N, B) \to \bf{R}\Hom_{A\wdh{\otimes}_{\O_C} B}(M\wdh{\otimes}_{\O_C} N, A\wdh{\otimes}_{\O_C} B)
    \]
    is an isomorphism\footnote{To construct the map, we note that $\bf{R}\Hom_{R\wdh{\otimes}_{\O_C} S}(M\wdh{\otimes}_{\O_C} N, R\wdh{\otimes}_{\O_C} S)$ is derived $p$-adically complete by Lemma~\ref{lemma:dual-complete} applied to $R\wdh{\otimes}_{\O_C} S$ and $M\wdh{\otimes}_{\O_C} N$.}, where $\wdh{\otimes}$ stands for the usual completed tensor product and $\wdh{\otimes}^L_{\O_C}$ for the derived completed derived tensor product. 
\end{lemma}
\begin{proof}
    For brevity, let us denote $A\wdh{\otimes}_{\O_C} B$ by $S$. The derived Nakayama Lemma (see \cite[\href{https://stacks.math.columbia.edu/tag/0G1U}{Tag 0G1U}]{stacks-project}) and \cite[\href{https://stacks.math.columbia.edu/tag/0A6A}{Tag 0A6A}]{stacks-project} imply that it suffices to show that the natural morphism
    \[
    \bf{R}\Hom_{A/p}(M/p, A/p) \otimes^L_{\O_C/p}\bf{R}\Hom_{B/p}(N/p, B/p) \to \bf{R}\Hom_{S/p}(M/p\otimes_{\O_C/p} N/p, S/p)
    \]
    is an isomorphism. Since $M/p$ and $N/p$ are pseudo-coherent, we can choose resolutions $P^\bullet \to M/p$ and $Q^\bullet \to N/p$ by bounded above complexes of finite free modules. Since the terms of $P^\bullet$ and $Q^\bullet$ are $\O_C/p$-flat, we conclude that $P^\bullet \otimes^\bullet_{\O_C/p} Q^\bullet \to M/p\otimes_{\O_C/p} N/p$ is a resolution by finite free $S/p$-modules. Then the claim boils down to showing that the natural morphism of complexes
    \[
    \Hom_{A/p}(P^\bullet, A/p) \otimes^\bullet_{\O_C/p} \Hom_{B/p}(Q^\bullet, B/p) \to \Hom_{S/p}(P^\bullet \otimes^\bullet_{\O_C/p} Q^\bullet, S/p).
    \]
    is a quasi-isomorphism. This morphism can be easily seen to be even an isomorphism of complexes. 
\end{proof}

\begin{prop}\label{prop:derived-duality-Rm} In the notation as above, the pairing
\[
R_m^+ \otimes^L_{R^+} R_m^+ \xr{ \cdot } R_m^+ \xr{\rm{tr}_m^+} R^+
\]
is perfect, i.e. the natural duality morphism
\[
R_m^+ \to \bf{R}\Hom_{R^+}(R_m^+, R^+)
\]
is an isomorphism.
\end{prop}
\begin{proof}
    First, Lemma~\ref{lemma:kunneth-polystable-model} implies that it suffices to treat the case when $l=1$ in Equation~(\ref{eqn:model-polystable}). So we assume that $R^+=\frac{\O_C\left\langle t_{0}, \dots, t_{n}\right\rangle}{\left(t_{0} \dots t_{n} - \varpi\right)}$ for some $\varpi\in \O_C \setminus \{0\}$. \smallskip
    
    {\it Case~$1$. $\varpi\in \O_C^\times$. } In this case, Lemma~\ref{lemma:decomposition-R-infty-2} and Lemma~\ref{lemma:smooth-case} guarantee that we have a natural decomposition
    \[
    R_m^+ \simeq \bigoplus_{\ov{\chi}} R_{\ov{\chi}}^+ \simeq \bigoplus_{\ov{\chi}} V_{\chi} \otimes_{\O_C} R^+,
    \]
    where the sum is taken over $\ov{\chi}\in \frac{\frac{1}{p^m}\bf{X}(T)}{\bf{X}(T)}$. Since $R^+_{\ov{\chi}}$ pairs non-trivially only with elements of $R^+_{-\ov{\chi}}$, we conclude that it suffices to show that the natural morphism
    \[
    \bigoplus_{\ov{\chi}} V_{\chi} \otimes_{\O_C} R^+ \to \bigoplus_{\ov{\chi}} \bf{R}\Hom_{R^+}(V_{-\chi} \otimes_{\O_C} R^+, R^+)
    \]
    is an isomorphism. This can be checked for each $\ov{\chi}$ separately. Now \cite[\href{https://stacks.math.columbia.edu/tag/0A6A}{Tag 0A6A}]{stacks-project} implies that it suffices to show that the natural morphism
    \[
    V_{\chi} \to \bf{R}\Hom_{\O_C}(V_{-\chi}, \O_C)
    \]
    is an isomorphism. This is obvious since both $V_{\chi}$ and $V_{-\chi}$ are free $\O_C$-modules, and $V_{\chi} \otimes V_{-\chi} \simeq \O_C$ via the multiplication map. \smallskip
    
    {\it Case~$2$. $\varpi\in \m_C\setminus \{0\}$.} Choose $\pi\in \m_C$ such that $1>|\pi|>|\varpi^{1/p^n}|$ and $1>|\pi|>|p|$. We note that Lemma~\ref{lemma:finite-cover-integrally-closed}(\ref{lemma:finite-cover-integrally-closed-2}) and Lemma~\ref{lemma:dual-complete} imply that both $R_m^+$ and $\bf{R}\Hom_{R^+}(R_m^+, R^+)$ are derived $\pi$-adically complete. Therefore, \cite[\href{https://stacks.math.columbia.edu/tag/0A6A}{Tag 0A6A}]{stacks-project} and the derived Nakayama Lemma (see \cite[\href{https://stacks.math.columbia.edu/tag/0G1U}{Tag 0G1U}]{stacks-project}) imply that it suffices to show that the natural morphisms
    \[
    R_m^+/\pi \to \bf{R}\Hom_{R^+/\pi}(R_m^+/\pi, R^+/\pi)
    \]
    are isomorphisms. Now we note that 
    \[
    R^+/\pi \simeq \frac{(\O_C/\pi)\left[t_0, \dots, t_n\right]}{\left(t_0 \dots  t_n\right)}, \, \text{ and }\, R^+_m/\pi \simeq \frac{(\O_C/\pi)\left[t_0^{1/p^m}, \dots, t_n^{1/p^m}\right]}{\left(t_0^{1/p^m} \dots  t_n^{1/p^m}\right)}.
    \]
    Since $p\in \pi\O_C$ by our choice of $\pi$, we see that 
    \[
    R^+/\pi \simeq \frac{\bf{F}_p\left[t_0, \dots, t_n\right]}{(t_0 \dots t_n)} \otimes_{\bf{F}_p} \O_C/\pi \simeq A\otimes_{\bf{F}_p} \O_C/\pi,
    \]
    \[
    R^+_m/\pi \simeq \frac{\bf{F}_p\left[t_0^{1/p^m}, \dots, t_n^{1/p^m}\right]}{\left(t_0^{1/p^m} \dots  t_n^{1/p^m}\right)} \otimes_{\bf{F}_p} \O_C/\pi \simeq A_m\otimes_{\bf{F}_p} \O_C/\pi,
    \]
    where $A$ and $A_m$ are as in Section~\ref{section:digression}. Furthermore, one sees that $\rm{tr}_m^+/\pi \colon R^+_m/\pi \to R^+/\pi$ can be identified with 
    \[
    \rm{tr}_m\otimes_{\bf{F}_p} \O_C/\pi \colon A_m\otimes_{\bf{F}_p} \O_C/\pi \to A\otimes_{\bf{F}_p} \O_C/\pi,
    \]
    where $\rm{tr}_m\colon A_m \to A$ is from Definition~\ref{defn:trace-am}. Therefore, an easy application of \cite[\href{https://stacks.math.columbia.edu/tag/0A6A}{Tag 0A6A}]{stacks-project} ensures that it suffices to show that the natural morphism (induced by $\rm{tr}_m$)
    \[
    A_m \to \bf{R}\Hom_A(A_m, A)
    \]
    is an isomorphism. This was already checked in Theorem~\ref{thm:algebraic-duality-semistable}. 
\end{proof}

\begin{cor}\label{cor:derived-duality} In the notation of Definition~\ref{defn:+-chi}, the pairing
\[
R_{\ov{\chi}}^+ \otimes^L_{R^+} R_{-\ov{\chi}}^+ \xr{ \cdot } R^+
\]
is a perfect (derived) pairing for any $\ov{\chi} \in \frac{\frac{1}{p^m}\bf{X}(T)}{\bf{X}(T)}$. More precisely, the natural duality morphism
\[
R_{\ov{\chi}}^+ \to \bf{R}\Hom_{R^+}(R_{-\ov{\chi}}^+, R^+)
\]
is an isomorphism, and the same with $\ov{\chi}$ and $-\ov{\chi}$ interchanged.
\end{cor}
\begin{proof}
    This follows directly from Proposition~\ref{prop:derived-duality-Rm}, Lemma~\ref{lemma:decomposition-R-infty-2}, and the observation that $R^+_{\ov{\chi}}$ pairs non-trivially only with the elements of $R^+_{-\ov{\chi}}$. 
\end{proof}

\begin{rmk}\label{rmk:duality-mod-p-chi} Corollary~\ref{cor:derived-duality} easily implies that the pairing
\[
R^+_{\ov{\chi}}/p\otimes^L_{R^+/p}  R^+_{-\ov{\chi}}/p \xr{\cdot} R^+/p
\]
is perfect. 
\end{rmk}

\subsubsection{Almost duality in group cohomology}

Throughout this section, we keep the notation of Section~\ref{section:poly-perfectoid-covering}. In particular, $R^+$ is a model polystable $\O_C$-algebra, $R^+_\infty$ is its standard perfectoid covering such that $\Spf R_\infty^+ \to \Spf R^+$ is a pro-\'etale $\Gamma\simeq \Z_p(1)^d$-torsor on generic fibers. \smallskip

The main goal of this section is to show that $\bf{R}\Gamma_\cont(\Gamma, R_\infty^+)$ is an almost self-dual element of $\bf{D}(R^+)$. We start by studying some general properties of continuous cohomology of the group $\Gamma$. 

\begin{defn}\label{defn:trace-map-group-cohomology} The {\it trace morphism} $\rm{tr}_\Gamma \colon \bf{R}\Gamma_{\cont}(\Gamma, R^+) \to R^+[-d]$ is the composition
\[
\bf{R}\Gamma_{\cont}(\Gamma, R^+) \to \rm{H}^d_{\cont}(\Gamma, R^+)[-d] \xr{\sim} R^+[-d],
\]
where the first map is a ``projection'' onto the top cohomology group and the second map is the canonical isomorphism from Lemma~\ref{cor:group-coh-compute-model-case}(\ref{cor:group-coh-compute-model-case-1}). Similarly, one defines
\[
\rm{tr}_\Gamma/p \colon \bf{R}\Gamma_{\cont}(\Gamma, R^+/p) \to R^+/p[-d].
\]
\end{defn}

\begin{lemma}\label{lemma:general-duality-group-coh} Let $M$ and $N$ be $R^+/p$-modules with a continuous $R^+/p$-linear $\Gamma$-action, and let $M\otimes^L_{R^+/p} N \xr{\varphi} R^+/p$ be a $\Gamma$-invariant perfect pairing (i.e. $M \to \bf{R}\Hom_{R^+/p}(N, R^+/p)$ is an isomorphism, and the same with $N$ and $M$ interchanged). Then the pairing
\[
\bf{R}\Gamma_{\cont}(\Gamma, M) \otimes^L_{R^+/p} \bf{R}\Gamma_{\cont}(\Gamma, N) \xr{\cup} \bf{R}\Gamma_{\cont}(\Gamma, M\otimes^L_{R^+/p} N) \xr{\bf{R}\Gamma_{\cont}(\Gamma, \varphi)} \bf{R}\Gamma_{\cont}(\Gamma, R^+/p) \xr{\rm{tr}_\Gamma/p} R^+/p[-d]
\]
is perfect, i.e., the natural morphism
\[
\bf{R}\Gamma_{\cont}(\Gamma, M) \to \bf{R}\Hom_{R^+/p}\left(\bf{R}\Gamma_{\cont}(\Gamma, N), R^+/p[-d]\right)
\]
is an isomorphism, and the same with $M$ and $N$ interchanged.
\end{lemma}
\begin{proof}
    First, we choose some trivialization $\Gamma \cong \Z_p^d$. Then we note \cite[Lemma 7.3]{BMS1} ensures that the natural morphism
    \[
    \bf{R}\Gamma(\Z^d, R^+/p) \to \bf{R}\Gamma_{\cont}(\Gamma, R^+/p)
    \]
    is an isomorphism, and the same applies to cohomology with coefficients with $M$ and $N$. Therefore, it suffices to show analogous statement where we replace all continuous cohomology of $\Gamma$ with cohomology of the (discrete) group $\Z^d$. \smallskip
    
    Now we recall that the classifying space $B(\Z^d)$ is isomorphic to the $d$-dimensional torus $(\bf{S}^1)^d$, and so $M$ (resp. $N$) (functorially) correspond to a locally constant sheaf of $R^+/p$-modules $\ud{M}$ (resp. $\ud{N}$) on $(\bf{S}^1)^d$. Furthermore, the induced pairing 
    \[
    \ud{M}\otimes^L_{\ud{R}^+/p} \ud{N} \to \ud{R}^+/p
    \]
    is perfect, and $\bf{R}\Gamma(\Z^d, M)\simeq \bf{R}\Gamma((\bf{S}^1)^d, \ud{M})$ (and the same for $N$ and $R^+/p$). Therefore, it suffices to show that the pairing
    \[
    \bf{R}\Gamma\left(\left(\bf{S}^1\right)^d, \ud{M}\right) \otimes^L_{\ud{R}^+/p} \bf{R}\Gamma\left(\left(\bf{S}^1\right)^d, \ud{N}\right) \xr{\cup} \bf{R}\Gamma\left(\left(\bf{S}^1\right)^d, \ud{M}\otimes^L_{\ud{R}^+/p} \ud{N}\right) \xr{\bf{R}\Gamma\left(\left(\bf{S}^1\right)^d, \ud{\varphi}\right)} \bf{R}\Gamma\left(\left(\bf{S}^1\right)^d, \ud{R}^+/p\right) \to R^+/p[-d]
    \]
    is perfect, where the last map is the trace map in topological Poincar\'e Duality. This follows\footnote{This implication explicitly uses that the trace map $\bf{R}\Gamma((\bf{S}^1)^d, \ud{R}^+/p) \xr{\rm{tr}_{(\bf{S}^1)^d}} R^+/p[-d]$ in topological Poincar\'e duality matches the explicitly defined trace map $\rm{tr}_{\Gamma}/p$ up to an invertible unit $u\in (R^+/p)^\times$. To see this, one should note that both trace maps induce an isomorphism of the top cohomology group with the free rank-$1$ module $R^+/p$.} from the usual topological Poincar\'e Duality on the $d$-dimensial torus $(\bf{S}^1)^d$.
\end{proof}

\begin{lemma}\label{lemma:duality-chi} Let $\ov{\chi} \in \bf{X}(T)[1/p]/\bf{X}(T)$. Then, in the notation of Definition~\ref{defn:+-chi}, the pairing
\[
\bf{R}\Gamma_{\cont}(\Gamma, R^+_{\ov{\chi}}) \otimes^L_{R^+} \bf{R}\Gamma_{\cont}(\Gamma, R^+_{-\ov{\chi}}) \xr{\cup} \bf{R}\Gamma_{\cont}(\Gamma, R_{\ov{\chi}}^+ \otimes^L_{R^+} R_{-\ov{\chi}}^+) \xr{\bf{R}\Gamma_{\cont}(\Gamma, \cdot)} \bf{R}\Gamma_{\cont}(\Gamma, R^+) \xr{\rm{tr}_\Gamma} R^+[-d]
\]
is perfect.
\end{lemma}
\begin{proof}
    First, Lemma~\ref{lemma:computation-chi} implies that $\bf{R}\Gamma_{\cont}(\Gamma, R^+_{-\ov{\chi}})$ lies in $D^b_{coh}(R^+)$. Then Lemma~\ref{lemma:kunneth-polystable-model} and the derived Nakayama Lemma (see \cite[\href{https://stacks.math.columbia.edu/tag/0G1U}{Tag 0G1U}]{stacks-project}) imply that it suffices to prove the result for the (continuous) cohomology of $R^+_{\ov{\chi}}/p$ and $R^+_{-\ov{\chi}}/p$. This follows directly from Lemma~\ref{lemma:general-duality-group-coh} and Remark~\ref{rmk:duality-mod-p-chi}. 
\end{proof}

\begin{lemma}\label{lemma:derived-complete-infty-cohomology} In the notation of Section~\ref{section:poly-perfectoid-covering}, $\bf{R}\Gamma_\cont(\Gamma, R_\infty^+)$ and $\bf{R}\Hom_{R^+}\left(\bf{R}\Gamma_\cont(\Gamma, R^+_\infty), R^+\right)$ are derived $p$-adically complete.
\end{lemma}
\begin{proof}
    \cite[\href{https://stacks.math.columbia.edu/tag/091T}{Tag 091T}]{stacks-project} and \cite[\href{https://stacks.math.columbia.edu/tag/091U}{Tag 091U}]{stacks-project} imply that it suffices to show that $\bf{R}\Gamma_\cont(\Gamma, R_\infty^+)$ and $\bf{R}\Hom_{R^+}\left(\bf{R}\Gamma_\cont(\Gamma, R^+_\infty), R^+\right)$ have $p$-adically complete cohomology modules. Then \cite[Lemma 2.12.7]{Z3} implies that it suffices to show that these cohomology modules are almost finitely generated. Then the result follows from Corollary~\ref{cor:group-coh-compute-model-case}(\ref{cor:group-coh-compute-model-case-3}) and \cite[Proposition 2.6.19]{Z3}. 
\end{proof}

\begin{defn}\label{defn:trace-map-group-cohomology-infty} The {\it trace morphism} $\rm{tr}_{\Gamma, \infty} \colon \bf{R}\Gamma_{\cont}(\Gamma, R^+_\infty) \to R^+[-d]$ is the composition
\[
\bf{R}\Gamma_{\cont}(\Gamma, R^+_\infty) \to \bf{R}\Gamma_{\cont}(\Gamma, R^+) \xr{\rm{tr}_\Gamma} R^+[-d],
\]
where the first map is induced by the projection onto $R^+$ coming from Decomposition~(\ref{eqn:decomposition-R-infty}) and the second map is the trace map from Definition~\ref{defn:trace-map-group-cohomology}. Similarly, one defines
\[
\rm{tr}_{\Gamma, \infty}/p \colon \bf{R}\Gamma_{\cont}(\Gamma, R^+_\infty/p) \to R^+/p[-d].
\]
\end{defn}

\begin{rmk}\label{rmk:trace-quotient} The trace map $\rm{tr}_{\Gamma, \infty}$ can also be defined as the composition
\[
\bf{R}\Gamma_{\cont}(\Gamma, R^+_\infty) \to \rm{H}^d_\cont(\Gamma, R^+_{\infty})[-d] \to \frac{\rm{H}^d_\cont(\Gamma, R^+_\infty)}{\rm{H}^d_\cont(\Gamma, R^+_\infty)[1-\zeta_p]}[-d] \simeq R^+[-d].
\]
\end{rmk}

\begin{thm}\label{thm:almost-duality-group-cohomology} In the notation of Section~\ref{section:poly-perfectoid-covering}, the pairing
\[
\bf{R}\Gamma_{\cont}(\Gamma, R_\infty^+) \otimes^L_{R^+} \bf{R}\Gamma_\cont(\Gamma, R^+_\infty) \xr{\cup} \bf{R}\Gamma_\cont(\Gamma, R^+_\infty) \xr{\rm{tr}_{\Gamma, \infty}} R^+[-d]
\]
is (derived) almost perfect, i.e., the natural morphism
\[
\bf{R}\Gamma_{\cont}(\Gamma, R_\infty^+) \to \bf{R}\rm{Hom}_{R^+}\left(\bf{R}\Gamma_{\cont}(\Gamma, R_\infty^+), R^+[-d]\right)
\]
is an {\it almost} isomorphism. 
\end{thm}
\begin{proof}
    Using Lemma~\ref{lemma:derived-complete-infty-cohomology}, \cite[Lemma A.4]{Z3}, and derived Nakayama Lemma (see \cite[\href{https://stacks.math.columbia.edu/tag/0G1U}{Tag 0G1U}]{stacks-project}), one reduces the question to showing that the natural morphism
    \begin{equation}\label{eqn:duality-map-group-coh}
        \bf{R}\Gamma_{\cont}(\Gamma, R_\infty^+/p) \to \bf{R}\rm{Hom}_{R^+/p}\left(\bf{R}\Gamma_{\cont}(\Gamma, R_\infty^+/p), R^+/p[-d]\right)
    \end{equation}
    is an almost isomorphism. For this, we note that 
    \[
    R^+_\infty/p \simeq \bigoplus_{\ov{\chi} \in \frac{\bf{X}(T)[1/p]}{\bf{X}(T)}} R^+_{\ov{\chi}}/p
    \]
    due to Decomposition~(\ref{eqn:decomposition-R-infty}). Furthermore, since continuous group cohomology commute with direct sums of discrete modules, the morphism in (\ref{eqn:duality-map-group-coh}) is equal to 
    \[
    \bigoplus_{\ov{\chi}} \bf{R}\Gamma_\cont(\Gamma, R^+_{\ov{\chi}}/p) \to \prod_{\ov{\chi}} \bf{R}\rm{Hom}_{R^+/p}\left(\bf{R}\Gamma_{\cont}(\Gamma, R_{-\ov{\chi}}^+/p), R^+/p[-d]\right).
    \]
    Since $R^+_{\ov{\chi}}/p$ pairs non-trivially only with elements of $R^+_{-\ov{\chi}}/p$, it suffices to show the following two things:
    \begin{enumerate}
        \item For each $\ov{\chi}\in \bf{X}(T)[1/p]/\bf{X}(T)$, the natural morphism 
        \[
        \bf{R}\Gamma_\cont(\Gamma, R^+_{\ov{\chi}}/p) \to \bf{R}\rm{Hom}_{R^+/p}\left(\bf{R}\Gamma_{\cont}(\Gamma, R_{-\ov{\chi}}^+/p), R^+/p[-d]\right)
        \]
        is an isomorphism;
        \item for each integer $n\geq 1$, cohomology groups of $\bf{R}\Gamma_\cont(\Gamma, R^+_{\ov{\chi}}/p)$ are annihilated by $p^{1/n}$ for all but finitely many $\ov{\chi} \in \bf{X}(T)[1/p]/\bf{X}(T)$.
    \end{enumerate}
    The first claim follows immediately from Lemma~\ref{lemma:duality-chi}, while the second claim follows from Corollary~\ref{cor:computation-chi-mod-p}(\ref{cor:computation-chi-mod-p-2}).
\end{proof}

\subsubsection{Almost duality on polystable formal schemes}

The main goal of this section is to show that Faltings' trace induces an almost perfect pairing on any separated rig-smooth polystable formal $\O_C$-scheme $\X$.

\begin{thm}\label{thm:almost-perfect-polystable} Let $\X$ be a separated polystable formal $\O_C$-scheme with generic fiber $X=\X_C$ of pure dimension $d$. Then the pairing
\[
\bf{R}\nu_* \wdh{\O}_X^{+, a} \otimes^L_{\O_{\X}} \bf{R}\nu_* \wdh{\O}_X^{+, a} \xr{\cup} \bf{R}\nu_* \wdh{\O}_X^{+, a} \xr{\rm{Tr}_{F, \X}} \omega_{\X}^{a}(-d)[-d]
\]
is almost perfect.
\end{thm}
\begin{proof}
{\it Step 1: We reduce to a model polystable formal $\O_C$-scheme $\X=\Spf R^+$.}
Since the question is symmetric in both variables, it is sufficient to show that the induced map
\[
\bf{R}\nu_* \wdh{\O}_X^{+, a} \to \bf{R}\ud{al\cal{H}om}_{\X}\left(\bf{R}\nu_* \wdh{\O}_X^{+, a}, \omega_{\X}^{a}(-d)[-d]\right)
\]
is an isomorphism (in $\bf{D}(\X)^a$). The question is local on $\X$, so we can assume that there is a zig-zag of affine formal $\O_C$-schemes 
\[
\begin{tikzcd}
& \mathfrak{U} \arrow{dr}{\mathfrak g} \arrow[dl, swap, "\mf"] & \\
\X & & \Spf R^+
\end{tikzcd}
\]
with \'etale $\mf$ and $\mathfrak{g}$, and a model polystable formal $\O_C$-scheme $\Spf R^+$. Furthermore, we can also assume that $\mathfrak{f}$ is surjective (and so it is faithfully flat). Now \cite[Theorem 6.13.6(1, 3), Lemma 4.9.7, Lemma 2.9.12]{Z3}, and flatness of the morphism $\O_\X(\X) \to \O_{\sU}(\sU)$ imply that the natural morphisms
\[
\bf{L}\mf^*\, \bf{R}\nu_{\X, *} \wdh{\O}_X^{+, a} \to \bf{R}\nu_{\mathfrak{U}, *}\wdh{\O}_{U}^{+, a}, 
\]
\[
\bf{L}\mf^*\,\bf{R}\ud{al\cal{H}om}_{\X}\left(\bf{R}\nu_{\X, *} \wdh{\O}_X^{+, a}, \omega_{\X}^{a}(-d)[-d]\right) \to \bf{R}\ud{al\cal{H}om}_{\mathfrak{U}}\left(\bf{L}\mf^*\,\bf{R}\nu_{\X, *} \wdh{\O}_X^{+, a}, \bf{L}\mf^*\omega_{\X}^{a}(-d)[-d]\right)
\]
are isomorphisms. Moreover, Lemma~\ref{lemma:faltings-trace-commutes-with-base-change} implies that we have a commutative diagram
\[
\begin{tikzcd}
\bf{L}\mf^*\, \bf{R}\nu_{\X, *} \wdh{\O}_X^{+, a} \arrow{dd} \arrow{r} &\bf{L}\mf^*\,\bf{R}\ud{al\cal{H}om}_{\X}\left(\bf{R}\nu_{\X, *} \wdh{\O}_X^{+, a}, \omega_{\X}^{a}(-d)[-d]\right) \arrow{d}\\
 & \bf{R}\ud{al\cal{H}om}_{\mathfrak{U}}\left(\bf{L}\mf^*\,\bf{R}\nu_{\X, *} \wdh{\O}_X^{+, a}, \bf{L}\mf^*\,\omega_{\X}^{a}(-d)[-d]\right) \arrow{d} \\
\bf{R}\nu_{\mathfrak{U}, *}\wdh{\O}_{U}^{+, a}  \arrow{r} & \bf{R}\ud{al\cal{H}om}_{\mathfrak{U}}\left(\bf{R}\nu_{\mathfrak{U}, *}\wdh{\O}_{U}^{+, a}, \omega_{\mathfrak{U}}^{a}(-d)[-d]\right)
\end{tikzcd}
\]
with vertical maps being isomorphisms (see Lemma~\ref{lemma:dualizing-etale-base-change} for the bottom right map). Thus, the (almost) faithfully flat base change implies that it is sufficient to prove the pairing is perfect for $\X=\mathfrak U$. Then the same argument for $\mathfrak{g}$ reduces the situation to the case $\X=\Spf R^+$. \smallskip

{\it Step 2: We reduce to the almost duality in group cohomology.} Now we note that Theorem~\ref{thm:structure-nearby-framed} implies that the trace map $\bf{R}\nu_* \wdh{\O}_X^{+, a} \to \omega^{a}_{\X}(-d)[-d]$ can be identified with the map
\[
\bf{R}\nu_* \wdh{\O}_X^{+, a} \to \rm{R}^d\nu_* \wdh{\O}_X^{+, a}[-d] \to \left(\frac{\rm{R}^d\nu_* \wdh{\O}_X^{+}}{\rm{R}^d\nu_* \wdh{\O}_X^{+}[1-\zeta_p]}\right)^a[-d] \simeq \omega^{a}_{\X_0}(-d)[-d].
\]
Therefore, we can reformulate the almost perfectness of the pairing induced by Faltings' trace intrinsically in terms of $\bf{R}\nu_* \wdh{\O}_X^{+, a}$. More precisely, it is sufficient to show that the map
\[
\bf{R}\nu_* \wdh{\O}_X^{+, a} \otimes^L_{\O_{\X}} \bf{R}\nu_* \wdh{\O}_X^{+, a} \xr{\cup} \bf{R}\nu_* \wdh{\O}_X^{+, a} \to  \left(\frac{\rm{R}^d\nu_* \wdh{\O}_X^{+}}{\rm{R}^d\nu_* \wdh{\O}_X^{+}[1-\zeta_p]}\right)^a[-d]
\]
is an almost perfect pairing. Now Lemma~\ref{lemma:almost-compute-nu}, Corollary~\ref{cor:group-coh-compute-model-case}, and \cite[Lemma 4.8.13]{Z3} ensure that we have a commutative diagram
\[
\begin{tikzcd}
\bf{R}\Gamma_{\cont}(\Gamma, R^+_\infty)^{a, L\Updelta} \otimes^L_{\O_{\X}} \bf{R}\Gamma_{\cont}(\Gamma, R^+_\infty)^{a, L\Updelta} \arrow{r}{\cup}  \arrow{d}& \bf{R}\Gamma_{\cont}(\Gamma, R^+_\infty)^{a, L\Updelta}[-d] \arrow{r} \arrow{d} & \rm{H}^d_{\cont}(\Gamma, R^+)^{a, \Updelta}[-d] \arrow{d} \\
\bf{R}\nu_* \wdh{\O}_X^{+, a} \otimes^L_{\O_{\X_0}} \bf{R}\nu_* \wdh{\O}_X^{+, a} \arrow{r}{\cup} & \bf{R}\nu_* \wdh{\O}_X^{+, a} \arrow{r} & \left(\frac{\rm{R}^d\nu_* \wdh{\O}_X^{+}}{\rm{R}^d\nu_* \wdh{\O}_X^{+}[1-\zeta_p]}\right)^a[-d] \\
\end{tikzcd}
\]
with vertical maps being isomorphisms. Now we use \cite[Lemma 4.9.7]{Z3} to reduce the question to the almost perfectness of the pairing
\[
\bf{R}\Gamma_{\cont}(\Gamma, R^+_\infty) \otimes^L_{R^+} \bf{R}\Gamma_{\cont}(\Gamma, R^+_\infty) \xr{\cup}\bf{R}\Gamma_{\cont}(\Gamma, R^+_\infty) \to \rm{H}^d_{\cont}(\Gamma, R^+)[-d]\simeq R^+[-d] \ .
\]
This follows directly from Theorem~\ref{thm:almost-duality-group-cohomology} (and Remark~\ref{rmk:trace-quotient}).
\end{proof} 

\begin{cor}\label{cor:faltings-polystable-mod-p} In the notation of Theorem~\ref{thm:almost-perfect-polystable}, the pairing
\[
\bf{R}\nu_* \left(\O_X^{+, a}/p\right) \otimes^L_{\O_{\X_0}} \bf{R}\nu_*\left(\O_X^{+, a}/p\right) \xr{\cup} \bf{R}\nu_* \left(\O_X^{+, a}/p\right) \xr{\rm{Tr}_{F, \X}} \omega_{\X_0}^{\bullet, a}(-d)[-2d]
\]
is almost perfect.
\end{cor}
\begin{proof}
    This follows directly from Theorem~\ref{thm:almost-perfect-polystable} by noticing that  Lemma~\ref{lemma:line-bundle-dualizing} and Theorem~\ref{thm:dualizing-complex} imply that $\omega^\bullet_{\X}\simeq \omega_{\X}[d]$, and $\omega^\bullet_{\X_0}\simeq\omega^\bullet_\X \otimes^L_{\O_\X} \O_{\X_0} \simeq \omega_\X \otimes^L_{\O_\X} \O_{\X_0}[d]$.
\end{proof}

\subsection{Local duality on a rig-smooth formal scheme} 
 
The main goal of this section is to prove Theorem~\ref{thm:main-thm-?} for any admissible separated rig-smooth formal $\O_C$-scheme $\X$ with reduced special fiber. \smallskip 

The main idea of our proof is to reduce to the polystable case using the version of the local uniformization result from \cite{Z1}. 

\begin{thm}\label{thm:local-uniformization}(\cite[Theorem 1.4]{Z1}) Let $\X$ be an admissible\footnote{In this paper, we assume that any admissible formal scheme is qcqs by definition} formal $\O_C$-scheme with smooth generic fiber $\X_C$. Then there is a finite set $(\X_i, \mf_i)_{i\in I}$ of admissible formal $\O_{C}$-schemes with morphisms $\mf_i\colon \X_i \to \X$ such that
\begin{itemize}\itemsep0.5em
\item The set $(\X_i, \mf_i)$ can be obtained from $\X$ as a composition of open Zariski coverings and rig-isomorphisms (let us denote by $\{\Y_j\}_{j\in J}$ the collection of all formal $\O_C$-schemes appearing in the presentation of $\X_i$ as a composition of open Zariski covers and rig-isomorphisms);
\item Each $\X_i$ is a geometric quotient of an admissible formal $\O_{C}$-scheme $\X'_i$ by an action of a finite group $G_i$ such that the quotient map $\mathfrak{g}_{i,C}\colon \X'_{i,C} \to \X_{i,C}$ is a $G_i$-torsor;
\item Each $\X'_{i}$ admits a rig-isomorphism $\pi_i\colon \X''_i \to \X'_{i}$ with a rig-smooth, polystable formal $\O_{C}$-scheme $\X''_i$;
\item Furthermore, one can assume that all $\X_i$, $\X'_i$, $\X''_i$, and $\Y_j$ have reduced, quasi-projective special fiber.  
\end{itemize}
\end{thm}
\begin{proof}
    \cite[Theorem 1.3]{Z1} constructs such $\X_i$, $\X'_i$, $\X''_i$ (and $\Y_j$) satisfying all these assumptions except for having reduced special fiber. First, we note that each $\X''_i$ has reduced special fiber since it is polystable. Therefore, one can construct desired formal schemes by replacing each $\X_i$, $\X'_i$, $\X''_i$ (and $\Y_j$) with their normalizations $\widetilde{\X}_i$, $\widetilde{\X}'_i$, $\widetilde{\X}''_i=\X''_i$ (and $\widetilde{\Y}_j$) in their generic fibers (see Definition~\ref{defn:normalization}). To make sure this does the job, we only need to show that each $\widetilde{\X}_i$, $\widetilde{\X}'_i$, $\widetilde{\X}''_i=\X''_i$ (and $\widetilde{\Y}_j$) has quasi-projective special fiber, and that $\widetilde{\X}_i=\widetilde{\X}'_i/G$ (this quotient exists by \cite[Remark 3.3.3 and Theorem 3.3.4]{Z2}). \smallskip
    
    The former property is easy since each normalization in the generic fiber is finite, and so preserves the property of having quasi-projective special fiber. Thus, we only need to show that $\widetilde{\X}_i=\widetilde{\X}'_i/G$. This can be checked locally on $\X_i$, so we can assume that $\X'_i=\Spf R_i$, $\X'_i = \Spf R_i^G$, $\widetilde{\X}_i=\Spf R_i^{\rm{ic}}$, and $\widetilde{\X}'_i=\Spf (R_i^G)^{\rm{ic}}$, where $\rm{ic}$ stands for the integral closure in the corresponding generic fiber. Then the question boils down to showing that 
    \[
    \left(R_i^{\rm{ic}}\right)^G = \left(R_i^G\right)^{\rm{ic}}. 
    \]
    This follows from the facts that $R_i^G \to \left(R_i^{\rm{ic}}\right)^G$ is integral, and $\left(R_i^{\rm{ic}}\right)^G$ is integrally closed in $\left(R_{i}[1/p]\right)^G=R_i^G[1/p]$. 
\end{proof}

\begin{defn}\label{defn:good-quotient} Let $\X'$ be an admissible formal $\O_C$-scheme with a quasi-projective special fiber, and let $G$ be a finite group with a right $\O_C$-action on $\X'$. We say that the geometric quotient\footnote{This quotient always exists due to \cite[Remark 3.3.3 and Theorem 3.3.4]{Z2}.} morphism $\mf \colon \X' \to \X=\X'/G$ is {\it a good quotient} if the generic fiber $\X'_C \to \X_C$ is a $G$-torsor.
\end{defn}

Theorem~\ref{thm:local-uniformization} and Corollary~\ref{cor:faltings-polystable-mod-p}  imply that, in order to show that Faltings' pairing is almost perfect, it suffices to show that almost perfectness of Faltings' pairing descends through rig-isomorphisms and good quotients. 

\subsubsection{Descent through rig-isomorphisms}

In this section, we show that almost perfectness of Faltings' trace map descends through rig-isomorphisms: 

\begin{lemma}\label{lemma:descent-blow-up} Let $\pi\colon \X' \to \X$ be a rig-isomorphism of nice admissible formal $\O_C$-schemes of dimension $d$. Suppose that Faltings' pairing is  almost perfect on $\X'$. Then the same holds on $\X$.  
\end{lemma}
\begin{proof}
We note that Corollary~\ref{cor:descend-rig-finite-etale} implies that the following diagram commutes
\[
\begin{tikzcd}[column sep = 4em]
\bf{R}\nu_{\X, *} \left(\O_X^{+, a}/p\right) \arrow{r}{D_{\X}} \arrow{dd}{\wr}& \bf{R}\ud{al\cal{H}om}_{\X_0}(\bf{R}\nu_{\X, *}\left(\O_{X}^{+, a}/p\right), \omega^{\bullet, a}_{\X_0}(-d)[-2d]) \arrow{d}{\wr} \\
 & \bf{R}\ud{al\cal{H}om}_{\X_0}(\bf{R}\pi_{0, *}\bf{R}\nu_{\X', *}\left(\O_{X'}^{+, a}/p\right), \omega^{\bullet, a}_{\X_0}(-d)[-2d])\\
\bf{R}\pi_{0, *}\bf{R}\nu_{\X', *}\left(\O_{X'}^{+, a}/p\right) \arrow{r}{\bf{R}\pi_{0, *}(D_{\X'})} & \bf{R}\pi_{0, *}\bf{R}\ud{al\cal{H}om}_{\X'_0}(\bf{R}\nu_{\X', *}\left(\O_{X'}^{+, a}/p\right), \omega^{\bullet, a}_{\X'_0}(-d)[-2d]) \arrow{u}{\wr}, 
\end{tikzcd}
\]
where the right bottom vertical map is the isomorphism from \cite[Lemma 5.5.6]{Z3}. Other vertical maps are also clearly isomorphisms, and the bottom horizontal morphism is an isomorphism by assumption. Therefore, the top horizontal morphism is an isomorphism as well. 
\end{proof}

\subsubsection{Digression: Sheaves with a finite group action}
 
In this section, we recall the definition of a sheaf with an action of a finite group $G$, and extend this to the almost setup. This notion will play a crucial role in our proof of the fact that almost perfectness of Faltings' trace map descends through ``good'' quotients. \smallskip
 
\begin{defn} Let $(X, \O_X)$ be a ringed site with a right action of a finite group $G$. A {\it $G$-equivariant $\O_X$-module} is an $\O_X$-module $\F$ together with a collection of isomorphisms
\[
\phi_g \colon \F \to g_*\F \ (g\in G)
\]
such that $\phi_e=\rm{id}$ and, for any $g, h\in G$, the following diagram commutes:
\[
\begin{tikzcd}
\F \arrow{r}{\phi_g}\arrow{d}{\phi_{hg}} & g_*\F \arrow{d}{g_*(\phi_h)} \\
(hg)_* \F \arrow{r}{\sim} & g_*\left(h_*\F\right). 
\end{tikzcd}
\]
Morphisms of $G$-equivariant $\O_X$-modules are defined to be morphisms of underlying $\O_X$-modules that commute with the $G$-action. We denote this category by $\bf{Mod}_X^G$ or $\bf{Mod}_{\O_X}^G$.
\end{defn}

\begin{rmk} For any ringed site $(X, \O_X)$ with a right action of $G$, the structure sheaf $\O_X$ is naturally an object of $\bf{Mod}_X^G$.
\end{rmk}

\begin{rmk} If $G$ acts trivially on $X$, we see that the category of $G$-equivariant $\O_X$-modules is equivalent to the category of $\O_X[G]$-modules.
\end{rmk}

\begin{rmk}\label{rmk:grothendieck-abelian} Arguing similarly to \cite[Proposition 5.1.1]{Tohoku}, we see that the category of $\bf{Mod}_X^G$ is a Grothendieck abelian category. We denote the derived category of $\bf{Mod}_X^G$ by $\bf{D}_G(X)$. If the action is trivial, we simply denote it by $\bf{D}(\O_X[G])$.
\end{rmk}

\begin{rmk}\label{rmk:equivariant-pushforward} Remark~\ref{rmk:grothendieck-abelian} and \cite[\href{https://stacks.math.columbia.edu/tag/079P}{Tag 079P}]{stacks-project} imply that the category of complexes $\rm{Comp}(\bf{Mod}_X^G)$ has enough $K$-injective complexes. Furthermore, arguing similarly to \cite[\textsection 9.1, p.\,498]{Lei-Fu}, we see that the forgetful functor $\rm{Comp}(\bf{Mod}_X^G) \to \rm{Comp}(\bf{Mod}_X)$ preserves $K$-injective complexes. Thus, for any $G$-equivariant morphism $f\colon (X, \O_X)\to (Y, \O_Y)$, we can define a functor $\bf{R}f_*\colon \bf{D}_G(X) \to \bf{D}_G(Y)$ such that the diagram
\[
\begin{tikzcd}
\bf{D}_G(X) \arrow{r}{\rm{forget}} \arrow{d}{\bf{R}f_*} &\bf{D}(X) \arrow{d}{\bf{R}f_*} \\
\bf{D}_G(Y)\arrow{r}{\rm{forget}} & \bf{D}(Y)
\end{tikzcd}
\]
commutes up to a specified equivalence.
\end{rmk}

\begin{rmk}\label{rmk:derived-action}  In particular, for any $G$-invariant morphism $f\colon (X, \O_X) \to (Y, \O_Y)$, $\bf{R}f_*\O_X$ naturally lifts to an object of $\bf{D}(\O_Y[G])$.
\end{rmk}

%\begin{lemma} Let $(X, \O_X) \to (Y, \O_Y)$ be a morphism of ringed sites with a trivial $G$-action. Consider $\O_Y$ and $\O_X$ as elements of $\bf{Mod}_{\O_Y[G]}$ and $\bf{Mod}_{\O_X[G]}$ respectively. Then there is a natural isomorphism
%\[
%\bf{L}f^* \O_Y \simeq \O_X.
%\]
%\end{lemma}
%We note that this lemma is not entirely trivial since $\O_X$ is usually not a flat $\O_X[G]$-module. 
%\begin{proof}
 %   We note that the $\O_Y[G]$-module $\O_Y$ admits a canonical bar-resolution
 %   \[
 %   F^\bullet_Y \to \O_Y
 %   \]
 %   such that each $F^i_Y \simeq \bigoplus_{G^{-i}} \O_Y[G]$ with the standard differentials. Similarly, $\O_X$ admits a bar resolution $F^\bullet_X$. Now the result follows from the observation that $F^\bullet_Y$ is a $\O_Y[G]$-flat (bounded above) resolution of $\O_Y$, and $f^*F^\bullet_Y \simeq F^\bullet_X$.
%\end{proof}

\begin{defn} Let $(X, \O_X)$ be a ringed space with a trivial action of a finite group $G$, and let $\F\in \bf{D}(\O_X[G])$. Then the complex of {\it homotopy invariants} (resp. {\it homotopy co-invariants}) is $\F^{hG} \coloneqq \bf{R}\cal{H}om_{\O_X[G]}(\O_X, \F)\in \bf{D}(\O_X)$ (resp. $\F_{hG}\coloneqq \F\otimes^L_{\O_X[G]} \O_X\in \bf{D}(\O_X)$).
\end{defn}

\begin{lemma}\label{lemma:commutes-with-homotopy-limits-and-colimits} Let $\X$ be an admissible formal $\O_C$-scheme with generic fiber $X$, $t\colon (X_\et, \O^+_{X}/p) \to (\X_0, \O_{\X_0})$ the natural morphism of ringed sites, and $G$ a finite group acting trivially on $\X$. Then $\bf{R}t_*\colon \bf{D}(\O^+_{X}/p[G]) \to \bf{D}(\O_{\X_0}[G])$ commutes with homotopy invariants and homotopy co-invariants.
\end{lemma}  
\begin{proof}
    We give a proof of the claim for homotopy co-invariants. The case of homotopy invariants is similar and even easier\footnote{It is easier because $\bf{R}t_*$ automatically commutes with (sequential) homotopy limits.}. \smallskip
    
    First, we note that $\bf{R}t_*$ commutes with (sequential) homotopy colimits (see \cite[\href{https://stacks.math.columbia.edu/tag/0A5K}{Tag 0A5K}]{stacks-project}). Indeed, for each integer $i\geq 0$, the functor $\rm{R}^it_*$ commutes with filtered colimits since $t$ is a coherent morphism of coherent topoi. Therefore, $\bf{R}t_*$ commutes with (sequential) homotopy colimits if all terms lie in $\bf{D}^{\geq -N}(\O^+_{X}/p[G])$ for some integer $N$. Since $\bf{R}t_*$ is of finite cohomological dimension (due to \cite[Corollary 2.8.3]{H3}), a standard argument extends the result to the unbounded derived category as well. \smallskip
    
    Now we note that the $\left(\O_{X}^+/p\right)[G]$-module $\O_X^+/p$ admits a canonical bar-resolution
    \[
    B^\bullet_X \to \O_X^+/p
    \]
    such that each 
    \[
        B^i_X \simeq \bigoplus_{\#G^{-i}} \left(\O_{X}^+/p\right)[G]
    \]
    with the standard differentials, and a similar bar resolution $B^\bullet_{\X_0} \to \O_{\X_0}$ exists on $\X_0$. In particular, we can write $\O_X^+/p \simeq \hocolim_{n} B^{\geq -n}_X$ and $\O_{\X_0}\simeq \hocolim_n B^{\geq -n}_{\X_0}$. \smallskip
    
    Now we fix some $\F\in \bf{D}(\O_X^+/p[G])$ and wish to show that the natural morphism
    \[
    \bf{R}t_*\left(\F \right) \otimes^L_{\O_{\X_0}[G]} \O_{\X_0} \to \bf{R}t_*\left(\F \otimes^L_{\O^+_X/p[G]} \O_X^+/p\right)
    \]
    is an isomorphism. Using the bar-resolutions, and the fact that homotopy colimits commute with $\bf{R}t_*$ and with derived tensor products, we conclude that it suffices to show that
    \[
    \left(\bf{R}t_*\F\right) \otimes^L_{\O_{\X_0}[G]} B^{\geq -n}_{\X_0} \to \bf{R}t_*\left(\F\otimes^L_{(\O_{X}^+/p)[G]} B^{\geq -n}_X\right)
    \]
    is an isomorphism for any integer $n$. A standard inductive argument reduces the question to showing that the natural morphism 
    \[
    \left(\bf{R}t_*\F\right) \otimes^L_{\O_{\X_0}[G]} B^{-n}_{\X_0} \to \bf{R}t_*\left(\F\otimes^L_{(\O_{X}^+/p)[G]} B^{-n}_X\right)
    \]
    is an isomorphism for any integer $n$. This is clear since $\bf{R}t_*$ commutes with (finite) direct sums, and $B^{-n}_X$, $B^{-n}_{\X_0}$ are finite free of the same rank. 
\end{proof}

Now we extend the definition of $G$-equivariant sheaves to the almost world (at least in the case of the trivial action). For this, we note that, for any ringed $\O_C$-site $(X, \O_X)$, the pair $(X, \O_X[G])$ also forms a ringed $\O_C$-site. \smallskip

\begin{defn} Let $(X, \O_X)$ be a ringed $\O_C$-site, and let $G$ be a finite group that trivially acts on $(X, \O_X)$. The {\it category of $G$-equivariant almost $\O_X$-modules} $\bf{Mod}_{\O_X[G]}^a$ is the category of almost $\O_X[G]$-modules (see \cite[Definition 3.1.9]{Z3}). 

The {\it derived category of $G$-equivariant almost $\O_X$-modules} $\bf{D}(\O_X[G])^a$ is the derived category of $\bf{Mod}_{\O_X[G]}^a$.
\end{defn}

\begin{rmk} The results of \cite[\textsection 3.1, 3.4, 3.5]{Z3} implicitly assume that the ringed $\O_C$-site $(X, \O_X)$ has a {\it commutative} sheaf of rings $\O_X$. However, almost all results hold under the weaker assumption that $\O_C$ is {\it central} in $\O_X$. In particular, we can apply the results of these sections to the ringed $\O_C$-site $(X, \O_X[G])$. 
\end{rmk}

\begin{defn} Let $(X, \O_X)$ be a ringed space with a trivial action of a finite group $G$, and let $\F^a\in \bf{D}(\O_X[G])^a$. Then the complex of {\it homotopy invariants} (resp. {\it homotopy co-invariants}) is $\F^{a, hG} \coloneqq \bf{R}\ud{al\cal{H}om}_{\O_X[G]^a}(\O_X^a, \F^a)\in \bf{D}(\O_X)^a$ (resp. $\F^a_{hG}\coloneqq \F^a\otimes^L_{\O_X[G]^a} \O_X^a\in \bf{D}(\O_X)^a$).
\end{defn}

\begin{lemma}\label{lemma:duality-invariants-coinvarints} Let $(X, \O_X)$ be a ringed $\O_C$-site, and $G$ a finite group. Let $\F\in \bf{D}(\O_X[G])$  and $\G\in \bf{D}(\O_X)$. Then the natural map $\bf{R}\ud{al\cal{H}om}_{X}(\F^a, \G^a)^{hG} \to \bf{R}\ud{al\cal{H}om}_{X}(\F^a_{hG}, \G^a)$ is an isomorphism in $\bf{D}(X)^a$.
\end{lemma}
\begin{proof}
Unravelling the definitions, we see that one has to show that the natural map
\[
\bf{R}\ud{al\cal{H}om}_{\O_X[G]^a}(\O_X^a, \bf{R}\ud{al\cal{H}om}_{\O_X^a}(\F^a, \G^a)) \to \bf{R}\ud{al\cal{H}om}_{\O_X^a}(\F^a\otimes^L_{\O_X[G]^a} \O_X^a, \G^a)
\]
is an isomorphism. Using \cite[Proposition 3.5.8 and Proposition 3.5.14 (or Proposition 3.5.20)]{Z3}, we see that it suffices to show that the natural morphism
\[
\bf{R}\ud{\cal{H}om}_{\O_X[G]}(\O_X, \bf{R}\ud{\cal{H}om}_{\O_X}(\F, \G)) \to \bf{R}\ud{\cal{H}om}_{\O_X}(\F\otimes^L_{\O_X[G]} \O_X, \G)
\]
is an isomorphism in $\bf{D}(X)$. This is the usual version of the (derived) tensor-Hom adjunction.
\end{proof}

\subsubsection{Descent through good quotients}

The main goal of this section is to prove that almost perfectness of Faltings' trace descends through good quotients (see Definition~\ref{defn:good-quotient}). The main idea is to adapt the proof of Lemma~\ref{lemma:descent-blow-up} to the $G$-equivariant setting. The main new technical difficulty is to promote the commutative diagram from Corollary~\ref{cor:descend-rig-finite-etale} to a commutative diagram in $\bf{D}(\O_{\X_0}[G])^a$. We do not quite do this in this section since it seems difficult to rigorously define a $G$-equivariant structure on $\bf{R}\mf_{0, *}\, \omega_{\X_0}^{\bullet, a}$, but it turns out that a slightly weaker form of the $G$-equivariant version of Corollary~\ref{cor:descend-rig-finite-etale} is sufficient for our purposes. \smallskip

For the next constructions, let $\mf \colon \X' \to \X$ be a good quotient (see Definition~\ref{defn:good-quotient}) with generic fiber $f\colon X'\to X$. 

\begin{construction} Remark~\ref{rmk:equivariant-pushforward} (or Remark~\ref{rmk:derived-action}) implies that $\bf{R}\mf_{0, *}\,\bf{R}\nu_{\X', *}\left(\O_{X'}^+/p\right)$ can be naturally promoted to an element of $\bf{D}(\O_{\X_0}[G])$. By applying the almostification functor, we see that $\bf{R}\mf_{0, *}\,\bf{R}\nu_{\X', *}\left(\O_{X'}^{+, a}/p\right)$ is naturally an object of $\bf{D}(\O_{\X_0}[G])^a$.
\end{construction}

\begin{construction}\label{rmk:pro-etale-trace-equivariant} The pro-\'etale trace (see Definition~\ref{defn:mod-p-proetale-trace})
\[
\rm{Tr}_{\rm{Zar}, \mf} \colon \bf{R}\mf_{0, *}\bf{R}\nu_{\X', *}\left(\O_{X'}^+/p\right) \to \bf{R}\nu_{\X, *}\left(\O_{X}^+/p\right)
\]
can be canonically promoted to a morphism in $\bf{D}(\O_{\X_0}[G])$ if we put the trivial $G$-structure on $\bf{R}\nu_{\X, *}\left(\O_{X}^+/p\right)$. Indeed, Remark~\ref{rmk:etale-pushforward-trace} ensures $\rm{Tr}_{\rm{Zar}, \mf}$ is identified with 
\[
\bf{R}t_{\X, *}\left(\rm{Tr}_{\et, f}\right) \colon \bf{R}t_{\X, *} \bf{R}f_{\et, *} \left(\O_{X'}^+/p^n\right)
 \to \bf{R}\nu_{\X, *}\left(\O_{X}^+/p\right).
\]
Using Remark~\ref{rmk:equivariant-pushforward}, we reduce the question to showing that the \'etale trace map
\[
\rm{Tr}_{\et, f, \O_X^+/p} \colon f_{\et, *} \O_{X'}^+/p \to \O_{X}^+/p
\]
is $G$-equivariant. This can be checked locally, so we can assume that $X' \to X$ is a split torsor. In this case, it is obvious. 
\end{construction}

\begin{construction} A similar argument shows that the restriction morphism
\[
\rm{Res}_{\mf} \colon \bf{R}\nu_{\X, *} \left(\O^{+}_{X}/p\right) \to \bf{R}\mf_{0, *} \left(\bf{R}\nu_{\X', *} \O^{+}_{X'}/p \right) 
\]
can be canonically promoted to a morphism in $\bf{D}(\X_0[G])$. 
\end{construction}

\begin{lemma}\label{lemma:invariant-coinvariant-torsor} Let $\mf \colon \X' \to \X$ be a good quotient (see Definition~\ref{defn:good-quotient}) with generic fiber $f\colon X'\to X$. Then the pro-\'etale trace map 
 \[
 \rm{Tr}_{\rm{Zar}, \mf}\colon \bf{R}\mf_{0, *}\,\bf{R}\nu_{\X', *}\left(\O_{X'}^+/p\right) \to \bf{R}\nu_{\X, *}\left(\O_X^+/p\right)
 \]
 induces an isomorphism 
 \[
    \bigg(\bf{R}\mf_{0, *}\,\bf{R}\nu_{\X', *}\left(\O_{X'}^+/p\right)\bigg)_{hG} \to \bf{R}\nu_{\X, *}\left(\O_X^+/p\right).
 \] 
 Similarly, the restriction map $\rm{Res}_{\mf}\colon \bf{R}\nu_{\X, *}\left(\O_X^+/p\right) \to \bf{R}\mf_*\bf{R}\nu_{\X', *}\left(\O_{X'}^+/p\right)$ induces an isomorphism
\[
\bf{R}\nu_{\X, *}\left(\O_X^+/p\right) \to \bigg(\bf{R}\mf_{0, *}\bf{R}\nu_{\X', *}\left(\O_{X'}^+/p\right)\bigg)^{hG}.
\]
\end{lemma}
\begin{proof}
We note that \cite[Corollary 3.17(i)]{Sch1} implies that 
\[
    \bf{R}\mf_{0, *}\bf{R}\nu_{\X', *}\left(\O_{X'}^+/p\right)\simeq \bf{R}t_{\X, *}f_{\et, *}\left(\O_{X'}^+/p\right),
\]
\[
    \bf{R}\nu_{\X, *}\left(\O_X^+/p\right) \simeq \bf{R}t_{\X, *}\left(\O_X^+/p\right)
\]
where $t\colon (X_\et, \O_X^+/p) \to (\X_{0, \rm{Zar}}, \O_{\X_0})$ is the natural morphism of ringed sites. Now we note that 
\[
    \rm{Tr}_{\rm{Zar}, \mf}\simeq \bf{R}t_{\X, *}(\rm{Tr}_{\et, f}),
\]
\[
    \rm{Res}_{\mf}\simeq \bf{R}t_{\X, *}(\rm{Res}_{f}).
\]
Lemma~\ref{lemma:commutes-with-homotopy-limits-and-colimits} implies that $\bf{R}t_{\X, *}$ commutes with both homotopy invariants and co-invariants, so it suffices to show that the maps
\[
\left(f_{\et, *}\,\O_{X'}^+/p\right)_{hG} \to \O_X^+/p,
\]
\[
\O_X^+/p \to \left(f_{\et, *}\,\O_{X'}^+/p\right)^{hG}
\]
are isomorphisms. \smallskip

The claim is \'etale local on $X$, so we can assume that $f$ is a split $G$-torsor. In this case, $f_{\et, *}\,\O_{X'}^+/p \simeq \O_X^+/p[G]$, and so it suffices to show that the natural morphisms
\[
\left(\O_X^+/p[G]\right)_{hG} \to \O_X^+/p,
\]
\[
\O_X^+/p \to \left(\O_X^+/p[G]\right)^{hG}
\]
are isomorphisms. The first morphism is an isomorphism for tautological reasons. The second map can be easily seen to be an isomorphism on $\cal{H}^0$, so it suffices to show that $\left(\O_X^+/p[G]\right)^{hG}$ has no higher cohomology sheaves. This follows from the observation that induced modules have no higher group cohomology. 
\end{proof}

\begin{rmk}\label{rmk:almost-makes-sense} Lemma~\ref{lemma:invariant-coinvariant-torsor} stays true in the almost category since derived (inner) Hom, derived tensor product, and derived pushforward in the almost category can be computed as the ``almostification'' of the usual analogues of derived (inner) Hom, derived tensor product, and derived pushforward respectively (see \cite[Propositions 3.5.8, 3.5.14, and 3.5.23]{Z3}).
\end{rmk}

Now we prove a $G$-equivariant version of Corollary~\ref{cor:descend-rig-finite-etale}. It will be more convenient to split it into $2$ independent parts. 

\begin{lemma}\label{lemma:descend-rig-finite-etale-equivariant} Let $\mf\colon \X' \to \X$ be a good quotient of admissible formal $\O_C$-schemes, and let $f\colon X' \to X$ be the induced morphism on generic fibers. Then the diagram 
\[
\begin{tikzcd}[column sep = 4em]
\bf{R}\mf_{0, *} \left(\bf{R}\nu_{\X', *} \O^{+, a}_{X'}/p \right)\otimes^{L} \bf{R}\mf_{0, *} \left(\bf{R}\nu_{\X', *} \O^{+, a}_{X'}/p\right) \arrow{r}{\cup} \arrow[d, shift left=7ex, "\rm{Tr}_{\rm{Zar}, \mf}"] & \bf{R}\mf_{0, *} \left(\bf{R}\nu_{\X', *} \O_{X'}^{+, a}/p\right) \arrow{d}{\rm{Tr}_{\rm{Zar}, \mf}}  \\
\bf{R}\nu_{\X, *} \left(\O^{+, a}_{X}/p\right) \otimes^{L}  \bf{R}\nu_{\X, *}\left(\O^{+, a}_{X}/p\right) \arrow[u, shift left=7ex, "\rm{Res}_{\mf}"] \arrow{r}{\cup} & \bf{R}\nu_{\X, *} \left(\O_{X}^{+, a}/p\right)
\end{tikzcd}
\]
commutes in $\bf{D}(\O_{\X_0}[G])^a$. 
\end{lemma}
\begin{proof}
    Arguing as in the proof of Corollary~\ref{cor:descend-rig-finite-etale}, we reduce the question to showing that the diagram
    \[
    \begin{tikzcd}
        f_{\proet, *} \left(\O^+_{X'}/p\right)\otimes^{L} f_{\proet, *}\left(\O^+_{X'}/p\right)  \arrow{r}{\cup} \arrow[d, shift left=6ex, "\rm{Tr}_{\proet, f}"] &  f_{\proet, *}\left(\O^+_{X'}/p\right) \arrow{d}{\rm{Tr}_{\proet, f}}  \\
        \O^+_{X}/p \otimes^{L} \O^+_{X}/p \arrow[u, shift left=6ex, "\rm{Res}_{f}"] \arrow{r}{\cup} & \O_{X}^+/p
    \end{tikzcd}
    \]
commutes in $\bf{Mod}_{\O^+_X/p}^G$. This equality can now be checked pro-\'etale locally, so one can assume that $f$ is a split $G$-torsor. In this case, the claim is obvious. 
\end{proof}

\begin{lemma}\label{lemma:descend-rig-finite-etale-equivariant-2} Let $\mf\colon \X' \to \X$ be a good quotient of nice admissible formal $\O_C$-schemes of dimension $d$ (see Definition~\ref{defn:nice-formal-scheme}), and let $f\colon X' \to X$ be the induced morphism on generic fibers. Then 
    \begin{enumerate}
        \item the morphism $\rm{Tr}_{\mf_0}(-d)[-2d]\circ \bf{R}\mf_{0, *}(\rm{Tr}_{F, \X'})\colon \bf{R}\mf_{0, *} \left(\bf{R}\nu_{\X', *} \O_{X'}^{+, a}/p\right) \to \omega^{\bullet, a}_{\X_0}(-d)[-2d]$ can be uniquely promoted to a morphism in $\bf{D}(\O_{\X_0}[G])^a$ (with the trivial $G$-equivariant structure on $\omega^{\bullet, a}_{\X_0}$);
        \item the diagram
    \begin{equation}\label{eqn:commutes-G-equi}
    \begin{tikzcd}[column sep = 4em, row sep = 4em]
        \bf{R}\mf_{0, *} \left(\bf{R}\nu_{\X', *} \O_{X'}^{+, a}/p\right) \arrow{d}{\rm{Tr}_{\rm{Zar}, \mf}} \arrow{rd}{\rm{Tr}_{\mf_0}(-d)[-2d]\circ \bf{R}\mf_{0, *}(\rm{Tr}_{F, \X'})} &  \\
        \bf{R}\nu_{\X, *} \left(\O_{X}^{+, a}/p\right) \arrow{r}{\rm{Tr}_{F, \X}} & \omega^{\bullet, a}_{\X_0}(-d)[-2d]
    \end{tikzcd}
    \end{equation}
commutes\footnote{We note that this diagram commutes in $\bf{D}(\X_0)^a$ due to Corollary~\ref{cor:descend-rig-finite-etale}.} in $\bf{D}(\O_{\X_0}[G])^a$.
    \end{enumerate}
\end{lemma}
\begin{proof}
    We first show (1). For this, we note that 
    \[
    \bf{R}\mf_{0, *} \left(\bf{R}\nu_{\X', *} \O_{X'}^{+, a}/p\right) \in \bf{D}^{[0, d]}_{acoh}(\X_0)^a
    \]
    due to \cite[Theorem 6.13.5]{Z3} (and finiteness of $\mf$), and $\omega_{\X_0}(-d)[-2d]\in \bf{D}^{[d, 2d]}_{coh}(\X_0)$ due to Theorem~\ref{thm:dualizing-complex}. In particular, the morphism $\rm{Tr}_{\mf_0}(-d)[-2d]\circ \bf{R}\mf_{0, *}(\rm{Tr}_{F, \X'})$ factors through the morphism
    \[
    \cal{H}^d\left(\rm{Tr}_{\mf_0}(-d)[-2d]\circ \bf{R}\mf_{0, *}(\rm{Tr}_{F, \X'}) \right) \colon \mf_{0, *} \left(\rm{R}^d\nu_{\X', *} \O_{X'}^{+, a}/p\right) \to \omega^a_{\X_0}(-d). 
    \]
    Therefore, the morphism $\rm{Tr}_{\mf_0}(-d)[-2d]\circ \bf{R}\mf_{0, *}(\rm{Tr}_{F, \X'})$ can be (uniquely) promoted to a morphism in $\bf{D}(\O_{\X_0}[G])^a$ if and only if $\cal{H}^d\left(\rm{Tr}_{\mf_0}(-d)[-2d]\circ \bf{R}\mf_{0, *}(\rm{Tr}_{F, \X'}) \right)$ is $G$-equivariant. In particular, we see that this promotion is unique if exists. \smallskip
    
    Corollary~\ref{cor:descend-rig-finite-etale} ensures that $\cal{H}^d\left(\rm{Tr}_{\mf_0}(-d)[-2d]\circ \bf{R}\mf_{0, *}(\rm{Tr}_{F, \X'}) \right)$ is equal to the composition
    \[
    \cal{H}^d\left( \rm{Tr}_{F, \X}\right) \circ \cal{H}^d\left(\rm{Tr}_{\rm{Zar}, \mf}\right)  \colon \mf_{0, *} \left(\rm{R}^d\nu_{\X', *} \O_{X'}^{+, a}/p\right) \to \omega^a_{\X_0}(-d).
    \]
    Now we note that $\cal{H}^d\left( \rm{Tr}_{F, \X}\right)$ is $G$-equivariant because it is a morphism between sheaves with trivial $G$-action, and $\cal{H}^d\left(\rm{Tr}_{\rm{Zar}, \mf}\right)$ is $G$-equivariant by Construction~\ref{rmk:pro-etale-trace-equivariant}. Therefore, the composition is also $G$-equivariant finishing the proof. \smallskip
    
    (2): By degree reasons, we can check that the diagram commutes after applying $\cal{H}^d$. Then it is clear from the discussion/construction above. 
\end{proof}

Finally we are able to prove the main result of this section:

\begin{lemma}\label{lemma:descent-good-quotients} Let $\mf\colon \X' \to \X$ be a good quotient of nice admissible formal $\O_C$-schemes of dimension $d$ (see Definition~\ref{defn:nice-formal-scheme}), and let $f\colon X' \to X$ be the induced morphism on generic fibers. If Faltings' pairing is almost perfect on $\X'$, then the same holds on $\X$.
\end{lemma}
\begin{proof}
We note that Corollary~\ref{cor:descend-rig-finite-etale} implies that the following diagram commutes
\[
\begin{tikzcd}[column sep = 4em]
\bf{R}\nu_{\X, *} \left(\O_X^{+, a}/p\right) \arrow{r}{D_{\X}} \arrow{dd}{\rm{Res}_{\mf}}& \bf{R}\ud{al\cal{H}om}_{\X_0}\left(\bf{R}\nu_{\X, *}\left(\O_{X}^{+, a}/p\right), \omega^{\bullet, a}_{\X_0}(-d)[-2d]\right) \arrow{d}{} \\
 & \bf{R}\ud{al\cal{H}om}_{\X_0}\left(\bf{R}\mf_{0, *}\bf{R}\nu_{\X', *}\left(\O_{X'}^{+, a}/p\right), \omega^{\bullet, a}_{\X_0}(-d)[-2d]\right)\\
\bf{R}\mf_{0, *}\bf{R}\nu_{\X', *}\left(\O_{X'}^{+, a}/p\right) \arrow{r}{\bf{R}\mf_{0, *}D_{\X'}} & \bf{R}\mf_{0, *}\bf{R}\ud{al\cal{H}om}_{\X'_0}\left(\bf{R}\nu_{\X', *}\left(\O_{X'}^{+, a}/p\right), \omega^{\bullet, a}_{\X'_0}(-d)[-2d]\right) \arrow{u},
\end{tikzcd}
\]
where the right bottom vertical map is the isomorphism from \cite[Lemma 5.5.6]{Z3}, and the right top vertical map is induced by 
\[
\rm{Tr}_{\rm{Zar}, \mf}\colon \bf{R}\mf_{0, *}\bf{R}\nu_{\X', *}\left(\O_{X'}^+/p\right) \to \bf{R}\nu_{\X, *}\left(\O_{X}^+/p\right).
\]
Since the bottom vertical arrow is an isomorphism, we slightly abuse the notation and denote the composition of the bottom horizontal map and the bottom vertical map simply by
\[
\bf{R}\mf_{0, *}D_{\X'}\colon \bf{R}\mf_{0, *}\bf{R}\nu_{\X', *}\left(\O_{X'}^{+, a}/p\right) \to \bf{R}\ud{al\cal{H}om}_{\X_0}\left(\bf{R}\mf_{0, *}\bf{R}\nu_{\X', *}\left(\O_{X'}^{+, a}/p\right), \omega^{\bullet, a}_{\X_0}(-d)[-2d]\right).
\]

Therefore, we get a commutative diagram
\begin{equation}\label{diagram:good-quotients-step-2}
\begin{tikzcd}[column sep = 4em]
\bf{R}\nu_{\X, *}\left( \O_X^{+, a}/p\right) \arrow{r}{D_{\X}} \arrow{d}{\rm{Res}_{\mf}}& \bf{R}\ud{al\cal{H}om}_{\X_0}\left(\bf{R}\nu_{\X, *}\left(\O_{X}^{+, a}/p\right), \omega^{\bullet, a}_{\X_0}(-d)[-2d]\right) \arrow{d}{} \\
\bf{R}\mf_{0, *}\bf{R}\nu_{\X', *}\left(\O_{X'}^{+, a}/p\right) \arrow{r}{\bf{R}\mf_{0, *}D_{\X'}} & \bf{R}\ud{al\cal{H}om}_{\X_0}\left(\bf{R}\mf_{0, *}\bf{R}\nu_{\X', *}\left(\O_{X'}^{+, a}/p\right), \omega^{\bullet, a}_{\X_0}(-d)[-2d]\right), 
\end{tikzcd}
\end{equation}
that (canonically) lifts to a commutative diagram in $\bf{D}(\O_{\X_0}[G])^a$ due to Lemma~\ref{lemma:descend-rig-finite-etale-equivariant} and Lemma~\ref{lemma:descend-rig-finite-etale-equivariant-2}. Since the $G$-actions on $\bf{R}\nu_{\X, *}\left( \O_X^{+, a}/p\right)$ and $\bf{R}\ud{al\cal{H}om}_{\X_0}\left(\bf{R}\nu_{\X, *}\left(\O_{X}^{+, a}/p\right), \omega^{\bullet, a}_{\X_0}(-d)[-2d]\right)$ are trivial (by construction), Diagram~(\ref{diagram:good-quotients-step-2}) induces the following commutative diagram:
\[
\begin{tikzcd}[column sep = 5em]
\bf{R}\nu_{\X, *}\left( \O_X^{+, a}/p\right) \arrow{r}{D_{\X}} \arrow{d}{}& \bf{R}\ud{al\cal{H}om}_{\X_0}\left(\bf{R}\nu_{\X, *}\left(\O_{X}^{+, a}/p\right), \omega^{\bullet, a}_{\X_0}(-d)[-2d]\right) \arrow{d}{} \\ 
\left(\bf{R}\mf_{0, *}\bf{R}\nu_{\X', *}\left(\O_{X'}^{+, a}/p\right)\right)^{hG} \arrow{d} \arrow{r}{\left(\bf{R}\mf_{0, *}D_{\X'}\right)^{hG}} & \left(\bf{R}\ud{al\cal{H}om}_{\X_0}\left(\bf{R}\mf_{0, *}\bf{R}\nu_{\X', *}\left(\O_{X'}^{+, a}/p\right), \omega^{\bullet, a}_{\X_0}(-d)[-2d]\right)\right)^{hG} \arrow{d}\\
\bf{R}\mf_{0, *}\bf{R}\nu_{\X', *}\left(\O_{X'}^{+, a}/p\right) \arrow{r}{\bf{R}\mf_{0, *}D_{\X'}} & \bf{R}\ud{al\cal{H}om}_{\X_0}\left(\bf{R}\mf_{0, *}\bf{R}\nu_{\X', *}\left(\O_{X'}^{+, a}/p\right)^a, \omega^{\bullet, a}_{\X_0}(-d)[-2d]\right), 
\end{tikzcd}
\]
Now we note that $\left(\bf{R}\mf_{0, *}D_{\X'}\right)^{hG}$ is an isomorphism since $D_{\X'}$ is an isomorphism by our assumption. Moreover, the top left vertical arrow is an isomorphism by Lemma~\ref{lemma:invariant-coinvariant-torsor} and Remark~\ref{rmk:almost-makes-sense}. Thus, it suffices to show that the top right vertical arrow is an isomorphism. Lemma~\ref{lemma:duality-invariants-coinvarints} ensures that it suffices to show that the composition
\begin{align*}
\bf{R}\ud{al\cal{H}om}_{\X_0}\left(\bf{R}\nu_{\X, *}(\O_{X}^{+, a}/p), \omega^{\bullet, a}_{\X_0}(-d)[-2d]\right) &\to \left(\bf{R}\ud{al\cal{H}om}_{\X_0}\left(\bf{R}\mf_{0, *}\bf{R}\nu_{\X', *}(\O_{X'}^{+, a}/p), \omega^{\bullet, a}_{\X_0}(-d)[-2d]\right)\right)^{hG}  \\
& \xrightarrow{\sim} \bf{R}\ud{al\cal{H}om}_{\X_0}\left(\left(\bf{R}\mf_{0, *}\bf{R}\nu_{\X', *}(\O_{X'}^{+, a}/p)\right)_{hG}, \omega^{\bullet, a}_{\X_0}(-d)[-2d]\right)
\end{align*}
is an isomorphism. This, in turn, follows from Lemma~\ref{lemma:invariant-coinvariant-torsor} as the morphism 
\[
\left(\bf{R}\mf_{0, *}\, \bf{R}\nu_{\X', *}\left(\O_{X'}^{+}/p\right)\right)_{hG} \to \bf{R}\nu_{\X, *}\left(\O_{X}^{+}/p\right)
\]
is already an isomorphism. 
\end{proof}

\begin{thm}\label{thm:local-duality} Let $\X$ be a separated admissible formal $\O_C$-scheme with smooth generic fiber $X=\X_C$ of pure dimension $d$. Then Faltings' pairing is almost perfect on $\X$.
\end{thm}
\begin{proof}
Let $\X_i, \X'_i, \X''_i$, and $\Y_j$ be the formal $\O_C$-schemes obtained by applying Theorem~\ref{thm:local-uniformization} to $\X$.

{\it Step 1. Prove the claim for each $\X''_i$}: This was already done in Theorem~\ref{thm:almost-perfect-polystable} because $\X''_i$ are polystable (and separated).  We only note that each $X''_{i}=\X''_{i, C}$ is indeed pure of dimension $d$ as it is \'etale over $X$.  \smallskip

{\it Step 2. Prove the claim for each $\X'_i$}: We recall that the morphism $\pi_i\colon \X''_i \to \X'_i$ is a rig-isomorphism between nice admissible formal $\O_C$-schemes of dimension $d$. Then Lemma~\ref{lemma:descent-blow-up} ensures that Faltings' trace is almost perfect on $\X'_i$ if it is almost perfect on $\X''_i$. \smallskip

{\it Step 3. Prove the claim for each $\X_i$}: This follows from Lemma~\ref{lemma:descent-good-quotients} since $\X'_i \to \X'$ is a good quotient of nice admissible formal $\O_C$-schemes of dimension $d$. \smallskip

{\it Step 4. Prove the claim for $\X$}: We note that the set $(\X_i, \mf_i)$ can be obtained from $\X$ as the composition of rig-isomorphisms and open Zariski coverings. Furthermore, we know that each $\Y_j$ that appears in the presentation of $\X_i$ as a composition of open Zariski covers and rig-isomorphisms has reduced special fiber. Thus, in order to show that Faltings' pairing is almost perfect on $\X$, it suffice to show that this property descends through Zariski open coverings and rig-isomorphisms of nice formal $\O_C$-schemes. The first claim is trivial, and the second one was proven in Lemma~\ref{lemma:descent-blow-up}.
\end{proof}

\section{Global Duality}\label{section:duality}

\subsection{Overview}
Throughout this section, we fix a complete rank-$1$ valued field $K$ with ring of integers $\O_K$, maximal ideal $\m\subset \O_K$, and residue field $k$. We also assume that $\O_K$ is of mixed characteristic $(0, p)$. We denote the completed algebraic closure of $K$ by $C$. \smallskip

Throughout this section, all cohomology groups are considered with respect to the pro-\'etale topology unless it is specified otherwise. We note that \cite[Corollary 3.17(i)]{Sch1} ensures that, for every integer $i$, the cohomology groups of $\bf{F}_p(i)$ or $\O_X^+/p(i)$ are the same in the pro-\'etale or \'etale topology on any rigid-analytic space $X$. \smallskip 

Section~(\ref{section:duality}) has two main goals. The first one is to construct the {\it global analogue} of Faltings' trace map and then to show that it induces an almost perfect pairing. The second goal is to deduce the usual Poincar\'e Duality for mod-$p$ \'etale cohomology of smooth proper rigid-analytic $K$-spaces.

\begin{thm}\label{thm:almost-global-duality-overview} Let $X$ be a smooth proper rigid $C$-space of pure dimension $d$. Then the there is a global trace map $\rm{Tr}_X\colon \bf{R}\Gamma(X, \O_X^{+, a}/p) \to \O_C^a/p(-d)[-2d]$ such that the induced pairing
\[
\bf{R}\Gamma(X, \O_X^{+, a}/p) \otimes^L_{\O_C/p} \bf{R}\Gamma(X, \O_X^{+, a}/p) \xr{-\cup -} \bf{R}\Gamma(X, \O_X^{+, a}/p) \xr{\rm{Tr}_X} \O_C^a/p(-d)[-2d]
\]
is almost perfect. 
\end{thm}

Theorem~\ref{thm:almost-global-duality-overview} is essentially a formal consequence of Theorem~\ref{thm:local-duality} and the almost version of Grothendieck Duality. One subtle point is that the construction of $\rm{Tr}_X$ will, a priori, depend on a choice of a nice formal model $\X$, but we show that it does not depend on it. \smallskip

To deduce the \'etale version of Poincar\'e Duality, we will need to use the \'etale trace map \[
t_X\colon \rm{H}^{2d}_\et(X_C, \bf{F}_p(d)) \to \bf{F}_p
\]
constructed by Berkovich in \cite[\textsection 7.2]{Ber} in the language of Berkovich spaces, and translated into the language of adic spaces in Section~\ref{section:berkovich-trace}.

\begin{thm}\label{thm:etale-PD-overview} Let $X$ be a smooth proper rigid $K$-space of pure dimension $d$. Then the pairing \[
\rm{H}^i_\et(X_C, \bf{F}_p) \otimes_{\bf{F}_p} \rm{H}^{2d-i}_\et(X_C, \bf{F}_p(d))\xr{-\cup -} \rm{H}^{2d}_\et(X_C, \bf{F}_p(d)) \xr{t_X} \bf{F}_p
\]
is a perfect pairing of continuous Galois representations for any $i\geq 0$. 
\end{thm}

The essential idea of the proof of Theorem~\ref{thm:etale-PD-overview} is to deduce it from Theorem~\ref{thm:almost-global-duality-overview} via the Primitive Comparison Theorem.

\subsection{Global almost duality}\label{section:global-almost-duality}

The main goal of this section is to prove Theorem~\ref{thm:almost-global-duality-overview}. Throughout this section, we fix a smooth proper rigid-analytic $C$-space $X$ of pure dimension $d$. \smallskip

We start with the construction of the global trace map by choosing a nice admissible formal $\O_C$-model $\mf\colon \X \to \Spf \O_C$ of the rigid-analytic space $f\colon X\to \Spa(C, \O_C)$ (its existence is guaranteed by the Reduced Fiber Theorem, see Theorem~\ref{thm:cofinal-reduced-family}). This model is automatically proper due to Lemma~\ref{lemma:proper-adic-formal}. 

\begin{defn} The {\it global trace} $\rm{Tr}_X^\X\colon \bf{R}\Gamma(X, \O_X^{+, a}/p) \to \O_C^a/p(-d)[2d]$ is the composition
\[
\bf{R}\Gamma\left(X, \O_X^{+, a}/p\right)\simeq \bf{R}\Gamma\left(\X_0, \bf{R}\nu_*\left(\O_X^{+, a}/p\right)\right) \xr{\bf{R}\Gamma(\X_0, \rm{Tr}_{F, \X})} \bf{R}\Gamma\left(\X_0, \omega^{\bullet, a}_{\X_0}\right)(-d)[-2d] \xr{\rm{Tr}_{\mf_0}(-d)[-2d]} \O^a_C/p(-d)[-2d],
\]
where $\rm{Tr}_{F, \X}$ is Faltings' trace from Definition~\ref{defn:F-trace-final}, and $\rm{Tr}_{\mf_0}$ is the trace map coming from Theorem~\ref{thm:existence-!}. 
\end{defn}

This construction, a priori, depends on the choice of an admissible nice model $\X$. We show that actually this construction is canonically independent of this choice.

\begin{lemma}\label{lemma:global-trace-independent} Let $X$ be as above. Then the trace maps $\rm{Tr}_X^\X$ and $\rm{Tr}_X^{\X'}$ are canonically identified for any two choices of admissible formal $\O_C$-models $\X$ and $\X'$.
\end{lemma}
\begin{proof}
Suppose we have two nice admissible formal models $\X$ and $\X'$ of the rigid-analytic space $X$. Using Theorem~\ref{thm:cofinal-reduced-family}, we can choose another nice admissible formal $\O_C$-model $\X''$ of $X$ that dominates both $\X$ and $\X'$. Therefore, for the purpose of proving that $\rm{Tr}_X^\X$ can be identified with $\rm{Tr}_X^{\X'}$, it is enough to assume that there is a rig-isomorphism $\pi \colon \X' \to \X$. \smallskip

Now Corollary~\ref{cor:descend-rig-finite-etale} implies that the following diagram commutes
\[
\begin{tikzcd}[column sep = 7 em]
\bf{R}\Gamma(X, \O_X^{+, a}/p) \arrow{r}{\sim} & \bf{R}\Gamma\left(\X_0, \bf{R}\nu_{\X, *}\left(\O_X^{+, a}/p\right)\right) \arrow{r}{\sim} \arrow{d}{\bf{R}\Gamma\left(\X_0, \rm{Tr}_{F, \X}\right)} & \bf{R}\Gamma\left(\X_0, \bf{R}\pi_*\bf{R}\nu_{\X', *}\left(\O_{X'}^{+, a}/p\right)\right) \arrow{d}{\bf{R}\Gamma(\X_0, \bf{R}\pi_*(\rm{Tr}_{F, \X'}))} \\
\O^a_C/p(-d)[-2d] & \arrow{l}{\rm{Tr}_{\mf_0}(-d)[-2d]} \bf{R}\Gamma\left(\X_0, \omega^{\bullet, a}_{\X_0}(-d)[-2d]\right) & \arrow{l}{\bf{R}\Gamma(\X_0, \rm{Tr}_{\pi_0}(-d)[-2d])} \bf{R}\Gamma\left(\X_0, \bf{R}\pi_*\omega^{\bullet, a}_{\X'_0}(-d)[-2d]\right).
\end{tikzcd}
\]
The left horseshoe defines the morphism $\rm{Tr}_X^\X \colon \bf{R}\Gamma(X, \O_X^{+, a}/p) \to \O_C^a/p(-d)[-2d]$, and the outer horseshoe defines the morphism $\rm{Tr}_X^{\X'}$ because $\rm{Tr}_{\mf_0}\circ \bf{R}\Gamma(\X_0, \rm{Tr}_{\pi_0})= \rm{Tr}_{\mf'_0}$. Therefore, $\rm{Tr}_{X}^\X = \rm{Tr}_X^{\X'}$.
\end{proof}

\begin{defn}\label{defn:global-trace} We define the {\it global trace map} 
\[
\rm{Tr}_X \colon \bf{R}\Gamma\left(X, \O_X^{+, a}/p\right) \to \O^a_C/p(-d)[-2d]
\]
to be $\rm{Tr}_X^\X$ for any choice of nice admissible formal model $\X$. Lemma~\ref{lemma:global-trace-independent} guarantees that it is well-defined. 
\end{defn}

The global trace map defines the pairing
\begin{equation}\label{eqn:trace-almost-mod-p}
\bf{R}\Gamma(X, \O_X^{+, a}/p) \otimes_{\O_C/p}^L \bf{R}\Gamma(X, \O_X^{+, a}/p) \xr{-\cup-} \bf{R}\Gamma(X, \O_X^{+, a}/p)  \xr{\rm{Tr}_X} \O_C^a/p(-d)[-2d].
\end{equation}
We show that this pairing is almost perfect, i.e. the {\it duality map}
\[
D_X\colon \bf{R}\Gamma(X, \O_X^{+, a}/p) \to \bf{R}\rm{alHom}_{\O_C/p}\left(\bf{R}\Gamma(X, \O_X^{+, a}/p), \O_C^a/p(-d)[-2d]\right)
\]
is an almost isomorphism\footnote{This pairing is symmetric, so this definition coincides with the definition from Section~\ref{section:notation}}. 

\begin{thm}\label{thm:global-duality-mod-p} Let $X$ be a smooth proper rigid $C$-space of pure dimensiond $d$. Then the pairing (\ref{eqn:trace-almost-mod-p}) is almost perfect, i.e., the morphism $D_X$ is an almost isomorphism. 
\end{thm}
\begin{proof} 
We note that \cite[\href{https://stacks.math.columbia.edu/tag/0FP6}{Tag 0FP6}]{stacks-project} gives that the diagram
\[
\begin{tikzcd}
\bf{R}\Gamma(X, \O_X^+/p) \otimes^L_{\O_C/p} \bf{R}\Gamma(X, \O_X^+/p) \arrow{r}{\sim} \arrow{d}{\cup}& \bf{R}\Gamma(\X_0, \bf{R}\nu_* \O_X^+/p) \otimes^L_{\O_C/p} \bf{R}\Gamma(\X_0, \bf{R}\nu_* \O_X^+/p) \arrow{d}{\cup} \\
\bf{R}\Gamma(X, \O_X^+/p) & \arrow{l}{\bf{R}\Gamma(\X_0,- \cup-)}  \bf{R}\Gamma(\X_0, \bf{R}\nu_* \O_X^+/p \otimes^L_{\O_{\X_0}} \bf{R}\nu_* \O_X^+/p)
\end{tikzcd}
\]
is commutative. As the trace map in Grothendieck duality is defined via cup products, this formally implies that the diagram
\[\label{eqn:diagram-global-duality}
\begin{tikzcd}[column sep = 5em]
\bf{R}\Gamma\left(\X_0, \bf{R}\nu_*\left(\O_X^{+, a}/p\right)\right)\arrow{d}{\wr} \arrow{r}{\bf{R}\Gamma(\X_0, D_{\X})} & \bf{R}\Gamma\left(\X_0, \bf{R}\ud{al\cal{H}om}_{\O_{\X_0}}\left(\bf{R}\nu_*\left(\O_X^{+, a}/p\right), \omega^{\bullet, a}_{\X_0}(-d)[-2d]\right)\right)  \arrow{d}{\wr} \\
\bf{R}\Gamma(X, \O_X^{+, a}/p) \arrow{r}{D_X} & \bf{R}\rm{alHom}_{\O_C/p}\left(\bf{R}\Gamma\left(X, \O_X^{+, a}/p\right), \O^a_C/p(-d)[-2d]\right)
\end{tikzcd}
\]
is commutative, where the right vertical arrow is the almost isomorphism coming from the almost version of Grothendieck Duality (use \cite[Lemma 5.5.6]{Z3} for $f=\mf_0\colon \X_0 \to \Spec \O_C/p$). Now we note that the top horizontal arrow is an almost isomorphism by Theorem~\ref{thm:local-duality}, and the left vertical map is an almost isomorphism for tautological reasons. This  implies that $D_X$ is an almost isormorphism. 
\end{proof}

Now we want to show a non-derived analogue of Theorem~\ref{thm:global-duality-mod-p}. We recall that $\bf{R}\Gamma(X, \O_X^+/p)$ is almost concentrated in degrees $[0,2d]$ by \cite[Theorem 6.3.3 and Lemma C.6.10]{Z3}. Therefore, $\rm{Tr}_X$ induces a morphism (in the almost category) 
\[
\rm{Tr}_X^{2d}\colon \rm{H}^{2d}\left(X, \O_X^{+, a}/p\right) \to \O_C^a/p(-d)
\]
This, in turn, induces the pairing
\begin{equation}\label{eqn:pairing-abelian}
    \rm{H}^i(X, \O_X^{+, a}/p) \otimes_{\O_C/p} \rm{H}^{2d-i}(X, \O_X^{+, a}/p) \xr{-\cup-} \rm{H}^{2d}(X, \O_X^{+, a}/p) \xr{\rm{Tr}_X^{2d}} \O_C^a/p(-d).
\end{equation}

\begin{thm}\label{thm:almost-perfect-underived} Let $X$ be a smooth proper rigid $C$-space of pure dimension $d$. Then the pairing (\ref{eqn:pairing-abelian}) is almost perfect for every $i\geq 0$.
\end{thm}
\begin{proof}
Theorem~\ref{thm:global-duality-mod-p} says that the morphism
\[
D_X\colon \bf{R}\Gamma\left(X, \O_X^{+, a}/p\right) \to \bf{R}\rm{alHom}_{\O_C/p}\left(\bf{R}\Gamma\left(X, \O_X^{+, a}/p\right), \O^a_C/p(-d)[-2d]\right)
\]
is an almost isomorphism. Now we recall the Primitive Comparison Theorem \cite[Theorem 5.1]{Sch1} that says that the natural map
\[
\rm{H}^i_\et(X, \bf{F}_p) \otimes_{\bf{F}_p} \O_C/p \to \rm{H}^i(X, \O_X^+/p)
\]
is an almost isomorphism for any $i\geq 0$. This implies that $\rm{H}^i(X, \O_X^{+, a}/p)$ are almost projective over $\O_C^a/p$ by \cite[Lemma 2.2.6]{Z3} (in fact, these groups are even almost finite free, but we do not need this). Therefore, we get isomorphisms 
\[
\rm{H}^i\left(\bf{R}\rm{alHom}_{\O_C/p}\left(\bf{R}\Gamma(X, \O_X^{+, a}/p), \O^a_C/p(-d)[-2d]\right)\right) \simeq \rm{alHom}_{\O_C/p}\left(\rm{H}^{2d-i}\left(X, \O_X^{+, a}/p\right), \O^a_C/p(-d)\right).
\]
Thus, Theorem~\ref{thm:global-duality-mod-p} implies that the natural map
\[
\rm{H}^i(X, \O_X^{+, a}/p) \to \rm{alHom}_{\O_C/p}(\rm{H}^{2d-i}(X, \O_X^{+, a}/p), \O_C^a/p(-d))
\]
is an almost isomorphism, i.e. the pairing~(\ref{eqn:pairing-abelian}) is almost perfect. 
\end{proof}

\subsection{Berkovich trace}\label{section:berkovich-trace}

We recall the trace map in \'etale cohomology constructed in \cite{Ber}. Berkovich defined the trace map for a smooth morphism of Berkovich spaces, we transfer his construction to the language of adic spaces. \smallskip

For the rest of the section, we fix a ring $\Lambda = \Z/n\Z$ for some $n>0$. For a morphism $f\colon X \to Y$ of rigid-analytic spaces, we will freely identify $\left(\bf{R}f_!\, \F\right)(m)$ with $\bf{R}f_! \left( \F(m)\right)$ for any $\F\in \bf{D}^b(X_\et; \Lambda)$. This is possible due to (the easy version) of the projection formula established in \cite[Theorem 5.5.9]{H3}. \smallskip

\begin{defn}\label{defn:dimension}(\cite[Def.\,1.8.1]{H3}) The {\it dimension} of a locally spectral space $X$ is the supremum of the length $d$ of the chains of specializations $x_0\succ x_1 \succ \dots \succ x_d$ of points of $X$.

A locally spectral space $X$ is of {\it pure dimension $d$} if every non-empty open subset $U\subset X$ has dimension $d$.

The {\it (relative) dimension} $\dim f$ of a morphism of analytic adic spaces $f\colon X \to Y$ is the supremum of the dimensions of the fibers of $f$,
\[
\dim f \coloneqq \sup_{y\in Y} \dim f^{-1}(y)\in \bf{Z}_{\geq 0} \cup \{\infty\}.
\]
A morphism $f\colon X \to Y$ is of {\it relative pure dimension $d$} if all non-empty fibers $f^{-1}(y)$ are of pure dimension $d$.
\end{defn}

\begin{lemma}\label{lemma:cohomological-dimension} Let $f\colon X \to Y$ be a partially proper morphism of rigid $K$-spaces of pure dimension $d$. Then $\rm{R}^if_!\,\F=0$ for any $\F\in \rm{Shv}(X_\et, \Lambda)$ and $i>2d$.
\end{lemma}
\begin{proof}
This follows directly from \cite[Proposition 5.3.11 and Corollary 1.8.7]{H3}. 
\end{proof}

Let $ f\colon X \to Y$ and $g\colon Y\to Z$ be smooth partially proper morphisms of taut rigid $K$-spaces of pure dimension $d$ and $e$, respectively. Suppose we are given three homomorphisms $\alpha \colon \rm{R}^{2d}f_!\,\ud{\Lambda}_X(d) \to \ud{\Lambda}_Y$, $\beta\colon R^{2e}g_!\,\ud{\Lambda}_Y(e) \to \ud{\Lambda}_Z$, and $\gamma\colon \rm{R}^{2(d+e)}(g\circ f)_!\, \ud{\Lambda}_X(d+e) \to \ud{\Lambda}_Z$. Using the Leray spectral sequence and Lemma~\ref{lemma:cohomological-dimension}, we get an isomorphism
\[
\rm{R}^{2(d+e)}(g\circ f)_!\,\ud{\Lambda}_X(d+e)\simeq \rm{R}^{2e}g_!\,\left(\rm{R}^{2d}f_!\,\ud{\Lambda}_X(d)(e)\right)
\]
This defines a morphism
\[
\rm{R}^{2(d+e)}(g\circ f)_!\, \ud{\Lambda}_X(d+e)\simeq \rm{R}^{2e}g_!\,(\rm{R}^{2d}f_!\ud{\Lambda}_X(d)(e)) \xr{\rm{R}^{2e}g_!\,(\alpha)(e)} \rm{R}^{2e}g_!\, \ud{\Lambda}_Y(e) \xr{\rm{\beta}}\ud{\Lambda}_Z
\]
that is denoted by $\beta \Box\alpha$. We say that $\alpha$, $\beta$, and $\gamma$ are {\it compatible with the composition} if $\gamma=\beta \Box \alpha$. 

\begin{thm}\label{thm:construction-trace-map} Let $X$ and $Y$ be rigid-analytic $K$-spaces, and let $f\colon X \to Y$ be a partially proper smooth morphism of pure dimension $d$. Then one can define the trace map $t_{f}\colon \rm{R}^{2d}f_!\,\ud{\Lambda}_X(d) \to \ud{\Lambda}_Y$ satisfying the following properties:
\begin{enumerate}
    \item\label{property:trace-1} $t_f$ is compatible with taking geometric fibers over maximal points,
    \item\label{property:trace-2} $t_f$ is compatible with composition, i.e. if $g\colon Y \to Z$ is another smooth, partially proper morphism of pure dimension $e$, then $t_g \Box t_f=t_{g \circ f}$,
    \item\label{property:trace-3} if $d=0$, then $t_f\colon f_!\,\ud{\Lambda}_X \to \ud{\Lambda}_Y$ is the map coming from the adjunction $(f_!, f^*)$.
    \item\label{property:trace-4}  $t_f$ is surjective if all fibers of $f$ are non-empty. 
\end{enumerate}
\end{thm}
\begin{proof}
We note that $\rm{R}^if_!\,\ud{\Lambda}_X(d)$ is an overconvergent sheaf on $Y$ by \cite[Corollary 8.2.4]{H3}. Therefore, (\ref{property:trace-1}) ensures that it suffices to construct $t_f$ locally on $Y$ because a map of overconvergent sheaves is uniquely defined by a map on stalks over maximal points. This uniqueness allows to glue $t_f$ from local pieces. So we may and do assume that $Y$ is affinoid. In particular, Remark~\ref{rmk:app-berkovich-taut-examples} gives that $Y$ is taut and $X$ is also taut as it is partially proper over $Y$.

We consider the associated morphism of Berkovich spaces $u(f)\colon u(X)\to u(Y)$. Lemma~\ref{lemma:properties-of-berkovization} implies that it is a smooth morphism of pure dimension $d$. Therefore, \cite[Theorem 7.2.1]{Ber} constructs the morphism 
\[
t_f^B\colon \rm{R}^{2d}u(f)_!\ud{\Lambda}_{u(X)} \to \ud{\Lambda}_{u(Y)}
\]
with all desired properties. Now we use Theorem~\ref{thm:berkozation-compact-cohomology} to define 
\[
t_f\colon \rm{R}^{2d}f_!\ud{\Lambda}_{X} \to \ud{\Lambda}_{Y}
\]
as $\theta_Y^*(t_f^B)$. \smallskip

It is easy to see that $t_f^B$ satisfies (\ref{property:trace-1}), (\ref{property:trace-2}). Now (\ref{property:trace-3}) follows from Lemma~\ref{lemma:trace-berkovich-adic}. So we only need to show (\ref{property:trace-4}). Since $\theta_Y^*$ is exact, \cite[Theorem 7.2.1]{Ber} ensures that it suffices to show that $u(f)$ is surjective if $f$ is surjective. This is clear because the underlying continuous morphism $|u(f)|\colon |u(X)| \to |u(Y)|$ can be identified with the morphism $f_{\max}\colon X_{\max} \to Y_{\max}$. So its surjectivity follows from surjectivity of $f$ and \cite[Lemma 1.1.10]{H3}.
\end{proof}

\subsection{Poincar\'e Duality for $\bf{F}_p$-coefficients}\label{section:poincare-duality-mod-p}

The main goal of this section is to show Poincar\'e Duality for $\bf{F}_p$-coefficients on a smooth proper rigid-analytic $K$-space $X$. We denote by $f\colon X \to \Spa(K, \O_K)$ the structure morphism of $X$. \smallskip

We recall that sheaves of $\bf{F}_p$-modules on the \'etale site of $\Spa(K, \O_K)_\et$ can be identified with continuous $\bf{F}_p$-representations of the Galois group $G_K$ (for example, this follows from \cite[Proposition 2.3.10]{H3} or it can be seen directly). Therefore, the trace morphism $t_f\colon \rm{R}^{2d}f_*\, \ud{\bf{F}}_p(d) \to \ud{\bf{F}}_p$ can be identified with the Galois-equivariant morphism
\[
t_X \colon \rm{H}^{2d}_{\et}(X_C; \bf{F}_p(d)) \to \bf{F}_p.
\]

\begin{rmk}\label{rmk:pass-to-alg-closure} Theorem~\ref{thm:construction-trace-map}(\ref{property:trace-1}) implies that $t_X=t_{X_C} \colon \rm{H}^{2d}_\et(X_C; \bf{F}_p(d)) \to \bf{F}_p$ as a morphism of {\it abelian groups}.
\end{rmk}

\begin{thm}\label{thm:trace-iso-connected} Let $X$ be a smooth proper geometrically connected rigid $K$-space of pure dimension $d$. Then the trace map $t_X\colon \rm{H}^{2d}_\et(X_C, \bf{F}_p(d)) \to \bf{F}_p$ is an isomorphism of Galois representations.
\end{thm}
\begin{proof}
First we note that $t_X$ is Galois-equivariant by its construction, so Remark~\ref{rmk:pass-to-alg-closure} allows us to assume that $K=C$ is algebraically closed. Then $t_X$ is surjective due to Theorem~\ref{thm:construction-trace-map}(\ref{property:trace-4}). Therefore, it suffices to show that $\rm{H}^{2d}_\et(X, \bf{F}_p(d))$ is a one dimensional vector space over $\bf{F}_p$. Using the Primitive Comparison Theorem, we conclude that 
\[
\rm{H}^{i}_\et(X, \bf{F}_p(d))\otimes_{\bf{F}_p} \O_C^a/p \simeq \rm{H}^{i}(X, (\O_X^{+,a}/p)(d))
\]
is an almost isomorphism for all $i$. 

Now we note that the classification of finitely presented torsion $\O_C$-modules \cite[Proposition 2.10]{Sch1} guarantees that it suffices to show  that $\rm{H}^{2d}(X, (\O_X^{+, a}/p)(d)) \simeq \O_C^a/p$. \smallskip

We know that $\rm{H}^{2d}(X, (\O_X^+/p)(d))$ is almost dual to $\rm{H}^0(X, \O_X^+/p)$ due to Theorem~\ref{thm:global-duality-mod-p}, so it suffices to show that $\rm{H}^0(X, \O_X^+/p)^a\simeq \O_C^a/p$. We use the Primitive Comparison Theorem again to conclude that
\[
\rm{H}^0(X, \O_X^+/p)^a \simeq \rm{H}^0_\et(X, \bf{F}_p)\otimes_{\bf{F}_p}\O_C^a/p \simeq \O_C^a/p
\]
because $X$ is (geometrically) connected. 
\end{proof}

\begin{lemma}\label{lemma:iso-after-tensor-product} Let $V$ and $W$ be finite dimensional $\bf{F}_p$-vector spaces, and let $f\colon V \to W$ be an $\bf{F}_p$-linear homomorphism. Then $f$ is an isomorphism if and only if the morphism 
\[
f\otimes_{\bf{F}_p} \O_C/p \colon V\otimes_{\bf{F}_p} \O_C/p \to W\otimes_{\bf{F}_p} \O_C/p
\]
is an almost isomorphism.
\end{lemma}
\begin{proof}
    Let $F$ be the kernel of $f$, and let $Q$ be the cokernel of $f$. Then the kernel and cokernel of $f\otimes_{\bf{F}_p} \O_C/p$ are given by $F'\coloneqq F\otimes_{\bf{F}_p} \O_C/p \cong (\O_C/p)^{\dim F}$ and $Q'\coloneqq Q\otimes_{\bf{F}_p} \O_C/p \cong (\O_C/p)^{\dim Q}$ respectively. Now \cite[Proposition 2.10]{Sch1} implies that $(\O_C/p)^{\dim F} \simeq^a 0$ (resp. $(\O_C/p)^{\dim Q} \simeq^a 0$) if and only if $\dim F=0$ (resp. $\dim Q=0$). Therefore, $f$ is an isomorphism if and only if $f\otimes_{\bf{F}_p}\O_C/p$ is an almost isomorphism. 
\end{proof}

\begin{thm}\label{thm:poincare-duality-mod-p} Let $X$ be a smooth and proper rigid-analytic $K$-space of pure dimension $d$. Then the pairing 
\[
\rm{H}^i_\et(X_{C}, \bf{F}_p) \otimes_{\bf{F}_p} \rm{H}^{2d-i}_\et(X_{C}, \bf{F}_p(d)) \to \rm{H}^{2d}_\et(X_{C}, \bf{F}_p(d)) \xr{t_X} \bf{F}_p
\]
is a perfect Galois-equivariant pairing.
\end{thm}
\begin{proof}
Similar to Theorem~\ref{thm:trace-iso-connected}, we can assume that $K=C$ is algebraically closed. Then $X$ is a disjoint union of its (geometrically) connected components $X=\sqcup_{i=1}^n X_i$. Theorem~\ref{thm:construction-trace-map}(\ref{property:trace-3}) ensures that $t_X$ is equal to 
\[
\sum_{i=1}^n t_{X_i} \colon \bigoplus_{i=1}^n \rm{H}^{2d}_\et(X_i, \bf{F}_p(d))=\rm{H}^{2d}_\et(X, \bf{F}_p(d)) \to \bf{F}_p.
\]
Therefore, we may and do assume that $X$ is connected and $K=C$ is algebraically closed. \smallskip

Then Theorem~\ref{thm:trace-iso-connected} gives that $t_X$ is an isomorphism. Thus, it suffices to show that the natural morphism
\[
\rm{H}^i_\et(X, \bf{F}_p) \to \rm{Hom}_{\bf{F}_p}(\rm{H}^{2d-i}_\et(X, \bf{F}_p), \rm{H}^{2d}_\et(X, \bf{F}_p))
\]
induced by the cup product is an isomorphism. The main advantage of this reformulation is that it is independent of the construction of the trace map. \smallskip

We consider the following commutative diagram
\[
\begin{tikzcd}[column sep = 5em]
\rm{H}^i_\et(X, \bf{F}_p) \otimes_{\bf{F}_p} \O_C/p \arrow{dd}{\gamma} \arrow{r}{d_1} & \rm{Hom}_{\bf{F}_p}\left(\rm{H}^{2d-i}_\et(X, \bf{F}_p), \rm{H}^{2d}_\et(X, \bf{F}_p)\right) \otimes_{\bf{F}_p} \O_C/p \arrow{d}{\alpha} \\
& \rm{Hom}_{\O_C/p}\left(\rm{H}^{2d-i}_\et(X, \bf{F}_p) \otimes_{\bf{F}_p} \O_C/p, \rm{H}^{2d}_\et(X, \bf{F}_p)\otimes_{\bf{F}_p}\O_C/p\right) \arrow{d}{\beta} \\
\rm{H}^i(X, \O_X^+/p) \arrow{r}{d_2} & \rm{Hom}_{\O_C/p}\left(\rm{H}^{2d-i}(X, \O_X^+/p), \rm{H}^{2d}(X, \O_X^+/p)\right) \\
\end{tikzcd}
\]
Lemma~\ref{lemma:iso-after-tensor-product} implies that it suffices to show that $d_1$ is an almost isomorphism. For this, we note that $\alpha$ is easily seen to be an isomorphism as $\rm{H}^{2d-i}(X, \bf{F}_p)$ is of finite dimension, and $\beta$ and $\gamma$ are almost isomorphisms by the Primitive Comparison Theorem. The map $d_2$ is an almost isomorphism by Theorem~\ref{thm:almost-perfect-underived} (and \cite[Proposition 2.2.1]{Z3}). This implies that $d_1$ is an almost isomorphism and finishes the proof. 
\end{proof}

We also show a version of Poincar\'e Duality for \'etale $\bf{F}_p$-local systems on $X$. For this, we recall that there is the natural ``evaluation'' morphism
\[
\rm{ev}\colon \bf{L} \otimes \bf{L}^\vee \to \bf{F}_p
\]
for any \'etale $\bf{F}_p$-local system $\bf{L}$ on $X$. Then we can combine it with the cup-product to get a morphism
\[
\rm{H}^i_\et(X_C, \bf{L}) \otimes \rm{H}^{2d-i}_\et(X_C, \bf{L}^\vee(d)) \xr{\cup} \rm{H}^{2d}_\et(X_C, \bf{L}\otimes \bf{L}^\vee(d)) \xr{\rm{H}^{2d}_{\et}(X; \rm{ev})} \rm{H}^{2d}_\et(X_C, \bf{F}_p(d)),
\]
functorial in $\bf{L}$.

\begin{thm}\label{thm:local-systems} Let $X$ be a smooth and proper rigid $K$-variety of pure dimension $d$, and let $\bf{L}$ be an \'etale $\bf{F}_p$-local system on $X$. Then the pairing 
\[
\rm{H}^i_\et(X_{C}, \bf{L}) \otimes_{\bf{F}_p} \rm{H}^{2d-i}_\et(X_{C}, \bf{L}^{\vee}(d)) \to \rm{H}^{2d}_\et(X_{C}, \bf{F}_p(d)) \xr{t_X} \bf{F}_p
\]
is a perfect Galois-equivariant pairing.
\end{thm}
\begin{proof}
    If $\bf{L}=f_*\bf{L'}$ for some finite \'etale morphism $f\colon X'\to X$, then Theorem~\ref{thm:construction-trace-map}(\ref{property:trace-2}) implies that Poincar\'e Duality for the local system $\bf{L}$ is equivalent to Poincar\'e Duality for the local system $\bf{L}'$. Using method de la trace (argue as in \cite[\href{https://stacks.math.columbia.edu/tag/03SH}{Tag 03SH}]{stacks-project}), we can find a finite \'etale morphism $f\colon X' \to X$ of degree $r$ prime to $p$ such that 
    \[
    \bf{L}|_{X'}
    \]
    is a successive extension of constant $\bf{F}_p$-local systems on $X'$. Using the trace morphism for finite \'etale morphisms, we see that the composition 
    \[
    \bf{L} \to f_*\left(\bf{L}|_{X'}\right)\to \bf{L}
    \]
    is equal to the multiplication by $r$ map, in particular it is invertible. Thus $\bf{L}$ is a direct summand of $f_*\left(\bf{L}|_{X'}\right)$, and so it suffices to show the claim for $f_*\left(\bf{L}|_{X'}\right)$. The argument above implies that, furthermore, it suffices to show the claim for $\bf{L}|_{X'}$ and $X'$. So we can assume that $\bf{L}$ is a successive extension of constant local systems. In this case, the claim follows from the Five Lemma and Theorem~\ref{thm:poincare-duality-mod-p}.
\end{proof}

\subsection{Poincar\'e Duality for $p$-adic coefficients}

The main goal of this section is to generalize Theorem~\ref{thm:poincare-duality-mod-p} to $\Z/p^n\Z$, $\bf{Z}_p$, and $\Q_p$-coefficients. In this section, we will freely identify $\bf{R}\Gamma_{\proet}(X, \Z/p^n\Z)$ and $\bf{R}\Gamma_\et(X, \Z/p^n\Z)$. \smallskip

We recall that, by definition, $\bf{R}\Gamma(X, \Z_p) = \bf{R}\Gamma(X_\proet, \wdh{\Z}_p)$ for the pro-\'etale sheaf\footnote{See \cite[Proposition 8.2]{Sch1} for the second isomorphism below.}
\[
\wdh{\Z}_p=\lim_n \ud{\Z/p^n\Z} \simeq \bf{R}\lim_n \ud{\Z/p^n\Z}.
\]
Similarly, $ \bf{R}\Gamma(X, \Q_p) \simeq \bf{R}\Gamma(X_\proet, \wdh{\Q}_p)$ for the pro-\'etale sheaf $\wdh{\Q}_p=\wdh{\Z}_p[1/p]$.

\begin{lemma}\label{lemma:cohomology-perfect} Let $X$ be a proper rigid $C$-space of pure dimension $d$, and $n$ a positive integer. Then $\bf{R}\Gamma_\et\left(X, \Z/p^n\Z\right)\in \bf{D}_{perf}^{[0, 2d]}\left(\Z/p^n\Z\right)$ and $\bf{R}\Gamma(X, \Z_p)\in \bf{D}_{perf}^{[0, 2d]}\left(\Z_p\right)$. Furthermore, the natural morphism
\[
\rm{H}^i(X, \Z_p) \to \lim_n\rm{H}^i_\et(X, \Z/p^n\Z)
\]
is an isomorphism for every $i\geq 0$. 
\end{lemma}
\begin{proof}
We note that 
\[
\bf{R}\Gamma_\et\left(X, \Z/p^n\Z\right) \otimes^L_{\Z/p^n\Z} \bf{F}_p \simeq \bf{R}\Gamma_\et\left(X, \bf{F}_p\right) \in \bf{D}^{[0,2d]}_{coh}\left(\bf{F}_p\right) = \bf{D}^{[0,2d]}_{perf}\left(\bf{F}_p\right). 
\]
by \cite[Theorem 3.17]{Schsurvey} (or \cite[Theorem 1.1.2]{Z3}). Therefore, \cite[\href{https://stacks.math.columbia.edu/tag/07LU}{Tag 07LU}]{stacks-project} and \cite[Lemma A.3]{Z3} ensure that $\bf{R}\Gamma_\et(X, \Z/p^n\Z)\in \bf{D}^{[0, 2d]}_{perf}(\Z/p^n\Z)$. \smallskip

Now we note that \cite[\href{https://stacks.math.columbia.edu/tag/0A07}{Tag 0A07}]{stacks-project} ensures that $\bf{R}\Gamma(X, \Z_p) \simeq \bf{R}\lim_n \bf{R}\Gamma_\et(X, \Z/p^n\Z)$. Therefore, \cite[\href{https://stacks.math.columbia.edu/tag/09AV}{Tag 09AV}]{stacks-project} and \cite[Lemma A.3]{Z3} directly imply that $\bf{R}\Gamma(X, \Z_p) \in \bf{D}_{coh}^{[0, 2d]}(\Z_p) = \bf{D}_{perf}^{[0, 2d]}(\Z_p)$. Finally, we consider the Milnor exact sequence 
\[
0 \to \rm{R}^1\lim_n \rm{H}^{i-1}_\et(X, \Z/p^n\Z) \to \rm{H}^i(X, \Z_p) \to \lim_n \rm{H}^i_\et(X, \Z/p^n\Z) \to 0.
\]
Thus, in order to show that $\rm{H}^i(X, \Z_p) \to \lim_n \rm{H}^i_\et(X, \Z/p^n\Z)$ is an isomorphism, it suffices to show that $(\rm{H}^{i-1}_\et(X, \Z/p^n\Z))_{n\in \N}$ satisfies the Mittag-Leffler condition. This follows from finiteness of $\rm{H}^{i-1}_\et(X, \Z/p^n\Z)$ for all $i$ and $n$. 
\end{proof}

\begin{thm}\label{thm:Poincare-duality-mod-pn} Let $X$ be a smooth proper rigid-analytic $K$-space of pure dimension $d$. Then the pairing 
\[
\bf{R}\Gamma_\et(X_C, \Z/p^n\Z) \otimes^L_{\Z/p^n\Z} \bf{R}\Gamma_\et(X_C, \Z/p^n\Z(d)) \xr{- \cup -} \bf{R}\Gamma_\et(X_C, \Z/p^n\Z(d)) \xr{t_{X, \Z/p^n\Z}} \Z/p^n\Z
\]
is a perfect Galois-equivariant pairing for any integer $n\geq 1$.
\end{thm}
\begin{proof}
Similarly to the proof of Theorem~\ref{thm:poincare-duality-mod-p}, we can assume that $K=C$ is algebraically closed. \smallskip

It suffices to show that the duality morphism
\[
D_{X, \Z/p^n\Z} \colon \bf{R}\Gamma_\et(X, \Z/p^n\Z) \to \bf{R}\rm{Hom}_{\Z/p^n\Z}\left(\bf{R}\Gamma_\et(X, \Z/p^n\Z(d)), \Z/p^n\Z\right)
\]
is an isomorphism. Lemma~\ref{lemma:cohomology-perfect} implies that $D_{X, \Z/p^n\Z}$ is a morphism of perfect $\Z/p^n\Z$ complexes. Therefore, it suffices to show that $D_{X, \Z/p^n\Z} \otimes^L_{\Z_p} \bf{F}_p$ is an isomorphism. \smallskip

Now we note that the diagram 
\[
\begin{tikzcd}[column sep = 8em] 
\rm{H}^{2d}_\et(X, \Z/p^n\Z(d)) \otimes_{\Z/p^n\Z} \bf{F}_p \arrow{r}{t_{X, \Z/p^n\Z}\otimes_{\Z/p^n\Z} \bf{F}_p} \arrow{d}{\wr}& \bf{F}_p\\
\rm{H}^{2d}_\et(X,\bf{F}_p(d)) \arrow[ru, swap, "t_{X, \bf{F}_p}"]
\end{tikzcd}
\]
commutes by the construction of the trace map (see Theorem~\ref{thm:construction-trace-map}). Therefore, the square 
\[
\begin{tikzcd}[column sep = 8em]
\bf{R}\Gamma_\et\left(X, \Z/p^n\Z\right) \otimes^L_{\Z/p^n\Z} \bf{F}_p \arrow{r}{D_{X, \Z/p^n\Z}\otimes^L_{\Z/p^n\Z}\bf{F}_p}\arrow{dd}{\gamma} & \bf{R}\rm{Hom}_{\Z/p^n\Z}\bigg(\bf{R}\Gamma_\et\left(X, \Z/p^n\Z(d)\right), \Z/p^n\Z\bigg) \otimes^L_{\Z/p^n\Z} \bf{F}_p \arrow{d}{\alpha} \\
& \bf{R}\rm{Hom}_{\bf{F}_p}\left(\bf{R}\Gamma_\et\left(X, \Z/p^n\Z(d)\right)\otimes^L_{\Z/p^n\Z}\bf{F}_p, \bf{F}_p\right) \arrow{d}{\beta} \\
\bf{R}\Gamma_\et\left(X, \bf{F}_p\right) \arrow{r}{D_{X, \bf{F}_p}} & \bf{R}\rm{Hom}_{\bf{F}_p}\left(\bf{R}\Gamma_\et\big(X, \bf{F}_p(d)\right), \bf{F}_p\big)
\end{tikzcd}
\]
is commutative. We note that $\alpha$ is an isomorphism by \cite[\href{https://stacks.math.columbia.edu/tag/0A6A}{Tag 0A6A}]{stacks-project} since $\bf{R}\Gamma_\et(X, \Z/p^n\Z(d))$ is perfect due to Lemma~\ref{lemma:cohomology-perfect}. Clearly, $\beta$, $\gamma$ are isomorphisms, and $D_{X, \bf{F}_p}$ is an isomorphism by Theorem~\ref{thm:poincare-duality-mod-p}. This implies that $D_{X, \Z/p^n\Z}\otimes^L_{\Z/p^n\Z}\bf{F}_p$ is an isomorphism, and therefore so is $D_{X, \Z/p^n\Z}$.  
\end{proof}

\begin{defn} The {\it trace map} 
\[
t_{X, \Z_p}\colon \rm{H}^{2d}(X_C, \bf{Z}_p(d)) \to \Z_p
\]
for a smooth proper rigid $K$-space $X$ of pure dimension $d$, is defined as follows:
\[
t_{X, \Z_p}\coloneqq \lim_{n}t_{X, \Z/p^n\Z} \colon \rm{H}^{2d}(X_C, \bf{Z}_p(d))=\lim_n \rm{H}^{2d}_\et(X_C, \Z/p^n\Z(d)) \to \Z_p.
\]
\end{defn}

\begin{thm}\label{thm:poincare-duality-integral} Let $X$ be a smooth proper rigid $K$-variety of pure dimension $d$. Then the pairing 
\[
\bf{R}\Gamma(X_C, \Z_p) \otimes^L_{\Z_p} \bf{R}\Gamma(X_C, \Z_p(d)) \xr{- \cup -} \bf{R}\Gamma(X_C, \Z_p(d)) \xr{t_{X, \Z_p}} \Z_p
\]
is a perfect Galois-equivariant pairing.
\end{thm}
\begin{proof}
The proof of Theorem~\ref{thm:Poincare-duality-mod-pn} works almost verbatim.
\end{proof}

In the rest of this section, we will freely use the fact that $\bf{R}\Gamma(X_C, \Q_p)=\bf{R}\Gamma(X_C, \Z_p)[1/p]$ for a {\it qcqs} rigid-analytic $K$-space $X$. This follows from the fact that $X_\proet$ is a coherent topos for a qcqs $X$ (see \cite[Proposition 3.12(vii)]{Sch1}), and the fact that $\rm{R}^i\Gamma(\cal{C}, -)$ commutes with filtered colimits for a coherent topos $\cal{C}$. 

\begin{defn} The {\it trace map} $t_{X, \Q_p}\colon \rm{H}^{2d}(X_C, \bf{Q}_p(d)) \to \Q_p$ for a smooth proper rigid $K$-space $X$ of pure dimension $d$, is defined as follows:
\[
t_{X, \Q_p}\coloneqq t_{X, \Z_p}[1/p] \colon \rm{H}^{2d}(X_C, \bf{Q}_p(d))=\rm{H}^{2d}(X_C, \Z_p(d))[1/p] \to \Q_p.
\]
\end{defn}

\begin{thm}\label{thm:poincare-duality-rational} Let $X$ be a smooth proper rigid $K$-variety of pure dimension $d$. Then the pairing 
\[
\bf{R}\Gamma(X_C, \Q_p) \otimes^L_{\Q_p} \bf{R}\Gamma(X_C, \Q_p(d)) \xr{- \cup -} \bf{R}\Gamma(X_C, \Q_p(d)) \xr{t_{X, \Q_p}} \Q_p
\]
is a perfect Galois-equivariant pairing. In particular, for every integer $i\geq 0$, this induces a perfect Galois-equivariant pairing
\[
\rm{H}^i(X_C, \Q_p) \otimes_{\Q_p} \rm{H}^{2d-i}(X_C, \Q_p(d)) \xr{- \cup -} \rm{H}^{2d}(X_C, \Q_p(d)) \xr{t_{X, \Q_p}} \Q_p.
\]
\end{thm}
\begin{proof}
The result easily follows from the isomorphism $\bf{R}\Gamma(X_C, \Q_p)=\bf{R}\Gamma(X_C, \Z_p)[1/p]$, Theorem~\ref{thm:poincare-duality-integral}, and the observation that $\Q_p$ is a field. 
\end{proof}

\newpage

\appendix

\section{Adic Spaces and Berkovich Spaces}\label{appendix:berkovich}

Throughout this section, we fix a complete rank-$1$ valued field $K$ with ring of integers $\O_K$, and a pseudo-uniformizer $\varpi\in \O_K$. In this section, all adic spaces are assumed to be {\it analytic and locally strongly noetherian}. \smallskip

\begin{defn1} A morphism of analytic adic spaces $f\colon X \to Y$ is called {\it partially proper} if $f$ is locally of $+$-weakly finite type, and satisfies the valuative criterion: for every complete microbial valuation ring $k^+$ and a commutative diagram
\[
\begin{tikzcd}
\Spa(k, k^\circ) \arrow{d} \arrow{r} & X \arrow{d}{f} \\
\Spa(k, k^+) \arrow{r} & Y
\end{tikzcd}
\]
there is a unique lifting $\Spa(k, k^+) \to X$ making the diagram commute. 
\end{defn1}

\begin{rmk1} In \cite{H3}, it is important to allow partially proper morphisms to be merely locally $+$-weakly finite type as opposed to the stronger condition of being locally of finite type (see \cite[Definition 1.2.1]{H3}). However, this extra level of generality will not play any role in this paper; all morphisms we consider are always going to be locally of finite type by design.
\end{rmk1}

\begin{defn1}\label{defn:taut}(\cite[Definition 5.1.2]{H3}) An analytic adic space $X$ is called {\it taut} if $X$ is quasi-separated and, for every quasi-compact subset $U$ of $X$, the closure $\ov{U}$ of $U$ in $X$ is quasi-compact.

A morphism $f\colon X \to Y$ of analytic adic spaces is called {\it taut} if, for every taut open subspace $U$ of $Y$, the inverse image $f^{-1}(U)$ is taut. 
\end{defn1}

\begin{rmk1}\label{rmk:app-berkovich-taut-examples} We recall that any morphism between analytic adic spaces is spectral. Therefore, Definition~\ref{defn:taut} recovers \cite[Definition 5.1.2]{H3}. In particular, \cite[Lemma 5.1.3, 5.1.4]{H3} ensure that any qcqs morphism or partially proper morphism of analytic adic spaces is taut. Furthermore, any morphism between taut analytic adic spaces is taut.
\end{rmk1}

\begin{defn1} The category of {\it adic spaces locally of finite type over $\Spa(K,\O_K)$} is denoted by $(A)$. \smallskip

The category of {\it taut adic spaces locally of finite type over $\Spa(K,\O_K)$} is denoted by $(A)'$. 
\end{defn1}

We recall that any point $x\in X$ of an analytic adic space $X$ has a unique rank-$1$ generalization  $x_{\rm{gen}}$. The topological space $X_{\max}$ is defined to be the set of rank-$1$ points of $X$ with the strongest topology such that the map
\[
\omega_X \colon X \to X_{\max}
\]
\[
x\mapsto x_{\rm{gen}}
\]
is continuous. In other words, we endow $X_{\max}$ with the quotient topology from $X$.

\begin{defn1} The category of {\it hausdorff strictly $K$-analytic Berkovich spaces} is denoted by $(An)$ (we also refer to \cite[\textsection 1.2, p.\,22]{Ber} for the precise definition). 
\end{defn1}

We recall that \cite[Proposition 8.3.1, Remark 8.3.2, Lemma 5.6.8(i)]{H3} and \cite[Proposition 4.5(iv)]{H1} guarantee that the functor
\[
s'\colon (An) \to  (A)'.
\]
constructed in \cite[\textsection 8.3]{H3} is an equivalence of categories. The quasi-inverse of this functor\footnote{We also refer to \cite{Henkel} for a more direct construction of $u$.} is denoted by 
\[
u\colon (A)' \to (An).
\]
We summarize the main properties of this functor below:
\begin{lemma1}\label{lemma:berkovization}(\cite[Remark 8.3.2]{H3} or \cite[Construction 7.1]{Henkel})
\begin{enumerate}\label{berkovization}
    \item\label{berkovization-1} $u$ sends an open immersion $U=\Spa(A, A^\circ) \subset X$ to an affinoid domain $\cal{M}(A) \subset u(X)$.
    \item\label{berkovization-2} $u$ sends a morphism $f\colon \Spa(B, B^\circ) \to \Spa(A, A^\circ)$ to the corresponding $u(f)\colon \cal{M}(B) \to \cal{M}(A)$
    \item\label{berkovization-3} the underlying topological space of $u(X)$ is functorially identified with $X_{\max}$. 
\end{enumerate}
\end{lemma1}

\begin{lemma1}\label{lemma:properties-of-berkovization} Let $f\colon X\to Y$ be a morphism in $(A)'$. Suppose that $f$ is partially proper (resp.\,proper, resp.\,\'etale, resp.\,smooth). Then $u(f)$ is boundaryless (resp.\,proper, resp.\,quasi-\'etale, resp.\,quasi-smooth).
\end{lemma1}
\begin{rmk1} We refer to  \cite[Definition 1.3.20, 5.2.4, and 5.2.6]{ducros} for the definitions of boundaryless, quasi-\'etale, and quasi-smooth morphisms of Berkovich spaces. We also note that boundaryless morphisms are sometimes called {\it closed} morphisms. We prefer to call such morphisms boundaryless because this notion is {\it different} from the notion of topologically closed morphisms. 
\end{rmk1}
\begin{proof}[Proof of Lemma~\ref{lemma:properties-of-berkovization}]
{\it Suppose $f$ is partially proper}. Then it suffices to show that the morphism of germs\footnote{Look at \cite[\textsection{3.4}]{Ber} for the definition of germs of Berkovich spaces.} 
\[
    \left(u(X),x\right) \to \left(u(Y), u(f)(x)\right)
\]
is boundaryless\footnote{By definition, this means that the map is induced by a boundaryless map $X' \to u(Y)$, where $X'$ is an open neighborhood of $x$ in $u(X)$.} for any $x\in u(X)=X_{\max}$. Now \cite[Theorem 4.1]{Temkin-proper} gives that this map is boundaryless if and only if the morphism of reductions 
\[
\widetilde{\left(u(X),x\right)} \to \widetilde{\left(u(Y), f(x)\right)}
\]
is proper, i.e. it induces the bijection\footnote{Look at \cite[p.1]{Temkin-proper} for the definition of a proper morphism of reductions.}
\[
\widetilde{(u(X),x)} \to \bf{P}_{\widetilde{\cal{H}(x)}/\widetilde{K}}  \times_{\bf{P}_{\widetilde{\cal{H}(f(x))}/\widetilde{K}}} \widetilde{(u(Y), f(x))}\footnote{$\bf{P}_{\widetilde{\cal{H}(f(x))}/\widetilde{K}}$ is the set of all valuations of the residue field of $\cal{H}(f(x))$ that are trivial on the residue field of $K$. In \cite{Temkin-proper}, this is denoted just by $\bf{P}_{\widetilde{\cal{H}(f(x))}}$.}.
\]
Now we use \cite[Remark 2.6]{Temkin-proper} to see that the underlying topological space of $\widetilde{(u(X), x)}$ coincides with $\omega_X^{-1}(x) \subset X$ and similarly  for $\widetilde{(u(Y), f(x))}$. So bijectivity of the above map follows from \cite[Corollary 1.3.9]{H3}. \medskip

{\it Suppose $f$ is proper}. For the purpose of showing that $u(f)$ is proper, it suffices to treat the case of affinoid $Y=\Spa(A, A^\circ)$. Then $X$ is a qcqs rigid-analytic $K$-space, so $f$ has a formal model $\mf\colon \X \to \Y$. Now Lemma~\ref{lemma:proper-adic-formal} guarantees that $f$ is proper if and only if so is $\mf$. Then \cite[Corollary 4.4]{Temkin-proper}, and compatibility of adic and Berkovich generic fibers guarantees that $u(f)$ is proper.
\medskip

{\it Suppose $f$ is \'etale}. Then \cite[5.2.10]{ducros} and Lemma~\ref{lemma:berkovization}(\ref{berkovization-1}) imply that the claim is local on $X$ and $Y$. Therefore, \cite[Lemma 2.2.8]{H3} implies that we can assume that $X$ and $Y$ are affinoids, and that $f=g\circ j$, where $j\colon X \to \Spa(B, B^\circ)$ is an open immersion and $g\colon \Spa(B, B^\circ) \to \Spa(A, A^\circ)$ is a finite \'etale morphism. Lemma~\ref{lemma:berkovization}(\ref{berkovization-1}) implies that $u(j)$ is an analytic domain, in particular, it is quasi-\'etale. And $u(g)\colon \cal{M}(B) \to \cal{M}(A)$ is finite \'etale because $A \to B$ is finite \'etale. \medskip

{\it Suppose $f$ is smooth}. We use \cite[5.2.10]{ducros}, Lemma~\ref{lemma:berkovization}(\ref{berkovization-1}), and \cite[Corollary 1.6.10]{H3} to reduce to the case when $f\colon X \to Y$ is a map of affinoids that factors as
\[
X \xr{g} \bf{D}^N_Y  \to Y,
\]
with \'etale $g$. Then we see that $u(f)$ can be written as the composition
\[
u(X) \xr{u(g)} \bf{D}^N_{u(Y)} \to u(Y)
\]
with $u(g)$ being quasi-\'etale. Therefore, $u(f)$ is quasi-smooth since the $N$-dimensional disc $\bf{D}^N_{u(Y)}$ is quasi-smooth over $u(Y)$. 
\end{proof}

\begin{rmk1} With more work, one can upgrade the results of Lemma~\ref{lemma:properties-of-berkovization} to the ``if and only if'' results. We do not discuss it here since we will never need the ``only if'' direction. 
\end{rmk1}

\begin{cor1}\label{cor:partially-proper-etale} Let $f\colon X\to Y$ be a morphism in $(A)'$. Then
\begin{enumerate}
    \item $f$ is partially proper and \'etale if and only if $u(f)$ is \'etale;
    \item $u(f)$ is smooth if $f$ is partially proper and smooth.
\end{enumerate}
\end{cor1}
\begin{proof}
The first part is proven in the proof of \cite[Part (a), p.\,427]{H3}. For the second one, we recall that  \cite[Corollary 5.4.8]{ducros} (and \cite[Remark 9.7]{ber-smooth} to make it work in a non-necessary good strictly $K$-analytic case) implies that a morphism of strictly $K$-analytic spaces is smooth (resp. \'etale) if and only if it is quasi-smooth (resp. quasi-\'etale) and boundaryless. Therefore, the claim follows from Lemma~\ref{lemma:properties-of-berkovization}.
\end{proof}

\begin{rmk1}\label{rmk:etale-berkovich-adic} Corollary~\ref{cor:partially-proper-etale} implies that $s'$ sends \'etale morphisms to partially proper \'etale morphisms. 
\end{rmk1}

\begin{defn1} The {\it strict \'etale site} $X_{\et .s}$ of a strictly $K$-analytic space $X$ is defined as follows: 
\begin{enumerate}
    \item the underlying category $(\Ett/X)_s$ of $X_{\et.s}$ is the full subcategory of $(An)$ consisting of all $X$-objects $Y$ such that the structure morphism $f\colon Y \to X$ is an \'etale morphism of $K$-analytic spaces,
    \item A family $(Y_i \xr{f_i} Y)$ is a covering of $Y$ when $Y=\cup_{i\in I}f_i(Y_i)$. 
\end{enumerate}
\end{defn1}

Similarly, for a rigid-analytic space $X$, we denote by $\Ett/X$ the category of \'etale $X$-spaces $Y \to X$ and by $X_\et$ the (small) \'etale site of $X$.

\begin{rmk1}\label{rmk:Berkovich-different-site} We note that Berkovich considers a different \'etale site $X_\et$ in \cite{Ber}. His site $X_\et$ consists of all \'etale $X$-spaces $Y$ that are not necessary $K$-strict. However, \cite[p.\,426]{H3} shows that the natural morphism of sites
\[
X_\et \to X_{\et.s}
\]
induces an isomorphism of corresponding topoi. Therefore, all results of \cite{Ber} still hold true if one uses $X_{\et.s}$ instead of $X_\et$. 
\end{rmk1}

Now we fix a taut rigid-analytic $K$-space $X$, and consider the natural morphism of sites
\[
\theta_X\colon X_\et \to u(X)_{\et.s}
\]
induced by the functor
\[
\left(\Ett/u(X)\right)_s \to \Ett/X
\]
\[
\big(Y \to u(X)\big) \mapsto \left(s'(Y) \to s'(u(X))=X\right).
\]
This functor is well-defined due to Remark~\ref{rmk:etale-berkovich-adic}. \smallskip

If $f\colon X \to Y$ is a morphism in $(A)'$, then $\theta_X$ and $\theta_Y$ fit into the following commutative diagram:
\begin{equation}\label{diagram:etale-strict}
\begin{tikzcd}
X_{\et} \arrow{d}{f} \arrow{r}{\theta_X} & u(X)_{\et.s} \arrow{d}{u(f)}\\
Y_\et \arrow{r}{\theta_Y} & u(Y)_{\et.s}.
\end{tikzcd}
\end{equation}

\begin{thm1}\label{thm:berkozation-compact-cohomology} Let $f\colon X \to Y$ be a partially proper morphism in $(A)'$. Then there is a natural isomorphism 
\[
\alpha_{f}(K)\colon \bf{R}f_!\,\theta_X^*\F \xr{\sim} \theta_Y^*\,\bf{R}u(f)_!\F
\]
for any $\F\in \bf{D}^+(u(X)_{\et.s}, \Z)$. 
\end{thm1}
\begin{rmk1} We recall that $\bf{R}f_!$ is defined in \cite[\textsection 5.3]{H3} and $\bf{R}u(f)_!$ is defined in \cite[\textsection 5.1]{Ber}.
\end{rmk1}
\begin{proof}[Proof of Theorem~\ref{thm:berkozation-compact-cohomology}]
We note that Lemma~\ref{lemma:properties-of-berkovization} ensures that $u(f)$ is boundaryless since $f$ is partially proper. Thus \cite[Proposition 8.3.6]{H3} proves the claim (\cite[Proposition 8.3.6]{H3} is written in terms of Tate rigid spaces, but the actual proof shows the statement on the level of corresponding adic spaces). 
\end{proof}

For the next definition, we fix an analytic adic space $X$.

\begin{defn1}(\cite[Definition 8.2.1]{H3}) An \'etale sheaf of abelian groups on $X$ is {\it overconvergent} if, for every specialization $u\colon \ov{\eta}_1 \to \ov{\eta}_2$ of geometric points of $X$, the map $u^*(\F) \colon \F_{\ov{\eta}_2} \to \F_{\ov{\eta}_1}$ is an isomorphism. 

We denote the {\it category of overconvergent sheaves} of abelian groups by $\cal{A}b_{\rm{ov}}(X_\et)$.
\end{defn1}

\begin{lemma1}\label{lemma:overconv-berkovich}\cite[Theorem 8.3.5]{H3} Let $X$ be a taut rigid $K$-space. Then the functor 
\[
\theta_X^*\colon \cal{A}b(u(X)_{\et.s}) \to \cal{A}b(X_{\et})
\]
induces an equivalence of categories 
\[
\theta_X^*\colon \cal{A}b(u(X)_{\et.s}) \to \cal{A}b_{\rm{ov}}(X_{\et}).
\]
\end{lemma1}

Now let $f\colon X \to Y$ be a partially proper \'etale morphism in $(A)'$. Then \cite[Lemma 2.7.6]{H3} implies that the functor $f_!$ is left adjoint to $f^*$. In particular, there is a trace map 
\[
    t_f\colon f_!f^*\F \to \F
\]
for any $\F \in \cal{A}b(Y_\et)$. \smallskip

Similarly, Corollary~\ref{cor:partially-proper-etale} guarantees that $u(f)\colon u(X) \to u(Y)$ is an \'etale morphism. Therefore, the functor $u(f)_!$ is a left adjoint to $u(f)^*$ due to \cite[Remark 5.4.2(ii)]{Ber}. So, for any $\G\in \cal{A}b(u(Y)_{\et.s})$, there is a trace map 
\[
    t_{u(f)}^B\colon u(f)_!u(f)^*\G \to \G. 
\]

\begin{lemma1}\label{lemma:trace-berkovich-adic} Let $f\colon X \to Y$ be a partially proper \'etale morphism of taut rigid $K$-spaces. Then, for any $\G\in \cal{A}b(u(Y)_\et)$, the following diagram is commutative
\[
\begin{tikzcd}[column sep = 5em, row sep = 4em]
\theta_Y^* u(f)_!u(f)^*\G  \arrow{r}{\theta_Y^*t^B_{u(f)}} & \theta_Y^* \G \\
f_!\,\theta_X^*u(f)^*\G \arrow{u}{\alpha_f(u(f)^*\G)} \arrow{r}{\sim} & f_!f^*\theta_Y^*\G \arrow{u}{t_f}.
\end{tikzcd}
\]
\end{lemma1}
\begin{proof}
First, both functors $f^*$ and $f_!$ preserve overconvergent sheaves by \cite[Proposition 8.2.3 and Corollary 8.2.4]{H3}. Thus the result follows formally from Lemma~\ref{lemma:overconv-berkovich}, Theorem~\ref{thm:berkozation-compact-cohomology}, and the adjunctions.
\end{proof}

%\section*{Appendices}
\section{Some Facts from Rigid Geometry}\label{rigid-appendix}

Throughout this appendix, we fix a complete rank-$1$ valuation ring $\O_K$ with fraction field $K$, maximal ideal $\m\subset \O_K$, residue field $k$, and a pseudo-uniformizer $\varpi\in \m$. For an $\O_K$-algebra $A$, we denote by $A_K$ (resp. $A_k$) the tensor product $A\otimes_{\O_K} K$ (resp. $A\otimes_{\O_K} k$). \smallskip

We note that for rigid-analytic $K$-spaces there are two possible definitions of proper and separated morphism: one is due to R.\,Huber in terms of adic spaces (see \cite[\textsection 1.3]{H3}) and one is due to R.\,Kiehl in terms of classical Tate rigid-analytic spaces (see \cite[\textsection 6.3]{B}). To avoid any confusion, we explicitly spell out that these definitions are equivalent and also equivalent to the definition via formal models: 

\begin{lemma1}\label{lemma:proper-adic-formal}(L\"{u}tkebohmert--Temkin) Let $\mf\colon \X \to \Y$ be a morphism of admissible formal $\O_K$-schemes. Then the following are equivalent:
\begin{enumerate}
    \item $\mf$ is separated (resp.~proper);
    \item the generic fiber $\mf_K\colon \X_K \to \Y_K$ is separated (resp.~proper) in the sense of \cite[\textsection 1.3]{H3};
    \item the generic fiber $\mf_K\colon \X_K \to \Y_K$ is separated (resp.~proper) in the sense of \cite[\textsection 6.3]{B}.
\end{enumerate}
\end{lemma1}
\begin{proof}
    First, \cite[Remark 1.3.18]{H3} (or \cite[Proposition 3.11.20 and 3.12.5]{H2}\footnote{Strictly speaking, \cite[Proposition 3.11.20 and 3.12.5]{H2} apply only when $\O_K$ is noetherian. However, the proofs in \textit{loc.cit.}~do not actually require this noetherian hypothesis and can be adapted to the case of admissible formal $\O_C$-schemes.})  implies (1) and (2) are equivalent. Now  \cite[Corollary 4.5]{Temkin-proper} (or \cite{Lutke-proper} if $\O_K$ is discretely valued) implies that $\mf$ is proper if and only if $\mf_K$ is Kiehl-proper. A similar argument shows that $\mf$ is separated if and only if $\mf_K$ is Kiehl-separated. This shows that (1) is equivalent to (3) (alternatively, one can use \cite[Remark 1.3.19(ii)]{H3} to say that (2) is equivalent to (3) for separated morphisms). 
\end{proof}

%\begin{lemma}\label{lemma:flat-free} Let $A$ be a flat topologically finitely presented $\O_K$-algebra. Then $A$ is topologically free, i.e., $A\simeq \wdh{\bigoplus}_I \O_K$ for some set $I$.
%\end{lemma}
%\begin{proof}
%    {\it Step~$1$. $A/\varpi$ is free over $\O_K/\varpi$.} We write $\O_K/\varpi = \colim_J R_j$ as a filtered colimit of its local noetherian subalgebras. Then there is an element $j\in J$ and a flat finitely presented $R_j$-algebra $A_j$ such that $A_j\otimes_{R_j} \O_K/\varpi \simeq A/\varpi$. Thus, it suffices to show that $A_j$ is a free $R_j$-algebra. This follows from \cite[\href{https://stacks.math.columbia.edu/tag/051G}{Tag 051G}]{stacks-project} since $R_j$ is an artinian local ring. \smallskip
    
  %  {\it Step~$2$. $A$ is topologically free over $\O_K$.} Choose some basis elements $\{\ov{e}_j\}_{j\in J} \in A/\varpi$ and lift them somehow to elements $\{e_{j}\}_{j\in J}\in A$. This defined a morphism
  %  \[
  %  \varphi\colon \wdh{\bigoplus}_J \O_K \to A.
 %   \]
 %   Let us denote the cone of this morphism by $C$. \cite[\href{https://stacks.math.columbia.edu/tag/091T}{Tag 091T}]{stacks-project} and \cite[\href{https://stacks.math.columbia.edu/tag/091U}{Tag 091U}]{stacks-project} imply that $C$ is derived complete. Furthermore, $\O_K$-flatness of $\wdh{\bigoplus}_K \O_K$ and $A$ implies that $C\otimes^L_{\O_K} \O_K/\varpi =0$. Therefore, \cite[\href{https://stacks.math.columbia.edu/tag/0G1U}{Tag 0G1U}]{stacks-project} implies that $C\simeq 0$ finishing the proof. 
%\end{proof}

Now we discuss a slightly refined version of Elkik's algebraization result that does not change the (relative) dimension of the algebraized algebra. For this, we need a number of preliminary results. 

\begin{lemma1}\label{lemma:pure-dimension-rigid}(\cite[Lemma 1.8.6]{H3}) Let $X$ be a rigid-analytic $K$-space. Then $X$ is of pure dimension $d$ (see \cite[Definition 1.8.1]{H3}) if and only if $\dim A=d$ for any open affinoid subset $\Spa(A, A^\circ) \subset X$.
\end{lemma1}

\begin{lemma1}\label{lemma:dimension-generic-special} Let $A$ be a flat, topologically finite type $\O_K$-algebra. Then $\dim A_K =\dim A_k$.
\end{lemma1}
\begin{proof}
    Noether Normalization Theorem (see \cite[Theorem 9.2.10]{FujKato}) implies that there is an integer $d$ and a finite {\it injective} morphism
    \[
    \O_K\langle T_1, \dots, T_{d}\rangle \to A
    \]
    such that the induced morphism $k[T_1, \dots, T_{d}] \to A_k$ is (finite and) injective. Therefore, \cite[Proposition 2.2/17]{B} and \cite[Theorem 20]{M} imply that
    \[
    \dim A_K=\dim K\langle T_1,\dots, T_{d}\rangle = d = \dim k[T_1, \dots, T_{d}] = \dim A_k. \qedhere
    \]
\end{proof}

\begin{cor1}\label{cor:pure-dimension-rigid-formal} Let $\X$ be an admissible formal $\O_K$-scheme with generic fiber $X=\X_K$ of pure dimension $d$. Then the special fiber $\ov{\X}$ is of pure dimension $d$ as well.
\end{cor1}
\begin{proof}
    It suffices to show that, for every open affine $\Spf A\subset \X$, $\dim A_k=d$. Then Lemma~\ref{lemma:pure-dimension-rigid} and Lemma~\ref{lemma:dimension-generic-special} imply that
    \[
    \dim A_k = \dim A_K =d.\qedhere
    \]
\end{proof}

\begin{lemma1}\label{lemma:pure-algebraization}(Elkik's algebraization) Let $\X=\Spf B$ be an admissible affine formal $\O_K$-scheme. Suppose the generic fiber $\X_K$ is smooth of pure dimension $d$. Then there is a flat, finitely presented $\O_K$-algebra $A$ such that
\begin{enumerate}
    \item there is an isomorphism $\wdh{A} \simeq B$, where $\wdh{A}$ is the $\varpi$-adic completion of $A$;
    \item $A_K$ is $K$-smooth;
    \item $A$ is of pure relative dimension $d$.
\end{enumerate}
\end{lemma1}
\begin{proof}
    We first note that \cite[Theorem 3.1.3]{T3} (that essentially boils down to \cite[Th\'eorem\`e 7 on page 582 and Remarque 2(c) on p.588]{Elkik} and \cite[Proposition 3.3.2]{T0}) implies that there is a flat, finitely presented $\O_K$-algebra $A'$ such that $\wdh{A}'\simeq B$ and $A'_K$ is $K$-smooth. The only remaining question is to make $A'$ to be of pure relative dimension $d$. \smallskip
    
    Since $X'=\Spec A'$ and $\X=\Spf B$ have isomorphic special fibers, Corollary~\ref{cor:pure-dimension-rigid-formal} implies that $X'$ has special fiber of pure dimension $d$. Since the underlying topological space of $|\Spec A'|$ is noetherian, it has finitely many irreducible components $X'_1, \dots, X'_n$. Furthermore, $\O_K$-flatness of $\Spec A'$ implies that each irreducible component $X'_i$ is $\O_K$-flat. So it either lies inside the generic fiber $\Spec A'_K$ or surjects onto $\Spec \O_K$. \smallskip
    
    Suppose that the components $X'_1, \dots, X'_s$ lie in the generic fiber and $X'_{s+1}, \dots, X'_n$ surject onto $\Spec \O_K$. Since irreducible components are closed, we conclude that the union $X'_1\cup \dots \cup X'_s$ is closed in $\Spec A'$. Furthermore, this union is easily seen to be open using that $\Spec A'_K$ is smooth. Therefore, we can replace $\Spec A'$ with 
    \[
    \Spec A = \Spec A' \setminus \cup_{i=1}^s X'_i
    \]
    to assume that every irreducible component of $\Spec A$ surjects onto $\Spec \O_K$ (note that $\wdh{A}\simeq \wdh{A}'\simeq B$). In this case, we claim that $\Spec A \to \Spec \O_K$ is a morphism of pure relative dimension $d$. \smallskip
    
    To prove this, we write $\Spec A$ as a union of $\O_K$-flat irreducible components $\Spec A = \cup_{i=1}^m X_i$. It suffices to show that each $X_{i, K}$ is a $K$-scheme of pure dimension $d$. This follows from \cite[Lemme 14.3.10]{EGA4_3}.
\end{proof}

Now we discuss a version of the Reduced Fiber Theorem that will be crucial for some proofs and constructions of this paper. We start with some preliminary results: 

\begin{lemma1}\label{lemma:enough-constants} Let $A$ be a flat, topologically finite type $\O_K$-algebra, and let $f\in A_K$ be a non-zero element. Then there is an element $c\in K^\times$ such that $cf \in A\subset A_K$ and the residue class $\ov{cf}\in A_k$ is non-zero.
\end{lemma1}
\begin{proof}
    We choose a surjection $\alpha\colon \O_K\langle T_1, \dots, T_n\rangle \to A$ and consider the associated residue norm $|.|_\alpha$ on $A_K$ (see \cite[Proposition 3.1/5]{B}). Since $f\neq 0$, we see that $|f|_{\alpha}\neq 0$, so there is an element $c\in K^\times$ such that $|1/c| = |f|_{\alpha}$. So $cf\in A_K$ is an element with the residue norm $1$. Therefore, $cf\in A$ and its residue class $\ov{cf} \in A_k$ is non-zero. 
\end{proof}

\begin{lemma1}\label{lemma:special-fiber-reduced} Let $A$ be a flat, topologically finite type $\O_K$-algebra such that its special fiber $A_k$ is reduced. Then $A$ is reduced and $A=A_K^\circ$. 
\end{lemma1}
\begin{proof}
    Let $f\in A$ be a non-zero nilpotent element. Then Lemma~\ref{lemma:enough-constants} ensures that there exists an element $c\in K^\times$ such that $g\coloneqq cf\in A \subset A_K$ and the residue class $\ov{g}\in A_k$ is non-zero. This implies that $\ov{g}$ is a non-zero nilpotent element of $A_k$ contradicting the assumption that $A_k$ is reduced.\smallskip
    
    The equality $A=A_K^\circ$ follows from \cite[Proposition 3.4.1]{lutke-jacobian}. 
\end{proof}

\begin{cor1}\label{cor:normalization} Let $A$ be a flat, topologically finite type $\O_K$-algebra. Then 
\begin{enumerate}
    \item $A$ is integrally closed in $A_K$ if $A_k$ is reduced;
    \item $A_k$ is reduced if $A$ is integrally closed in $A_K$ and $K$ is algebraically closed.
\end{enumerate}
\end{cor1}
\begin{proof}
    The first part follows from Lemma~\ref{lemma:special-fiber-reduced} and the observation that $A_K^\circ$ is integrally closed in $A_K$ (see \cite[Theorem 3.1/17(ii)]{B}). \smallskip
    
    To show the second part, we note that the assumption that $A$ is integrally closed in $A_K$ and \cite[Theorem 3.1/17(ii)]{B} imply that $A=A_K^\circ$. Furthermore, $A_K^{\circ\circ}=\m A_K^\circ$ since $K$ is algebraically closed. Therefore, the result follows from the observation that $A_k = A_K^\circ/\m A_K^\circ = A_K^\circ/A_K^{\circ\circ}$ is always reduced since $A^{\circ\circ}_K$ is a radical ideal in $A_K^\circ$. 
\end{proof}

\begin{thm1}\label{thm:reduced-fiber-theorem}(Reduced Fiber Theorem) Let $K$ be an algebraically closed complete rank-$1$ valued field, $\X$ an admissible formal $\O_K$-scheme with reduced generic fiber $X=\X_K$. Then there is a unique finite rig-isomorphism $\pi\colon \widetilde{\X} \to \X$ such that the special fiber of $\widetilde{\X}$ is reduced. Furthermore, it satisfies the following universal property: for any admissible formal $\O_K$-scheme $\X'$ with reduced special fiber and a morphism $\mf\colon \X'\to \X$, there is a unique morphism $\widetilde{\mf} \colon \X' \to \widetilde{\X}$ such that the diagram
\[
\begin{tikzcd}
    \X' \arrow{d}{\mf} \arrow{r}{\widetilde{\mf}} & \widetilde{\X} \arrow{dl}{\pi} \\
    \X &
\end{tikzcd}
\]
commutes.
\end{thm1}

Before we embark on the proof, we want to mention that existence of $\pi$ is an easier special case of much more general Reduced Fiber Theorem (see \cite[Theorem 2.1]{BLR4}).
\begin{proof}
    We first verify the universal property of a finite rig-isomorphism $\pi\colon \widetilde{\X} \to \X$ with $\widetilde{\X}$ having reduced special fiber (in particular, it implies its uniqueness). First, we note that, for any open formal subscheme $\sU\subset \X$, the morphism $\pi|_{\pi^{-1}(\sU)}\colon \pi^{-1}(\sU) \to \sU$ is a finite rig-isomorphism with $\pi^{-1}(\sU)$ having reduced special fiber. Therefore, it suffices to verify the universal property under the additional assumption that $\X=\Spf A$ is affine. Furthermore, we can assume that $\X'=\Spf B$ is affine. \smallskip

    In this situation, Corollary~\ref{cor:normalization} implies that $\widetilde{\X}=\Spf A_K^\circ$. Therefore, the question boils down to showing that the natural morphism $A \to B$ uniquely extends to a morphism $A_K^\circ \to B$. Uniqueness is clear since $A$ and $B$ are $\O_K$-flat and $A_K^\circ \subset A_K$, while existence follows from the equality $B=B_K^\circ$ coming from Lemma~\ref{lemma:special-fiber-reduced} and the assumption that $\Spf B$ has reduced special fiber. \smallskip

    Now we show existence of $\pi\colon \widetilde{\X} \to \X$. The universal property of $\pi$ and a standard gluing argument reduce the question to the case of an affine $\X=\Spf A$. In this case, we claim that $\widetilde{\X}=\Spf A_K^\circ \to \X=\Spf A$ does the job. Indeed,  \cite[Theorem 1.3]{BLR4} implies that $\Spf A_K^\circ$ is an admissible formal $\O_K$-scheme with reduced special fiber, while \cite[Theorem 3.1/17]{B} implies that $\Spf A_K^\circ \to \Spf A$ is integral and thus is finite.
\end{proof}

For the next definition, we fix an algebraically closed rank-$1$ valued field $K$ and an admissible formal $\O_K$-scheme $\X$ with reduced generic fiber $X=\X_K$.

\begin{defn1}\label{defn:normalization} The {\it normalization of $\X$ in $\X_K$} is the finite rig-isomorphism $\pi\colon \widetilde{\X} \to \X$ constructed in Theorem~\ref{thm:reduced-fiber-theorem}. 
\end{defn1}

\begin{rmk1} The universal property established in Theorem~\ref{thm:reduced-fiber-theorem} implies that $\widetilde{\X}$ is functorial in $\X$.
\end{rmk1}

We end this appendix with a number of results concerned with the existence of some ``good'' open subsets of admissible formal schemes. 

\begin{lemma1}\label{lemma:connected} Let $\X$ be an admissible formal $\O_K$-scheme with generic fiber $\X_K=X$ and reduced special fiber $\ov{\X}$. Suppose that $T\subset X$ is a connected component of $X$. Then there is a connected component $\mathfrak T \subset \X$ such that $\mathfrak T_K =T$.
\end{lemma1}
\begin{proof}
It suffices to show every idempotent of $\Gamma(X, \O_X)$ comes from an idempotent of $\Gamma(\X, \O_\X)$. Since $\O_X$ is a sheaf on $X$ and $\X$ is $\O_K$-flat, it is easy to see that it suffices to treat the case of an affine admissible formal $\O_K$-scheme $\X=\Spf A$. In this case, the result follows from Corollary~\ref{cor:normalization}.
\end{proof}

\begin{lemma1}\label{lemma:good-open} Let $\mf\colon \X \to \Y$ be a rig-finite morphism of admissible formal $\O_K$-schemes. Suppose that the special fiber $\ov{\Y}$ of $\Y$ is geometrically reduced. Then there is an open dense formal subscheme $\mathfrak U \subset \Y$ such that $\sU$ lies in the smooth locus $\Y^{\rm{sm}}$ and the restriction $\mf|_{\X_{\mathfrak U}} \colon \X_{\mathfrak U} \to \mathfrak U$ is finite and flat.
\end{lemma1}
\begin{proof}

Since $\Y$ is $\O_K$-flat and its special fiber is geometrically reduced, we conclude that $\Y^{\rm{sm}}$ is a dense open of $\Y$, so we can replace $\Y$ with $\Y^{\rm{sm}}$ to assume that $\Y$ is smooth. Then we can pass to its connected (equivalently, irreducible) components to assume that $\Y$ is smooth and irreducible. \smallskip

Now \cite[\href{https://stacks.math.columbia.edu/tag/052B}{Tag 052B}]{stacks-project} implies that there is a dense open $U\subset \ov{\Y}$ such that $\ov{\mf}|_{\ov{\mf}^{-1}(U)}\colon \ov{\mf}^{-1}(U) \to U$ is flat. The dense open subscheme $U\subset \ov{\Y}$ defines an open dense formal subscheme $\sU \subset \Y$. The fiber-by-fiber flatness criterion implies that the morphism $\mf_n\colon \X_n \to \Y_n$ is flat over $\sU_n$ for any $n\geq 1$, and then \cite[Proposition 7.3/11]{B} guarantees that $\mf\colon \X \to \Y$ is flat over $\sU \subset \Y$. So we can replace $\Y$ with $\sU$ to assume that $\mf$ is a flat morphism and $\Y$ is smooth.  \smallskip

Now the question boils down to finding a dense open formal subscheme $\sU\subset \Y$ such that $\mf|_{\X_\sU}\colon \X_\sU \to \sU$ is finite. We note that \cite[Proposition 4.2.3]{FujKato} implies that it suffices to find a dense open $U\subset \ov{\Y}$ such that $\ov{\mf}\colon \ov{\X}_U \to U$ is finite. \smallskip

Since $\ov{\Y}$ is $k$-smooth and connected, it is of pure dimension $d$ for some integer $d$. Using  Lemma~\ref{lemma:dimension-generic-special} and finiteness of $\mf_K$, we conclude that 
\[
\dim \ov{\X} = \dim \X_K \leq \dim \Y_K = \dim \ov{\Y} = d. 
\]
This implies that $\ov{\mf}$ is quasi-finite over the generic point of $\ov{\Y}$. Since $\ov{\mf}$ is proper (see Lemma~\ref{lemma:proper-adic-formal}), \cite[Corollaire 13.1.5]{EGA4_3} implies that there is an open dense subset $U \subset \ov{\Y}$ such that $\ov{\mf}$ is quasi-finite over $U$. As $\ov{\mf}$ is proper, it is automatically finite over $\ov{\sU}$ by \cite[Théorème 8.11.1]{EGA4_3}. This finishes the proof.
\end{proof}

\begin{cor1}\label{cor:super-good-open} Let $\mf\colon \X \to \Y$ be a rig-finite morphism of admissible formal $\O_K$-schemes with geometrically reduced special fibers. Then there is an open dense formal subscheme $\mathfrak U \subset \Y$ such that $\sU$ lies in the smooth locus $\Y^{\rm{sm}}$, the restriction $\mf|_{\X_{\mathfrak U}} \colon \X_{\mathfrak U} \to \mathfrak U$ is finite and flat, and $\X_\sU$ is $\O_K$-smooth.
\end{cor1}
\begin{proof}
    Lemma~\ref{lemma:good-open} implies that we can assume that $\Y=\Y^{\rm{sm}}$ and $\mf\colon \X \to \Y$ is finite and flat. Furthermore, we can assume that $\Y$ is irreducible by passing to connected (equivalently, irreducible) components of $\Y$.
    
    Now we note that the smooth locus $\X^{\rm{sm}}\subset \X$ contains all generic points of the special fiber $|\ov{\X}|=|\X|$. Using that $\mf$ is finite and flat, we conclude that $\sU\coloneqq \Y - \mf(\X - \X^{\sm})$ is a dense open subset of $\Y$. By construction, $\mf^{-1}(\sU)\subset \X^{\sm}$, so this open subset does the job.
\end{proof}

\begin{cor1}\label{cor:iso-on-open} Let $\mf\colon \X \to \Y$ be a rig-isomorphism of admissible formal $\O_K$-schemes. Suppose that $\Y$ has geometrically reduced special fiber. Then there is an open dense formal subscheme $\mathfrak U \subset \Y$ such that $\sU$ lies in the smooth locus $\Y^{\rm{sm}}$ and the restriction $\mf|_{\X_{\mathfrak U}} \colon \X_{\mathfrak U} \to \mathfrak U$ is an isomorphism.
\end{cor1}
\begin{proof}
We use Lemma~\ref{lemma:good-open} to find an open dense subset $\mathfrak U \subset \Y^{\rm{sm}}$ such that $\mf|_{\X_{\mathfrak U}} \colon \X_{\mathfrak U} \to \mathfrak U$ is finite. We claim that it is an isomorphism, it is sufficient to check this locally. So we can assume that $\Y=\Spf A$ is affine and the morphism is given by a finite morphism $\Spf B \to \Spf A$ such that $A_K=B_K$. Then $A\to B$ is an isomorphism due to Corollary~\ref{cor:normalization} and finiteness of $A \to B$.
\end{proof}

\section{Generic Fibers of $\O_\X$-modules}\label{generic-appendix}

We collect some results about generic fibers of certain $\O_\X$-modules on an admissible, formal model $\X$. Even though all the results seem to be standard, it is hard to find a precise reference. \smallskip

Throughout this section, we fix a complete rank-$1$ {\it perfectoid} valuation ring $\O_K$ with fraction field $K$, maximal ideal $\m\subset \O_K$, and a pseudo-uniformizer $\varpi$. We also do almost mathematics with respect to $\m\subset \O_K$. \smallskip

We also fix an admissible formal $\O_K$-scheme $\X$  with generic fiber $X=\X_K$. Then we have a morphism of ringed sites 
\[
\sp_\X \colon (X_{\rm{an}}, \O_X) \to (\X_{\rm{Zar}}, \O_\X) 
\] 
that is flat by \cite[Corollary II.5.1.4]{FujKato}. \smallskip

\begin{defn1} The {\it generic fiber} of $\F\in \bf{D}(\X)$ is the complex
$
\F_K \coloneqq \bf{L}\sp_\X^*\left(\F\right).
$
\end{defn1}

\begin{rmk1} Flatness of $\sp_\X$ guarantees that $\F_K \simeq \sp_\X^*\F$ for $\F\in \bf{Mod}_\X$.
\end{rmk1}

\begin{defn1} An $\O_\X$-module $\F$ is $\varpi^\infty$-torsion if the natural morphism
\[
\colim_{n} \F[\varpi^n] \to \F
\]
is an isomorphism. 
\end{defn1}

\begin{lemma1}\label{lemma:kill-torsion} Let $\F$ be a $\varpi^\infty$-torsion $\O_\X$-module. Then $\F_K$ is the zero sheaf. More generally, if $\F\in \bf{D}(\X)$ with $\varpi^\infty$-torsion cohomology sheaves, then $\F_K$ is the zero complex.
\end{lemma1}
\begin{proof}
It suffices to show that $\F_K \simeq 0$ for a $\varpi^\infty$-torsion $\O_\X$-module $\F$. Since $\sp_\X^*$ commutes with colimits, we see that the natural morphism
\[
\colim_n (\F[\varpi^n])_K \to \F_K
\]
 is an isomorphism. Now, clearly $\F[\varpi^n]_K \simeq 0$. And so $\F_K\simeq 0$.
\end{proof}

\begin{cor1} The functor $(-)_K \colon \bf{D}(\X) \to \bf{D}(X)$ canonically descends to the functor 
\[
(-)_K \colon \bf{D}(\X)^a \to \bf{D}(X).
\]
\end{cor1}
\begin{proof}
We recall that \cite[Definition 3.4.2 and Theorem 3.4.9]{Z3} guarantee that $\bf{D}(\X)^a$ is the Verdier quotient of $\bf{D}(\X)$ by the subcategory of $\bf{D}_{\Sigma_\X}(\X)$ of complexes with almost zero cohomology sheaves. Therefore, the results follow from the universal properties of Verdier quotients, Lemma~\ref{lemma:kill-torsion}, and the observation that any almost zero $\O_\X$-module is $\varpi^\infty$-torsion. 
\end{proof}

\begin{lemma1}\label{lemma:generic-fiber-almost-coherent} The functor $(-)_K\colon \bf{D}(\X)^a \to \bf{D}(X)$ induces a functor $(-)_K\colon \bf{D}_{acoh}(\X)^a \to \bf{D}_{coh}(X)$. More precisely, if $\X=\Spf A$ is affine and $\F \simeq M^\Updelta$ for some $M\in \bf{Mod}_A^{acoh}$, then the natural morphism $\widetilde{M[1/\varpi]} \to \F_K$ is an isomorphism.
\end{lemma1}
\begin{proof}
The claim is local on $\X$, so we may assume that $\X=\Spf A$ is affine. Then using flatness of $\sp_\X$ and Lemma~\ref{lemma:kill-torsion}, we easily reduce to the case of an adically quasi-coherent, almost coherent $\O_\X$-module $\F$. Now \cite[Lemma 4.6.1]{Z3} guarantees that $\F\simeq M^\Updelta$ for some almost coherent $A$-module $M$. Clearly, $M[1/\varpi]$ is a finite $A[1/\varpi]$-module. So it suffices to show that $\F_K \simeq \widetilde{M[1/\varpi]}$. \smallskip

Choose a finitely presented $A$-module $N$ and map $N \to M$ with kernel and cokernel annihilated by $\varpi$. Then \cite[Lemma 4.5.14]{Z3} guarantees that $N^\Updelta \to M^\Updelta$ has kernel and cokernel annihilated by $\varpi$. Therefore, Lemma~\ref{lemma:kill-torsion} guarantees that 
\[
(N^\Updelta)_K \simeq (M^\Updelta)_K \simeq \F_K. 
\]
Thus the isomorphism $\widetilde{M[1/\varpi]} \simeq \widetilde{N[1/\varpi]} \simeq (N^\Updelta)_K$ finishes the proof.  
\end{proof}

\begin{lemma1}\label{coh-pullback} Let $\F \in \bf{D}^+_{coh}(X)$, then the natural morphism $\F \to (\bf{R}\sp_{\X, *} \F)_K$ is an isomorphism.
\end{lemma1}
\begin{proof}
First of all, we use flatness of $\sp_\X$ and the convergent spectral sequence
\[
\rm{E}^{p,q}_2=\rm{R}^p\sp_{\X, *}(\mathcal H^q(\F)) \Rightarrow \rm{R}^{p+q}\sp_{\X, *}(\F)
\] 
to reduce the question to the case of a coherent sheaf $\F$ concentrated in degree $0$. Then we use the vanishing of higher cohomology groups of coherent sheaves on affinoid spaces \cite[Proposition 6.5.1]{FujKato} to conclude that $\sp_{\X, *} \F=\bf{R}\sp_{\X, *} \F$. Thus, we only need to show that the natural map
\[
\F \to (\sp_{\X, *} \F)_K
\]
is an isomorphism. \smallskip

The claim is local on $\X$, so we can assume that $\X=\Spf A$ is affine. Since $\F$ is a coherent sheaf on an affinoid space, it admits a presentation
\[
\O_X^n \to \O_X^m \to \F \to 0
\]
Note that the kernels of the maps $\O_X^m \to \F$ and $\O_X^n \to \O_X^m$ are coherent, so the higher derived functors $\rm{R}^n \sp_{\X, *}$ vanish on these sheaves. Thus the sequence
\[
\sp_{\X, *}\left(\O_X^n\right) \to \sp_{\X, *}\left(\O_X^m\right) \to \sp_{\X, *}\left(\F\right) \to 0
\]
is exact. Flatness of $\sp_\X$ implies that the sequence
\[
\left(\sp_{\X, *}(\O_X^n)\right)_K \to \left(\sp_{\X, *}(\O_X^m)\right)_K \to \left(\sp_{\X, *}(\F)\right)_K \to 0
\]
is still exact. Now we use the fact that both $\sp_{\X, *}$ and $\sp_{\X}^*$ commute with finite direct sums to conclude that it suffices to show that $\O_X \to \left(\sp_{\X, *} \O_X\right)_K$ is an isomorphism. \smallskip

We note that the map $\O_\X \to \sp_{\X, *}\O_X$ is injective with the $\varpi^\infty$-torsion cokernel $\mathcal Q$. Therefore, Lemma~\ref{lemma:kill-torsion} and flatness of $\sp_\X$ imply that 
\[
\O_X \simeq (\O_\X)_K \xr{\sim} \left(\sp_{\X, *} \O_X\right)_K. \qedhere
\]
\end{proof}

Now we consider a morphism $\pi\colon \X' \to \X$ of admissible formal $\O_K$-schemes with generic fiber $\pi_K\colon X' \to X$. Then we have a commutative diagram
\[
\begin{tikzcd}
X'\arrow{r}{\sp_{\X'}} \arrow{d}{\pi_K} & \X' \arrow{d}{\pi} \\
X \arrow{r}{\sp_\X} & \X
\end{tikzcd}
\]
that defines a natural base change map 
\[
\bf{L}\sp^*_{\X} \bf{R}\pi_*\F \to \bf{R}\pi_{K, *}\bf{L}\sp^*_{\X'} \F  
\]
for any $\F \in \bf{D}(\X')$. In our notation, this can be rewritten as
\[
\left(\bf{R}\pi_*\F\right)_K \to \bf{R}\pi_{K, *}\left( \F_K\right).
\]
Similarly, the map $\left(\bf{R}\pi_*\F\right)_K \to \bf{R}\pi_{K, *}\left( \F_K\right)$ can be defined for any $\F\in \bf{D}(\X')^a$.

\begin{lemma1}\label{generic-fiber-commutes} Suppose that $\pi\colon \X' \to \X$ is a proper map. Then the base change morphism
\[
\left(\bf{R}\pi_*\F\right)_K \to \bf{R}\pi_{K, *}\left( \F_K\right)
\]
is an isomorphism for any $\F\in \bf{D}^+_{acoh}(\X')^a$ or $\F\in \bf{D}^+_{acoh}(\X')$. 
\end{lemma1}
\begin{proof}
The claim is local on $\X$, so we may and do assume that $\X=\Spf A$ is affine. We use the functor $(-)_!\colon \bf{D}^+_{acoh}(\X')^a \to \bf{D}^+_{qc, acoh}(\X')$ and Lemma~\ref{lemma:kill-torsion} to reduce to the case $\F\in \bf{D}^+_{qc, acoh}(\X')$. Finally, we use the spectral sequences 
\[
\rm{E}^{p,q}_2=\rm{R}^p\pi_*(\mathcal H^q(\F)) \Rightarrow \rm{R}^{p+q}\pi_*\,\F,
\]
\[
\rm{E}'^{p,q}_2=\rm{R}^p\pi_{K, *}(\mathcal H^q(\F_K)) \Rightarrow \rm{R}^{p+q}\pi_{K, *}(\F_K)
\]
and exactness of $(-)_K$ to reduce to the case of an adically quasi-coherent, almost coherent $\O_{\X'}$-module $\F$. \smallskip

So far, we have reduced the question to showing that the natural morphisms
\[
(\rm{R}^i\pi_* \F)_K \to \rm{R}^i\pi_{K, *}(\F_K)
\]
are isomorphisms for any adically quasi-coherent, almost coherent $\F$ and $i\geq 0$. Now \cite[Theorem 5.1.6]{Z3} implies that $\rm{H}^i(\X', \F)$ is an almost coherent $A$-module, and the natural morphism
\[
\rm{H}^i(\X', \F)^\Updelta \to \rm{R}^i\pi_*\F
\]
is an isomorphism for any $i\geq 0$. Therefore, Lemma~\ref{lemma:generic-fiber-almost-coherent} implies that the sheaves $(\rm{R}^i\pi_*\F)_K$ are coherent and canonically isomorphic to $\widetilde{\rm{H}^i(\X', \F)[1/\varpi]}$. Likewise, Lemma~\ref{lemma:proper-adic-formal} and \cite[Theorem 7.5.19]{FujKato} (and its proof) guarantee that $\rm{H}^i(X', \F_K)$ is a coherent $A$-module, and the natural morphism
\[
\widetilde{\rm{H}^i(X', \F_K)} \to \rm{R}^i\pi_{K, *}\F_K
\]
is an isomorphism for any $i\geq 0$. Therefore, the question is reduced to showing that the natural map
\[
\rm{H}^i(\X', \F)[1/\varpi] \to \rm{H}^i(X', \F_K)
\]
is an isomorphism for any adically quasi-coherent, almost coherent $\O_\X$-module $\F$ and any $i\geq 0$. \smallskip

We choose a finite covering $\X' = \cup_{j=1}^N \sU_j$ of $\X'$ by open affines $\sU_j$. Then \cite[Theorem I.7.1.1]{FujKato} implies that cohomology groups of $\F$ can be computed using the \v{C}ech complex $\check{C}^{\bullet}(\mathfrak U, \F)$ associated with $\F$. \smallskip

The complex $\check{C}^{\bullet}(\mathfrak U_K, \F_K)$ is equal to $\check{C}^{\bullet}(\mathfrak U, \F)[1/\varpi]$ by Lemma~\ref{lemma:generic-fiber-almost-coherent}. Since coherent sheaves do not have higher cohomology groups on affinoid spaces, we conclude that 
\[
\rm{H}^i(\X', \F)[1/\varpi]=\rm{H}^i\left(\check{C}^{\bullet}(\mathfrak U, \F)[1/\varpi]\right)=\rm{H}^i\left(\check{C}^{\bullet}(\mathfrak U_K, \F_K)\right)=\rm{H}^i(\X'_K, \F_K)
\]
for $i\geq 0$. This finishes the argument. 
\end{proof}

Now we assume that $K$ is of mixed characteristic $(0, p)$, and $\X$ is an admissible formal $\O_K$-model with smooth generic fiber $X=\X_K$. We consider morphisms of ringed topoi 
\begin{align*}
\mu\colon (X_\proet, \wdh{\O}^+_X) \to (X_{\rm{an}}, \O^+_X), \\
\nu\colon (X_\proet, \wdh{\O}^+_X) \to (\X, \O_\X), \\
\sp_\X\colon (X_{\rm{an}}, \O_X) \to (\X, \O_\X),
\end{align*}
and the integral version
\[
\sp^+_\X\colon (X_{\rm{an}}, \O^+_X) \to (\X, \O_\X). 
\]

\begin{lemma1}\label{lemma:integral-rational-pullback} Let $X$ be a smooth rigid $K$-space with an admissible formal model $\X$. Then the natural maps 
\[
\left(\bf{R}\nu_*\wdh{\O}^+_X\right)_K \to \left(\bf{R}\nu_*\wdh{\O}_X\right)_K \text{ and } \left(\bf{R}\nu_*\wdh{\O}_X\right)_K \to \bf{R}\mu_*\wdh{\O}_X
\]
are isomorphisms.
\end{lemma1}
\begin{proof}

We recall that $\wdh{\O}_X\simeq \wdh{\O}^+_X[1/p]$, so the cokernel $\wdh{\O}_X^+ \to \wdh{\O}_X$ is $p^{\infty}$-torsion. Thus, the cone of the map $\bf{R}\nu_*\wdh{\O}^+_X \to \bf{R}\nu_*\wdh{\O}_X$ has $p^{\infty}$-torsion cohomology sheaves. Therefore, Lemma~\ref{lemma:kill-torsion} implies that $(\bf{R}\nu_*\wdh{\O}^+_X)_K \to (\bf{R}\nu_*\wdh{\O}_X)_K$ is an isomorphism. 

The map $(\bf{R}\nu_*\wdh{\O}_X)_K \to \bf{R}\mu_*\wdh{\O}_X$ comes as the adjunction
\[
\eta \colon \left(\bf{R}\sp_{\X, *}(\bf{R}\mu_* \wdh{\O}_X)\right)_K=\bf{L}\sp^*_{\X}\bf{R}\sp_{\X, *}(\bf{R}\mu_* \wdh{\O}_X) \to \bf{R}\mu_*\wdh{\O}_X
\]
We note that the complex $\bf{R}\mu_*\wdh{\O}_X \in \bf{D}^b_{coh}(X)$ by \cite[Proposition 3.23]{Schsurvey}, so Lemma~\ref{coh-pullback} implies that $\eta$ is an isomorphism. 
\end{proof}

\section{Pro-\'Etale Trace Maps}\label{trace-appendix}

For the rest of the section, we fix a complete, algebraically closed rank-$1$ valued field $C$ of mixed characteristic $(0, p)$. We also fix a finite \'etale morphism $f\colon X' \to X$ of rigid-analytic $C$-spaces. \smallskip

The main goal for this section is to construct trace maps associated with such $f$. \smallskip

We note that a finite map $f$ is proper, so \cite[Definition 5.2.1 and Propositions 5.2.4]{H3} imply that $f_*=f_!\colon \cal{A}b(X'_\et) \to \cal{A}b(X_\et)$. Therefore, there is an adjunction $(f_*, f^{-1})$ for any finite \'etale $f$. 

\begin{defn1} We define the {\it \'etale trace map} 
\[
\rm{Tr}_{\et, f, \F} \colon f_*f^{-1}\F \to \F
\]
for any $\F \in \cal{A}b(X_\et)$ to be the counit of the adjunction $(f_*, f^{-1})$.
\end{defn1}

We will be particularly interested in the \'etale traces
\[
\rm{Tr}_{\et, f, \O_X^+}\colon f_*\,\O_{X'}^+ \to \O_X^+, \text{ and }
\]
\[
\rm{Tr}_{\et, f, \O_X^+/p^n}\colon f_* \left(\O_{X'}^+/p\right) \to \O_X^+/p.
\]
\begin{rmk1} If there is no ambiguity, we will often denote any of these maps just by $\rm{Tr}_{\et, f}$.
\end{rmk1}

Now we generalize the construction of the trace map to the pro-\'etale sheaves $\wdh{\O}^+$ and $\wdh{\O}$. We recall that we have a commutative diagram
\[
\begin{tikzcd}
 \left(X'_\proet, \wdh{\O}^+_{X'}\right) \arrow[rr, bend left, "\mu_{X'}"] \arrow{d}{f_{\proet}} \arrow{r}{\lambda_{X'}} &  \left(X'_\et, \O^+_{X'}\right) \arrow{d}{f_{\et}} \arrow{r} &  \left(X'_{\text{an}}, \O^+_{X'}\right) \arrow{d}{f_{\an}}\\
 \left(X_\proet, \wdh{\O}^+_{X}\right) \arrow[rr, bend right, "\mu_{X}"] \arrow{r}{\lambda_{X}} & \left(X_\et, \O^+_{X}\right) \arrow{r} & \left(X_{\text{an}}, \O^+_{X}\right),
\end{tikzcd}
\]
and \cite[Corollary 3.17(ii)]{Sch1} guarantees that $f_{\proet, *}\left(\O_{X'}^+/p^n\right) \simeq \lambda_X^{-1}(f_{\et, *}\O_{X'}^+/p^n)$.

\begin{defn1}\label{defn:trace-mod-p-proetale} We define the {\it (mod-$p^n$) pro-\'etale trace map} 
\[
\rm{Tr}_{\proet, f, \O_X^+/p^n}\colon f_{\proet, *}\left(\O_{X'}^+/p^n\right) \to \O_X^+/p^n
\]
as $\rm{Tr}_{\proet, f, \O_X^+/p^n} \coloneqq \lambda_X^{-1}\left(\rm{Tr}_{\et, f, \O_X^+/p^n}\right)$. If there is no ambiguity, we will often denote $\rm{Tr}_{\proet, f, \O_X^+/p^n}$ just by $\rm{Tr}_{\proet, f}$.
\end{defn1}

\begin{lemma1}\label{vse-v-deg-0} Let $f\colon X'\to X$ be a finite \'etale morphism of rigid-analytic $C$-spaces. Then the functor $f_{\proet, *}(-)$ is exact, i.e. $\bf{R}f_{\proet, *}\F$ is concentrated in degree $0$ for any $\F\in \cal{A}b(X'_\proet)$. 
\end{lemma1}
\begin{proof}
The claim is (pro-)\'etale local on $X$, so we can assume that $X' \to X$ is a split finite \'etale morphism. In this case, the claim is trivial. 
\end{proof}

We recall that the sheaf $\wdh{\O}^+_X$ is defined as $\lim_n \O_{X}^+/p^n$ on $X_\proet$, and $\wdh{\O}_X \coloneqq \wdh{\O}_X^+[1/p]$. Therefore, Lemma~\ref{vse-v-deg-0} implies that 
\[
\bf{R}f_{\proet, *}\,\wdh{\O}^+_{X'} \simeq f_{\proet, *}\,\wdh{\O}^+_{X'} \simeq \lim_n f_{\proet, *}\left(\O_{X'}^+/p^n\right), \text{ and }
\]
\[
\bf{R}f_{\proet, *}\,\wdh{\O}_{X'} \simeq f_{\proet, *}\,\wdh{\O}_{X'} \simeq\footnote{Here, we use that the pro-\'etale site of an affinoid is coherent by \cite[Proposition 3.12]{Sch1}, so cohomology commute with filtered colimits.} \left(f_{\proet, *}\wdh{\O}_{X'}^+\right)[1/p].
\]

\begin{defn1} We define the {\it (integral) pro-\'etale trace} 
\[
\rm{Tr}^+_{\proet, f}\colon  f_{\proet, *}\,\wdh{\O}^+_{X'} \to  \wdh{\O}^+_{X}
\]
as $\rm{Tr}^+_{\proet,  f} \coloneqq \lim_n \rm{Tr}_{\proet,  f,  \O^+_X/p^n}$. 
\end{defn1}

\begin{defn1} We define the {\it (rational) pro-\'etale trace} 
\[
\rm{Tr}_{\proet, f}\colon  f_{\proet, *}\,\wdh{\O}_{X'} \to  \wdh{\O}_{X}
\]
as $\rm{Tr}_{\proet,  f} \coloneqq \rm{Tr}^+_{\proet, f}[1/p]$. 
\end{defn1}

\begin{defn1}\label{defn:analytic-trace} We also define {\it pro-\'etale traces}
\[
\rm{Tr}_{\rm{an}, f}\colon \bf{R}f_{\rm{an}, *} \left(\bf{R}\mu_{X', *}\,\wdh{\O}_{X'}\right) \to \bf{R}\mu_{X, *}\,\wdh{\O}_{X}
\]
as 
\[
\rm{Tr}_{\rm{an}, f} \coloneqq \bf{R}\mu_{X, *}\left(\rm{Tr}_{\proet, f}\right)\colon \bf{R}\mu_{X, *}\left(f_{\proet,*} \wdh{\O}_{X'}\right) \to \bf{R}\mu_{X, *}\,\wdh{\O}_X.
\]
\end{defn1}

Now suppose that $f\colon X' \to X$ comes as the generic fiber of a morphism $\mf\colon \X' \to \X$ between admissible formal $\O_C$-models. Then we have a commutative diagram
\[
\begin{tikzcd}
(X'_\proet, \wdh{\O}_{X'}^+) \arrow{d}{f_\proet} \arrow{r}{\nu_{\X', *}} & (\X'_{\rm{Zar}}, \O_{\X'}) \arrow{d}{\mf} \\
(X_\proet, \wdh{\O}_X^+) \arrow{r}{\nu_{\X, *}} & (\X_{\rm{Zar}}, \O_\X).
\end{tikzcd}
\]

\begin{defn1}\label{defn:integral-proetale-trace} We define the {\it pro-\'etale trace}
\[
\rm{Tr}^+_{\rm{Zar}, \mf} \colon \bf{R}\mf_*\left(\bf{R}\nu_{\X', *}\wdh{\O}_{X'}^+\right) \to \bf{R}\nu_{\X, *}\wdh{\O}_{X}^+
\]
as
\[
\rm{Tr}^+_{\rm{Zar}, \mf} \coloneqq \bf{R}\nu_{\X, *}\left(\rm{Tr}^+_{\proet, f}\right)\colon \bf{R}\nu_{\X, *}\left(f_{\proet,*} \wdh{\O}^+_{X'}\right) \to \bf{R}\nu_{\X, *}\left(\wdh{\O}^+_X\right).
\]
\end{defn1}

\begin{defn1}\label{defn:mod-p-proetale-trace} Similarly, we define the {\it pro-\'etale trace}
\[
\rm{Tr}_{\rm{Zar}, \mf} \colon \bf{R}\mf_{0, *}\bf{R}\nu_{\X', *}\left(\O_{X'}^+/p\right) \to \bf{R}\nu_{\X, *}\left(\O_{X}^+/p\right)
\]
as
\[
\rm{Tr}_{\rm{Zar}, \mf} \coloneqq \bf{R}\nu_{\X, *}\left(\rm{Tr}_{\proet, f, \O^+_X/p}\right)\colon \bf{R}\nu_{\X, *}\left(f_{\proet,*} \O^+_{X'}/p\right) \to \bf{R}\nu_{\X, *}\left(\O^+_X/p\right).
\]
\end{defn1}

\begin{rmk1}\label{rmk:etale-pushforward-trace} Using \cite[Corollary 3.17(i)]{Sch1}, one can alternatively write $\rm{Tr}_{\rm{Zar}, \mf}$ as
\[
\bf{R}t_{\X, *}\left(\rm{Tr}_{\et, f, \O_X^+/p^n}\right)\colon \bf{R}\mf_{0, *}\bf{R}\nu_{\X', *}\left(\O_{X'}^+/p\right) \simeq \bf{R}t_{\X, *} \bf{R}f_{\et, *} \left(\O_{X'}^+/p^n\right) \to \bf{R}t_{\X, *}\left(\O_X^+/p\right) \simeq \bf{R}\nu_{\X, *} \left(\O_X^+/p^n\right),
\]
where $t\colon (X_\et, \O_X^+/p) \to (\X_{0, \rm{Zar}}, \O_{\X_0})$ is the natural morphism of topoi. 
\end{rmk1}

\begin{lemma1}\label{zar-an-trace} Let $\mf \colon \X' \to \X$ be a morphism of admissible formal $\O_C$-schemes such that its generic fiber $f\colon X' \to X$ is finite and \'etale. Then the following diagram
\[
\begin{tikzcd}[column sep = 5em]
\left(\bf{R}\mf_*\circ \bf{R}\nu_{\X', *} \wdh{\O}^+_{X'}\right)_C\arrow{r} \arrow{d}{\left(\rm{Tr}^+_{\rm{Zar}, \mathfrak f} \right)_C}& \bf{R}f_{\rm{an}, *}\left(\bf{R}\nu_{\X',  *}\wdh{\O}^+_{X'}\right)_C \arrow{r}{\bf{R}f_{\rm{an}, *}\left(\delta_{\X'}\right)} & \bf{R}f_{\rm{an}, *}\circ \bf{R}\mu_{X',*} \wdh{\O}_X \arrow{d}{\rm{Tr}_{\rm{an}, f}}\\
\left(\bf{R}\nu_{\X, *} \wdh{\O}^+_X\right)_C \arrow{rr}{\delta_{\X}} & & \bf{R}\mu_{X, *}\wdh{\O}_X
\end{tikzcd}
\]
is commutative. Furthermore, the top left arrow is an isomorphism.
\end{lemma1}
\begin{proof}
The top left arrow is an isomorphism by Lemma~\ref{generic-fiber-commutes} and almost coherence of $\bf{R}\nu_{\X', *}\,\wdh{\O}_{X'}^+$ (see \cite[Theorem 6.13.6]{Z3}). Commutativity of the diagram is easy but tedious and left to the reader.
\end{proof}

Our next goal is to get an explicit formula for the pro-\'etale trace 
\[
 \rm{Tr}_{\rm{an}, f}\colon \bf{R}f_{\rm{an}, *} \left(\bf{R}\mu_{X', *}\,\wdh{\O}_{X'}\right) \to \bf{R}\mu_{X, *}\wdh{\O}_X
\]
for a finite \'etale morphism $f\colon X' \to X$ of smooth rigid-analytic $C$-varieties.
We start with the preliminary lemma: 

\begin{lemma1}\label{trace-etale-trace-explicit-2} Let $f\colon X'=\Spa(B, B^+) \to X=\Spa(A, A^+)$ be a finite \'etale morphism of affinoid rigid-analytic $C$-spaces, $U\to X$ a pro-\'etale morphism with $\wdh{U}=\Spa(R, R^+)$ being an affinoid perfectoid space, and $U'=U\times_X X'$. Then $U'$ is an affinoid perfectoid object with $\wdh{U}'\simeq \Spa(S, S^+)$ with $S=R\otimes_A B$, and the pro-\'etale trace map 
\[
\rm{Tr}_{\proet, f}(U)\colon S=\left(f_{\proet, *} \wdh{\O}_{X'}\right)(U) \to R=\wdh{\O}_{X}(U)
\]
coincides with the explicit trace map
\[
\rm{Tr}_{S/R}\colon S \to R.
\]
\end{lemma1}
\begin{proof}
Without loss of generality we can assume that $X$ is connected. In particular, $f$ is surjective in this case (or empty). The empty case is obvious, so we assume that $f$ is surjective. \smallskip

We note that \cite[Lemma 4.5]{Sch1} implies that $U'$ is an affinoid perfectoid object, and that $\wdh{U'}= \wdh{U}\times_X X'=\Spa(S, S^+)$ for $S\cong R\wdh{\otimes}_A B$. Furthermore, the proof shows that $R\wdh{\otimes}_A B\simeq R\otimes_A B$, i.e., $R\otimes_A B$ is already complete. \smallskip

Furthermore, {\it loc.\,cit.\,} implies that, for any finite \'etale morphism $V \to U$ with $\wdh{V}=\Spa(T, T^+)$, the fiber product $V'\coloneqq U'\times_U V$ is an affinoid perfectoid space with the associated affinoid perfectoid space $\wdh{V}'=\Spa\left(T\otimes_R S, (T\otimes_R S)^+\right)$. Therefore, we can check that $\rm{Tr}_{\proet, f}(U)$ is equal to $\rm{Tr}_{S/R}$ locally in the finite \'etale topology on the $U$. In particular, we can assume that $f$ is a split torsor, in which case the claim is obvious. 
\end{proof}

In what follows, we assume that $X=\Spa(A, A^+)$ is an affinoid adic space that admits an \'etale morphism $g\colon X \to \bf{T}^d_C$ that is a composition of rational embeddings and finite \'etale morphism. In this case, $X'=\Spa(B, B^+)$ is also an affinoid adic space, and $g\circ f\colon X' \to \bf{T}^d$ is an \'etale morphism with the same properties. \smallskip

Notation~\ref{notation:perfectoid-covering} defines a commutative diagram
\[
\begin{tikzcd}
X'_\infty \coloneqq \Spa(B_\infty, B_\infty^+) \arrow{d} \arrow{r} & X_\infty\coloneqq \Spa(A_\infty, A_\infty^+) \arrow{d} \arrow{r} & \bf{T}^d_\infty \arrow{d} \\
X'=\Spa(B, B^+) \arrow{r} & X=\Spa(A, A^+) \arrow{r} & \bf{T}^d
\end{tikzcd}
\]
such that both squares are Cartesian and vertical arrows are (surjective) pro-\'etale $\Gamma=\bf{Z}_p(1)^{d}$-torsors. 

\begin{lemma1}\label{lemma:explicit-trace-etale} In the situation as above, let $\rm{Tr}_{B/A}\colon B \to A$ and $\rm{Tr}_{B_\infty/A_\infty}$ be finite locally free trace morphisms. Then the following diagram 
\[
\begin{tikzcd}[column sep = 6em]
\rm{H}^i_{\rm{cont}}(\Gamma, B) \arrow{r}{ \rm{H}^i_{\rm{cont}}\left(\Gamma, \rm{Tr}_{B/A}\right)}  \arrow{d} & \rm{H}^i_{\rm{cont}}(\Gamma, A) \arrow{d} \\
\rm{H}^i_{\rm{cont}}(\Gamma, B_\infty) \arrow{r}{ \rm{H}^i_{\rm{cont}}\left(\Gamma, \rm{Tr}_{B_{\infty}/A_{\infty}}\right)} \arrow{d} & \rm{H}^i_{\rm{cont}}(\Gamma, A_\infty) \arrow{d} \\
\rm{H}^i(X'_{\proet}, \wdh{\O}_{X'})= \rm{H}^i(X_{\proet}, f_{\proet, *}\wdh{\O}_{X'}) \arrow{r}{ \rm{H}^i(X_\proet, \rm{Tr}_{\proet, f})} \arrow{d} & \rm{H}^i(X_{\proet}, \wdh{\O}_{X})  \arrow{d}\\
\rm{H}^0(X_{\rm{an}}, f_{\rm{an}, *}\rm{R}^i\mu_{X', *}\wdh{\O}_{X'}) \arrow{r}{\rm{H}^0\left(\mathcal{H}^i(\rm{Tr}_{\rm{an}, f})\right)} & \rm{H}^0(X_{\rm{an}}, \rm{R}^i\mu_{X, *}\wdh{\O}_{X})\\
\end{tikzcd}
\]  
is commutative with vertical arrows being isomorphisms. 
\end{lemma1}
\begin{proof}
The discussion in Notation~\ref{notation:perfectoid-covering} implies that the top and middle vertical arrows are isomorphisms. We see that the left vertical arrow is an isomorphism by passing to $\rm{H}^i(-)$ in the following sequence of isomorphisms
\[
\bf{R}\Gamma(X'_{\proet}, \wdh{\O}_{X'}) \simeq \bf{R}\Gamma(X'_{\rm{an}}, \bf{R}\mu_{X', *} \wdh{\O}_{X'}) \simeq \bf{R}\Gamma(X_{\rm{an}}, f_*\bf{R}\mu_{X', *} \wdh{\O}_{X'}),
\]
where the last isomorphism uses that $f$ is finite and $\bf{R}\mu_{X', *} \wdh{\O}_{X'}$ has coherent cohomology sheaves (see Remark~\ref{rmk:coherent-proetale-etale}). A similar proof shows that the right vertical map is an isomorphism. \smallskip 

Commutativity of the top and bottom squares is obvious. The middle square commutes due to functoriality of the Cartan-Leray spectral sequence (for the pro-\'etale covers $\Spa(A_\infty, A_\infty^+) \to \Spa(A, A^+)$ and $\Spa(B_\infty^+, B_\infty) \to \Spa(B, B^+)$) and Lemma~\ref{trace-etale-trace-explicit-2}.
\end{proof}

\printbibliography

\end{document}